\documentclass{article}
\usepackage[a4paper, total={5.5in, 8in}]{geometry}
\usepackage[utf8]{inputenc}
\usepackage{amsmath,amssymb,amsthm, xcolor, enumitem, comment, todonotes}
\usepackage{url}
\usepackage[all]{xy}

\usepackage[utf8]{inputenc}
\usepackage{amsmath,amssymb,amsthm, xcolor, enumitem, comment, todonotes}
\usepackage{url}
\usepackage[all]{xy}
\usepackage[english]{babel}

\newtheorem{definition}{Definition}[section]

\newtheorem{example}[definition]{Example}

\newtheorem{notation}[definition]{Notation}
\newtheorem{remark}[definition]{Remark}

\newtheorem{theorem}{Theorem}[section]
\newtheorem{claim}[definition]{Claim}
\newtheorem{subclaim}[definition]{Subclaim}
\newtheorem{proposition}[definition]{Proposition}
\newtheorem{lemma}[definition]{Lemma}
\newtheorem{corollary}[definition]{Corollary}

\newcommand{\po}{\mathbb{P}}
\newcommand{\qo}{\mathbb{Q}}
\newcommand{\la}{\langle}
\newcommand{\ra}{\rangle}
\newcommand{\fr}{{}^{\frown}}
\newcommand{\name}{\dot}
\newcommand{\elem}{\prec}
\newcommand{\uhr}{\upharpoonright}

\newcommand{\power}{\mathcal{P}}

\renewcommand{\S}{\mathcal{S}}
\newcommand{\E}{\bar{E}}
\newcommand{\mo}{\mathbb{M}}
\newcommand{\co}{\mathbb{C}}
\newcommand{\ro}{\mathbb{R}}

\newcommand{\C}{\mathcal{C}}

\DeclareMathOperator{\dom}{dom}

\DeclareMathOperator{\cf}{cof}
\DeclareMathOperator{\cp}{cp}
\DeclareMathOperator{\Ult}{ult}

\DeclareMathOperator{\supp}{supp}

\DeclareMathOperator{\otp}{otp}

\newcommand{\MU}{M}
\newcommand{\jU}{j}
\newcommand{\NU}{N}
\newcommand{\iU}{i}
\newcommand{\kappaU}{\kappa}
\newcommand{\Mx}{M^*}
\newcommand{\jx}{j^*}
\newcommand{\kappax}{\kappa^*}
\newcommand{\Ex}{E^*}
\newcommand{\hqo}{\hat{\qo}}
\newcommand{\bqo}{\overline{\qo}}

\title{A Kunen-Like Model with a Critical Failure of the Continuum Hypothesis}
\author{Omer Ben Neria\footnote{The first author would like to thank the Israel Science Foundation (Grants 1832/19 and 1302/23) for their support.}\\ Hebrew University\\ Jerusalem, Israel 
\and Eyal Kaplan\\ UC Berkeley\\ Berkeley, CA}
\date{\today}

\begin{document}

\maketitle
\begin{abstract}
We construct a model of the form  $L[A,U]$  that exhibits the simplest structural behavior of  $\sigma$-complete ultrafilters in a model of set theory with a single measurable cardinal  $\kappa$ , yet satisfies  $2^\kappa = \kappa^{++}$. This result establishes a limitation on the extent to which structural properties of ultrafilters can determine the cardinal arithmetic at large cardinals, and answers a question posed by Goldberg concerning the failure of the Continuum Hypothesis at a measurable cardinal in a model of the Ultrapower Axiom. 
The construction introduces several methods in extensions of embeddings theory and fine-structure-based forcing, designed to control the behavior of non-normal ultrafilters in generic extensions.
\end{abstract}

\section{Introduction}
A hallmark of large cardinal properties in set theory is their capacity to determine fundamental aspects of the set-theoretic universe.
For example: 
\begin{enumerate}
    \item In cardinal arithmetic, Solovay (\cite{solovay1974SCHabovestronglycompact}) proved that the Singular Cardinal Hypothesis (SCH) holds above a strongly compact cardinal $\kappa$. 

    \item In the study of compactness in set theory, various large cardinal principles have been shown to secure (or are equivalent to) various compactness principles. In particular, incompactness combinatorial principles such as squares, nonreflecting stationary sets, or Aronszajn trees have been shown to fail around or above large cardinals (e.g., \cite{MR520190},\cite{KanamoriBook},\cite{CumForMag},\cite{10.1305/ndjfl/1125409326} provide a comprehensive account to the subject).

\end{enumerate}

The effect of large cardinal hypotheses on other key properties is  amplified under further structural assumption regarding the behavior of large cardinal ultrafilters. These are typically given in either terms of  properties of ultrafilters,
or in terms of the connections between them. 
An example of the first is the normality assumption that separates ultrafilters witnessing supercompactness from those witnessing strong compactness. 
An example of the second is the linearity of the Mitchell order 
$\triangleleft$
on $\sigma$-complete ultrafilters.

A remarkable demonstration of the extent structural assumptions about ultrafilters shed light on other key properties is Goldberg's work about the Ultrapower Axiom  \cite{GoldbergUABook}. The Ultrapower Axiom (UA) states that for every two $\sigma$-complete ultrafilters $U,W$ with associated ultrapower maps $j_U : V \to M_U$, $j_W : V \to M_W$, there are $\sigma$-complete ultrafilters $W' \in M_U$ and $U' \in M_W$ such that the internal ultrapowers $\Ult(M_U,W')$ of $M_U$ and $\Ult(M_{W},U')$ of $M_W$ are isomorphic to the same transitive class $P$, and the associated ultrapower maps  $k_{U'} : M_{W} \to P$ and $k_{W'} : M_{U} \to P$ generate a commutative system with $j_W,j_U$ respectively, i.e., $k_{U'} \circ j_{W} = k_{W'} \circ j_U.$
It is shown in \cite{GoldbergUABook} that under the Ultrapower Axiom every strongly compact cardinal is either supercompact or a measurable limit of supercompacts, and that the Generalized Continuum hypothesis (GCH) holds at and above a supercompact cardinal $\kappa$.\footnote{I.e., $2^\lambda =\lambda^+$ for every cardinal $\lambda \geq \kappa$.}
The prospects of these result are promising in light of a recent Theorem of Woodin, showing that for every $A \subseteq \mathbb{R}$ such that $L(A,\mathbb{R}) \models AD^+$, $HOD^{L(A,\mathbb{R})} \models UA$. The last implies that the axiom $V = \text{Ultimate}-L$ implies $UA$.\\

These results involving the Ultrapower Axiom can be viewed as extensions of fundamental results in inner model theory. Indeed, all known fine-structural based extender models $L[\vec{E}]$ satisfy the UA as well as many other strong properties such as the GCH. For example, by results of Kunen and Silver on the structure of a model of the form $L[U]$, where $U$ is normal\footnote{An ultrafilter $U$ on $\kappa$ is normal if for every $A\in U$ and $f\colon \kappa\to \kappa$ such that for every $\alpha\in A$, $f(\alpha)<\alpha$, there exists $B\in U$, $B\subseteq A$ and $\gamma<\kappa$ such that for every $\alpha\in B$, $f(\alpha) = \gamma$.} $\kappa$-complete ultrafilter on a measurable cardinal $\kappa$, nowadays known as a ``\emph{Kunen model}", $L[U]$ satisfies 
 \begin{itemize}
     \item The General Continuum Hypothesis (\cite{Silver-KunenModel}).

     \item Every $\sigma$-complete ultrafilter $W \in L[U]$ is isomorphic to a finite power $U^n$ of $U$ (\cite{KunenModel}). In particular, UA holds. 
 \end{itemize} 

 The results from UA + a supercompact cardinal $\kappa$, and the broader indications given by inner model theory, suggest that the ability to derive instances of GCH from structural properties of  ultrafilters can go beyond UA and a supercompact cardinal, and can possibly extend to ``local" versions in which structural assumptions such as UA around a measurable or a strong cardinal $\kappa$ could yield local instances of GCH around $\kappa$. 

 The main result of this paper provides a limitation to this perspective by constructing a ``Kunen-like" model $L[A,U]$ that satisfies the simplest possible (nontrivial) behavior of $\sigma$-complete ultrafilter, yet fails to satisfy GCH at its unique measurable cardinal $\kappa$.

 \begin{theorem}\label{Thm:Main}
 There exists a ``Kunen-Like" model of set theory, of the form $L[A,U]$ satisfying 
 \begin{enumerate}
 \item There is a unique measurable cardinal $\kappa$ and a unique $\kappa$-complete normal ultrafilter $U$
 \item Every $\sigma$-complete ultrafilter is isomorphic to a finite power $U^n$ of $U$
     \item $2^\kappa = \kappa^{++}$
 \end{enumerate} \end{theorem}

The Large cardinal consistency assumption used in our construction is of a cardinal $\kappa$ that is $(\kappa+2)$-strong.

Theorem \ref{Thm:Main} also provides the first instance of a forcing over a model of UA, which adds an unbounded subset to a measurable cardinal $\kappa$, and preserves its measurablity and UA.

In order to prove the main result, we make several contributions to both the study of the possible behavior of $\sigma$-complete ultrafilters and especially to non-normal ultrafilters,  and to the study of forcing iteration theory.
Concerning the former, the possible behavior of normal ultrafilters goes back to problems from the 1970s regarding the possible number of normal ultrafilters on a measurable cardinal $\kappa$, and the possible structure of the Mitchell order on them.
Questions about the number of normal measures and their ordering are closely related to some of the earliest developments in forcing theory and inner model theory, relating to large cardinals. For example, in forcing theory, the work of Kunen and Paris \cite{MR277381} on elementary embeddings in forcing extensions, constructs a model with the maximal possible number of normal ultrafilters on a single measurable cardinal. In inner model theory, the work of Mitchell \cite{Mitchell1974innermodelsforcoherentsequences,MR716621} on coherent sequences of normal ultrafilters $\vec{U}$ and their associated canonical inner models $L[\vec{U}]$, constructs models with any number $\lambda \in [1,\kappa^{++}]$ of normal ultrafilters on a measurable cardinal $\kappa$. 
Over the years new ideas and methods in forcing theory and inner model theory have been developed to tackle these questions. 
Results by Baldwin \cite{MR820124}, Apter–Cummings–Hamkins \cite{MR2299507}, Leaning \cite{MR3274970}, and Friedman-Magidor \cite{FriMag09} regarding the possible number of normal ultrafilters, and by Cummings \cite{MR1257466},\cite{MR1312304}, Witzany \cite{MR1286010}, and the first author \cite{MR3544708},\cite{MR3397347}, on the possible behavior of the Mitchell order.\\
The work of Friedman and Magidor \cite{FriMag09} settled the question regarding possible number normal ultrafilters by introducing a forcing method that can construct models with any number $\lambda \in [1,\kappa^{++}]$ of normal ultrafilters, from the minimal large cardinal assumption. In addition \cite{FriMag09} constructs a model with a single measurable cardinal $\kappa$ with a single normal ultrafilter $U$, while $2^\kappa = \kappa^{++}$. The Friedman-Magidor model is the starting point to Theorem \ref{Thm:Main}, as it satisfies the portion of the ``Kunen-like" model regarding normal ultrafilters. Nevertheless, it does not satisfy that every $\sigma$-complete ultrafilter on $\kappa$ is isomorphic to a finite power $U^n$ of the unique normal ultrafilter $U$.
The missing part of the relevant theory is to control the behavior of non-normal ultrafilters on a measurable cardinals. Section \ref{Section-BluePrint} provides a blueprint for addressing the missing part. It introduces abstract properties, referred to as "blueprints", of forcing notions and iterated forcings (i.e., see Definitions \ref{Def:FM-blueprint}, \ref{Def:KLblueprint}) and shows that generic extensions of suitable canonical inner models $V = L[\E]$ by a poset $\po$ satisfying the relevant blueprint are ``Kunen-Like" (see Theorem \ref{Thm:KLblueprint}).

The rest of the paper is devoted to constructing a poset $\po$ which satisfies the "Kunen-Like" blueprint assumptions. 
The first main ingredient of the poset construction is the iteration theory of generalized Miller forcings, introduced by Friedman and Zdomskyy \cite{FriedmanZdomskyy2010}. A variation of their forcing that will be used is described in Section \ref{Section:GeneralMiller}. 
The second main ingredient is the use of the fine-structure theory of the ground model $V = L[\E]$ for the construction of $\po$. The use of fine-structure is key in successfully fulfilling one of the main blueprint assumptions of $\po$, which requires it to incorporate a  nonstationary support iteration of coding posets. In addition, each coding poset $\qo_\alpha$ destroys the stationarity of various non-reflecting stationary sets $S_\alpha$ that consist of low-cofinality ordinals (i.e., cofinality $\omega_1$). It is well-known that each single instance $\qo_\alpha$ of such a poset is sufficiently distributive, and, therefore, does not collapse cardinals. The problem resides in the iteration of such posets, which in general can collapse cardinals. 
Our solution builds on the fine-structure theory of $L[\E]$, and specifically on the construction methods of squares and morasses. This approach was motivated by a recent work of Foreman, Magidor, and Zeman, that is outlined in \cite{ForMagZem} and expected to appear in details in a sequel work. The iteration constructed in  \cite{ForMagZem}  is based on an Easton support iteration of coding posets. Our construction extends this method to large (nonstationary) support iterations and in establishing a strong form of distributivity  given in terms of a certain fusion property (definition \ref{Def:BlueprintPrelim} (2)). The fine-structure background is given in Section \ref{Section:FineStucture}, and its application to the main poset construction is given in Section \ref{Section:FinalIteration}.\\

\noindent Our notations for ultrafilters, ultrapowers, fine-structure, and forcing, are mostly standard. The most notable exception is our use of the Jerusalem forcing convention, according to which a condition $p$ extends (i.e., is more informative than) a condition $q$ is denoted by $p \geq q$, and the weakest condition in a poset $\po$ is denoted by $0_{\po}$.

\section{Blueprints for controlling $\kappa$-complete Ultrafilters in Generic Extensions}\label{Section-BluePrint}

The goal of this section is to provide two blueprints for posets $\po$  that violate GCH at a measurable cardinal $\kappa$ while maintaining certain control over the measures on $\kappa$ (i.e., $\kappa$-complete ultrafilters on $\kappa$ in the generic extension).

The first blueprint will be an abstraction of the Friedman Magidor construction that yields a generic extension in which $\kappa$ is the only measurable cardinal, $2^\kappa = \kappa^{++}$, and there is a single normal measure $U$ on $\kappa$. We shall refer to it as the Friedman-Magidor (FM) blueprint. 

The second blueprint will extend the FM-blueprint by adding assumptions to $\po$ which guarantee that every measure on $\kappa$ in the generic extension is a finite power $U^n$ of $U$, therefore yielding a ``Kunen-like" model. We shall refer to the second blueprint as the ``Kunen-like" blueprint.

\begin{notation}${}$
\begin{enumerate}
  
    \item We force with the Jerusalem convention for forcing by which a condition $q'$ is stronger (more informative) than $q$ is denoted by $q' \geq q$.

    \item For a poset $\qo$,
    the weakest condition of $\qo$ is denoted by $0_{\qo}$, and the 
     canonical $\qo$-name for the generic filter is denoted by  $\name{G}_{\qo}$.

    \item For a transitive inner model $N$ of $V$, and an extender/measure $E \in N$ such that the ultrapower of $N$ by $E$ is well-founded,  we denote the ultrapower embedding  by $j_E : V \to N_E \cong \Ult(N,E)$, where $N_E$ is the transitive collapse of the ultrapower model.

    \item A full-supported iterated forcing of length $\kappa$ is a forcing $\po_{\kappa}=\la \po_{\alpha}, \name{\qo}_{\alpha} \colon \alpha<\kappa \ra$, where for each $\alpha<\kappa$, $\name{\qo}_{\alpha}$ is a $\po_\alpha$-name for a forcing notion in $V^{\po_\alpha}$. For every $\alpha\leq \kappa$, a condition $p\in \po_\alpha $ is a function with domain $\alpha$, such that for every $\beta<\alpha$, $p\uhr \beta\in \po_\beta$ and $p\uhr \beta \Vdash p(\beta)\in \name{Q}_{\beta}$. Its support is the set $\supp(p) = \alpha\setminus \{ \beta<\alpha \colon p\uhr \beta \Vdash p(\beta) = \name{0}_{\qo_\beta} \}$. 
    Given $\alpha< \kappa$ and a generic $G_{\po_\alpha}$ for $\po_\alpha$ over the ground model $V$, we use the notation $\po / \po_{\alpha}\in V[G_{\po_\alpha}]$ to denote the associated quotient forcing.
    \item An iterated forcing $\la \po_{\alpha}, \name{\qo}_{\alpha} \colon \alpha<\kappa \ra$ with an Easton support is defined similarly, by adding the additional requirement that for every $\alpha\leq \kappa$ and $p\in \po_\alpha$, the set $\supp(p)$ is an Easton set. In other words, for every $\lambda\leq \kappa$ inaccessible, $\supp(p)$ is bounded in $\lambda$. Similarly, we define an iterated forcing with a nonstationary support by requiring that for every $\alpha\leq \kappa$ and $p\in \po_\alpha$, the set $\supp(p)$ is nonstationary in every inaccessible cardinal. In other words, for every inaccessible $\lambda \leq \kappa$, $\supp(p)\cap \lambda$ is a nonstationary set.
\end{enumerate}    
\end{notation}

\begin{remark}\label {Rmk:RKisom}
    If $W,U$ are two $\kappa$-complete ultrafilters with the same ultrapower models $M_W = M_U$ and embeddings $j_W = j_U$ then $U \cong_{RK} W$.
\end{remark}

\begin{definition}\label{Def:ForcingBasics}${}$
    \begin{enumerate}

    \item 
    Let $Z \subseteq \qo$. A condition $q_Z \in \qo$ is an \emph{upper bound} for $Z$ if $q_Z \geq q$ for all $q \in Z$.
    We say that $q_Z$ is an \emph{exact upper-bound} if it is an upper-bound and for every condition $q' \in \qo$, if $q'$ extends all conditions in $Z$ then it also extends  $q_Z$.
    
    \item A poset $\qo$ is $\kappa$-distributive if for every sequence $\la D_\alpha \mid \alpha < \beta\ra$ of length $\beta < \kappa$, consisting of dense open subsets of $\qo$, then $\cap \vec{D} = \bigcap_{\alpha < \beta} D_\alpha$ is dense in $\qo$. 

     \item 
    Let $\qo$ be a poset and $X \elem (H_{|\qo|^+},\qo)$.
    We say that a subset $G_X \subseteq X \cap \qo$ is an \emph{$X$-generic set} for $\qo$ if every two conditions in $G_X$ have a common extension in $G_X$, and $G_X \cap D \neq\emptyset$ for every $D \in X$ which is dense in $\qo$.\\
    We say that a condition $q \in \qo$ is a \emph{generic condition for $X$} if 
    $$q \Vdash (\name{G}_{\qo} \cap X) \text{ is } X\text{-generic .} $$

            
    \end{enumerate}
\end{definition}

\begin{definition}\label{Def:BlueprintPrelim}${}$
\begin{enumerate}

    \item A sequence $\vec{X} = \la X_\alpha \mid \alpha < \kappa\ra$ is said to be a \emph{Continuous Elementary Chain} in a transitive structure $H$ if $X_\alpha \elem H$ for all $\alpha < \kappa$, $\alpha \subseteq X_\alpha \subseteq X_\beta$ for all $\alpha < \beta$, and $X_\delta = \bigcup_{\alpha < \delta}X_\alpha$ for every limit ordinal $\delta \leq \kappa$.\\
   
    We say that a continuous chain $\vec{X}$ is \emph{Internally Approachable}  if it satisfies further that  $\vec{X}\uhr \delta+1 = \la X_\alpha \mid \alpha \leq \delta\ra \in X_{\delta+1}$ for every successor ordinal $\delta+1 < \kappa$.

    \item Let $\kappa$ be a regular cardinal and $S \subseteq \kappa$ stationary.  
    We say that a poset $\qo$ of size $|\qo| \geq \kappa$ has the \emph{$\kappa$-Continuous-Fusion} property with respect to $S$ ($\kappa$ C-Fusion(S)) if for every condition $q \in \qo$, a sequence $\vec{D} = \la D_\alpha \mid \alpha < \kappa\ra$ of dense open subsets of $\qo$, and a structure 
    $(H_{|\qo|^+},<_{wo},\qo,\vec{D},A)$ that extends $H_{|\qo|^+}$,
    there is an extension $q^* \geq q$, a club $C \subseteq \kappa$, and a continuous elementary chain $\vec{X} = \la X_\alpha \mid \alpha < \kappa\ra$ in $(H_{|\qo|^+},\qo,\vec{D})$, so that 
    \begin{enumerate}
        \item $|X_\alpha|<\kappa$ for all $\alpha < \kappa$.

        \item For every $\alpha \in S \cap C$ and an $X_\alpha$-generic set $G_{X_\alpha} \subseteq X_\alpha \cap \qo$ (in $V$). If $G_{X_\alpha}$ is compatible with $q^*$,  then $G_{X_\alpha} \cup \{q^*\}$ has an exact upper-bound, denoted $q^*_{G_{X_\alpha}}$, which also belongs to $D_\alpha$.    

        \item $q^*$ is a generic condition for $X = \bigcup_{\alpha < \kappa}X_\alpha$.
    \end{enumerate} 
    We shall call a condition $q^*$ with these properties a $\kappa$-Continuous-Fusion witness for $\vec{X},\vec{D}$.
    
    
    \noindent 
    If $\kappa$ is Mahlo then we say $\qo$ has the $\kappa$ C-Fusion property if it has the $\kappa$ C-Fusion(S) property with respect the set of regular cardinals $S = \kappa \cap Reg$.

    \item We say that a poset $\qo$ is \emph{self coding} via sequences  $\vec{S}^{\qo}$ and $\vec{q}^{\qo}$ when
   
    \begin{enumerate}
        \item  $\vec{q}^{\qo} = \la q^{\qo}_\tau \mid \tau < |\qo|\ra$ is an enumeration of $\qo$, 

        \item there is a regular cardinal $\lambda_{\qo} \leq |\qo|$ such that  $\vec{S}^{\qo} = \la S^{\qo}_\tau \mid \tau < |\qo|\ra$ consists of pairwise almost disjoint stationary subsets of $\lambda_{\qo}$,
    
        \item for every $\tau < |\qo|$, 
        $$ 0_{\qo} \Vdash (\check{S}^{\qo}_{2\tau} \text{ is not stationary } \iff q^{\qo}_\tau \in \name{G}_{\qo}) \text{ and }   $$ 

        $$ 0_{\qo} \Vdash (\check{S}^{\qo}_{2\tau+1} \text{ is not stationary } \iff q^{\qo}_\tau \not\in \name{G}_{\qo}),   $$ 

    \end{enumerate}

    \item Let $\kappa$ be a regular cardinal, $S \subseteq \kappa$ stationary, and $\po_\kappa = \la \po_\alpha,\qo_\alpha\mid \alpha < \kappa\ra$ an iterated forcing of length $\kappa$. 
    We say that $\po_\kappa$ has the $\kappa$-Iteration-Fusion property with respect to $S$ ($\kappa$ Iteration-Fusion(S)) if for every condition $p \in \po_\kappa$ and a sequence $\la D_\alpha \mid \alpha < \kappa\ra$ of dense open subsets of $\po_\kappa$ there is $p^* \geq p$ and a club $C \subseteq \kappa$ such that for every $\alpha \in C \cap S$, 
    the set 
    $$\{ q \in \po_{\alpha+1} \mid q \fr (p^*\setminus (\alpha+1)) \in D_\alpha\} \text{ is dense in }\po_{\alpha+1}.$$

     If $\kappa$ is Mahlo then we say $\qo$ has the $\kappa$ Iteration-Fusion property if it has the $\kappa$ Iteration Fusion(S) property with respect the set of regular cardinals $S = \kappa \cap Reg$.
    \end{enumerate}
\end{definition}

\begin{example}
We show that for the generalized Sacks forcing from \cite{Kanamori} and \cite{FriedmanThompson}, one can transform the standard tree-type fusion construction to prove the $\kappa$-Continuous-Fusion property from Definition \ref{Def:BlueprintPrelim} above.\\
Assume that $\kappa$ is Mahlo, and let $Sacks^*(\kappa)$ be the Sacks forcing with splitting at singulars, namely $p\in Sacks^*(\kappa)$ if and only if:
\begin{itemize}
    \item $p$ is a sub-tree of the binary tree $2^{<\kappa}$.
    \item $p$ is closed under union of increasing sequences of length $<\kappa$ of its elements.
    \item There exists a club $C\subseteq \kappa$ such that for every singular $\alpha\in C$, every $t\in p\cap 2^{\alpha}$ is a splitting point of $p$ (namely, both $t^{\frown} \la 0 \ra$ and $t^{\frown}\la 1 \ra$ belong to $p$).
    \item For every regular $\xi<\kappa$ and $t\in p\cap 2^{\xi}$, $t$ is not a splitting point of $p$.
\end{itemize}
We order $Sacks^*(\kappa)$ by reverse inclusion. 
It's not hard to see that for every pair of conditions $p,q\in Sacks^*(\kappa)$, $p\vee q := p\cap q$ is the least upper bound of ${p,q}$. Furthermore, $Sacks^*(\kappa)$ is $\kappa$-directed closed. 

\begin{claim}
    $Sacks^*(\kappa)$ satisfies the $\kappa$-continuous-Fusion property with respect to $S = \kappa\cap Reg$.
\end{claim}

\begin{proof}
Assume that $q\in Sacks^*(\kappa)$, $\vec{D} = \la D_{\alpha} \colon \alpha<\kappa \ra$ is a sequence of dense open subsets of $Sacks^*(\kappa)$. Let $\langle X_{\alpha} \colon \alpha<\kappa \rangle$ be any continuous and internally approachable chain of elementary substructures of $\left( H_{\kappa^{++}}, Sacks^*(\kappa), \vec{D} \right)$, such that  for every $\alpha<\kappa$, $X_{\alpha+1}$ is closed under $\alpha$-sequences. 
such that $q, Sacks^*(\kappa), \la D_\alpha \colon \alpha<\kappa \ra$ are all in $X_0$. 
We construct an increasing Fusion sequence of conditions, namely an increasing sequence $\la q_{\alpha} \colon \alpha<\kappa \ra$ extending $q$, alongside a continuous, increasing sequence $\la \xi_{\alpha} \colon \alpha<\kappa \ra$ of ordinals below $\kappa$, such that:
\begin{itemize}
    \item For every $\alpha<\kappa$, if $\xi_{\alpha}$ is singular, then it is a splitting level of $q_\alpha$. Here, a splitting level of a condition $p$ is a singular $\xi<\kappa$ such that every node in $\mbox{Lev}_{\xi}(p)$ is a splitting point of $p$, where $\mbox{Lev}_\xi(p) := p\cap 2^{\xi}$. 

    \item For every $\alpha<\beta<\kappa$, $\mbox{Lev}_{\xi_\alpha+1}(q_{\beta}) =  \mbox{Lev}_{\xi_\alpha+1}(q_{\alpha})$.
    \item For every $\alpha<\kappa$, $q_{\alpha+1}\in X_{\alpha+1}$.
    \item For every $\alpha<\kappa$ and $t\in  \mbox{Lev}_{\xi_\alpha+1}(q_{\alpha})\cap X_{\alpha+1} $, the tree 
    $$(q_{\alpha+1})_t = \{ s\in q_{\alpha} \colon s,t \mbox{ are compatible sequences} \}$$ 
    belongs to 
    $$D^*_{\alpha} := D_\alpha\cap \left( \bigcap\{ D \colon D\in X_{\alpha} \mbox{ is dense and open subset of } Sacks^*(\kappa) \}\right).$$
    We point out that if $\alpha = \xi_\alpha$ is regular then $q_{\alpha+1}$ does not split at level $\alpha$, and has a unique extension $t$, and therefore $q_{\alpha+1} = (q_{\alpha+1})_t \in D_\alpha$.
\end{itemize}
The construction of a Fusion sequence of conditions is done by induction. At limit steps, take $q_{\alpha} = \bigcap\{ q_{\beta} \colon \beta<\alpha \} $, $\xi_{\alpha} = \sup\{ \xi_\beta \colon \beta<\alpha \} $. For the successor step, assume that $q_{\alpha}, \xi_{\alpha}$ have been constructed, and let $\xi_{\alpha+1}$ be the first splitting level of $q_{\alpha}$ above $\xi_{\alpha}$. For every $t\in q_{\alpha}\cap 2^{\xi_{\alpha+1}+1}$, shrink $\left(q_{\alpha}\right)_{ t }$ so that it enters the dense open set $D^*_\alpha$. Note that by $\kappa$-closure, and the fact that $|X_{\alpha}|<\kappa$, $D^*_{\alpha}$ is indeed dense open. It's not hard to verify that the resulting tree $q_{\alpha+1}$ has an associated club whose singular points are splitting levels, and thus $q_{\alpha+1}\in Sacks^*(\kappa)$. Finally, note that  $q_{\alpha}\in X_{\alpha+1}$ by induction and internal approachability. Since $D^*_{\alpha}\in X_{\alpha+1}$ (again, by internal approachability),  we can choose $q_{\alpha+1}$ such that  $q_{\alpha+1}\in X_{\alpha+1}$. 

This concludes the inductive construction. Let $q^*$ be the tree generated from 
$$ \bigcup_{\alpha<\kappa} \mbox{Lev}_{ \xi_\alpha+1 }(q_\alpha).$$ 
The fact that $q^*\in Sacks^*(\kappa)$ follows since every singular point in the club $\{ \xi_{\alpha} \colon \alpha<\kappa \}$ is a splitting point of $q^*$.
We argue that $q^*$, $\la X_{\alpha} \colon \alpha<\kappa \ra$ and the club $C = \{ \alpha < \kappa \mid \xi_\alpha = \alpha\}$ are witnesses for the $\kappa$-continuous-Fusion property with respect to $S = \kappa\cap Reg$:

\begin{enumerate}
    \item Assume that $\alpha\in C$ is regular and $G_\alpha$ is an $X_\alpha$ generic set which is compatible with $q^*$. We argue that $G_{X_{\alpha}}\cup\{ q^* \}$ has an exact upper bound, which is the subtree of $q^*$ obtained by extending the stem of $q^*$ to level $\xi_{\alpha}$, as determined by  $\cap G_{X_{\alpha}}$. More formally, we will show below that $\cap G_{X_\alpha}$ has a stem $t$ of length $\xi_\alpha = \alpha$, which belongs to $q^*$. The desired exact upper bound will be the tree $q^*_{X_\alpha}:=(q^*)_t$, which is the sub-tree of $q^*$ which includes only nodes of $q^*$ which are compatible with $t$.
   
    First note that for every $\beta<\alpha$, $\xi_{\beta}\in X_{\alpha}$ (by the above construction, $\xi_{\beta+1}\in X_{\beta+1}$ is the least splitting point of $q_\beta$ above $\xi_{\beta}$. Thus, by internal approachability, every initial segment of the sequence $\la \xi_\beta \colon \beta<\alpha \ra$ belongs to $X_\alpha$). Thus, $q^*_{X_{\alpha}}$ meets the dense open set of conditions which decide the generic stem up to height $\xi_{\beta}$. This is true for every $\beta<\alpha$ separately, and thus $q^*_{X_{\alpha}}$ might split only at levels $\geq \xi_{\alpha}$. But $\xi_\alpha = \alpha$, is regular, and thus $q^*_{X_\alpha}$ does not split in level $\xi_\alpha$ as well, and by the point made at the end of the construction of the sequence $\la q_\alpha \mid \alpha < \kappa\ra$, we get that $q^*_{X_\alpha} \in D_\alpha$. 

    We argue that $q^*_{X_\alpha} = (q^*)_t$ is an exact upper bound of $G_{X_\alpha}\cup\{q^* \}$. Indeed, for every $p\in G_{X_\alpha}$, $(q^*)_t$ meets the dense open set of conditions which either extend $p$ or incompatible with it (this dense open set belongs to $X_{\beta}$ for some $\beta<\alpha$, and thus by the above construction, $(q^*)_{t\uhr (\xi_{\beta}+1) }$ belongs to it). But $q^*$ is compatible with every condition of $G_{X_\alpha}$, and thus $(q^*)_t \geq p$. For exactness, note that if $q'$ is an upper bound of $G_{X_\alpha}\cup \{ q^* \}$, then, arguing as above, $t$ is the stem of $q'$ since $q'$ extends $G_{X_{\alpha}}$. But $q'$ extends $q^*$, and thus $q'$ extends $(q^*)_t$.

    \item We argue that $q^*$ is a generic condition for $X = \bigcup_{\alpha<\kappa} X_{\alpha}$. Assume that $G\subseteq Sacks^*(\kappa)$ is generic over $V$ and $q^*\in G$. We argue that $G\cap X$ is $X$-generic. Indeed, given $D\in X$ dense open in $Sacks^*(\kappa)$, there exists a regular $\alpha<\kappa$ with $\alpha = \xi_{\alpha}$ and $D\in X_{\alpha}$. It follows that for every $t\in \mbox{Lev}_{\alpha+1}(q^*)\cap X_{\alpha+1}$, $(q^*)_t\in D$. Let $t$ be the initial segment of the generic branch up to height $\alpha$. Then $(q^*)_t\in G$. We argue that also $(q^*)_t\in D$. Since $X_{\alpha+1}$ is closed under $\alpha$-sequences, $t\in X_{\alpha+1}$. It follows that $(q^*)_t\in G\cap D$ as desired.
\end{enumerate}
    
\end{proof}
\end{example}

\begin{lemma}\label{Lem:IA-Fusion-NewClubGeneric}
    Let $\qo$ be a poset, $\kappa$ an uncountable regular cardinal, and $\la X_\alpha \mid \alpha < \kappa\ra$ be a continuous elementary chain such that $X_\alpha \elem (H_{{|\qo|}^+},\qo)$  and $|X_\alpha|<\kappa$ for all $\alpha < \kappa$. Suppose that $\kappa$ remains regular in generic extensions by $\qo$ and $q^* \in \qo$ is a generic condition for $X = \bigcup_{\alpha < \kappa}X_\alpha$. Then there is a $\qo$-name of a club $\name{C} \subseteq \kappa$ such that $q^*$ is a generic condition for $X_\alpha$ for every $\alpha \in \name{C}$. Namely, 
    $$q^* \Vdash 
    \forall\alpha \in \name{C}. \thinspace (\name{G}_\qo \cap X_\alpha) \text{ is } X_\alpha\text{-generic.}$$
\end{lemma}
\begin{proof}
    It suffices to verify that in every generic extension $V[G]$ by $G\subseteq \qo$ with $q^* \in G$, there is a club $C \subseteq \kappa$ such that for every $\alpha \in C$, $(G \cap X_\alpha)$ is an $X_\alpha$ generic set. Since $q^*$ is a generic condition for $X$, for every dense open set $D \subseteq \qo$, $D \in X$, the intersection $G \cap D \cap X$ is nonempty. Let $\beta_D < \kappa$ be the minimal $\beta < \kappa$ so that $G \cap D \cap X_\beta \neq \emptyset$. Define a function $f :\kappa \to \kappa$ by 
    \[
    f(\alpha) = \sup\{\beta_D \mid D \in X_\alpha \text{ dense open in }\qo \}.
    \]
    Let $C \subseteq \kappa$ be the club of closure points of $f$. Since the sequence $\la X_\alpha \mid \alpha < \kappa\ra$ is a continuous chain, it follows that for every limit ordinal $\alpha \in C$ and a dense open set $D \in X_\alpha$, $G \cap D \cap X_\alpha \neq \emptyset$. Hence $G$ is a generic set for $X_\alpha$.
    \end{proof}

\begin{lemma}\label{Lem: CFusionApproximation} 
Let $\qo$ be a $\kappa$-distributive poset satisfying the $\kappa$ C-fusion property. Assume that $p\in \qo$, $n<\omega$ and $\name{f}$ is a $\qo$-name for a function from $\kappa^n$ to $V$. Then there are a condition $p^*\geq q$, a $\qo$-name $\name{C}^*$ for a club in $\kappa$, and a function $F\in V$, $F\colon \kappa\to V$, such that, for every Mahlo $\alpha<\kappa$, $|F(\alpha)|\leq 2^\alpha$, and-- 
$$p^* \Vdash \forall \alpha \in (\name{C}^* \cap Reg)\thinspace \left(\cup_{\vec{\beta}\in \alpha^{n-1}}\name{f}(\vec{\beta} \fr \{\alpha\}) \right)\subseteq \check{F}({\alpha}).$$
\end{lemma}

\begin{proof}
We can bound the values taken by the function $\name{f}$ and pick sufficiently large $\theta$ such that $0_{\po} \Vdash \forall \alpha < \kappa\thinspace \name{f}(\check{\alpha}) \in H_\theta^V$. Let $<_\theta$ be a well-ordering of $H_\theta$. Let $\vec{D} = \la D_\alpha \mid \alpha < \kappa\ra$ be a sequence of  dense sets in $\qo_\kappa$ given by 
$$D_\alpha = \{ q \in \qo \mid \exists \la x^\alpha_{\vec{\beta}}\ra_{\vec{\beta} \in \alpha^{n-1}} \subseteq H_{\theta}^V. \thinspace \forall \vec{\beta} \in \alpha^{n-1} \thinspace q \Vdash   \name{f}(\vec{\beta} \fr \{\alpha\}) = \check{x}^\alpha_{\vec{\beta}}\}.$$
The density of $D_{\alpha}$ follows since $\qo$ is $\kappa$-distributive.
Since $\qo$ has the $\kappa$ C-Fusion property, there are $q^* \geq q$, a club $C_1 \subseteq \kappa$, and a continuous elementary chain $\vec{X} = \la X_\alpha \mid \alpha < \kappa\ra$ of substructures $X_\alpha \elem (H_{\kappa^{++}}^{V[G]},\qo,\vec{D})$, such that for each inaccessible cardinal $\alpha \in C_1$, given a condition $q_{X_\alpha} \in \qo$ which is an exact upper bound to a $X_\alpha$-generic set, if $q_{X_\alpha} \parallel q^*$ then $q^* \vee q_{X_\alpha} \in D^*_\alpha$. 
Moreover, by Lemma \ref{Lem:IA-Fusion-NewClubGeneric} there is a $\qo$-name $\name{C}_2$ of a club in $\kappa$ such that for every $\alpha\in \name{C}_2$, $q^*$ forces that $\name{G}_{\qo} \cap X_\alpha$ is a generic set for $X_\alpha$. Since $\qo$ is $\kappa$-distributive, $\name{G}_{\qo} \cap  X_\alpha \in V$ for every such $\alpha$. Therefore by the C-Fusion property, $q^*$ forces that for every inaccessible $\alpha\in C_1\cap \name{C}_2$, we must have an exact upper bound $\name{q}_\alpha$ for $\name{G}_{\qo}\cap X_{\alpha}$. Such $
\name{q}_{\alpha}$ is compatible with $q^*$.

For each $\alpha \in C_1$, let
$$\Gamma(X_\alpha) = \{ q \in \qo_\kappa \mid q || q^* \text{ and } q \text{ is an exact upper-bound of some } X_\alpha \text{-generic set } G_{X_\alpha}^q \}.$$
For every $q \in \Gamma(X_\alpha)$ let $\la x^\alpha_{\vec{\beta}}(q)\ra_{\vec{\beta} \in \alpha^{n-1}} \subseteq H_\theta^V$ be such that for every $\vec{\beta} \in \alpha^{n-1}$, 
$$q \vee q^* \Vdash \name{f^*}(\vec{\beta} \fr\{\alpha\}) = \check{x}^\alpha_{\vec{\beta}}(q).$$
Since $\vec{X}$ is continuous we may assume $|X_\alpha| = \alpha$ for all $\alpha \in C_1$, and thus $|\Gamma(X_\alpha)| \leq (2^\alpha)^{V[G_\kappa]} = \alpha^{++}$. Let 
$$F(\alpha) = \{ x^{\alpha}_{\vec{\beta}}(q)  \mid \vec{\beta} \in \alpha^{n-1}, q \in \Gamma(X_\alpha)\}.$$
 It follows that $|F(\alpha)| \leq \alpha^{++}$ and
 $$q^* \Vdash \forall \alpha \in \check{C}_1 \cap \name{C}_2 \cap Reg \ \forall \vec{\beta} \in \alpha^{n-1} \ \name{f}(\vec{\beta} \fr \{\alpha\}) \in \check{F}(\alpha).$$
\end{proof}

\begin{lemma}\label{Lem:IAF-CapturingDenseSets}
    Suppose that $\po_\kappa =\la \po_\alpha,\qo_\alpha \mid \alpha < \kappa\ra$ is an iterated forcing which has the 
    $\kappa$ Iteration-Fusion property, and for each $\alpha < \kappa$, 
        $\qo_\alpha$ is trivial if $\alpha$ is not inaccessible,
    and the quotient tail forcing $\po/\po_\alpha$ is (forced to be) $\alpha$-distributive.
    Then for every $p \in \po_\kappa$ and a sequence $\la D_\alpha \mid \alpha < \kappa\ra$ of dense open subsets of $\po_\kappa$, there is an extension $p' \geq p$ and a closed unbounded set $C \subseteq \kappa$ such that for each regular $\alpha \in C$, the set
    \[
    \{ w \in \po_{\alpha+1} \mid w \fr (p'\setminus \alpha+1) \in \bigcap_{\beta < \alpha^{++}}D_\beta\}
    \]
    is dense in $\po_{\alpha+1}/(p'\uhr \alpha+1)$.
 \end{lemma}
 \begin{proof}
     For each $\alpha < \kappa$, let $\name{d}_\alpha$ be the $\po_{\alpha+1}$-name of the set of conditions 
     $$\name{d}_\alpha = \{ q \in \po_\kappa/\name{G}(\po_\alpha) \mid \exists w \in \name{G}(\po_\alpha) \thinspace w \fr q \in D_\alpha\}.$$
     For each $\alpha < \kappa$, $0_{\po_{\alpha+1}}$ forces $\name{d}_\alpha$ is dense open in $\po_\kappa/\name{G}(\po_{\alpha+1})$. Moreover, since $\po_\kappa/\name{G}(\po_{\alpha+1}) = \po_\kappa/\name{G}(\po_{\alpha^*})$ where $\alpha^*$ is the first inaccessible cardinal above $\alpha$, and $\po_\kappa/\name{G}(\po_{\alpha^*})$ is $\alpha^*$-distributive, $0_{\po_{\alpha+1}}$ forces $\name{d}^*_\alpha = \bigcap_{\beta < \alpha^{++}} \name{d}_\beta$ is also dense open.
    It follows that for each $\alpha < \kappa$ the set 
    $$D'_\alpha = \{ q \in \po_\kappa \mid q\uhr \alpha+1 \Vdash_{\po_{\alpha+1}} q\setminus \alpha+1 \in \name{d}^*_\alpha\}$$
    is dense open in $\po_\kappa$. 
    Applying the $\kappa$ Iteration-Fusion assumption for $\po_\kappa$ and the sequence $ \la D'_\alpha \mid \alpha < \kappa\ra$ we find $p' \geq p$ and a closed unbounded set $C \subseteq \kappa$ such that for every inaccessible cardinal $\alpha \in C$, the set $$\{ w \in \po_{\alpha+1} \mid w \fr (p'\setminus \alpha+1) \in D'_\alpha\} = 
    \{ w \in \po_{\alpha+1} \mid w \fr (p'\setminus \alpha+1) \in \bigcap_{\beta < \alpha^{++}}D_\alpha\}$$
    is dense in $\po_{\alpha+1}/(p'\uhr \alpha+1)$
 \end{proof}

 \begin{lemma}\label{Lem: PreservationOfStationarySetsForIterationFusionProperty}
     Suppose that $\kappa$ is Mahlo and $\po_\kappa =\la \po_\alpha,\qo_\alpha \mid \alpha < \kappa\ra$ is an iterated forcing which has the $\kappa$ Iteration-Fusion property, and, for every $\alpha<\kappa$, $|\po_\alpha|<\kappa$. Then $\po_\kappa$ preserves stationary subsets of $\kappa$. In particular, $\kappa$ remains Mahlo in $V^{\po_\kappa}$.
 \end{lemma}
 \begin{proof}
     It suffices to prove that every club subset of $\kappa$ in $V^{\po_\kappa}$ contains a club subset of $\kappa$ from $V$. 
     
     Assume that $\name{C}$ is a $\po_\kappa$-name for a club subset of $\kappa$. For every $\alpha<\kappa$, let $D_{\alpha}$ be the dense open subset of conditions deciding the $\alpha$-th element $\name{c}_\alpha$ in the increasing enumeration of $\name{C}$. By the $\kappa$ Iteration-Fusion property, every condition $p\in \po_\kappa$ can be extended to a condition $p^*\geq p$, such that for some club $E\subseteq \kappa$ and for every regular $\alpha\in E$, the set
     $$\{ q\in \po_{\alpha+1} \colon \exists c<\kappa, q^{\frown}\left( p^*\setminus (\alpha+1) \right)\Vdash c= \name{c}_\alpha  \}$$
     is dense open in $\po_{\alpha+1}$.

     Given such $p^*$ and club $E\subseteq \kappa$, let $C^*$ be the set of closure points of the function which maps each $\alpha\in E\cap Reg$ to-- $$\sup\{ c<\kappa \colon \exists q\in \po_{\alpha+1}, q^{\frown}\left( p^*\setminus (\alpha+1) \right)\Vdash c = \name{c}_\alpha\}<\kappa$$
     where the supremum above is indeed below $\kappa$ since $|\po_{\alpha+1}|<\kappa$. Then $C^*$ is a club in $\kappa$, and $p^*\Vdash C^*\subseteq \name{C}$.     
 \end{proof}

 \begin{lemma}\label{Lem: PowersetOfKappaIsPreserved}
     Suppose that $\kappa$ is Mahlo and $2^\kappa = \kappa^+$. Let $\po_\kappa =\la \po_\alpha,\qo_\alpha \mid \alpha < \kappa\ra$ be an iterated forcing which has the $\kappa$ Iteration-Fusion property, such that for every $\alpha<\kappa$, $|\po_\alpha|<\kappa$, $\qo_\alpha$ is trivial if $\alpha$ is not inaccessible, and the quotient tail forcing $\po_\kappa / \po_\alpha$ is (forced to be) $\alpha$-distributive.      
     Then $2^{\kappa} = \kappa^+$ in $V^{\po_\kappa}$.
 \end{lemma}
\begin{proof}
    Assume that $G\subseteq \po_\kappa$ is generic over $V$. Let $\name{A}$ is a $\po_\kappa$ name for a subset of $\kappa$. For every $\alpha<\kappa$, let 
    $$ D_{\alpha} = \{ q\in \po_\kappa \colon q\uhr (\alpha+1) \Vdash \exists A_{\alpha}\subseteq \alpha, q\setminus(\alpha+1)\Vdash \name{A}\cap \alpha = A_{\alpha} \}. $$
    Note that $D_{\alpha}$ is indeed dense for every $\alpha<\kappa$, since $\po_\kappa / \po_{\alpha+1}$ is $\alpha^+$-distributive.
    By the $\kappa$ Iteration-Fusion property, there exist $p\in G$ and a club $C\subseteq \kappa$ such that, for every $\alpha\in C\cap Reg$, the set
    $$ \{ q\in \po_{\alpha+1} \colon q\Vdash \exists A_{\alpha}\subseteq \alpha, \ p\setminus(\alpha+1)\Vdash \name{A}\cap \alpha = A_{\alpha} \} $$
    is dense in $\po_{\alpha+1}$. It follows that, for  every $\alpha\in C\cap Reg$, there exists a $\po_{\alpha+1}$-name $\name{A}_{\alpha}$, such that the weakest condition of $\po_{\alpha+1}$ forces that 
    $$p\setminus (\alpha+1)\Vdash \name{A}\cap \alpha = \name{A}_\alpha.$$
    The set $(\name{A})_G$ is uniquely determined from the sequence $\la \name{A}_\alpha \colon \alpha\in C\cap Reg \ra$ and $G$. By standard arguments, we can identify each name $\name{A}_{\alpha}$ with a function from $\alpha$ to the set $X_{\alpha}$ of maximal antichains in $\po_{\alpha+1}$. Since $\kappa$ is strongly inaccessible and $|\po_{\alpha+1}|<\kappa$, there are $2^{\kappa} = \kappa^{+}$ elements in $\prod_{\alpha\in C\cap Reg} (X_{\alpha})^{\alpha}$. It follows that $2^{\kappa} = \kappa^+$ holds in $V[G]$.
\end{proof}

\subsection{The Friedman-Magidor Blueprint}

\begin{definition}[FM-blueprint]\label{Def:FM-blueprint}${}$\\
    Let $\kappa$ be a measurable cardinal and  $\po = \la \po_\alpha,\name{\qo}_\alpha\mid \alpha \leq \kappa\ra$ a \textbf{nonstationary-support} iteration of length $\kappa+1$. The Friedman-Magidor (FM) blueprint includes the following additional assumptions for $\po$:
    \begin{enumerate}[label = FM\arabic*.]
        \item For each $\alpha \leq \kappa$, the ${\po}_\alpha-$name $\name{\qo}_\alpha$ is the trivial poset unless $\alpha$ is Mahlo.

        \item For every Mahlo cardinal $\alpha \leq \kappa$, 
        $0_{\po_\alpha}$ forces $\name{\qo}_\alpha$ satisfies the following requirements
        
        \begin{enumerate}
            
            \item $\qo_\alpha$ has size $|\name{\qo}_\alpha| =\alpha^{++}$ and adds $\alpha^{++}$ new subsets to $\alpha$, 
            
            \item $\name{\qo}_\alpha$ is $\alpha$-distributive, satisfies the $\alpha^{++}$.c.c, and preserves the Mahloness of $\alpha$. 
            
            \item $\name{\qo}_\alpha$ has the $\alpha$ C-Fusion property,

            \item $\name{\qo}_\alpha$ is self coding via a pair $(\name{\vec{q}}^\alpha, \name{\vec{S}}^\alpha)$.

        \end{enumerate}

    \item For every Mahlo cardinal $\alpha \leq \kappa$, 
    \begin{enumerate}

        \item $\po_\alpha$ has the $\alpha$ Iteration-Fusion property, 
        
        \item $0_{\po_\alpha} \Vdash (\po/\po_\alpha) \text{ is } \alpha\text{-distributive}$.
        
   \end{enumerate}

    \item The iteration is uniformly definable in sense that there are parameter-free formulas $\varphi_{\po}(v,x)$, $\varphi_{\qo}(v,x,y)$, and $\varphi_F(v,x)$, such that the following holds in every Mahlo cardinal $\alpha \leq \kappa$ and an inner model $M$ containing $H_{\alpha^{++}}$, 
    \begin{enumerate}
        \item $\varphi_{\po}(\alpha,x)$  defines $\po_{\alpha+1}$,

        \item $\varphi_F(\alpha,x)$ defines a $\po_{\alpha+1}$-name for a sequence of functions $\vec{f}^\alpha = \la f^\alpha_\tau \mid \tau < \alpha^{++}\ra \subseteq {}^\alpha \alpha$ that witnesses $\qo_\alpha$ adds $\alpha^{++}$ new subsets to $\alpha$,
        
        \item $\varphi_{\qo}(\alpha,x,y)$ defines the $\po_\alpha$-name for a pair $(\vec{q}^\alpha,\vec{S}^\alpha)$, which are forced by $0_{\po_\alpha}$ to be witnesses that $\qo_\alpha$ is self coding.
        \end{enumerate}    
    \end{enumerate}
\end{definition}

Throughout this section, we assume GCH holds in the ground model $V$.
We start with a lemma which states a useful combinatorial property of the a poset $\po$ that satisfies the FM-blueprint.

\begin{lemma}\label{Lem:IAFusionCombinatorics}
 (GCH) Suppose that $\po$ satisfies the FM-blueprint assumptions. Let $\name{f}$ be a $\po$-name for a function from $\kappa^n$ for some $n < \omega$ to ground model sets in $V$, and $p \in \po$ a condition. Then there are  $p^* \geq p$, a $\po$-name $\name{C}$ of a club in $\kappa$, and a function $F: \kappa \to V$ such that $|F(\alpha)| \leq \alpha^{++}$ for every Mahlo cardinal, and 
    $$p^* \Vdash \forall \alpha \in (\name{C} \cap Reg)\thinspace \left(\cup_{\vec{\beta}\in \alpha^{n-1}}\name{f}(\vec{\beta} \fr \{\alpha\}) \right)\subseteq \check{F}({\alpha}).$$
\end{lemma}

\begin{proof}
We can bound the values taken by the function $\name{f}$ and pick sufficiently large $\theta$ such that $0_{\po} \Vdash \forall \alpha < \kappa\thinspace \name{f}(\check{\alpha}) \in H_\theta^V$.
We fix a well-ordering $<_\theta$ of $H_\theta$.
Given $p^*  \in \po = \po_\kappa * \qo_\kappa$. We write it as $p = p\uhr \kappa * \name{p}_\kappa$ where $p\uhr \kappa \in \po_\kappa$ and $p\uhr \kappa \Vdash \name{p}_\kappa \in \qo_\kappa$. We break the ground approximation argument of $\name{f}$ into two parts, where we first approximate the $\po$-name $\name{f}$ by a $\po_\kappa$ name of a function $\name{F^*}$, and then approximate $\name{F^*}$ by a function $F \in V$.\\

For the first approximation step, fix $G_\kappa\subseteq \po_\kappa$ generic over $V$, and apply Lemma  \ref{Lem: CFusionApproximation} in $V[G_\kappa]$ to find $F^*\in V[G_\kappa]$, $p^*_\kappa \geq p_\kappa$ and a $\qo_\kappa$-name $\name{C}^*$ for a club subset of $\kappa$ such that, for every Mahlo $\alpha<\kappa$, $ V[G_\kappa] \models |F(\alpha)|\leq 2^\alpha$, and-- 
$$p_\kappa^* \Vdash \forall \alpha \in (\name{C}^* \cap Reg)\thinspace \left(\cup_{\vec{\beta}\in \alpha^{n-1}}\name{f}(\vec{\beta} \fr \{\alpha\}) \right)\subseteq \check{F}^*({\alpha}).$$

 For each inaccessible cardinal $\alpha \in C^*$, let $\bar{x}^\alpha = \la x^\alpha_i \mid i < \alpha^{++}\ra$ be an enumeration of $F^*(\alpha)$ according to the fix well-ordering $<_\theta$, of length $\alpha^{++}$ (possibly with repetitions).\\

\noindent
 Moving back to $V$ for the second approximation step, there is $\bar{p} \geq_{\po_\kappa} p\uhr \kappa$ which forces the relevant properties above for $\po_\kappa$-names, $\name{F}^*,\la \name{\bar{x}}^\alpha \mid \alpha < \kappa\ra$, and $\name{p}^*_\kappa$, of $F^*$, $\la \bar{x}^\alpha \mid \alpha < \kappa\ra$, and $p^*_\kappa$ respectively. Also, let $\name{C}_1$ be the
  $\po = (\po_\kappa*\qo_\kappa)$-name for $\name{C}^*$. 
 Working in $V$, we may take a sequence of dense open subsets of $\po_\kappa$,
 $\la D^\alpha_i \mid \alpha < \kappa, i < \alpha^{++}\ra$ such that for an inaccessible cardinal $\alpha < \kappa$ and $i < \alpha^{++}$, $D^\alpha_i$ consists of conditions 
 $r \in \po_\kappa$ for which there is a set $x^\alpha_i(r) \in H_\theta$ so that
         $r \Vdash \name{x}^\alpha_i = \check{x}_i^{\alpha}(r)$.
By Lemma \ref{Lem:IAF-CapturingDenseSets} there is an extension $p' \geq \bar{p}$ and a closed unbounded set $C_2 \subseteq \kappa$ such that for every inaccessible cardinal $\alpha \in C_2$, the set 
$$W_\alpha = \{ w \in \po_{\alpha+1} \mid w \fr (p'\setminus \alpha+1) \in \bigcap_{\beta \leq \alpha, i < \beta^{++}}D^\beta_i\}$$ is dense in $\po_{\alpha+1}/(p'\uhr \alpha+1)$.
For each $w \in W_\alpha$ denote $w \fr (p'\setminus \alpha+1)$ by $r_w$, and 
define 
$$F(\alpha) = \{ x^\alpha_i({r_w}) \mid w\in W_\alpha \text{ and } i < \alpha^{++}\}.$$
Define $p^* = p' \fr \name{p}^*_\kappa \in \po = \po_\kappa * \qo_\kappa$ and $\name{C}^* = \name{C}_1 \cap \check{C}_2 $. It follows that $|F(\alpha)| \leq \alpha^{++}$ for every inaccessible cardinal $\alpha < \kappa$, and $p^*$ forces that for each $\alpha \in \name{C}^*$ and $\vec{\beta} \in \alpha^{n-1}$, $\name{f}(\vec{\beta} \fr \{\alpha\}) \in \check{F}(\alpha)$.
\end{proof}

\begin{theorem}\label{Thm:FMblueprint}
Suppose that $V = L[\E]$ is a  fine-structural extender model (in the sense of \cite{Steel-HB})  which is minimal for the existence of a measurable cardinal $\kappa$ carrying a $(\kappa,\kappa^{++})$-extender $E$.
 Let $\po = \la \po_\alpha,\name{\qo}_\alpha \mid \alpha \leq \kappa\ra$ be an iterated forcing poset which satisfies the FM-blueprint. Then in a generic extension $V[G]$ by $G \subseteq \po$, $\kappa$ is the only measurable cardinal, there is a unique normal measure $U$ on $\kappa$, and $j_U \uhr V = j_E$.
\end{theorem}

Fix a $V$-generic filter $G \subseteq \po$.
We need to show that in $V[G]$, there is a unique normal measure $U$ on $\kappa$, and that $j_U \uhr V = j_E$ is the ultrapower embedding by the unique $(\kappa,\kappa^{++})$-extender $E \in V = L[\E]$. 
We break the proof of this statement into the following four parts, given in the next Lemmas.
\begin{lemma}\label{Lem:FM1}
    In $V[G]$ there exists a unique $M_E$-generic filter $G^* \subseteq j_E(\po)$ which contains $j_E``G$. 
\end{lemma}
    By Silver's argument for extensions of elementary embedding in forcing generic extension,  such $M_E$-generic $G^*$ gives rise to a unique extension $j^* : V[G] \to M_E[G^*]$ of the embedding $j_E$, defined by $j^*(\sigma_G) = j_E(\sigma)_{G^*}$ for a name $\sigma \in V^{\po}$.

\begin{lemma}\label{Lem:FM2}

    Every $a \in M_E[G^*]$ has the form $j^*(g)(\kappa)$ for some $g \in {}^\kappa V[G]$, $g \in V[G]$.
\end{lemma}
This implies the model $M_E[G^*]$ and the embedding $j^*$ identify with the ultrapower of $V[G]$ by the $j^*$-derived normal measure
    \[
    U = \{ X \subseteq \kappa \mid \kappa \in j^*(X)\}.
    \]

\begin{lemma}\label{Lem:FM2.5}
    Let $U^*$ be a $\kappa$-complete ultrafilter in $V[G]$, and  $j_{U^*} : V[G] \to M_{U^*}$ be its induced ultrapower embedding. Then 
    there is 
a finite normal iteration $\la \Mx_i,\jx_{i,j} : i \leq j \leq \ell\ra$ of $V = \Mx_0$ by extenders $\Ex_i \in \Mx_i$, $i < \ell$ (i.e., $\jx_{i,i+1} = j^{\Mx_i}_{\Ex_i} : \Mx_i \to \Mx_{i+1} \cong \Ult(\Mx_i,\Ex_i)$ is the ultrapower map), whose critical points $\kappax_i = cp(\Ex_i)$, $i < \ell$, are increasing, 
such that $j_{U^*}\uhr V = \jx_{0,\ell}$, $j_{U^*}``V = \Mx_{\ell}$, $M_{U^*} = \Mx_\ell[H^*]$ is a generic extension of $\Mx_{\ell}$ by $H^* = j_{U^*}(G)$.
\end{lemma}

\begin{lemma}\label{Lem:FM3}
    $\kappa$ is the only measurable cardinal in $V[G]$, and
    if $U^*$ is a normal measure on $\kappa$ in $V[G]$ whose associated ultrapower embedding is $j_{U^*} : V[G] \to M_{U^*}$, then $j_{U^*}\uhr V = j_E$ and $M_{U^*} = M_E[H^*]$ for some $M_E$-generic filter $H^* \subseteq j_E(\po)$.
\end{lemma}

By combining the statement of Lemma \ref{Lem:FM3} with the uniqueness part of Lemma \ref{Lem:FM1}, it follows that $H^* = G^*$ and that $U^* = U$, which concludes the proof of Theorem \ref{Thm:FMblueprint}. 
We proceed to prove Lemmas \ref{Lem:FM1}-\ref{Lem:FM3}.

\begin{proof}(Lemma \ref{Lem:FM1})\\
Write $j = j_E$, $M = M_E$, and $G = G_{\kappa} * G(\kappa)$,  where $G_\kappa \subseteq \po_\kappa$ is $V$-generic, and $G(\kappa) \subseteq \qo_\kappa = (\name{\qo}_\kappa)_{G_\kappa}$.
We first show that in $V[G]$, $j : V \to M$ extends to $j^1 : V[G_\kappa] \to M[G^*_{j(\kappa)}]$ where $G^*_{j(\kappa)} \subseteq j(\po_\kappa)$ is $M$-generic and contains $j``G_\kappa$. 
The uniform definability property of the FM-blueprint and the fact $H_{\kappa^{++}} \subseteq M$ imply that $j(\po)_{\kappa+1} = \po_{\kappa+1}$. In particular, $G$ is $M$-generic for $j(\po_\kappa)_{\kappa+1}$. 
Moving forward to extend $j^1$ to have domain $V[G] = V[G_\kappa * G(\kappa)]$, define 
 $$G^*(j(\kappa)) = \la j^1``G(\kappa) \ra = \{ q^* \in j^1(\qo_\kappa) \mid \forall q \in G(\kappa) \thinspace q^*\parallel j^1(q)\}.$$

To show $G^*_{j(\kappa)}$ is $M$-generic it suffices to verify that in $M[G]$, for every dense open subset $D \subseteq j(\po_\kappa)/G$ there is $p \in G_\kappa$ such that $(j(p)\setminus \kappa+1)_G \in D$.
Let $\name{D}$ be a $\po_{\kappa+1}$-name for $D$.
 Since $E$ is $(\kappa,\kappa^{++})$, $D = j(f)(\nu)$ for some $f : \kappa \to \power(\po_\kappa)$ and $\nu \in [\kappa,\kappa^{++})$, so that for every Mahlo cardinal $\alpha < \kappa$ and $\beta \in [\alpha,\alpha^{++})$,  $f(\beta)$ is a $\po_{\alpha+1}$-name of a dense open subset of $\po_\kappa/\po_{\alpha+1}$.
 For each Mahlo cardinal $\alpha < \kappa$, by the FM-blueprint assumption, $\po_\kappa/\po_{\alpha+1}$ is $<\beta_\alpha$-distributive, where $\beta_\alpha$ is the first Mahlo cardinal above $\alpha$. Therefore, $0_{\po_{\alpha+1}} \Vdash \bigcap_{\beta \in [\alpha,\alpha^{++})}f(\beta)$ is dense open. 
 Hence the set 
 $$D_\alpha = \{ \ p \in \po_\kappa \mid p\uhr (\alpha+1) \Vdash p\setminus (\alpha+1) \in \bigcap_{\beta \in [\alpha,\alpha^{++})}f(\beta)\ \}$$ is dense open in $\po_\kappa$, and 
 by the $\kappa$ Iteration-Fusion property of $\po_\kappa$ there is $p^* \in G$ and a club $C \subseteq \kappa$ so that for every Mahlo cardinal $\alpha \in C$, $p^*\uhr (\alpha+1) \Vdash p^*\setminus (\alpha+1) \in \bigcap_{\beta\in [\alpha,\alpha^{++})} f(\beta)$.
 Applying $j$, we see that $j(p^*)\uhr \kappa+1 \Vdash j(p^*) \setminus (\kappa+1) \in \name{D}$. Since $j(p^*)\uhr \kappa+1 \in G$, it follows that $j(p^*)\setminus (\kappa+1) \in D$.\\
 Having established $G^*_{j(\kappa)} \subseteq j(\po_{\kappa})$ is generic over $M$ and contains $j``G_\kappa$, we conclude $j$ extends to 
 $$j^1 : V[G_\kappa] \to M[G^*_{j(\kappa)}].$$ 
 
 Moving forward to extend $j^1$ to have domain $V[G] = V[G_\kappa * G(\kappa)]$ we would like to define $G^*(j(\kappa)) \subseteq j^1(\qo_\kappa)$ that is generic over $M[G^*_{j(\kappa)}]$. To this end, 
 we consider some properties of $\qo_\kappa$ and $j^1(\qo_\kappa)$ that follow from the assumption that $\qo_\kappa$ has the $\kappa$-Continuous-Fusion ($\kappa$-CF) property. By $\kappa$-CF, for every $q \in G(\kappa)$ and a sequence of dense open sets $\vec{D} = \la D_\alpha \mid \alpha < \kappa\ra$ there is a continuous chain of structures $\vec{X} = \la X_\alpha \mid \alpha < \kappa\ra$ and $q^* \in G(\kappa)$ that is a $\kappa$-Continuous-Fusion witness for $q,\vec{D},\vec{X}$.
As $|X_\alpha|<\kappa$ for all $\alpha < \kappa$ we have that $j(X_\alpha) = j[X_\alpha]$, and 
since the sequence $\vec{X}$ is continuous, we see that the $\kappa$-th structure $j(\vec{X})_\kappa$ of $j^1(\vec{X})$ is equal to 
$$
j^1(\vec{X})_\kappa = \bigcup_{\alpha<\kappa}j^1(X_\alpha) = j^1[\bigcup \vec{X}].
$$

Having $q^* \in G(\kappa)$ being a generic condition for $X = \bigcup \vec{X}$ we conclude that 
$$G_{j^1(\vec{X})_\kappa} = j^1[G(\kappa) \cap \bigcup \vec{X}]$$ 
is generic for $j^1(\vec{X})_\kappa$, and consists of conditions that are compatible with $j^1(q^*)$. 
Moreover, as $\kappa \in j(C \cap S)$ for $S = Reg \cap \kappa$, it follows from the elementarity of $j^1$ and the fact $q^*$ is a $\kappa$-Continuous-Fusion witness for $\vec{D},\vec{X}$ that the set
$$
\{ j^1(q^*) \} \cup G_{j^1(\vec{X})_\kappa}
$$
has an exact upper bound, denoted $j^1(q^*)_{G_{j(\vec{X})_\kappa}}$.
With this in mind, we 
define $G^*_{j(\kappa)} \subseteq j^1(\qo_\kappa)$ by

$$G^*(j^1(\kappa)) = \la \{ j^1(q^*)_{G_{j^1(\vec{X})_\kappa}} \mid q^*\in G_{\kappa} \text{ is a } \kappa-CF \text{ witness for some } \vec{D},\vec{X} \}\ra$$

Namely, $G^*(j^1(\kappa))$ is the set generated (i.e., closing under weaker conditions) by the set of conditions of the form $j^1(q^*)_{G_{j^1(\vec{X})_\kappa}}$, where $q^* \in G(\kappa)$, where $q^* \in G(\kappa)$ is $\kappa$-CF witness for some $\vec{X},\vec{D}$. 


Let us first verify that $G^*(j^1(\kappa))$ is a filter, i.e., that every two conditions of the form $j^1(q_1^*)_{G_{j^1(\vec{X})_\kappa}}$ and $j^1(q_2^*)_{G_{j^1(\vec{X})_\kappa}}$ are compatible, whenever  $q^*_1, q^*_2 \in G(\kappa)$ are $\kappa$-CF witnesses for pairs $\vec{X}^1 = \la X^1_\alpha\ra_{\alpha<\kappa},\vec{D}^1$ and $\vec{X}^2 = \la X^2_\alpha \mid \alpha < \kappa\ra,\vec{D}^2$, respectively. \\
By density, there is $q^* \in G(\kappa)$ that extends both $q^*_1,q^*_2$ and is a $\kappa$-CF witness for a sequence $\vec{X}$ (and a trivial $\vec{D}$) consisting of structure $$X_\alpha \elem (H_{|\qo|^+},\in,<_{wo},\qo,\vec{X}^1,\vec{X}^2).$$

By continuity of the sequences $\vec{X},\vec{X}^1,\vec{X}^2$, there is a club set of $\alpha < \kappa$ such that $X^1_\alpha,X^2_\alpha \subseteq X_\alpha$.
Applying $j^1$, we conclude  $j^1(q^*)$ extends both $j^1(q^*_1),j^1(q^*_2)$, and that 
$G_{j^1(\vec{X})_\kappa}$ contains both $G_{j^1(\vec{X^1})_\kappa}$ and $G_{j^1(\vec{X^2})_\kappa}$.
It follows that $j^1(q^*)_{j^1(\vec{X})_\kappa}$ is an upper bound of $\{ j^1(q^*_i)\} \cup G_{j^1(\vec{X^i})_\kappa}$ for both $i = 1,2$, and hence, is a common extension of their exact upper bounds $j^1(q^*)_{j^1(\vec{X^1})_\kappa}$, $j^1(q^*)_{j^1(\vec{X^2})_\kappa}$. \\

Next, we need to show $G^*(j^1(\kappa)) \subseteq j^1(\qo_\kappa)$ is generic over $M[G^*_{j(\kappa)}]$. Suppose $D^* \subseteq j^1(\qo_\kappa)$, $D^* \in M[G^*_{j(\kappa)}]$ is dense open, and let $\name{D^*}$ be a $j(\po_\kappa)$-name for $D^*$. Choosing a function $f^*$ and $\nu < \kappa^{++}$ such that $\name{D}^* = j(f^*)(\nu)$, we may assume that for every $\alpha < \kappa$,  $f^*(\alpha)$ is a $\po_\kappa$-name of a dense open subset of $\name{\qo}_\kappa$. Define for each $\alpha < \kappa$, $D^*_\alpha = \bigcap_{\beta \in [\alpha,\alpha^{++})} (f^*(\beta))_{G_\kappa}$. Since $\qo_\kappa$ is $\kappa$-distributive,  $D^*_\alpha$ is dense open. 
 Working in $V[G_\kappa]$ where $\qo_\kappa$ has the $\kappa$ C-Fusion property, there is a dense set of conditions $q^* \in \qo_\kappa$ for which there is a club $C \subseteq \kappa$, and an internally approachable continuous chain $\vec{X} = \la X_\alpha \mid \alpha < \kappa\ra$ of elementary substructures of $(H_{\kappa^{++}}[G_\kappa],\qo_\kappa)$ that witness the $\kappa$ C-Fusion property with respect to the sequence of dense sets $\la D^*_\alpha \mid \alpha < \kappa\ra$. 
 Let $X = \bigcup_{\alpha < \kappa} X_\alpha$
 and $j^1(\vec{X}) = \vec{X}^* = \la X^*_\alpha \mid \alpha < j(\kappa)\ra$. By the continuity of $\vec{X}^*$, we have $X^*_\kappa = j^1``X$. By the $\kappa$ C-Fusion property, $q^*$ is a generic condition for $X$. Moving to $V[G]= V[G_\kappa * G(\kappa)]$ we may assume $q^* \in G(\kappa)$ and conclude that $G(\kappa) \cap X$ is an $X$-generic set for $\qo_\kappa$. 
 Since $j\uhr X : X \leftrightarrow j``X = X^*_\kappa$ belongs to $M$ (as $|X| = \kappa$ and ${}^\kappa M \subseteq M$) and $G(\kappa) \in M[G^*_{j(\kappa)}]$, we conclude that $G_{X^*_\kappa} = j^1``(G(\kappa) \cap X)$ is an $X^*_\kappa$ generic set for $j^1(\qo_\kappa)$. Moreover, this is forced by $j^1(q^*)$ with respect to $\vec{X}^*$. It follows that $G_{X^*_\kappa} \cup \{ j^1(q^*)\}$ an exact upper-bound $q^*_{G_{X^*_\kappa}}$, which belongs to $D^*_\kappa \subseteq D^*$. We conclude that $G^*(j(\kappa)) \cap D^* \neq \emptyset$, and that 
 $G^* = G^*_{j(\kappa)} * G^*(j(\kappa))$ is a generic subset of $j(\po)$ over $M$, and contains $j`` G$.

 Next, we show $G^*$ is unique. Suppose  $H^* \subseteq j(\po)$ is $M$-generic and contains $j``G$.  First, $G_\kappa \subseteq j``G$ as $j(p) \uhr \kappa = p\uhr \kappa$ for all $p \in \po = \po_{\kappa+1}$. Therefore $H^*\uhr \po_\kappa = G_\kappa$. Second, the uniformly definable self coding properties of $\po$ guarantee $H^*\uhr \qo_\kappa$ must be  equal to $G(\kappa)$. Indeed, recall the sequences $\vec{q}^\kappa = \la q^\kappa_\tau \mid \tau < \kappa^{++}\ra$ and $\vec{S}^\kappa = \la S^\kappa_\tau \mid \tau < \kappa^{++}\ra$ which are used to code $\qo_\kappa$ (over both $V$ and $M$), if  $q^\kappa_\tau \in (H^*\uhr \qo_\kappa) \triangle G(\kappa)$ for some $\tau < \kappa^{++}$ then $V[G]$ and $M[H^*]$ will disagree on which of the two sets $S^{\kappa}_{2\tau}, S^{\kappa}_{2\tau+1}$ are not stationary. This is absurd as $M[H^*] \subseteq V[G]$.
 Having established that $H^*\uhr \po_{\kappa+1} = G$, it follows from the construction of $G^*$ that $G^*/G$ is generated by $j``G/G$. Since $j``G \subseteq H^*$ we must have $H^* = G^*$.
 \end{proof}

 \begin{proof}(Lemma \ref{Lem:FM2})\\
 For every $a \in M[G^*]$ there is a function $f : \kappa \to V^{\po}$ in $V$ and $\nu < \kappa^{++}$ such that $a = (j(f)(\nu))_{G^*}$. Therefore, to show that $a = j^*(g)(\kappa)$ for some $g :\kappa \to V[G]$ in $V[G]$, it suffices to verify that for every $\nu < \kappa^{++}$ there is $h : \kappa \to \kappa$ in $V[G]$ so that $\nu = j^*(h)(\kappa)$. To this end, let $\varphi_F(v,x)$ be the parameter-free formula from the uniformly definable requirement of the FM-blueprint. Since $H_{\kappa^{++}}$ is contained in $M = M_E$ then $\varphi_F(\kappa,x)$ defines the same $\po_{\kappa+1}$ sequence $\vec{f}^\kappa = \la f^\kappa_\nu \mid \nu < \kappa^{++}\ra \subseteq {}^\kappa \kappa$ of pairwise distinct functions in both $M[G^*]$ and $V[G]$.
 Similarly, for each Mahlo cardinal $\alpha < \kappa$, let $\vec{f}^\alpha = \la f^\alpha_\tau \mid \tau < \alpha^{++}\ra \subseteq {}^
\alpha \alpha$ be the sequence of pairwise distinct functions definable in $V[G]$ by $\varphi_F(\alpha,x)$.
 Working in $V[G]$, define for each $\nu < \kappa^{++}$ the function $h_\nu \in {}^\kappa \kappa$. For each Mahlo cardinal $\alpha < \kappa$, if $f^\kappa_\nu \uhr \alpha = f^\alpha_\tau$ for some $\tau < \alpha^{++}$ then $h_\nu(\alpha) = \tau$. Otherwise, $h_\nu(\alpha) = 0$. 
 For  $\nu < \kappa^{++}$, as $j^*(f^\kappa_\nu) \uhr \kappa = f^\kappa_\nu$, $j^*(h_\nu)(\kappa) = \nu$. 
 \end{proof}

\begin{proof}
(Lemma \ref{Lem:FM2.5})\\
Let $U^*$ be a measure in $V[G]$ on a measurable cardinal $\kappa^*$ (not necessarily normal). Denote its ultrapower embedding by $j_{U^*} : V[G] \to M_{U^*}$. 
Having $V = L[\E]$ be a fine-structural extender model we have $M_{U^*} = L[j_{U^*}(\E)][j_{U^*}(G)]$. I.e., it is a generic extension of 
its core model $L[j_{U^*}(\E)]$ by a generic filter $j_{U^*}(G) \subseteq j_{U^*}(\po)$. Denote $L[j_{U^*}(\E)]$ by $M^*$ and $j_{U^*}(G)$ by $H^*$.
Moreover, by a Theorem of Schindler (\cite{schindler2006iteratesofthecoremodel})  the restriction $\jx := j_{U^*}\uhr V  : V \to \Mx$ is an iterated ultrapower embedding of $V$ which is uniquely determined by applying comparison of $V$ and $\Mx$. 
Having, $M_{U^*}^\omega \subseteq M_{U^*}$ and $i(\po)$ being $\omega_1$-distributive means $(\Mx)^\omega \subseteq \Mx$.  Hence, the iterated ultrapower of $V$ that generates $\Mx$ must have finite length $\ell < \omega$, namely
 $$\jx  = \jx_{\ell-1,\ell} \circ \jx_{\ell-2,\ell-1}\circ \dots \circ \jx_{0,1}$$
    where $\ell < \omega$ and for each $i < \ell$, 
    $$\jx_{i,i+1} = j^{\Mx_i}_{
    \Ex_i} : \Mx_i \to \Mx_{i+1} \cong \Ult(\Mx_i,\Ex_i)$$ is an ultrapower embedding by an extender $\Ex_i \in \Mx_i$ ($\Mx_0 = V$) is an extender in $\Mx_i$ with critical point $cp(\Ex_i) = \kappax_i$ and height $ht(\Ex_i) \leq ((\kappax_i)^{++})^{\Mx_i}$, 
    $\kappa = \kappax_0 < \kappax_1 < \dots < \kappax_{\ell-1}$.
\end{proof}

\begin{proof}(Lemma \ref{Lem:FM3})\\

The statement of the lemma follows from the next two claims. 
\begin{claim}\label{Claim:F_0isE}
    $\Ex_0 = E$. In particular $\kappa^* = \kappa_0 = \kappa$ is the only measurable cardinal in $V[G]$.
\end{claim}
\begin{proof}(Claim \ref{Claim:F_0isE})\\
    The minimality assumption of $V = L[\E]$ being a minimal model for the property of $\kappa$ carrying a $(\kappa,\kappa^{++})$-extender implies $\kappa$ is the maximal measurable cardinal in $V$. Hence $\kappa_0 \leq \kappa$. Clearly $\kappa_0$ is a nontrivial iteration stage of $\po$ and so $2^{\kappa_0} = \kappa_0^{++}$ in $V[G]$.\\
    
    Suppose that $\Ex_0 \neq E$. Then $ht(\Ex_0) < \kappa_0^{++}$ which implies $(\kappa_0^{++})^{M_1} < j_{\Ex_0}(\kappa_0) < \kappa_0^{++}$. 
    Since $(\kappa_0^{++})^{M_1} = (\kappa_0^{++})^{N}$, and $M_{U^*} = N[H^*] \models 2^{\kappa_0} = (\kappa_0)^{++}$, there is a bijection $f \in N[H^*]$, 
    $f : (\kappa_0^{++})^N \leftrightarrow \power(\kappa_0)^{M_{U^*}}$. But $\power(\kappa_0)^{M_{U^*}}= \power(\kappa_0)^{V[G]}$, and so the last contradicts the fact $|\power(\kappa_0)|^{V[G]} = \kappa_0^{++}$.
    This conclude the proof of Claim \ref{Claim:F_0isE}.
\end{proof}

\begin{claim}\label{Claim:ellEq1}
    If $U^*$ is normal then $\ell = 1$. 
\end{claim}
\begin{proof}(Claim \ref{Claim:ellEq1})\\
Suppose otherwise. Let 
$$j_{1,\ell} = j_{\ell-1,\ell} \circ j_{\ell-2,\ell-2}\circ \dots \circ j_{2,1}: M_1 \to M_{\ell}.$$ 
Then $j_{1,\ell}$ is not trivial and $cp(j_{1,\ell}) = \kappa_1 > \kappa_0^{++} = \kappa^{++}$.
Since $U^*$ is normal there is $f : \kappa \to \kappa$ in $V[G]= V[G_\kappa * G(\kappa)]$ such that $\kappa_1 = j_{U^*}(f)(\kappa)$. 
By Lemma \ref{Lem:IAFusionCombinatorics} there is a function $F \in V$ with $\dom(F)= \kappa$ and $|F(\alpha)| \leq \alpha^{++}$ for all inaccessibles $\alpha < \kappa$, a $\po$-name of a club $\name{C}^*$, and a condition  $p^* \in G$ such that 
$p^* \Vdash \forall \alpha \in \name{C}^*. \ \name{f}(\check{\alpha}) \in \check{F}(\check{\alpha})$. 
 Applying $j_U^*$ we see that 
 $$j_{U^*}(p^*) \Vdash \left(\check{\kappa} \in j_{U^*}(\name{C}^*)\right) \implies \left(\kappa_1 \in j_{U^*}(F)(\kappa)\right).$$ 
 Having $j_{U^*}(p^*) \in H^*$, and $\kappa \in j_{U^*}(C^*)$ as $C^* = \name{C}^*_{G}$ is a club and $\cp(j_{U^*}) = \kappa$, we conclude that
 $$\kappa_1 \in j_{U^*}(F)(\kappa) = j_{1,\ell}(j_{0,1}(F))\left( j_{1,\ell}(\kappa)\right) = j_{1,\ell}(j_{0,1}(F)(\kappa)) = j_{1,\ell}``[j_{0,1}(F)(\kappa)].
 $$
 This is absurd as $\kappa_1$ is the critical point of $j_{1,\ell}$ and cannot belong to a pointwise image of $j_{1,\ell}$.
 This conclude the proofs of Claim \ref{Claim:ellEq1} and Lemma \ref{Lem:FM3}, in turn.
\end{proof}
\end{proof}

\subsection{The Extended ``Kunen-Like" Blueprint}

\begin{definition}\label{Def:<_sing}
    Let $\kappa$ be an uncountable regular cardinal. Let $<^\kappa_{Sing}$ be the order relation on ${}^\kappa \kappa$ defined by  
     $f <^\kappa_{Sing} g$ if and only if there is a club $C \subseteq \kappa$ such that $f(\alpha) < g(\alpha)$ for every singular ordinal $\alpha \in C$. 
\end{definition}

Let us list some basic properties of $<^\kappa_{Sing}$.
\begin{lemma}\label{Lem:<*Basics}${}$
    \begin{enumerate}
        \item For every uncountable regular cardinal $\kappa$, $<^\kappa_{Sing}$ is a well-founded pre-order on ${}^\kappa \kappa$.

        \item Let  $M_1,M_2$ be transitive classes of $ZFC$  such that
        \begin{itemize}
            \item $\kappa$ is a regular uncountable cardinal in both, 
            \item $M_1,M_2$ agree on cofinalities of cardinal $\alpha < \kappa$,
            \item  $\power(\kappa)^{M_1} \subseteq \power(\kappa)^{M_2}$.
        \end{itemize}  
        Suppose that in $M_1$, there is an increasing sequence in $<^\kappa_{Sing}$ of length $\rho_1 \in On$, then    $\rho_1 \leq  \left(rank(<^\kappa_{Sing})\right)^{M_2}.$
    \end{enumerate}
\end{lemma}
\begin{proof}${}$
    \begin{enumerate}
        \item This is an immediate consequence of the fact that the club filter on a regular uncountable cardinal $\kappa$ is $\sigma$-complete.

        \item 
        $\power(\kappa)^{M_1} \subseteq \power(\kappa)^{M_2}$ implies that every club $C \subseteq \kappa$ in $M_1$ belongs to $M_2$. Since $M_1,M_2$ agree on cofinalities of cardinal $\alpha < \kappa$, $(C \cap Sing)^{M_1} = (C \cap Sing)^{M_2}$. It is therefore clear that for every $f,g \in {}^\kappa \kappa \cap M_1$, $M_1 \models f <^\kappa_{Sing} g$ implies $M_2 \models f <^\kappa_{Sing} g$ as well. 
It follows by a straightforward induction on ordinals $\rho$, that for every $f \in {}^\kappa \kappa \cap M_1$ with $\rho = rank_{<^\kappa_{Sing}}(f)^{M_1}$ that $rank_{<^\kappa_{Sing}}(f)^{M_2} \geq \rho$.
As the assumption clearly implies $rank^{M_1}(<^\kappa_{Sing}) \geq \rho_1$, we conclude that  $rank^{M_2}(<^\kappa_{Sing}) \geq \rho_1$. 
    \end{enumerate}
\end{proof}

\begin{definition}[``Kunen-like" blueprint]\label{Def:KLblueprint}${}$\\
    Let $\kappa$ be a measurable cardinal and  $\po = \la \po_\alpha,\name{\qo}_\alpha\mid \alpha \leq \kappa\ra$ an iteration of length $\kappa+1$. The ``Kunen-like" blueprint includes the assumptions of the Friedman-Magidor blueprint \ref{Def:FM-blueprint}, with the following two additional assumptions:
    \begin{enumerate}[label = KL\arabic*.]
            \item For every Mahlo cardinal $\alpha \leq \kappa$, $\name{\qo}_\alpha$ adds an increasing sequence in $<^\alpha_{Sing}$ of length $\alpha^{++}$,
            
            \item $\name{\qo}_\alpha$ is self coding, by sequences $\vec{q}^\alpha$ and $\vec{S}^\alpha$, where $\vec{S}^\alpha = \la S^\alpha_\tau \mid \tau < \alpha^{++}\ra$ consists of almost disjoint stationary subsets of $\alpha^+ \cap Cof(\rho_\alpha)$, for some regular cardinal $\rho_\alpha < \alpha$.
    \end{enumerate}
\end{definition}

\begin{theorem}\label{Thm:KLblueprint}
Suppose that $V = L[\E]$ is a  fine-structural extender model which is minimal for the existence of a measurable cardinal $\kappa$ carrying a $(\kappa,\kappa^{++})$-extender $E$.
 Let $\po = \la \po_\alpha,\name{\qo}_\alpha \mid \alpha \leq \kappa\ra$ be a iterated forcing poset which satisfies the ``Kunen-like" blueprint. Then in a generic extension $V[G]$ by $G \subseteq \po$, $2^\kappa = \kappa^{++}$, $\kappa$ is the only measurable cardinal,  there is a unique normal measure $U$ on $\kappa$, and every $\sigma$-complete ultrafilter $U^*$ is Rudin-Kiesler isomorphic to a finite power $U^\ell$ of $U$.
\end{theorem}

Before turning to the proof of Theorem \ref{Thm:KLblueprint} let us describe the elementary embedding and generic filter associated to $U^{\ell}$ for $1 \leq \ell < \omega$.

\begin{remark}\label{Rmk:DsecribeU^ell}
    We give a precise construction of finite powers $U^\ell$, $\ell < \omega$ of the normal measure $U \in V[G]$ on $\kappa$ from the previous section. For this, we define by induction on $n$, ultrapower inner models $\NU_n$, elementary embeddings $\iU_{0,n} : \NU_0 \to \NU_n$.\\
    Set $\NU_0 = V[G]$ and 
    $$\iU_{0,1} = j^{\NU_0}_{U} : \NU_0 \to \NU_1 \cong \Ult(\NU_0,U).$$ 
    Denote the critical point $\kappa$ by $\kappaU_0$.
    For $n \geq 1$, given $\iU_{0,n}: \NU_0 \to \NU_n$ define 
    $$\iU_{n,n+1} = j^{\MU_n}_{\iU_{0,n}(U)} : \NU_n \to \NU_{n+1} \cong \Ult(\NU_n,\iU_{0,n}(U))$$ to be the ultrapower embedding, and set $\kappaU_n := \iU_{0,n}(\kappa)$ and 
    $\iU_{0,n+1} = j^{\NU_n}_{\iU_{0,n}(U)} \circ \iU_{0,n}$.\\
    
    \noindent 
    For $\ell \geq 1$ we can describe $U^\ell$ by
    $$U^\ell = \{ X \subseteq \kappa^{\ell} \mid \left(\thinspace \kappaU_0,\dots, \kappaU_{\ell-1}\thinspace \right) \in \iU_{0,\ell}(X)\}.$$
    Clearly, the ultrapower of $V[G]$ by $U^{\ell}$ is $\MU_{\ell}$, and the ultrapower embedding is $\iU_{0,\ell}$.\\
    
    \noindent
    For $n< \omega$, the construction of the map $\iU_{n,n+1} : \NU_n \to \NU_{n+1}$ as an (internal) ultrapower map of $\NU_n$ by $\iU_{0,n}(U)$. 
    This means that the description of this ultrapower as an extension a ground embedding in a generic extension by $\iU_{0,n}(\po)$.
    More specifically, we have
    \begin{itemize}
        \item $\kappaU_n$ is the only measurable cardinal in $\NU_n$, and its unique normal measure is $\iU_{0,n}(U)$.
        
        \item $\NU_n = \MU_n[G^{(n)}]$ is the generic extension of $\MU_n := L[\iU_{0,n}(\E)]$, is the ultrapower  by  $G^{(n)} = \iU_{0,n}(G) \subseteq \iU_{0,n}(\po)$.

        \item The restriction of the ultrapower map $j^{\NU_n}_{\iU_{0,n}(U)}$ to the ground $\MU_n$ is the ultrapower embedding of $\MU_n$ by its unique $(\kappaU_n,\kappaU_n^{++})$-extender $E_n := \iU_{0,n}(E)$.\\
        
        \item For each finite $\ell \geq 1$, the restriction of $\iU_{0,\ell} = j^{V[G]}_{U^\ell}$ to $V$ is the $\ell$-th iterated ultrapower of $V$ by $E$ (and its images). 
        Namely, given by the  iteration sequence $\la \MU_n, \jU_{n,m} \mid n \leq  m \leq \ell\ra$, where
        $$\MU_0 = V, \text{ and } \jU_{n,n+1} = j^{\MU_n}_{\jU_{0,n}(E)}: \MU_n \to \MU_{n+1} \cong \Ult(\MU_n,\jU_{0,n}(E)) \text{ for every } n < \ell.$$
        In particular, for each $n < \omega$, $E_n = \iU_{0,n}(E) = \jU_{0,n}(E)$.
        
        \item $G^{(0)} = G$ and for each $n < \omega$, 
        $G^{(n+1)} \subseteq \iU_{0,n+1}(\po)$ is generated by the set $G^{(n)} \cup j^{\MU_n}_{E_n}``[G^{(n)}]$. 
        Moreover, $\jU_{0,n+1}(\po_\kappa) = \po^{\MU_{n+1}}_{\kappaU_{n+1}} $ and 
        $G^{(n+1)}\uhr \jU_{0,n+1}(\po_\kappa)$ 
        is generated from 
        $G^{(n)} \cup j^{\MU_n}_{E_n}``[G^{(n)}\uhr \po^{\MU_n}_{\kappa_n}]$ (see the first part of the proof of Lemma \ref{Lem:FM1}), and 
        $G^{(n+1)} \uhr \qo^{\MU_{n+1}}_{\kappaU_{n+1}}$ is generated from $j^1_{E_n}``[G^{(n)}\uhr \qo^{\MU_n}_{\kappaU_n}]$, where $j^1_{E_n}$ is the extension of $j^{\MU_n}_{E_n}$ to $\MU_n[G^{(n+1)}\uhr \po^{\MU_{n+1}}_{\kappaU_{n+1}}]$ (Second part of the proof of Lemma \ref{Lem:FM1}).\\
        
        In particular we see that for each $n < \ell$
        \[
        G^{(n)} = G^{(\ell)}\uhr \po^{\MU_{\ell}}_{\kappa_n+1} \text{ and }
        \]
        \[
        \iU_{0,n}``[G\uhr \qo_\kappa] \text{ generates the generic } G^{(n)}\uhr\qo^{\MU_n}_{\kappaU_n} \text{ over } \MU_n[G^{(n)}\uhr \po^{\MU_n}_{\kappaU_n}].
        \]
    \end{itemize}
\end{remark}

\subsection{Proof of Theorem \ref{Thm:KLblueprint}}
In this subsection we prove Theorem  \ref{Thm:KLblueprint}.
Let $U^*$ be a $\sigma$-complete ultrafilter in $V[G]$. Since by Theorem \ref{Thm:FMblueprint}, $\kappa$ is the only measurable cardinal in $V[G]$, then $U^*$ is $\kappa$-complete. We first appeal to the analysis and notations from the proof of Lemma \ref{Lem:FM2}, whose first part (up to and including Claim \ref{Claim:F_0isE}) does not assume normality for $U^*$.
Let $j_{U^*} : V[G] \to M_{U^*}$ be the ultrapower embedding. Summarizing the relevant findings in Lemma \ref{Lem:FM2} we know $U^*$ and $j_{U^*}$ have the following properties:
\begin{itemize}
\item $j_{U^*}\uhr V = \jx_{0,\ell} : V \to N = \Mx_\ell$ is a an iterated ultrapower of $V$ of some finite length $\omega > \ell \geq 1$ by extenders $\Ex_i \in M_i$ whose critical points are $\kappax_i = cp(\Ex_i)$, for  $i < \ell$.

\item $\Mx_0 = V$ and $\Ex_0 = E$ by Claim \ref{Claim:F_0isE}. In particular $\kappa_0 = \kappa$.

\item for each $i < \ell$, 
$\jx_{i,i+1} : \Mx_i \to \Mx_{i+1}\cong \Ult(M_i,\Ex_i)$ by an extender $\Ex_i \in \Mx_i$ whose critical point is $\kappax_i = cp(\Ex_i) > ((\kappax_{i-1})^{++})^{\Mx_i}$.

\item $M_{U^*} = \Mx_{\ell}[H^*]$ where $H^* \subseteq \jx_{0,\ell}(\po)$ is generic over $\Mx_{\ell}$.
\end{itemize}

\begin{definition}
We refer to $\ell$ as the \emph{ground iteration length of} $U^*$ and further denote it by $\ell(U^*)$.    
\end{definition}

$\ell(U^*)$ clearly depends on $U^*$. For example if $U^* = U^{k}$ is a finite power of the normal measure $U$ then $\ell(U^*)  = k$. 

We establish the statement of Theorem \ref{Thm:KLblueprint} by showing that $\ell(U^*)$ identifies $U^*$.

\begin{proposition}\label{Prop:U*ISUell}
    For every $\kappa$-complete ultrafilter $U^*$ in $V[G]$, $U^* \cong_{RK} U^{\ell}$ where $\ell = \ell(U^*)$ is the ground iteration length of $U^*$.
\end{proposition}

We prove Proposition 
\ref{Prop:U*ISUell} by \textbf{induction on $\ell = \ell(U^*)$.}

The proof will be based on an analysis of $U^*$ and its ultrapower, and a comparison  with the description of $U^{\ell}$ and its ultrapower parameters, as given in Remark \ref{Rmk:DsecribeU^ell}.
The analysis of $U^*$ and its ultrapower will make extensive use of the following measures $U^*_i$
$1 \leq i \leq \ell$, derived from $j_{U^*}$. 

\begin{definition}\label{Def:U*_i}
    For $1 \leq i \leq \ell$, let $U^*_i$ be a $\kappa$-complete ultrafilter on $\kappa^{i}$ defined by 
    \[ U^*_i = \{A \subseteq \kappa^{i} \mid (\kappax_0,\dots,\kappax_{i-1}) \in j_{U^*}(A)\}\]

We denote the ultrapower of $V[G]$ with $U^*_i$ by $j_{U^*_i} : V[G] \to M_{U^*_i}$, and the induced embedding of $M_{U^*_i}$ in $M_{U^*}$ by 
$k_{U^*_i,U^*} : M_{U^*_i} \to M_{U^*}$, which is given by 
$$k_{U^*_i,U^*}([f]_{U^*_i}) = j_{U^*}(f)(\kappax_0,\dots\kappax_{i-1}).$$
    
\end{definition}

Therefore $U^*_i \leq_{RK} U^*$. For example, $U^*_1$ is the normal projection of $U^*$, which by Theorem \ref{Thm:FMblueprint} must be the unique normal measure $U$ on $\kappa$ in $V[G]$. \\

We proceed to proof the inductive step of the statement of Proposition \ref{Prop:U*ISUell}. 
Fix a $\kappa$-complete ultrafilter $U^*$ in $V[G]$ with $\ell(U^*) = \ell$, and suppose that the statement holds for all ultrafilters $\bar{U}$ with $\ell(\bar{U}) < \ell$.

Let $\Ex_i, \jx_{i,i+1} : \Mx_i \to \Mx_{i+1} \cong \Ult(\Mx_i,\Ex_i)$, be the normal iteration sequence that generate $j_{U^*}\uhr V : V \to \Mx_{\ell}$. 

The following lemma gathers the main properties of $j_{U^*}$, given in terms of  $\Ex_i,\Mx_i$, $U^*_i$ and $k_{U^*_i,U^*}$. These will be key in proving that $U^* \cong_{RK} U^\ell$.
    \begin{lemma}\label{Lem:KLMainList}
    The following holds for each $0 \leq i < \ell$
    \begin{enumerate}
        \item $\Ex_i = E_i$,\\
        
        \noindent
        Notice that having $E^*_j = E_j$ for all $j \leq i$ this implies that $ \Mx_{i+1} = \MU_{i+1}$, $\jx_{0,i+1} = \jU_{0,i+1}$, and $\kappax_j = \kappaU_j$ for $0 \leq j \leq i+1$. 
        Also, $\po^{\Mx_{\ell}}_{\kappax_{i+1} +1} = \po^{\MU}_{\kappax_{i+1}+1} = \jU_{0,i}(\po)$.
        
        \item $H^*\uhr \po^{\Mx_\ell}_{\kappax_{i}+1} = G^{(i)}$
        
        \item  $U^*_{i+1}= U^{i+1}$ for $i < \ell-1$ 

        \item $cp(k_{U^*_{i+1},U^*}) \geq \kappax_{i+1}$ for $i < \ell-1$.
    \end{enumerate}
    \end{lemma}

The proof of the Lemma will make use of the following notation.
\begin{notation}
For $i \leq \ell$, 
let $\la f^{\kappax_i,H^*}_{\tau} \mid \tau < ((\kappax_i)^{++})^{\Mx_\ell}\ra$ be the $H^*$-induced generic sequence of functions at stage $\kappax_i$ of the iteration (i.e., produced by the $H^*$ generic for $\qo^{\Mx_\ell}_{\kappax_i}$), which is definable from the  uniform definition $j_{U^*}(\varphi_F)(\kappax_i,x)$ (specified by blueprint assumptions).
\end{notation}

\begin{proof}(Lemma \ref{Lem:KLMainList})\\
We proof the statements in the lemma by induction on $i < \ell$.\\

Starting at $i = 0$, the FM-blueprint argument shows that $\Ex_0 = E_0$ (Claim \ref{Claim:F_0isE}).
Also, the proof of Lemma \ref{Lem:FM1}, 
shows that $H^*\uhr \po_{\kappax_0+1} = G = G^{(0)}$. 
This also implies that $H^*$ contains both  $\jx_{0,\ell}``G$ and $G = G^{(0)}$. 
Also,  since $U^*_1$ is a normal measure, the FM-blueprint argument (\ref{Claim:ellEq1}) implies $U^*_1 = U= U^1$. 


Next, by Lemma \ref{Lem:FM2}, we have that for every $\gamma < \kappa^{++} = ((\kappax_0)^{++})^{\Mx_\ell}$ there is a function $h \in {}^\kappa \kappa$ in $V[G]$ such that $\gamma = j_{U^*_1}(h)(\kappax_0)$. 
This, together with the proof of Lemma \ref{Lem:FM1}, which shows that
$G^{(1)}$ is generated by $G \vee \jx_{0,1}``G$, implies 
that for every  $\po$-name $\name{h}$  with $h = \name{h}_G$, there is some $p \in G$ such that 
$$p \vee \jx_{0,1}(p) \Vdash_{\jx_{0,1}(\po)} \jx_{0,1}(\name{h})(\check{\kappax_0} ) = \check{\gamma}.$$

Applying $\jx_{1,\ell} : \Mx_1 \to \Mx_\ell$, and recalling that its critical point is $\kappax_1 > ((\kappax_0)^{++})^{\Mx_1}> \gamma$, we get that
$$p \vee \jx_{0,\ell}(p) \Vdash_{\jx_{0,\ell}(\po)} \jx_{0,\ell}(\name{h})(\check{\kappax_0} ) = \check{\gamma}.$$

As $H^*$ contains $G \cup \jx_{0,\ell}``G$, we conclude that  in $M_{U^*} = \Mx_{\ell}[H^*]$, $j_{U^*}(h)(\kappax_0) = \gamma$.
It follows that $((\kappax_0)^{++})^{\Mx_{1}} \subseteq im(k_{U^*_1,U^*})$. 
Next, by standard iterated ultrapower arguments and the normality of the iteration leading to $\Mx_\ell$, we have that every $\gamma < \kappax_1$ is of the form $\jx_{0,\ell}(g)(\gamma_0)$ for some $\gamma_0 < ((\kappax_0)^{++})^{\Mx_1}$ and $g \in {}^\kappa \kappa$ in $V$. Combining the findings, we conclude that $\kappax_1 \subseteq im(k_{U^*_1,U^*})$, hence 
$\cp(k_{U^*_1,U^*}) \geq \kappax_1$.\\

This conclude the verification of the inductive assumptions for $i = 0$.\\

Suppose now that $1 \leq i < \ell$ is such that the inductive assumptions hold for all $0 \leq j \leq i-1$. 
In particular, we know that 
$U^*_i = U^i$ and $\cp(k_{U^*_{i},U^*}) \geq \kappax_i$. We proceed to verify the inductive statements for $i$.
\begin{claim}
    $\Ex_i = E_i$.
\end{claim}
\begin{proof}
Suppose otherwise. 
Since $E_i$ is the unique $(\lambda,\lambda^{++})$- extender in $\Mx_{i}$ (for any $\lambda)$ then $\Ex_i \neq E_i$ implies 
$$(2^{\kappax_i})^{M_{U^*}} = ((\kappax_i)^{++})^{\Mx_{\ell}}  < \jx_{i,i+1}(\kappax_i) <  ((\kappax_i)^{++})^{\Mx_{i}}.$$
As $rank^{M_{U^*}}(<^{\kappax_i}_{Sing}) < ((2^{\kappax_i})^+)^{M_{U^*}} = ((\kappax_i)^{+3})^{\Mx_{\ell}}  < \jx_{i,i+1}(\kappax_i)$, we get that $rank^{M_{U^*}}(<^{\kappax_i}_{Sing}) < ((\kappax_i)^{++})^{\Mx_{i}}$.\\
Let us see that this contradicts our inductive assumption for $i$. By the inductive assumptions, 
\begin{itemize}
    \item $cp(k_{U^*_i,U^*}) \geq \kappax_i$, which implies $\power(\kappax_i)^{M_{U^*_i}} \subseteq \power(\kappax_i)^{M_{U^*}}$, and

    \item $((\kappax_{i})^{++})^{M_{U^*_i}} = ((\kappax_{i})^{++})^{M_{U^i}} = ((\kappax_{i})^{++})^{\MU_i} = ((\kappax_{i})^{++})^{\Mx_i}$.
\end{itemize}
Since $M_{U^*_i}$ contains a (generic) $<^{\kappax_i}_{Sing}$-increasing sequence of functions in ${}^{\kappax_i}\kappax_i$, it follows from the last points and Lemma \ref{Lem:<*Basics} that 
$rank^{M_{U^*}}(<^{\kappax_i}_{Sing}) \geq ((\kappax_i)^{++})^{\Mx_{i}}$, which contradicts the first finding above.
\end{proof}

Next, we need to show that $G^{(i)} = H^*\uhr \po_{\kappax_i+1}$.
To see this, note that
 $k_{U^*_{i},U^*}``G^{(i)} \subseteq H^*$ and in particular $G^{(i)}\uhr \po_{\kappax_i}  \subseteq H^*\uhr \po_{\kappax_i}$. As the first set is already generic over $\Mx_{\ell}$, it must be that $G^{(i)}\uhr \po_{\kappax_i} = H^*\uhr \po_{\kappax_i}$.

It remains to show that $H^*$ and $G^{(i)}$ agree on the generic filter at stage $\kappax_i = \kappaU_i$ of the iteration. 

\begin{claim}
$H^*\uhr \qo_{\kappax_i} = G^{(i)}\uhr \qo_{\kappax_i}$.
\end{claim}
\begin{proof}
    As pointed out at the end of Remark \ref{Rmk:DsecribeU^ell}, 
    $\jU_{0,i}``[G\uhr \qo_\kappa]$ generates the generic filter $G^{(i)}\uhr \qo^{\MU_i}_{\kappaU_i}$ over $\MU_i[G^{(i)}\uhr \po^{\MU_i}_{\kappaU_i}]= \Mx_i[H^*\uhr \po^{\MU_i}_{\kappaU_i}]$. As $E_i = \Ex_i$, $\Mx_i$ and $\Mx_\ell$ agree on $H_{(\kappax_{i})^{++}}$. Thus $\jU_{0,i}``[G\uhr \qo_\kappa] = \jx_{0,i}``[G\uhr \qo_\kappa]$ also generates a generic for $\qo^{\Mx_i}_{\kappaU_i}$ over $\Mx_\ell[ H^*\uhr \po^{\Mx_\ell}_{\kappaU_i}]$. 
    It is therefore sufficient to prove that 
    $\jU_{0,i}``[G\uhr \qo_\kappa] \subseteq H^*\uhr \qo^{\MU_i}_{\kappaU_i}$. 
    $\jU_{0,i}``[G\uhr \qo_\kappa] \subseteq H^*\uhr \qo^{\MU_i}_{\kappaU_i}$.
    
    Following the notations of the Extended Blueprint \ref{Def:KLblueprint}, let $\vec{S}^\kappa = \la S^\kappa_\tau \mid \tau < \kappa^{++}\ra$ and $\vec{q}^\kappa = \la q^\kappa_\tau \mid \tau < \kappa^{++}\ra$ be the sequence of stationary sets and enumeration of $\qo_\kappa$, which are used by $\qo_\kappa$ to self code its generic filter (As described in Definition \ref{Def:BlueprintPrelim}).
    Similarly, let $\vec{S}^{\kappaU_i} = \jx_{0,i}(\vec{S}^\kappa)$ and $\vec{q}^{\kappaU_i} = \jx_{0,i}(\vec{q}^\kappa)$. be the sequences used for the self coding of $\qo^{\Mx_\ell}_{\kappaU_i}$ in $\Mx_\ell$. 
    By the blueprint assumptions for each $\tau< \kappa^{++}$, the stationary set $S^\kappa_\tau \subseteq \kappa^+$ consists of ordinals of cofinality $\rho_\kappa < \kappa$. Since $\jx_{0,i}(\rho_\kappa) = \rho_\kappa$, the same is true for $\jx_{0,i}(S^\kappa_\tau) = S^{\kappaU_i}_{\jx_{0,i}(\tau)}$. 
    Moreover, as the set $\jx_{0,i}``\kappa^+$ is unbounded in $(\kappaU_i^+)^{\Mx_\ell}$ and closed to limits of its  $<\kappa$-increasing sequences, it follows that
    
    $$V[G] \models S^\kappa_{\tau} \text{ is stationary} 
\iff V[G] \models S^{\kappaU_i}_{\jx_{0,i}(\tau)} \text{ is stationary} \iff
\Mx_\ell[H^*] \models S^{\kappaU_i}_{\jx_{0,i}(\tau)} \text{ is stationary}$$
For the direction $(\Longleftarrow)$ of the last equivalence, we use the fact that the stationary sets are coded in pairs, where each $S^{\kappa_i}_{ \nu }$ is stationary in $\Mx_\ell[H^*]$ if and only if $S^{\kappa_i}_{ \nu' }$ is not stationary there, for $\nu'$ which is either $\nu+1$ if $\nu$ is an even ordinal, or $\nu-1$ if $\nu$ is odd. Thus, the trivial direction $(\Longrightarrow)$ suffices to imply the equivalence.

    By the self coding properties of $\qo_\kappa$ and elementary properties of $\jx_{0,i}$, it follows that for every $\tau < \kappa^{++}$, $q^\kappa_\tau \in G\uhr \qo_\kappa$ if and only if $\jx_{0,i}(q^\kappa_\tau) \in H^*$. The claim follows. 
\end{proof}

Moving on, we would like to show that 
\begin{claim}\label{Claim:U*_i=U^i}
$U^*_{i+1} = U^{i+1}$ if $i + 1 < \ell$.   
\end{claim}
\begin{proof}
Let $j_{U^*_{i+1}} : V[G] \to M_{U^*_{i+1}}$ denote the ultrapower embedding and model of $V[G]$ by $U^*_{i+1}$.\\
 Our strategy for showing that $U^*_{i+1} = U^{i+1}$ is to prove that the restriction of $M_{U^*_{i+1}}$ to its ground model is $\Mx_{i+1}$. Once this is established, we can  use the fact $\Mx_{i+1}$ is  iterated ultrapower of $V$ by the extenders $E_j$ $0 \leq j \leq i$, and therefore the ground iteration length $U^*_{i+1}$ is $\ell(U^*_{i+1}) = i+1$. And as we assume $i+1 < \ell$, it will immediately follow from our main inductive hypothesis that $U^*_{i+1} \cong_{RK} U^{i+1}$. This means that $j_{U^*_{i+1}} = j_{U^{i+1}} : V[G] \to M_{U^*_{i+1}} = M_{U^{i+1}}$. Moreover, as $[id]_{U^*_{i+1}} = (\kappax_0,\dots,\kappax_{i}) = (\kappaU_0,\dots,\kappaU_i) = [id]_{U^{i+1}}$ we conclude that in fact, $U^*_{i+1} = U^{i+1}$.\\
 
\noindent To identify the ground of $M_{U^*_{i+1}}$ we study its image by the elementary embedding 

$$k_{U^*_{i+1},U^*} : M_{U^*_{i+1}} \to M_{U^*} = \Mx_\ell[H^*],$$
with the observation that the ground of $M_{U^*_{i+1}}$ is the transitive collapse of its pointwise image. 
By the elementarity of $k_{U^*_{i+1},U^*}$, the image of the ground of $M_{U^*_{i+1}}$  is intersection $\Mx_{\ell} \cap im(k_{U^*_{i+1},U^*})$. Therefore, in order to identify the ground of $M_{U^*_{i+1}}$ with $\Mx_{i+1}$, it suffices to prove 
    $$\Mx_\ell \cap im(k_{U^*_{i+1},U^*}) = im(\jx_{i+1,\ell}).$$
The left-hand-side class
    $\Mx_\ell \cap im(k_{U^*_{i+1},U^*})$ consists of all sets of the form $j_{U^*}(g^*)(\kappax_0,\dots\kappax_i)$, where $g^* \in V[G]$ is a function from $\kappa^{i+1}$ to $V$.
    The right-hand-side class $im(\jx_{i+1,\ell})$ consists of all sets of the form $\jx_{0,\ell}(g)(\gamma_0,\dots,\gamma_i)$ where $g : \kappa^{i+1} \to V$ is a function in $V$, and for each $j \leq i$, 
    $\gamma_j < ((\kappax_j)^{++})^{\Mx_{i+1}}$.
The following two subclaims prove the mutual containment between the two classes. 
    \begin{subclaim}\label{Claim:U*(i+1)capturingM(i+1)}
    For every $x \in im(\jx_{i+1,\ell})$ there is 
    a function  $h : \kappa^{i+1} \to \kappa$ in $V[G]$ such that $x = j_{U^*}(h)(\kappax_0,\dots,\kappax_{i})$. 
    \end{subclaim}

    Since every set $x \in im(\jx_{i+1,\ell})$ is of the form $\jx_{0,\ell}(f)(\gamma_0,\dots,\gamma_i)$ where $f : \kappa^{i+1} \to V$ is a function in $V$, and 
    $\gamma_j \leq ((\kappax_j)^{++})^{\Mx_\ell}$ for every $j \leq i$, 
    it suffices to show that for every $\gamma < ((\kappax_i)^{++})^{\Mx_\ell}$
    there is a function $h^*\in M_{U^*_i}$ so that $\gamma = k_{U^*_i,U^*}(h^*)(\kappax_i)$. To show this, 
    we use the inductive assumptions for $j = i-1$, which include  $U^*_i = U^i$,  $H^*\uhr \po_{\kappax_i+1} = G^{(i)}$, and $\cp(k_{U^*_i,U^*}) \geq \kappax_i$.\\
    Consider the function $f^{\kappax_i,G^{(i)}}_{\gamma} \in  {}^{\kappax_i} \kappax_i \cap M_{U^*_i}$, which is the $\gamma$-th element from the $G^{(i)}$-induced generic sequence of function. We see that 
    $$k_{U^*_i,U^*}(f^{\kappax_i,G^{(i)}}_{\gamma}) \uhr \kappax_i = f^{\kappax_i,G^{(i)}}_{\gamma} = f^{\kappax_i,H^*}_{\gamma}.$$
    It follows that $\gamma = k_{U^*_i,U^*}(h^*)(\kappax_i)$ where $h^* \in {}^{\kappax_i} \kappax_i \cap M_{U^*_i}$ is defined mapping a Mahlo cardinal $\alpha < \kappax_i$ to $h^*(\alpha) < \alpha^{++}$ for which $f^{\kappax_i,G^{(i)}}_{\gamma} \uhr \alpha = f^{\alpha,G^{(i)}}_{h^*(\alpha)}$, if such an ordinal $h^*(\alpha)$ exists.
    This conclude the proof of Subclaim \ref{Claim:U*(i+1)capturingM(i+1)}.

    \begin{subclaim}\label{Claim:M(i+1)capturingU*(i+1)}
    
         For every set $x \in im(k_{U^*_{i+1},U^*}) \cap \Mx_{\ell}$ there is a function $g \in V$, $g :\kappa^2 \to V$, and $\gamma < ((\kappax_i)^{++})^{\Mx_\ell}$ to that $x = j_{U^*}(g)(\kappax_i,\gamma)$.    \end{subclaim}

         Let $f = \name{f}_G \in V[G]$ be a function from $\kappa^{i+1}$ to $V$ so that $x = j_{U^*}(f)(\kappax_0,\dots,\kappax_i)$.
             By Lemma \ref{Lem:IAFusionCombinatorics} applied to the name $\name{f}$, there are $F \in V$, $F : \kappa \to V$ and a club $C^* \in \kappa$ in $V[G]$,  such that $|F(\alpha)| \leq \alpha^{++}$ for every inaccessible cardinal $\alpha < \kappa$, and for every $\alpha\in C^* \cap Reg$ and $\vec{\beta} \in \alpha^{i}$, 
             $f(\vec{\beta} \fr \{\alpha\}) \in F(\alpha)$.
             Having ${U^*_i} = U^i$ we get $j_{U^*_i}(C^*) \in M_{U^*_i}$ is a club in $\kappax_{i}$. Moreover,  $\cp(k_{U^*_i,U^*}) \geq \kappax_{i}$ implies that $j_{U^*}(C^*)\cap \kappax_{i} = j_{U^*_i}(C^*)$,  is unbounded in $\kappax_i$. Hence, $\kappax_{i} \in j_{U^*}(C^*)$. 
             We conclude that 
             in $M_{U^*} = \Mx_\ell[H^*]$, the set 
             $Y = j_{U^*}(F)(\kappax_i) = \jx_{0,\ell}(F)(\kappax_i) \in  \Mx_\ell$ has size $|Y|^{\Mx_\ell} \leq ((\kappax_i)^{++})^{\Mx_\ell}$, and 
             $$ x = j_{U^*}(f)(\kappax_0,\dots,\kappax_i) \in Y.$$
             Letting $\rho_Y \in \Mx_\ell$, $\rho_Y : ((\kappax_i)^{++})^{\Mx_\ell} \to Y$ be the minimal surjection in the canonical well-ordering of $\Mx_\ell$, and that $x = \rho_Y(\gamma)$ for some $\gamma < ((\kappax_i)^{++})^{\Mx_\ell}$.
             We conclude that $\rho_Y$ is definable from $Y$, and hence has the form $\jx_{0,\ell}(h)(\kappax_i)$ for some $h \in V$.
             It follows that $x = j_{U^*}(g)(\kappax_i,\gamma)$ where 
             $g \in V$ is defined by $g(\alpha,\beta) = h(\alpha)(\beta)$, where $h(\alpha) = \rho_{F(\alpha)} : \alpha^{++} \to V$ is a surjection from $\alpha^{++}$ onto $F(\alpha)$, which is minimal in the canonical well ordering of $V= L[\E]$ and $\beta < \alpha^{++}$. 
         This conclude the proof of Subclaim \ref{Claim:M(i+1)capturingU*(i+1)}, and Claim \ref{Claim:U*_i=U^i}.
         \end{proof}

Finally, it remains to to show that

\begin{claim}
$cp(k_{U^*_{i+1},U^*})\geq \kappax_{i+1}$ if $i+1 < \ell$.
\end{claim}
\begin{proof}

By Subclaim \ref{Claim:U*(i+1)capturingM(i+1)}, 
$im(k_{U^*_{i+1},U^*})$ contains all ordinals of the form $j_{U^*}(f)(\gamma_0,\dots,\gamma_i)$ where $f : \kappa^{i+1} \to On$ is a function in $V$, and $\gamma_0,\dots,\gamma_i < ((\kappax_i)^{++})^{\Mx_\ell}$.
By the normality of the iteration leading to $\Mx_\ell$, every $\gamma < \kappax_{i+1}$ is of the form $\gamma = j_{U^*}(g)(\gamma_0,\dots,\gamma_i)$, where $\gamma_0< \dots< \gamma_i < ((\kappax_i)^{++})^{\Mx_\ell}$. 
Hence $\kappax_{i+1} \leq 
cp(k_{U^*_{i+1},U^*})$.
\end{proof}

This concludes the proof of Lemma \ref{Lem:KLMainList}.
\end{proof}

Having shown Lemma \ref{Lem:KLMainList}, we finish the proof of Proposition \ref{Prop:U*ISUell} by showing that $U^*_{\ell} \cong_{RK} U^*$.  By Lemma \ref{Lem:KLMainList}, $\Ex_i = E_i$ for all $i < \ell$. Therefore $\Mx_\ell = \MU_{\ell}$ is the ultrapower of $\Mx_{\ell-1} = \MU_{\ell-1}$ by $\Ex_{\ell-1} = E_{\ell-1}$, whose associated elementary embedding is $\jx_{\ell-1,\ell} : \Mx_{\ell-1} \to \Mx_{\ell} \cong \Ult(\Mx_{\ell-1},E_{\ell-1})$.
Also, by the proof of Claim \ref{Claim:U*_i=U^i} we see that 
\begin{itemize}
    \item $\Mx_{\ell-1}$ is the ground of $M_{U^*_{\ell-1}} = M_{U^{\ell-1}} = \Mx_{\ell-1}[G^{(\ell-1)}]$,
    
    \item   $k_{U^*_{\ell-1},U^*}(G^{(\ell-1)}) \subseteq H^*$, and

    \item  $k_{U^*_{\ell-1},U^*} : \MU_{\ell-1}[G^{(\ell-1)}] \to \Mx_{\ell}[H^*]$  is an extension of the embedding $\jx_{\ell-1,\ell}$, which is the ultrapower embedding of $\Mx_{\ell-1}$ by its extender $E_{\ell-1}$.
\end{itemize}

It follows from the uniqueness assertion of Lemma 
    \ref{Lem:FM1} that $H^* = G^{(\ell)}$, and from \ref{Lem:FM3} that 
    $k_{U^*_{\ell-1},U^*}$ is the ultrapower embedding of $M_{U^*_{\ell-1}} = M_{U^{\ell-1}}$ by its unique normal measure $j_{U^{\ell-1}}(U)$. Therefore, it is clear from the description of $U^\ell$ from the ultrapower by $U^{\ell-1}$ followed by the ultrapower with the image of $U$ that $M_{U^*} =M_{U^\ell}$ and that $j_{U^*} = j_{U^{\ell}}$.
    Using Remark \ref{Rmk:RKisom} we conclude that $U^* \cong_{RK} U^\ell$. This concludes the proof of Proposition \ref{Prop:U*ISUell}, and in turn, the proof of Theorem 
\ref{Thm:KLblueprint}.

\newpage
\section{Thin stationary sets in $L[\E]$} \label{Section:FineStucture}

\subsection{Introduction}

Our implementation of the self-coding mechanism of Definition \ref{Def:KLblueprint} will be based on a careful choice of sequences $\la \mathcal{S}^\alpha \mid \alpha \leq \kappa \text{ is Mahlo}\ra$ 
of non-reflecting stationary sets $\mathcal{S}^\alpha \subseteq \alpha^+ \cap cof(\omega_1)$. 
The coding part of the stage $\alpha$ poset $\qo_\alpha$ will add a closed unbounded subset of $\alpha^+$ that is disjoint from some stationary subset of $\mathcal{S}^\alpha$ (depending on the information to be coded). 
The fact $\mathcal{S}^\alpha$ is non-reflecting suffices to secure that the coding poset at each individual $\qo_\alpha$ is $\alpha^+$-distributive. However, an iteration of posets $\qo_\alpha$, $\alpha < \kappa$, each $\qo_\alpha$ being $\alpha^+$-distributive may collapse cardinals, let alone satisfy stronger preservation properties such as the $\kappa$-Iteration-Fusion property needed for the blueprint. 

Starting from a strong limit cardinal $\kappa$, the goal of this section is to construct non-reflecting stationary sets $\mathcal{S}^\alpha \subseteq \alpha^+ \cap cof(\omega_1)$ for Mahlo cardinals $\alpha < \kappa$ so the non-stationary support iteration of posets adding disjoint closed unbounded subsets from each $\mathcal{S}^\alpha$, has distributive tails and has the $\kappa$ Iteration Fusion property (Definition \ref{Def:FM-blueprint}(3)). The construction is motivated by a recent result of Foreman, Magidor and Zeman about Games with Filters with Strong Ideals, that is announced in \cite{ForMagZem} and expected to appear in a sequel paper. The forcing  described \cite{ForMagZem} is based on an Easton support iteration of similar club shooting posets. Our use of an iteration that allows unbounded supports, and the requirement of the $\kappa$-Iteration-Fusion property, leads to a somewhat different analysis of the relevant fine-structural theory (most notably, the results around Proposition \ref{Lemma:NowhereStationary}). Specifically, our construction is based on a correspondence between certain sequences 
$$\la \eta_\alpha \mid \alpha \in Z\ra$$
of points $\eta_\alpha \in \mathcal{S}^\alpha$, where the domain  $Z \subseteq \kappa$ is nowhere stationary.\footnote{I.e., for every $\beta \leq \kappa$, $Z \cap \beta$ is not stationary in $\beta$.}
The sets $\mathcal{S}^\alpha$ are constructed in subsection \ref{subsec:FineStructurePrimer}, and the correspondence mentioned above is developed in subsection \ref{subsec:Correspondence}.
For the most part, our construction is based the fine structure theory and notations from \cite{SchindlerZeman2009FineStruture} and 
\cite{Steel-HB}. 
To simplify the presentation and avoid certain technical details for the benefit of readers who are not experts in fine structure, we make use of some ad-hoc notions of good and very good points, given in subsection \ref{subsec:GoodAndVeryGood}.

\subsection{A Fine structure primer}\label{subsec:FineStructurePrimer}


For the most part, we make use of standard fine structure theory and notations following the handbook chapters by Schindler-Zeman \cite{SchindlerZeman2009FineStruture} and Steel \cite{Steel-HB}. 

For a premouse $M = (J^{E^M}_\tau,E^M,F^M)$ we denote the $n$-th projectum by $\rho^M_n$, the $n$-th standard parameter by $p^M_n$, and the $n$-th reduct by $M^n$.
Recall that, for every $n<\omega$, the $n$-th projectum, parameter, standard code and reduct are defined inductively:
\begin{itemize}
    \item The $(n+1)$-th projectum of $M$ is the least ordinal $\rho = \rho_{n+1}^{M}$ such that there exists a new $\Sigma_1$-definable subset of $\rho$ over the $n$-th reduct, $M^{(n)}$, of $M$. We say that such a subset is $r\Sigma^{(n)}_1$-definable.
    \item The $(n+1)$-th standard parameter, $p_{n+1}^{M}$, is the lexicographically-least sequence 
    $$\la p(0), \ldots,p(n+1) \ra$$ of finite sets of ordinals, such that for every $i\leq n$, $p(i+1)\in \left[ \rho^{M}_{i+1}, \rho^{M}_i \right)^{<\omega}$ is a parameter used to define a new $r\Sigma^{(i+1)}_1$-subset of $\rho^{M}_{i+1}$ 
    \item The $(n+1)$-th reduct is the structure $ M^{(n+1)} = ( H^{M}_{\rho}, A^{M}_{n+1} )$, where  $\rho = \rho^{M}_{n+1}$ and $A^{M}_{n+1}$ is the standard code associated with the standard parameter $p^{M}_{n+1}$. Namely, the set consisting of pairs $\la i ,x \ra$ where $i$ is the Godel number of a $\Sigma^{(n)}_1$ formula $\varphi$, $x\in H^{M}_{\rho}$ and $M^{(n)}\vDash \varphi(x,p^{M}_{n+1})$.
\end{itemize}
Given a premouse $M$, we shall denote $\rho_{\omega}(M) = \min\{ \rho_{n}(M) \colon n<\omega \}$. 

\begin{remark} \label{Rmk:sigma1Skolemfunction} 
Assume that $M$ is a premouse. Let us elaborate on the definition of the Skolem functions $h^{M}_{n+1} \colon \omega \times (M^{(n+1)})^{<\omega}\to M$. 
\begin{enumerate}
    \item We begin by recalling the definition of the first Skolem function. The $\Sigma_1$-definition describing $h^{M}_1(i,x) = y$ for $i < \omega$ and $x,y \in M$, has the form: 
\begin{align*}
\label{Definition of the Sigma1 Skolem function}
 \exists z (\  &(z = S^{\E}_\gamma \text{ for some ordinal } \gamma) \wedge (x,y \in z) \wedge\\
&(z \models y \text{ is constructibly minimal with } \varphi^i_0(x,y))   \  )   
\end{align*}
where:
\begin{itemize}
    \item $\varphi^i_0$ denotes the $\Delta_0$ component in the $\Sigma_1$-formula coded by $i$, which is $\exists w \varphi^i_0(x,w)$.
    \item The hierarchy $\la S^{\E}_{\alpha} \colon \alpha\in Ord \ra$ refines $\la J^{\E}_{\alpha} \colon \alpha\in Ord \mbox{ and is limit} \ra$. Each $S^{\E}_{\alpha+1}$ is obtained from $S^{\E}_{\alpha}$ by taking $S^{\E}_{\alpha+1} = \bigcup_{ f\in F } f^{''}((S^{\E}_\alpha \cup  \{ S^{ \E}_\alpha   \} )^{n_f})$, where $F$ is a finite set of functions, each $f\in F$ in an $n_f$-ary function for some $n_f<\omega$, and $F$ forms a basis for all the Rudimentary functions (including one for additional predicates such as $\E$). For $\alpha$ limit, $S^{\E}_{\alpha} = \bigcup \{S^{\E}_{\beta} \colon \beta<\alpha\}$ is Rudimentary closed of the same height of $J^{\E}_\alpha$, and thus it is the same structure as $J^{\E}_{\alpha}$.
\end{itemize}
\item Let us inductively define $h^{M}_{n+1} \colon \omega \times (M^{(n+1)})^{<\omega}\to M$. By coding sequences in $\omega^{<\omega}$ as a single natural number, it's enough to define $h^{M}_{n+1} \colon \omega^{<\omega} \times (M^{(n+1)})^{<\omega}\to M$ as follows:
$$  h^{ M}_{n+1}\left( \langle \vec{i}, i_0, \ldots, i_k \rangle, \langle x_0,\ldots, x_k \rangle \right) = h^{ 
M }_{n}\left( \vec{i}, \la  h^{ M^{(n)} }_{1}(i_0, x_0),\ldots, h^{ M^{(n)} }_{1}(i_k, x_k) \ra  \right). $$
\end{enumerate}    
\end{remark}
Following definition 5.12 in \cite{SchindlerZeman2009FineStruture}, given premouses $M,N$, we write $M \elem_{r\Sigma_{n+1}} N$  if there exists an $r\Sigma_{n+1}$-preserving embedding $\sigma \colon M\to N$: that is, an embedding satisfying the following--
\begin{itemize}
    \item $\sigma(p^{M}_{n}) = p^{N}_{n}$.
    \item For every $i\leq n$, $\sigma \uhr M^{(i)} \to N^{(i)}$ is $\Sigma_1$-preserving.
\end{itemize}
If the second requirement above is weakened so that $\sigma \uhr M^{(i)} \to N^{(i)}$ is $\Sigma_1$-preserving for every $i<n$ and $\sigma\uhr M^{(n)}\to N^{(n)}$ is $\Sigma_0$-preserving, then the embedding $\sigma\colon M\to N$ is called a weakly $r\Sigma_{n+1}$-preserving embedding.

\subsection{The stationary sets $\mathcal{S}^\alpha$}

Let $L[\E]$ be a minimal extender model for the existence of a $(\kappa,\kappa^{++})$-extender for some cardinal $\kappa$. Throughout this section, we denote $K = L[\E]$. For every ordinal $\eta$, $K|\eta$ is the structure $\left(J^{\E}_\eta,\in,E_\eta\right)$, which is active if $E_\eta \neq \emptyset$, and passive otherwise. 

Throughout this subsection, fix an uncountable cardinal $\alpha\leq \kappa$. $L[\E]$ carries $\square_\alpha$-sequence, $\vec{c^\alpha}= \la c^\alpha_\eta \mid \eta \in \C^\alpha\ra$ where 
$$\C^\alpha = \{ \eta \in (\alpha,\alpha^+) \mid K|\eta \elem_{\Sigma_\omega} K|\alpha^+\}$$ 
is a closed unbounded subset of $\alpha^+$. For every $\eta\in \mathcal{C}^\alpha$, the definition of $c^\alpha_\eta$ is based on the collapsing structure $N^{\E}_\eta = K|\tau(\eta)$ of $\eta$, defined as following.

\begin{definition} \label{Def: CollapsingStructure}
For $\eta \in \C^\alpha$, let $\tau(\eta)$ be the maximal $\tau\in \left( \eta, \alpha^{+} \right)$, such that $\eta$ is a cardinal in $K|\tau$ (and, if there is no such $\tau$, set $\tau = \eta$). The structure $K|\tau(\eta)$ is called the collapsing structure of $\eta$, and is denoted by $N^{\E}_{\eta}$. 
\end{definition}
Whenever $\E$ is clear from the context, we denote $N_{\eta} = N^{\E}_\eta$. The projecta, standard parameters and reducts of a collapsing structure $N_{\eta}$ will be denoted by $\rho^{N_{\eta}}_{n}$, $p^{N_\eta}_n$ and $N^{(n)}_{\eta}$. Denote by $h^{N_{\eta}}_{n+1} \colon \omega\times \left( N^{(n+1)}_{\eta} \right)^{<\omega} \to N_{\eta} $ the $(n+1)$-th Skolem function (see remark \ref{Rmk:sigma1Skolemfunction} below). The $h^{N_\eta}_{n+1}$-hull of $\rho^{N_\eta}_{n+1}\cup \{ p^{N_\eta}_{n+1} \}$ is $N_{\eta}$. Thus, every element of $N_\eta$ has the form $h^{N_\eta}_{n+1}(k, \xi, p^{N_\eta}_{n+1})$ for some $k<\omega$ and $\xi<\rho^{N_\eta}_{n+1}$. Let $n_\eta < \omega$ be minimal such that $\alpha = \rho^{N_\eta}_{n_{\eta}+1}$.


By the maximality of $\tau(\eta)$ in the definition of the collapsing structure $N_\eta$ and the fact it is acceptable, there exists a surjection from $\alpha$ onto $\eta$ which is definable over $N_{\eta}$. It follows that $\alpha$ is the ultimate projectum of $N_{\eta}$.

It will be useful for our purposes to prove that in $L[\E]$, for every $\alpha \geq \omega_2$ there is a stationary set of nice collapsing structures.

\begin{proposition} \label{Prop: Stationary set of passive collapsing structures}
    For every cardinal $\alpha > \omega_2$ the set
\begin{align*}
\mathcal{S}^\alpha = 
\{ \eta \in \C^\alpha  \mid \  &
 N_\eta \text{ is passive, } \eta\in N_\eta,\\
 &n_\eta \geq 2, \  \rho^{N_\eta}_1 = \ldots = \rho^{N_\eta}_{n_\eta} = N_\eta\cap Ord,\\
 &p_{n_\eta+1}^{N_\eta} = \{ \delta \} \mbox{ where } \delta\in (\alpha,\alpha^{+}) \mbox{ is the least ordinal outside of }  h^{N_\eta}_{n_\eta+1}[\omega\times \alpha], \\
 &\cf(N_\eta \cap Ord) = \omega_1, \mbox{ and } h^{N_\eta}_{n_\eta+1}[\omega \times \omega_1] \text{ is cofinal in } N_\eta \cap Ord.\} 
    \end{align*}
    is stationary in $\alpha^+$.
\end{proposition}

\begin{proof}
    Suppose otherwise, and let $C^* \subseteq \alpha^+$ be a closed unbounded set disjoint from $S^\alpha$,  which is minimal in the canonical well-order of $L[
    \E]$. Therefore $C^* \in K|\alpha^{++}$ is definable in $K|\alpha^{++}$ without parameters. 
    For each $\nu < \omega_1$, let $\xi_\nu$ be the $\nu$-th ordinal in the interval $(\alpha^{++},\alpha^{+3})$ satisfying  $K|\xi_\nu \elem_{r\Sigma_2} K|(\alpha^{+3})$.\footnote{Namely, the identity map from $K|\xi_\nu$ to $K|\alpha^{+3}$ is $r\Sigma_2$-preserving.} 
    Let $\xi^* = \bigcup_{\nu < \omega_1}\xi_\nu$ and  $N^* = \bigcup_{\nu < \omega_1} K|\xi_\nu$. 
    We have:
    \begin{itemize}
        \item $N^*\elem_{r\Sigma_{2}}K|\alpha^{+3}$.
        \item $N^*$ is passive. This follows since $K\mid \alpha^{+3}$ itself is passive by our anti-large cardinals assumption that $K$ is the minimal model containing a $(\mu, \mu^{++})$-extender for some cardinal $\mu$.
        
        \item $\alpha^{+2} \in N^*$ is the largest cardinal in $N^*$, and is $\Pi_2$-definable  without parameters in $N^*$.
        Similarly, $\alpha,\alpha^+ \in N^*$ are definable without parameters, and therefore also $C^*$. In particular, $C^*\in h^{N^*}_3[\omega \times \{0\}]$. 
        
        \item For each $\nu < \omega_1$, $\xi_\nu$ is the $\nu$-th ordinal above $\alpha^{+2}$ in $N^*$ which is closed under $h_2^{N^*}$. Therefore, $\xi_\nu \in  h^{N^*}_3[\omega \times \omega_1]$. Hence, $h^{N^*}_3[\omega \times \omega_1]$ is cofinal in $N^* \cap Ord = \xi^*$. 
    
       \item $\rho_1^{N^*} = \rho^{N^*}_2 = \xi^*$ (this follows since since $N^*\elem_{r\Sigma_2}K\mid \alpha^{+3}$, and so $N^*$ and its first reduct, satisfy $\Sigma_1$-separation). 
       
        \item $cf(N^* \cap Ord) = cf(\xi^*) = \omega_1$.
        \end{itemize}
Let $\delta$ be the least ordinal such that $\delta\notin h^{N^*}_{3}[\omega\times \alpha]$. Clearly $\delta\in (\alpha, \alpha^{+})$. Define
$$X = {h}^{N^*}_3[\omega \times (\alpha\cup \{  \delta  \})]$$
Denote by $N$ the transitive collapse of $X$. Let $\sigma$ be the inverse of the transitive collapse map. We will argue below, in a sequence of claims, that for some $\eta\in \mathcal{C}^{\alpha}$, $N = N_{\eta}$, and  $\eta\in C^*\cap \mathcal{S}^{\alpha}$, which is a contradiction. 
\begin{claim}\label{Claim: prp Calpha, collapsing structure}
    $N$ is a premouse and has the form $N = J^{\E'}_{\tau}$ for some coherent extender sequence $\E'$ and ordinal $\tau$. Moreover, $\sigma\colon N\to N^*$ is weakly $r\Sigma_3$-preserving, $\rho^{J^{\E'}}_{3} = \alpha$ and $p^{J^{\E'}}_{3} = \{\delta \} $. 
\end{claim}

\begin{proof}
    denote by $(N^*)^{(2)}$ the second reduct of $N^*$ (since $N^*$ and its first reduct are models of $\Sigma_1$-separation, the universe of $(N^*)^{(2)}$ is the same as the universe of $N^*$). Consider the hull 
    $$Y = {h}^{(N^*)^{(2)}}_1[\omega \times (\alpha\cup \{  \delta  \})].$$
    Let $N'$ be the transitive collapse of $Y$, and let $\pi \colon N'\to (N^*)^{(2)}$ be the inverse of the collapse map. Note that $\delta$ is collapsed to itself, since every ordinal below $\delta$ already belongs to $h^{N^*}_{3}[\omega\times \alpha]$. Since $\pi$ is $\Sigma_0$, $N'$ is the second reduct of a $J$-structure $J^{B}_{\tau}$ for some predicate $B$ and ordinal $\tau$ (see Lemma 3.3 in \cite{SchindlerZeman2009FineStruture}). Moreover, $\pi$ lifts to an $r\Sigma_3$-preserving embedding $\bar{\pi}$ from $J^{B}_{\tau}$ to $N^*$ (this follows from the downward extension of embeddings lemma, see Lemma 3.1 in \cite{SchindlerZeman2009FineStruture}). Also, the lift $\bar{\pi}$ is $\Sigma_1$-preserving, and a $J$-structure which can be $\Sigma_1$-embedded into a premouse is a premouse (see Lemma 4.1.3 (a) in \cite{Zeman-Book}). Thus, for some extender sequence $\E'$, $J^{B}_{\tau} = J^{\E'}_{\tau}$. 
    
    Since $J^{\E'}_\tau$ is $r\Sigma_3$ embedded in $N^*$, we have that $\rho^{J^{\E'}_\tau}_1 = \rho^{J^{\E'}_\tau}_{2} = \tau$. By the uniform definability of Skolem functions, and the fact that $\delta$ collapses to itself, we have that $J^{\E'}_\tau = h^{J^{\E'}_\tau}_{3}[\omega\times (\alpha\cup \{\delta \})]$. In particular $\rho^{J^{\E'}_\tau}_3 = \alpha$. 
    
    We argue that $p^{J^{\E'}_\tau}_3 = \{\delta\}$. Assume that $A\subseteq \alpha$ is $\Sigma_1$ definable over $\left(J^{\E'}_\tau\right)^{(2)}$ from some parameter $a<_{lex}{\delta}$. Let $A^*\subseteq \alpha$ be the set defined from $a$ using the same $\Sigma_1$-formula over $(N^*)^{(2)}$. Then $A^*\in h^{(N^*)^{(2)}}_{1}[\omega\times \alpha]$, by the choice of $\delta$ and since $a<_{lex} \{\delta\}$. Thus $A^*\in Y$ collapses to $A$.
    
    Finally, note that $\bar{\pi}$ maps each element of the form $h^{J^{\E'}_\tau}_{3}(n,\xi, \delta)$ to $h^{N^*}_3(n,\xi, \delta)$ (where $n<\omega$ and $\xi<\alpha$), and thus $X$ transitively collapse to $J^{\E'}_\tau$. It follows that $\bar{\pi} = \sigma$.
\end{proof}

$N$ is passive, since $N^*$ is passive. $N$ is iterable, by embedding each of its iterates to a corresponding iterate of $N^*$. We argue that $\E' = \E$, and thus it will follow that $N$ is an initial segment of $K$. This is proved by a  standard comparison argument 
(see for example, \cite{Mitchell-HB} , \cite{Zeman-Book}, \cite{Steel-HB}). We include a proof for readers who are less familiar with comparison arguments.

\begin{claim} \label{claim: comparison argument}
$N= K|\tau$ where $\tau = N\cap \mbox{Ord}$.
\end{claim}
\begin{proof}    
    Denote $M = L[\E]$. Let us argue that $N$ is an initial segment of $M$. Assume otherwise, and run the comparison between the structures $M, N$. Assume that the length of the comparison is $\theta\leq \mbox{Ord}$, and the structures built in the comparison process are $\la N_{i} \colon i\leq\theta \ra$, $\la M_{i} \colon i\leq\theta \ra$. Note that the structures $N, M$ agree on extenders indexed below $\alpha$, since $\mbox{crit}(\sigma)\geq \alpha$. Thus, if an extender is applied on one of the sides in the comparison process, its critical point is at least $\alpha$ (it cannot be below $\alpha$, since this, together with the fact that the extender taken is indexed above $\alpha$, contradicts the smallness assumption of $L[\E]$ being a minimal model with a $(\kappa,\kappa^{++})$-extender).

    Since $M$ is universal, there is no truncation on the $N$-side. We argue that $N_{\theta} = N$, namely, $N$ does not move in the iteration. Indeed, otherwise, since there is no truncation on the $N$-side, $N_\theta$ is not sound. Thus it cannot be a strict initial segment of $M_\theta$. So, again by universality, $N_{\theta} = M_{\theta}$, and thus $\left(\mathcal{P}(\alpha)\right)^{N} = \left(\mathcal{P}(\alpha)\right)^{N_\theta} = \left(\mathcal{P}(\alpha)\right)^{M_\theta}$. Since the powerset of $\alpha$ in $N$ is partial to the powerset of $\alpha$ in $M$, the $M$-side must have been truncated in the iteration process. At least one extender was applied in the iteration process on the $M$-side, since otherwise, $N_\theta$ is an initial segment of $M$, contradicting the fact that $N_\theta$ is not sound. Thus, $M_{\theta}$ was truncated in order to apply an extender on the appropriate level to which it can be applied.

    Let $M'$ be the model obtained after the last truncation on the $M$-side (in particular, $M_{\theta}$ is an iteration of $M'$ with no drops). $N$ is the core of $N_\theta = M_\theta$ (that is, the Skolem hull of $N_{\theta}$ below its projectum joint with its standard parameter), and since $M' $ is sound, $N = M'$. This contradicts the fact that some iterates have been taken above $M'$ with its extenders.

    Thus the $N$-side does not move in the comparison process. So either $N = M _{\theta}$ or $N$ is a strict initial segment of $M_{\theta}$. If an extender is applied on the $M$-side, it must have critical point $\alpha$, and index below $\left(\alpha^+\right)^{N}$. Thus, such an extender is a partial extender, and $M_\theta$ is not sound. So the case $N = M_{\theta}$ is possible only in the case that $N$ is an initial segment of $M$. The other possible case is that $N$ is a strict initial segment of $M_{\theta}$, and furthermore, $N$ is an initial segment of $M_{\theta}\uhr (\alpha^+)^{M_{\theta}}$ (since $N$ definably collapses to $\alpha$). It follows that that $N $ is actually an initial segment of $M\uhr  \left( \alpha^{+} \right)^{ M }$, since $M_\theta \uhr \left( \alpha^{+} \right)^{ M _{\theta} } = M\uhr \left( \alpha^{+} \right)^{ M _{\theta} }$ by the coherence of extender sequences. Thus, in any case, $N$ is an initial segment of $M$, as desired.
\end{proof}

The following claim yields the contradiction which  concludes the proof of proposition \ref{Prop: Stationary set of passive collapsing structures}. 

\begin{claim} \label{claim: the transitive collapse is indeed a collapsing structure}
Let $\eta = \alpha^+ \cap X = \mbox{crit}(\sigma)$. Then $N = N_\eta$, and $\eta\in \mathcal{S}^{\alpha}\cap C$.
\end{claim}

\begin{proof}
We first argue that $\eta\in \mathcal{C}^{\alpha}$, namely $K\mid \eta \elem_{\Sigma_\omega} K\mid \alpha^{+}$. Given  $a\in K\mid \eta$ and formula $\varphi(a)$, consider the $\Delta_0$-formula $\varphi^{ J^{\E}_{\eta} }(a)$ obtained from $\varphi$ by bounding every quantifier into $J^{\E}_\eta$. Then-- 
\begin{align*}
        K\mid \eta \vDash \varphi(a) & \iff N\vDash \varphi^{ J^{\E}_{\eta} }(a)\\
        &\iff N^* \vDash \varphi^{ J^{\E}_{\alpha^+} }(\pi(a))\\
        &\iff K\mid \alpha^+ \vDash \varphi(\pi(a))
\end{align*}
Since $\pi(a) = a$, we conclude that $K\mid \eta \vDash \varphi(a) \iff   K\mid \alpha^+ \vDash \varphi(a)$, as desired.

We argue that $N = N_{\eta}$. $\eta$ is a cardinal in $N = J^{\E}_{\tau}$ since it is equal to $(\alpha^+)^N$. However, $\eta$ collapses in $J^{\E}_{\tau+\omega}$, since $h^{N}_{3}\uhr \omega\times \alpha\cup\{\delta\} $ is onto $\eta$ and $r\Sigma_3$-definable over $N$.

Finally, by the construction of $N^*$, $h^N_3[\omega \times \omega_1]$ is cofinal in $N \cap Ord$, and $\cf(N \cap Ord) = \cf(N^* \cap Ord)  = \omega_1$. 

We showed above that $\eta\in \mathcal{S}^\alpha$. We also have that $\eta \in C^*$, as $\eta = X \cap \alpha^+$ must be a limit point of $C^* \in X$. So $\eta\in \mathcal{S}^{\alpha}\cap C^*$.
\end{proof}
\end{proof} 

A similar argument can be used to establish the following:

\begin{proposition} \label{Prop: Stationary set of passive collapsing structures on an arbitrary cofinality}
For every cardinal $\alpha > \omega_2$, $2\leq n<\omega$ and a regular cardinal $\omega_1\leq \zeta\leq\alpha$, the set    
        \begin{align*}
    \mathcal{S}^{\alpha, \zeta , n} = 
    \{ \eta \in \C^\alpha  \mid \  &
    N_\eta \text{ is passive, }, \eta\in N_\eta\\
    &n_\eta = n, \ \rho^{N_\eta}_1 = \ldots = \rho^{N_\eta}_{n_\eta} = N_\eta\cap Ord,\\
    &p_{n_\eta+1}^{N_\eta} = \{ \delta \} \mbox{ for some } \delta\in (\alpha,\alpha^{+}), \\
    &\cf(N_\eta \cap Ord) = \zeta, \mbox{ and } h^{N_\eta}_{n_\eta+1}[\omega \times \zeta]     \text{ is cofinal in } N_\eta \cap Ord.\} 
        \end{align*}
is stationary in $\alpha^+$. Furthermore, for every set $X\in K \mid \alpha^{+3}$ which is lightface $\Sigma_{n}$-definable over $K \mid \alpha^{+3}$, the subset of $S^{\alpha.\zeta,n}$ consisting of $\eta$-s such that there exists a weakly $r\Sigma_{n+1}$-preserving map $\pi\colon N_\eta\to K$ with $X\in \mbox{Im}(\pi)$ is stationary.
\end{proposition}

\begin{proof}
Assume that $\mathcal{S}^{\alpha,\zeta,n}$ is not stationary. Repeating the same argument of Proposition \ref{Prop: Stationary set of passive collapsing structures}, let $C^*\subseteq \alpha^+$ be the constructibly-least club disjoint from $\mathcal{S}^{\alpha,\zeta,n}$. We can find $N^*\elem_{r\Sigma_{n}} J^{\E}_{\lambda^{+3}}$ such that:
    \begin{itemize}
        \item $N^*$ is passive.
        \item $C^*\in h^{N^*}_{n+1}[\omega\times \{ 0 \}]$.
        \item $h^{N^*}_{n+1}[\omega\times \zeta]$ is cofinal in $N^*\cap \mbox{Ord}$.
        \item $\rho^{N^*}_1 =\ldots = \rho^{N^*}_{n} =  N^*\cap \mbox{Ord}$. 
        \item For the "furthermore" part, assume that $X$ is given and $C^*$ is disjoint from the relevant subset of $\mathcal{S}^{\alpha,\zeta,n}$. Since $X$ is lightface $\Sigma_n$-definable in $J_{\alpha^{+3}}$, we can assume  $ X\in h^{N^*}_{n+1}[\omega\times \{ 0 \}]$.
    \end{itemize}
    To find such $N^*$, pick, for every $\nu<\zeta$, the $\nu$-th ordinal $\xi_\nu\in (\alpha^{++}, \alpha^{+3})$ such that $J^{\E}_{\xi_{\nu}} \elem_{r\Sigma_{n}} J^{\E}_{\alpha^{+3}}$ and $X\in J^{\E}_{\xi_\nu}$. Define $N^* = \bigcup_{\nu<\delta} J^{\E}_{\xi_\nu}$. Then $N^*$ is passive and   $N^*\elem_{r\Sigma_{n}} J^{\E}_{\alpha^{+3}}$. We have that $h^{N^*}_{n+1}[\omega\times \zeta]$ cofinal in $\sup_{\nu<\zeta}\xi_{\nu} = N^*\cap \mbox{Ord}$, since each $\xi_{\nu}$ belongs to $h^{N^*}_{n+1}[\omega\times \zeta]$, as the $\nu$-th ordinal $\xi > \alpha^{++}$ such that $J^{\E}_{\xi}$ is closed under $h^{N^*}_{n}$ (the fact that $n\geq 2$ ensures that $\alpha\in h^{N^*}_{n+1}[\omega\times \{0\}] $). Also, each of the first $(n-1)$-many reducts of $N^*$ satisfies $\Sigma_1$-separation, and thus $\rho^{N^*}_{i} = N^*\cap \mbox{Ord}$ for every $i\leq n$. 

    Let $\delta\in (\alpha,\alpha^+)$ be the least element outside of $h^{N^*}_{n+1}[\omega\times \alpha]$. Let $N$ be the transitive collapse of $h^{N^*}_{n+1}[\omega \times (\alpha\cup \{\delta\})]$. Arguing as in Proposition \ref{Prop: Stationary set of passive collapsing structures}, $N = N_{\eta}$ for some $\eta\in \mathcal{S}^{\alpha,\zeta,n}\cap C^*$. For the "furthermore" part, let $\pi \colon N_{\eta}\to J^{\E}_{\alpha^{+3}}$ be the inverse of the transitive collapse map which collapses $N^*$ to $N$. Then $X\in \mbox{Im}(\pi)$ since $X\in h^{N^*}_{n+1}[\omega\times \{ 0 \} ]$.
\end{proof}

\subsection{Good and Very Good Points}\label{subsec:GoodAndVeryGood}

In this subsection we revisit the construction of the $\square_{\alpha}$-sequence $\la c^{\alpha}_{\eta} \colon \eta\in \mathcal{C}^{\alpha} \ra$. The main goal is to associate to each point of $c^{\alpha}_{\eta}$ an ordinal $\epsilon$ of the reduct $N^{n_\eta}_{\eta}$. Such points $\epsilon$ will be utilized in the next subsection to establish a connection between the $\square_{\alpha}$-sequence, and other $\square_{\alpha'}$-sequences for suitably chosen cardinals $\alpha'<\alpha$.

We follow the $\square_\alpha$ constrution to produce clubs $c^\alpha_\eta$ for every $\eta \in \bigcup_{\zeta < \kappa} S^{\alpha,\zeta}$. The special properties of the collapsing structures $N_\eta$ for $\eta \in S^{\alpha,\zeta}$ reduces the technical complexity of the construction of $c^\eta_\alpha$, and for the most part, allows us to follow the construction of square in $L$, as described in Handbook chapter \cite{SchindlerZeman2009FineStruture} by Schindler and Zeman.  Recall that the construction consists of two steps.
\begin{itemize}
    \item In the first step, a closed subset $d^{\alpha}_{\eta}\subseteq \eta$ is defined, for every $\eta\in {\C}^\alpha$, such that the system $\la d^{\alpha}_\eta \colon \eta\in {\C}^{\alpha} \ra$ is coherent. 
    \item In the second step, each club $d^{\alpha}_{\eta}$ is thinned out to the club $c^{\alpha}_{\eta}\subseteq \eta$ with $\otp( c^{\alpha}_\eta )\leq \alpha$, so that the system $\la c^{\alpha}_\eta \colon \eta\in {\C}^\alpha, \cf(\eta)>\omega \ra$ is the desired coherent sequence of clubs. 
\end{itemize}

We concentrate on the first step, and postpone the details of the second step to definition \ref{Def: finestructurethinout}.

Following \cite{SchindlerZeman2009FineStruture}, $d^{\alpha}_{\eta}$ is defined for every $\eta\in {\C}^\alpha$ as follows.

\begin{definition}
Fix $\eta\in {\C}^{\alpha}$. Define $d^{\alpha}_{\eta}\subseteq \eta$ to be the set of points $\eta'\in {\C}^{\alpha}\cap \eta$, such that:
\begin{enumerate}
    \item $n_{\eta'} = n_{\eta}$. In what follows, $n$ denotes this common value.
    \item There exists a weakly $r\Sigma_{n+1}$-embedding $\sigma \colon N_{\eta'}\to N_{\eta}$, such that:
    \begin{enumerate}
        \item $\sigma\uhr \eta' = id$.
        \item $\sigma( p^{N_{\eta'}}_{n+1} ) = p^{N_{\eta}}_{n+1}$.
        \item $\sigma(\eta') = \eta$ (under the assumption that $\eta'\in N_\eta$).
    \end{enumerate}
\end{enumerate}
\end{definition}

The square construction asserts that $d^\alpha_\eta$ is cofinal in $\eta$ whenever $\cf(\eta) > \aleph_0$. This has been shown by Jensen for $L$, and in a series of developments for stronger canonical inner models,   culminating in the work of Schimmerling and Zeman \cite{schimmerlingZemanSquare}.

\begin{theorem}
For each $\eta \in \mathcal{C}^\alpha$, if $cf(\eta) > \omega$ then $d^\alpha_\eta \subseteq \eta$ is cofinal in $\eta$.
\end{theorem}

For our purposes, it will be convenient to regard the structures $N_{\eta'}$ for $\eta'\in d^\alpha_\eta$ as the result of a transitive collapse. For that, we introduce the machinery of restricted Skolem hulls (Definition \ref{Def:RestrictedSkolem}) and good points (Definition \ref{Def: GoodPoints}).

\begin{definition}(Restricted Skolem functions)\label{Def:RestrictedSkolem}
Let $\eta\in \mathcal{C}^{\alpha}$ and denote $n = n_{\eta}$. Assume that $\epsilon\in N_{\eta}$ an ordinal and $\alpha< \epsilon < \rho^{N_{\eta}}_{n}$. 
\begin{enumerate}
    \item Let $J^{ \E, n }_{ \epsilon }$ denote the structure $N^{(n)}_{\eta} \mid \epsilon$ (that is, the $\epsilon$-th level of the $J$-hierarchy in the suitable language).
    \item $h_{1}^{N_\eta , \epsilon}$ is the function  defined by the $\Delta_0$-formula, which is obtained by replacing the existential quantifier $\exists z$ in the definition of the first Skolem function from remark \ref{Rmk:sigma1Skolemfunction} with $ \exists z \in J^{\E,n}_\epsilon$.
    \item The function $h^{N_{\eta}  , \epsilon}_{n+1}$ is defined similarly to the inductive definition of the function $h^{N_{\eta}}_{n+1}$ from remark \ref{Rmk:sigma1Skolemfunction} as follows:
    $$  h^{ N_\eta , \epsilon }_{n+1}\left( \langle \vec{i}, i_0, \ldots, i_k \rangle, \langle x_0,\ldots, x_k \rangle \right) = h^{ N_{\eta} }_{n}\left( \vec{i}, \la  h^{ N^{(n)}_{\eta}   , \epsilon  }_{1}(i_0, x_0),\ldots, h^{ N^{(n)}_{\eta}  , \epsilon }_{1}(i_k, x_k) \ra  \right). $$
\end{enumerate}
\end{definition}

In comparison to $h^{N_\eta}_{n+1}$-Skolem hulls, which produce $r\Sigma_{n+1}$-preserving maps by taking inverse to the transitive collapse, restricted Skolem hull produce weakly $r\Sigma_{n+1}$-preserving maps. This point will be further explained in Lemma \ref{Lem: FineStructureGoodPts}.

%

We now introduce the concept of a good point. Each such good point $\epsilon$ induces a club point $\eta'\in d^{\alpha}_{\eta}$ (the other direction is also true, but requires to extend the notion of a 'good point' to a more general context; see Remark \ref{Rmk: FineStructureGoodPts}).

\begin{definition} \label{Def: GoodPoints}
Assume that $N_{\eta}$ is the collapsing structure of some $\eta\in {\S}^{\alpha, \zeta,n}$ for some $\zeta\leq \alpha$ and $n<\omega$. Let $\delta\in (\alpha,\alpha^+)$ be such that $p^{N_{\eta}}_{n_\eta+1} = \{\delta\}$. We say that an ordinal $\epsilon<\rho^{N_\eta}_{n} $ is good if:
\begin{enumerate}
    \item  $h^{N^{(n)}_\eta, \epsilon}_1[\omega \times (\alpha\cup \{\delta\})]$ is unbounded in $\epsilon$ (equivalently, $h^{ N_{\eta}  , {\epsilon} }_{n+1}[\omega\times (\alpha\cup \{\delta \}) ]$ is unbounded in $\epsilon$).
    \item  $\alpha\in h^{N_\eta^{(n)},\epsilon}_1[\omega \times (\alpha\cup \{\delta \})]$ (equivalently, $\alpha \in h^{ N_{\eta} , {\epsilon} }_{n+1}[\omega\times(\alpha\cup \{\delta\})]$).
    \item Letting $W$ be the transitive collapse of $h^{N^{(n)}_\eta, \epsilon}_{1}[\omega\times \delta]$, $W\in h^{N^{(n)}_{\eta}, \epsilon}_{1}[\omega\times (\alpha\cup\{\delta\})]$ (equivalently, $W\in h^{N_\eta,\epsilon}_{n+1}[\omega\times (\alpha\cup \{\delta \} )]$).
\end{enumerate}
\end{definition}

\begin{remark} \label{Rmk: FineStructureGoodPts}
In order to simplify the proofs and avoid dealing with solidity witnesses, we concentrated above on collapsing structures $N_\eta$ for $\eta\in \mathcal{S}^{\alpha,\zeta,n}$ (where $\zeta\leq\alpha$ is regular, and $n<\omega$). 

Working with such values of $\eta$ also allows us to restrict our discussion to collapsing structures with a singleton in $(\alpha, \alpha^+)$ as a standard parameter. This will be key in the proof that the forcing realizing the ``Kunen-like" blueprint, which will be constructed in later sections, preserves the stationarity of certain stationary sets (see Lemma \ref{Lem: FinalForcingPreservationOfStatSets}).

The definition of a good point can be adapted to collapsing structures with an arbitrary parameter, by replacing $\{\delta\}$ in the definition above with the standard parameter $p^{N_\eta}_{n+1}$, and adding the requirements:
\begin{itemize}
    \item $\epsilon > \max\left(p^{N_\eta}_{n+1}\right)$.
    \item For every $i\leq n+1$ and $\nu\in p_i = p^{N_\eta}_{i}$, $W^{\nu,p_i}_{N_{\eta}   , {\epsilon}}\in h^{ N_{\eta}   , {\epsilon} }_{n+1}[\omega\times(\alpha\cup \{\delta \}) ]$, where $W^{\nu,p_i}_{N_{\eta}   , {\epsilon}}$ is the transitive collapse of $h^{N_\eta,\epsilon}_{i+1}[\omega \times (\nu \cup (p_i\setminus (\nu+1)))]$. The structure $W^{\nu,p_i}_{N_{\eta}   , {\epsilon}}$ is called the standard witness to the fact that $\nu\in p_i$ (see Definitions 7.3 and 7.10 in \cite{SchindlerZeman2009FineStruture} for more details).
\end{itemize}
In general, all the results in this section can be generalized to arbitrary (passive) collapsing structures by integrating the parameters and their solidity witnesses into the hulls we collapse.
    
\end{remark}


\begin{lemma} \label{Lem: FineStructureGoodPts}
Assume that $\eta\in {\mathcal{S}}^{\alpha, \zeta,n}$ for some regular $\zeta\leq \eta $ and $n<\omega$. Let $\delta\in (\alpha,\alpha^+)$ be such that $p^{N_\eta}_{n+1} = \{\delta\}$, and let $\epsilon \in \left(\delta, \rho^{N_\eta}_{n}\right)$ be a good point. Then:
\begin{enumerate}
    \item There exists $\tau$ such that $h^{N_\eta , \epsilon}_{n+1}[\omega\times(\alpha\cup \{\delta \})]$ transitively collapses to $J^{\E}_{\tau}$.
   \item Let $\eta' = \eta \cap h^{N_\eta , \epsilon}_{n+1}[\omega\times(\alpha\cup \{\delta\} )]$. Then $\alpha < \eta' < \eta$. Moreover, $\eta'\in \mathcal{C}^{\alpha}$, namely $J^{\E}_{\eta'} \prec J^{\E}_{\alpha^+}$. Furthermore, the collapsing structure $N_{\eta'}$ of $\eta'$ is $J^{\E}_{\tau}$.
    \item The inverse of the transitive collapse of $h^{N_\eta , \epsilon}_{n+1}[\omega\times(\alpha\cup \{\delta \})]$ is the unique weakly $r\Sigma_{n+1}$-embedding $\sigma \colon N_{\eta'}\to N_{\eta}$, with the properties:
    \begin{enumerate}
        \item $\sigma\uhr \eta' = id$.
        \item $\sigma(p_{n+1}^{N_{\eta'}}) = p^{N_\eta}_{n+1}$.
        \item $\sigma(\eta') = \eta$.
    \end{enumerate}
      This witnesses the fact that $\eta'\in d^{\alpha}_{\eta}$. We denote the above embedding by 
      $$\sigma_{\eta',\eta} \colon N_{\eta'}\to N_{\eta}.$$
\end{enumerate}
\end{lemma}

\begin{proof}${}$
\begin{enumerate}
    
    \item Denote by $X$ the transitive collapse of $h^{N_\eta , \epsilon}_{n+1}[\omega\times (\alpha\cup \{\delta\})]$, and by $Y$ the transitive collapse of $h^{N^{(n)}_\eta , \epsilon}_{1}[\omega\times(\alpha\cup \{\delta \})]$.

    We first argue that $Y$ is rudimentary closed. Namely, that $Y$ is closed under each function in the finite set of functions which form a basis for the class of rudimentary functions. For that, we should prove that for every basis function $f(\nu_1,\ldots, \nu_k)$, if $x_1,\ldots, x_k\in h^{N^{(n)}_{\eta}  , \epsilon}_{1}[\omega\times (\alpha\cup \{\delta \})]$, then $f(x_1,\ldots, x_k)\in h^{N^{(n)}_{\eta}  , \epsilon}_{1}[\omega\times (\alpha\cup \{\delta\})]$. By the definition of the restricted Skolem function, it suffices to prove that there exists an ordinal $\gamma<\epsilon$ such that $x_1, \ldots, x_k, f(x_1,\ldots, x_k)\in N^{(n)}_\eta \mid \gamma$ 
    (where $N^{(n)}_\eta \mid \gamma$ is the $\gamma$-th member in the $S$-hierarchy of the $n$-th reduct, see Remark \ref{Rmk:sigma1Skolemfunction}), since such $N^{(n)}_\eta \mid \gamma\in J^{\E}_{\epsilon}$ can be taken as a witness for the $\Sigma_1$-formula which defines $f(x_1,\ldots, x_n)$ from $x_1,\ldots, x_n$. 
    Indeed, by the very definition of the $S$-hierarchy, if $\gamma<\epsilon$ is such that $x_1,\ldots, x_n\in N^{(n)}_\eta \mid \gamma$, then $f(x_1,\ldots, x_n)\in N^{(n)}_\eta \mid \gamma+1$. 

    Since $Y$ is rudimentary closed, $Y$ is collapsed to a $J$-structure of the form $J^{A}_{\tau'}$ for some ordinal $\tau'$ and some predicate $A$. Let $\pi \colon J^{A}_\tau \to N^{(n)}_\eta $ be the inverse of the transitive collapse. Let $\delta'$ be such that $\pi(\delta') = \delta$. By the uniform definability of the Skolem function, and the fact that $h^{N^{(n)}_\eta, \epsilon}_1[\omega\times \left( \alpha\cup \{\delta \} \right)]$ is unbounded in $\epsilon$, $J^{A}_{\tau'}=h^{J^{A}_{\tau'}}_{1}[\omega\times (\alpha\cup \{\delta'\} )]$. Then $\rho_1^{J^{A}_{\tau'}} = \alpha$ and $p^{J^{A}_{\tau'}}_{1} = \{\delta'\}$, and for some ordinal $\tau$ and an extender sequence $\E'$, $J^{A}_{\tau'}$ is the $n$-th reduct of a structure $J^{\E'}_{\tau}$ with $\rho^{J^{\E'}_\tau}_{n+1} = \alpha$, $p^{J^{\E'}_\tau} = \{  \delta' \}$, and $\pi$ lifts to an $r\Sigma_{n+1}$-preserving embedding $\sigma\colon J^{\E'}_{\tau}\to N_{\eta}$. We concentrate on the proof that $p^{J_{\tau'}}_1 = \{ \delta' \}$, since the rest is similar to Claim  \ref{Claim: prp Calpha, collapsing structure}. Recall that $W\in h^{N^{(n)}_\eta, \epsilon}_{1}[\omega\times (\alpha\cup \{\delta \})]$    where $W$ is the transitive collapse of $h^{N^{(n)}_\eta, \epsilon}_{1}[\omega\times \delta]$. But $W$ is also the transitive collapse of $h^{J^{A}_{\tau'}}[\omega\times \delta']$, and in the  transitive collapse of $Y$, $W$ collapses to itself, since it's transitive. So $W\in J^{A}_{\tau'}$, which witnesses the fact that $\delta'\in p^{J^{A}_{\tau'}}_{1}$ (see Definition 7.3 and Lemmas 7.2 and 7.4 in \cite{SchindlerZeman2009FineStruture}). Since $\delta'$ is already a parameter satisfying $J^{A}_{\tau'}[\omega\times (\alpha\cup \{\delta' \})]$, $p^{J^{A}_{\tau'}}_1 = \{\delta' \}$. 
    
    The embedding $\sigma$ is actually the unique $r\Sigma_{n+1}$-preserving embedding which lifts $\pi$ and satisfies $\sigma\left( p^{J^{\E'}_{\tau'}}_{n+1} \right) = p^{N_\eta}_{n+1}$ (that is, $\sigma(\delta') = \delta$), and, since $\sigma$ is weakly $r\Sigma_{n+1}$-preserving, we have
    $$\sigma\left(     h^{J^{\E'}_\tau}_{n}\left( \vec{i} , \la h^{ J^{A}_{\tau'} }_1(i_0, \nu_0, \delta'), \ldots, h^{ J^{A}_{\tau'} }_{1}(i_k, \nu_k, \delta')  \ra\right)    \right) = h^{N_\eta}_{n}\left( \vec{i} , \la h^{ N^{(n)}_{\eta} }_1(i_0,  \nu_0, \delta ), \ldots, h^{  N^{(n)}_{\eta} }_{1}(i_k,\nu_k ,\delta)  \ra\right)$$
    for every $\vec{i}, \la i_0,\ldots, i_k \ra$ sequences of natural numbers, and $\nu_0, \ldots, \nu_k<\alpha$. Therefore, $X$ is equal to the image of $\sigma$, and $\sigma$ is the inverse of the transitive collapse map of $X$ to $J^{\E'}_{\tau}$. Finally, by running a comparison argument between $J^{\E'}_{\tau}$ and $L[\E]$, as in Claim \ref{claim: comparison argument}, we conclude that $J^{\E'}_{\tau}$ is an initial segment of $L[\E]$. Thus, $X \simeq J^{\E}_\tau$, and $Y$ is its $n$-th reduct.

    \item We first argue that $h^{N_{\eta}  , \epsilon}_{n+1}[\omega\times (\alpha\cup \{\delta \} )]\cap \eta$ is transitive. Once we prove that, it will follow that $\eta' =h^{N_{\eta}  , \epsilon}_{n+1}[\omega\times (\alpha\cup  \{\delta \} )]\cap \eta$ is an ordinal in the interval $(\alpha,\eta)$, for the following reasons:
    \begin{itemize}
        \item $\eta'> \alpha$ since $\alpha+1\subseteq h^{N_{\eta}  , \epsilon}_{n+1}[\omega\times (\alpha\cup \{ \delta\})]$ for every good point $\epsilon$.
        \item $\eta'<\eta$: otherwise,   $h^{J^{\E}_\tau}_{n+1}[\omega\times ( \alpha\cup \{ \delta\} )]$ witnesses that $\eta$ is collapsed in $J^{\E}_{\tau+\omega}$, namely $J^{\E}_\tau = N_\eta$. Comparing the $n$-th reducts, it follows that  $Y = h^{N^{(n)}_\eta, \epsilon}_{1}[\omega \cup ( \alpha\cup \{ \delta \} )]$ transitively collapses to $N^{(n)}_\eta$. However, letting $J^{\E, n}_\epsilon = N^{(n)}_\eta \mid \epsilon$, the function $h_{1}^{N^{(n)}_{\eta}, {\epsilon}} \colon \omega\times J^{\E,n}_{\epsilon} \to J^{\E,n}_{\epsilon} $ is $\Sigma_1$-definable over $J^{\E,n}_{\epsilon}$. Thus, $h^{N_{\eta}  , {\epsilon}}_{1}$ belongs to $J^{\E,n}_{\epsilon+\omega}  \subsetneq N^{(n)}_{\eta}$, and so does its pointwise image $Y = h_1^{N_{\eta}, \epsilon}[\omega\times (\alpha\cup \{\delta \} )]$.
    \end{itemize}
    
    We proceed and prove that transitivity of $h^{N_{\eta}  , \epsilon}_{n+1}[\omega\times (\alpha\cup \{\delta \} )]\cap \eta$. Assume that $\eta^*$ belongs to $h^{N_{\eta}  , \epsilon}_{n+1}[\omega\times (\alpha\cup \{ \delta\})]\cap \eta$. Since $J^{\E}_{\eta}\elem J^{\E}_{\alpha^{+}}$,
    $J^{\E}_{\eta}\vDash \left| \eta^* \right|\leq \alpha$. This means that the $\Sigma_1$-statement 
       $$"\exists f. f: \alpha \to \eta^* 
    \text{ is a surjection}"$$
    has a witness in $J^{\E}_\epsilon$ (even in $J^{\E}_{\eta}$), and therefore such $f$ belongs to 
    $h^{N_\eta , \epsilon}_{n+1}[\omega\times (\alpha\cup \{\delta\})]$. It is then clear that $\eta^* = f[\alpha] \subseteq h^{N_\eta , \epsilon}_{n+1}[\omega\times (\alpha\cup \{\delta \} )]$. 

    The fact that $J^{\E}_{\eta'}\elem_{\Sigma_{\omega}} J^{\E}_{\alpha^{+}}$ follows by a similar argument as in claim \ref{claim: the transitive collapse is indeed a collapsing structure}: denote by $\sigma \colon J^{\E}_{\tau} \to N_\eta$ the inverse of the transitive collapse map of $X$. Note that $\sigma(\eta') = \eta$ and $\sigma\uhr \eta' = id$. For every $a\in J^{\E}_{\eta'}$ and formula $\varphi(a)$, consider the $\Delta_0$ formula $\varphi^{ J^{\E}_{\eta'} }(a)$ obtained from $\varphi$ be bounding every quantifier into $J^{\E}_{\eta'}$. Then $\sigma(a) = a$, and--
    \begin{align*}
        J^{\E}_{\eta'}\vDash \varphi(a) & \iff J^{\E}_{\tau}\vDash \varphi^{J^{\E}_{\eta'}}(a)\\
        &\iff N_{\eta}\vDash \varphi^{J^{\E}_{\eta}}(\sigma(a))\\
        &\iff J^{\E}_{\eta} \vDash \varphi(a) \\
        &\iff J^{\E}_{\alpha^{+}} \vDash \varphi(a)
    \end{align*}
    where the last equivalence follows since $\eta\in \mathcal{C}^\alpha$.
    We should justify that $N_{\eta'} = J^{\E}_\tau$. It suffices to prove that $\eta'$ collapses in $J^{\E}_{\tau+\omega}$. Let $\delta'\in J^{\E}_\tau$ be such that $\pi(\delta') = \delta$ (we will see below that $\delta' = \delta$ for $\eta\in \mathcal{S}^\alpha$). Then indeed $h^{J^{\E}_\tau}_{n+1}[\omega\times (\alpha\cup \{ \delta'\})]$ contains $\eta'$ and thus witnesses the fact that $\eta'$ is collapsed in $J^{\E}_{\tau+\omega}$.

    \item This follows from the above points.
\end{enumerate}
\end{proof}

\begin{lemma}
Let $\eta\in \mathcal{S}^{\alpha, \zeta,n}$ for some $\zeta\leq \alpha$ regular and $n<\omega$. Let  $\delta\in (\alpha,\alpha^+)$ be such that $p_{n_\eta+1}^{N_{\eta}} = \{\delta \}$.  Assume that $\eta'\in d^{\alpha}_{\eta}$. Then there exists a good point $\epsilon \in N_\eta$ such that $ \sigma_{\eta', \eta} $ is the inverse of the transitive collapse of $ h^{N_{\eta} , \epsilon }_{n_\eta+1}[\omega\times(\alpha\cup \{\delta \})] $.
\end{lemma}

\begin{proof}
Denote $n = n_\eta = n_{\eta'}$. Let $\epsilon = \sup\left( \sigma_{\eta',\eta} [\mbox{Ord}\cap N^{(n)}_{\eta'}]\right)$. It suffices to prove that $\epsilon\in N^{(n)}_\eta$, and then it follows that $\epsilon$ is a good point, and $\sigma_{\eta',\eta}$ is the inverse of the transitive collapse of $h^{N_\eta,\epsilon}_{n+1}[\omega\times (\alpha \cup \{ \delta \} )]$.

Assume that $ \sigma_{\eta',\eta} [\mbox{Ord}\cap N^{(n)}_{\eta'}]$ is unbounded in $N^{(n)}_{\eta}$. It follows that  $\sigma_{\eta',\eta}\uhr N^{(n)}_{\eta'} \colon N^{(n)}_{\eta'}\to N^{(n)}_{\eta}$ is cofinal and thus $\Sigma_1$. In other words, $\sigma_{\eta',\eta}$ is $r\Sigma_{n+1}$-preserving, rather than just weakly $r\Sigma_{n+1}$-preserving. Therefore, $N_{\eta'}$ is the transitive collapse of $h^{N_\eta}_{n+1}[\omega \times (\alpha \cup p^{N_\eta}_{n+1})]$, which is, by $n$-soundness of $N_\eta$, $N_\eta$ itself. So $\eta' = \eta$.
\end{proof}

We will make use of the following standard property of collapsing structures $N_\eta$ which connect their cofinality with the cofinality of $\eta$. 
\begin{lemma}\label{Lemma:Cf(eta)=Cf(N)}
For every $\eta\in \mathcal{S}^{\alpha, \zeta,n}$, $cf(\eta) = \cf(N_\eta \cap Ord)$.    
\end{lemma}
\begin{proof}
Let $n$ be minimal such that $\alpha = \rho_{n+1}^{N_\eta}$.
    For each good point $\epsilon \in N^{(n)}_\eta\cap \mbox{Ord}$, $J^{\E}_{\epsilon} \in N^{(n)}_\eta$, and thus the $\alpha$-sized set  $X_\epsilon := h^{N_{\eta}  , \epsilon}_{n+1}[\omega \times \alpha]$ belongs to $N^{(n)}_\eta$. Since $\eta > \alpha$ is a regular cardinal in $N_\eta$, $\eta_{\epsilon} := \sup(X_\epsilon \cap \eta) < \eta$. Clearly, $\eta_{\epsilon_1} \leq \eta_{\epsilon_2}$ for $\epsilon_1<\epsilon_2$, and for every $\eta' < \eta$ there is some $\epsilon \in N^{(n)}_\eta$ such that $\eta'<\eta_\epsilon$. We conclude that both maps 
    $$\epsilon \mapsto \eta_\epsilon \text{, and }
    \epsilon \mapsto \sup(X_\epsilon \cap Ord) < (N_\eta \cap Ord)$$
    are (weakly) increasing and cofinal map from $N^{(n)}_\eta \cap Ord$ to $\eta$, $N_\eta \cap Ord$, respectively.  Hence $cf(\eta) = cf(N_\eta\cap Ord)$.
\end{proof}

We proceed to the definition of the clubs $c^{\alpha}_{\eta}$. Each such club is obtained from $d^{\alpha}_{\eta}$ by a screening process, as in the following definition. We remark that we slightly abuse the notation and identify pairs in $\omega^{<\omega}\times \alpha$ with ordinals in $\alpha$.

\begin{definition} \label{Def: finestructurethinout}
Let $\eta\in \mathcal{S}^{\alpha, \zeta, n}$ for some regular $\zeta\leq \alpha$ and $n<\omega$. Let $\delta\in (\alpha,\alpha^+)$ be such that $\{\delta\}= p^{ N_\eta }_{n+1}$. We define $c^{\alpha}_{\eta} = \la \eta_{i} \colon i< i^* \ra$, where $i^*$ and $\la \eta_{i} \colon i<i^* \ra$ are defined inductively, along with an increasing, continuous  sequence $\la \alpha_i \colon i<i^* \ra $ of ordinals below $\alpha$, as follows: Let $\eta_0 = \min(d^{\alpha}_{\eta})$.
Assume that $i<i^*$ and $\eta_i$ has been constructed. Let $\alpha_{i}$ be the least ordinal $\gamma<\alpha$ such that $h^{N_{\eta}}_{n+1}(\gamma, \delta)\notin \mbox{ran}(\sigma_{ \eta_i, \eta })$. Set $ \eta_{i+1} $ to be the least $\eta'\in d^{\alpha}_{\eta}$ above $\eta_i$ such that $h^{N_{\eta}}_{n+1}(\alpha_i,\delta)\in \mbox{ran}(\sigma_{\eta',\eta})$. For $i$ limit, assuming that $\la \eta_j \colon j<i \ra$ were chosen, let $\eta_i = \sup\{ \eta_j  \colon j<i\}$. 

\end{definition}

The filtering process producing ordinals $\eta_i$ along with a strictly increasing sequence of $\alpha_i < \alpha$ must halt at some stage $i^* \leq \alpha$. 
This process can be translated to a parallel filtering process on the good points $\epsilon_i$ which correspond to $\eta_i$ ($i<i^*$). 

\begin{definition}\label{Definition:Witnesses-alpha^epsilon}
In the above notations, given $\alpha < \kappa$ let $\epsilon_\alpha$ be the minimal good point $\epsilon$ so that $h^{N_{\eta, \epsilon}}_{n+1}\uhr \alpha\times\{\delta \} = h^{N_\eta}_{n+1}\uhr \alpha\times\{\delta\}$.
    We say that $\epsilon$ is \text{ very good } if there is $\alpha$ such that $\epsilon = \epsilon_\alpha$, and define its witness $\alpha^\epsilon < \kappa$ by 
$$ \alpha^\epsilon = \sup\{ \alpha \mid \epsilon = \epsilon_\alpha\}.$$
\end{definition}

Note that the set $\{ \alpha \colon \epsilon = \epsilon_{\alpha} \}$ is closed under limits, and thus the supremum in the above definition is actually a maximum.

\begin{lemma}
    Let $\eta\in \mathcal{S}^{\alpha}$, and let $\delta\in (\alpha, \alpha^+)$ be such that $p^{N_\eta}_{n_\eta+1} = \{\delta \} $. Assume that $\epsilon$ is a good point and $\eta'\in d^\alpha_{\eta}$ is its corresponding element in $d^\alpha_{\eta}$. Then $\eta'\in c^\alpha_{\eta}$ if and only if $\epsilon$ is a very good point.
\end{lemma}

\begin{proof}
    For simplicity we assume that $n_{\eta} =0$. Otherwise, argue from $N^{(n_\eta)}_\eta$.
    Assume first that $\eta' \in c^\alpha_{\eta}$. In the notations above, let $i<i^*$ be such that $\eta' = \eta_i$. Then for every $\gamma<\alpha_i$ with $h^{N_{\eta}}_{1}(\gamma, \delta)$ defined, $h^{N_{\eta}}_{1}(\gamma,\delta)\in \mbox{ran}\left( \sigma_{\eta_i,\eta} \right)$. But $\sigma_{ \eta_i, \eta }$ is the inverse of the transitive collapse of $h^{N_{\eta},\epsilon}_1[\alpha\cup \{ \delta \} ]$, and thus $h^{N_{\eta}, \epsilon}_1(\gamma, \delta)$ is defined. Therefore $$h^{N_{\eta},\epsilon}_{1} \uhr \alpha_i\times \{ \delta \} = h^{N_{\eta}}_{1}\uhr \alpha_i\times\{\delta \}.$$ 
    Let us argue that $\epsilon$ is minimal with this property; this implies that $\epsilon = \epsilon_{\alpha_i}$ is very good. Assume that  $\zeta<\epsilon$ is a good point. We argue that 
    $$h^{N_{\eta}, \zeta}_{1} \uhr \alpha_i\times\{\delta \} \neq h^{N_{\eta}}_{1}\uhr \alpha_i\times\{ \delta\}.$$ 
    Since  $\zeta$ is a good point, it defines a point $\eta_0\in d^\alpha_{\eta}$ below $\eta'$. Then $\mbox{ran}(\sigma_{\eta_0, \eta}) \subsetneq \mbox{ran}( \sigma_{\eta_i, \eta} )$, where the strict inequality follows since otherwise, $\eta'$ would have been chosen as the $i$-th element of $c^\alpha_{\eta}$, instead of $\eta_i$. Therefore, $ h^{N_{\eta}, \zeta}_{1}\uhr \alpha_i\times \{\delta \}  \neq  h^{N_{\eta}}_{1}\uhr \alpha_i\times \{\delta \}$.

    Assume now that $\epsilon$ is a very good point, and let us argue that $\eta'\in c^\alpha_{\eta}$. Write $\epsilon = \epsilon_{\alpha}$ for some $\alpha = \alpha^{\epsilon}<\kappa$. We argue that there exists $i<i^*$ such that $\alpha = \alpha_i$. Assume otherwise. Let $i<i^*$ be the least such that $\alpha< \alpha_{i}$. By the definition, for every $\beta\leq\alpha$, if $h^{N_{\eta}}_{1}(\beta,\delta)$ is defined, then $$h^{N_{\eta}}_1(\beta, \delta)\in \mbox{ran}(\sigma_{\eta_i, \eta}).$$ Denote by $\zeta$ the good point which corresponds to $\eta_i$. Then for every $\beta$ as above, $(\beta,\delta)\in \mbox{dom}(h^{N_{\eta}, \zeta}_{1})$. It follows that 
    $$h^{N_{\eta}, \zeta}_{1}\uhr (\alpha+1)\times \{\delta\} = h^{N_{\eta}}_{1}\uhr (\alpha+1)\times\{ \delta\}.$$ 
    This contradicts the minimality of $\epsilon$. 

    Thus, $\epsilon = \epsilon_{\alpha_i}$. Now, if $\eta'\notin c^\alpha_{\eta}$ then, since $c^\alpha_{\eta}$ is a club, there exists a maximal point $\eta_j\in c^\alpha_{\eta}$ below $\eta'$, and $\mbox{ran}(\sigma_{\eta_j,\eta}) = \mbox{ran}( \sigma_{ \eta',\eta } )$. So, if $\zeta <\epsilon $ is the very good point which defines $\eta_j$, then $h^{ N_{\eta}, \zeta }_{1}\uhr \alpha_i\times\{ \delta\} =h^{ N_{\eta} }_{1}\uhr \alpha_i \times \{ \delta\}$, contradicting the fact that $\epsilon = \epsilon_{\alpha_i}$ is the minimal good point $\mu$ such that  $h^{ N_{\eta}, \mu }_{1}\uhr \alpha_i\times \{\delta\} =h^{ N_{\eta} }_{1}\uhr \alpha_i\times\{\delta \} $.
    \end{proof}

It is clear from the definition that the map $\alpha \mapsto \epsilon_\alpha$ is (weakly) increasing, continuous, and $\bigcup_{\alpha < \kappa} \epsilon_\alpha = Ord \cap N^{(n_\eta)}_\eta$. 
It follows from the continuity property that $\alpha^\epsilon$ is the maximal value $\alpha < \kappa$ with $\epsilon = \epsilon_\alpha$.\\


Finally, focusing on the stationary sets ${\S}^\alpha$ of Proposition \ref{Prop: Stationary set of passive collapsing structures}, we have the following:
\begin{lemma} \label{Lemma: elements in Salpha have omega1-clubs}
    For every cardinal $\alpha \geq \omega_1$ and $\eta \in {\S}^\alpha \subseteq \alpha^+$,  $otp(c^\alpha_\eta) = \omega_1$.
\end{lemma}
\begin{proof}
    For each $\eta \in S^\alpha$, $cf(\eta) = cf(N_\eta \cap Ord) = \omega_1$ by Lemma \ref{Lemma:Cf(eta)=Cf(N)}. Therefore $\mbox{otp}(c^\eta_\alpha) \geq \omega_1$.
    For the other equality, note that the fact $h^{N_\eta}_{n_\eta+1}[\omega \times \omega_1]$ is cofinal in $N_\eta$ implies that the ordertype of very good points  in $N_\eta$ is at most $\omega_1$.\\
\end{proof}

\begin{corollary}\label{Cor: FineStructureS^alphaDoesNotReflect}
$\mathcal{S}^\alpha$ does not reflect. \end{corollary}
    
\begin{proof}
    For every $\eta$ of uncountable cofinality, $S^\alpha \cap c^\alpha_\eta$ contains at most one limit point of $c^\alpha_\eta$ below $\eta$ (even if $\eta\notin \mathcal{S}^\alpha$), by the previous lemma. Thus, there exists a club in $\eta$ disjoint from $\mathcal{S}^\alpha$.
\end{proof}

\subsection{Correspondence between Good Points at Different Heights}\label{subsec:Correspondence}

In this subsection, assume that $\kappa$ is a limit of regular cardinals.


\begin{lemma} \label{Lem: FineStructureCorrespondenceBetweenkappaalpha}
Let $\eta\in {\mathcal{S}}^{\kappa, \zeta,n}$ for some regular $\zeta\leq \kappa$ and $n<\omega$. Let $\delta\in (\kappa, \kappa^+)$ be such that $p^{N_\eta}_{n_\eta+1} = \{\delta \} $ is the associated parameter. Let $\alpha<\kappa$ be a cardinal such that:
\begin{itemize}
    \item $\alpha\in h^{N_{\eta}}_{n_\eta+1}[\omega\times(\alpha\cup \{\delta \}) ]$. 
    \item The transitive collapse of $h^{N_\eta}_{n_\eta+1}[\omega \times \delta ]$ belongs to $h^{N_{\eta}}_{n_\eta+1}[\omega\times(\alpha\cup \{\delta \}) ]$.
\end{itemize}
Then the transitive collapse of $h_{n_{\eta}+1}^{N_{\eta}}[\omega\times (\alpha\cup \{\delta \} )]$ is a collapsing structure  $N_{\eta(\alpha)}$ for some 
 $\eta(\alpha)\in \mathcal{C}^{\alpha} \setminus \mathcal{S}^{\alpha}$.
\end{lemma}

\begin{remark}The definition above can be adapted to collapsing structures with arbitrary parameters, by adding the requirement that for every $i\leq n_\eta+1$ and $\nu\in p_i := p_{1}^{N^{(i)}_\eta}$, the standard witness $W^{\nu,p_i}_{N^{(i)}_{\eta}}$ to the fact that $\nu\in p_i$ belongs to $h^{N_{\eta}}_{n_\eta+1} \left[ \omega\times(\alpha\cup p_{n_\eta+1} ) \right]$.\footnote{The standard witness $W^{\nu,p_i}_{N^{(i)}_{\eta}}$ is the transitive collapse of 
$h^{N^{(i)}_{\eta}}_{1}  \left[\omega\times (\nu\cup p_i\setminus (\nu+1))\right]$, also see definitions 7.3 and 7.10 of \cite{SchindlerZeman2009FineStruture}.}
\end{remark}

\begin{proof}

Arguing as in Lemma \ref{Lem: FineStructureGoodPts}, the transitive collapse of $h^{N_\eta}_{n_\eta+1}[ \omega\times (\alpha\cup \{\delta\})]$ is a $J$-structure, of the form $J^{\E'}_{\tau}$ for some extender sequence $\E'$ and ordinal $\tau$. The inverse of the collapse map $\pi \colon J^{\E'}_{\tau} \to N_{\eta}$ is $r\Sigma_{n_\eta+1}$-preserving. The ultimate projectum of $J^{\E'}_\eta$ is $\alpha$, and $n_\eta$ is the least such that $\rho^{ J^{\E'}_{\tau} }_{n_\eta+1} = \alpha$. Since the solidity witness for $p^{N_{\eta}}_{n_\eta+1} = \{\delta\}$ was included in the structure we collapsed, the standard parameter of $J^{\E'}_{\tau}$ is $\{\delta(\alpha
) \} $, where $\delta(\alpha)$ is the image of $\delta$ under the collapse (see the relevant argument from Lemma \ref{Lem: FineStructureGoodPts}). Finally, by carrying a comparison argument similar to Claim \ref{claim: comparison argument}, $\E'$ is an initial segment of $\E$. Let 
$$\eta(\alpha) = \alpha^{+}\cap h^{N_\eta}_{n_\eta+1}[ \omega\times (\alpha\cup \{\delta\})]\in (\alpha,\alpha^+).$$ 
We argue that $\eta(\alpha) = \mbox{crit}(\pi)$, $\pi(\eta(\alpha)) = \alpha^{+}$,  $\eta(\alpha)\in \mathcal{C}^\alpha$ and $J^{\E}_\tau$ is actually the collapsing structure associated to $\eta(\alpha)$. 
In order to prove that $\mbox{crit}(\pi) = \eta(\alpha)$, it suffices to argue that $X = h^{N_\eta}_{n_{\eta}+1}[\omega\times (\alpha\cup \{\delta \})]$ is transitive below $\eta(\alpha)$. Thus, assume that $\alpha\leq \gamma<\eta(\alpha)$. Let $\gamma^* = \pi(\gamma)$. Then $\gamma^*< \alpha^+$ and thus there exists a bijection $f\colon \alpha\to \gamma^*$ in $N_\eta$. Since this is a $\Sigma_1$ statement in parameters $\alpha,\gamma^*$, and $\alpha\in X$, we have such $f\in X$. Let $g\colon \alpha\to \gamma$ be a bijection function which is the image of $f$ under the collapse. Since both $f,g$ have domain $\alpha$, they coincide. So $f$ belongs to the transitive collapse of $X$, and $\gamma^* = \gamma$.

In order to see that $\eta(\alpha)\in \mathcal{C}^{\alpha}$, we shall prove that $J^{\E}_{\eta(\alpha)}\elem J^{\E}_{\alpha^+}$. Let $a\in J^{\E}_{\eta(\alpha)}$ be a parameter, and let $\varphi(a)$ be a formula. Let $\varphi^{J^{\E}_{\eta(\alpha)}}(a)$ be the formula obtained from $\varphi$ by bounding each quantifier into $J^{\E}_{\eta(\alpha)}$. Then:
\begin{align*}
 J^{\E}_{\eta(\alpha)}\vDash  \varphi( a) \iff & J^{\E}_{\tau} \vDash \varphi^{J^{\E}_{\eta(\alpha)}}(a) \\ \iff & N_{\eta}\vDash \varphi^{ J^{\E}_{\alpha^{+}} }(\pi(a)) \\ \iff & N_{\eta}\vDash \varphi^{ J^{\E}_{\alpha^{+}} }(a) \\ \iff& J^{\E}_{\alpha^{+}} \vDash \varphi( a )   
\end{align*}
where we used the fact that $\pi(a) = a$ since $a\in J^{\E}_{\eta(\alpha)}$ and $\eta(\alpha) = \mbox{crit}(\pi)$. $J^{\E}_{\tau} = N_{\eta(\alpha)}$ is the collapsing structure associated with $\eta(\alpha)$.\footnote{Note that $N_{\eta(\alpha)}$ is the collapsing structure of $\eta(\alpha)$ as an ordinal in $(\alpha,\alpha^+)$, while $N_{\eta}$ is the collapsing structure of $\eta\in (\kappa, \kappa^{+})$.} This follows since $h^{N_{\eta(\alpha)}}_{n_{\eta}+1}[\omega\times ( \alpha\times \{\delta(\alpha) \} )]$ belongs to $J^{\E}_{\eta(\alpha)+\omega}$ and contains $\eta(\alpha)$.

Finally, we argue that $\eta(\alpha)\notin \mathcal{S}^{\alpha}$. Indeed, $p^{N_{\eta(\alpha)}}_{n_{\eta(\alpha)+1}} = \{ \delta(\alpha)\}$, but $\delta(\alpha)$ is not the least ordinal outside of $h^{N_{ 
\eta(\alpha) }}_{n_{\eta(\alpha)+1}}[\omega\times \alpha]$, since the preimage of such an ordinal under the collapse is an ordinal of cardinality $\alpha$, but the preimage of $\delta(\alpha)$ is $\delta\in (\kappa, \kappa^+)$. So $\eta(\alpha)\notin \mathcal{S}^{\alpha}$. 
\end{proof}

The following lemma establishes a connection between good and very good points in $N_\eta$ and in $N_{\eta(\alpha)}$. 

\begin{lemma}\label{Lemma: FineStructureReflectDownGoodPts}
    Assume that $\eta\in \mathcal{S}^{\kappa, \zeta, n}$ for some regular $\zeta\leq \kappa$ and $n<\omega$.  Let $\alpha<\kappa$ be such that $\alpha$, and  the transitive collapse of $h^{N_\eta}_{n_\eta+1}[\omega\times \delta]$, belong to $h^{N_\eta}_{n_\eta+1}[\omega \times (\alpha\cup \{ \delta \})]$. Let $\epsilon\in N^{(n)}_\eta$, and denote by $\epsilon(\alpha)\in N^{(n)}_{\eta(\alpha)}$ its image under the transitive collapse. Then:
    \begin{enumerate}
        \item \label{Clause: FineStructureUprwardsGood} Assume that $\epsilon(\alpha)$ is a good point of $N_{\eta(\alpha)}$. Assume also that $\kappa \in h^{N_\eta, \epsilon}[\omega\times (\kappa \cup \{\delta \})]$. Then $\epsilon$ is a good point of $N_{\eta}$.
        \item \label{Clause: FineStructureDownwardsGood}Assume that $\epsilon$ is a good point of $N_{\eta}$. Assume also that $\alpha\in h^{N_\eta, \epsilon}_{n+1}[\omega\times (\alpha\cup \{\delta \} )]$ and $h^{N_\eta, \epsilon}_{n_\eta+1}[\omega\times (\alpha\cup \{\delta \} )]$ is unbounded in $\epsilon$. Then $\epsilon(\alpha)$ is a good point of $N_{\eta(\alpha)}$.
        \item \label{Clause: FineStructureDownwardsVeryGood} Assume  in addition to the assumptions of clause \ref{Clause: FineStructureDownwardsGood}, that $\kappa \in h^{N_\eta, \epsilon}_{n+1}[\alpha^{\epsilon}\cup \{\delta \}]$,\footnote{We abused the notation and identified pairs in $\omega\times \kappa$ with ordinals in $\kappa$ in a natural way. Formally, we want $\kappa\in h^{N_\eta, \epsilon}_{n+1}[\omega\times (\alpha^*\cup \{\delta \})]$ for some $\alpha^*$, such that the image of $\omega\times \alpha^*$ under this identification is bounded below $\alpha^{\epsilon}$. This could be ensured for every high enough $\epsilon$.} and $\alpha> \alpha^\epsilon$ (where $\alpha^\epsilon$ is the witness for $\epsilon$ being a very good point). Then $\epsilon(\alpha)$ is a very good point of $N_{\eta(\alpha)}$. 
    \end{enumerate}
\end{lemma}

\begin{proof}
\begin{enumerate}${}$
    \item Note that $h^{N_\eta, \epsilon}_{n+1}[\omega\times ( \alpha\cup \{\delta \} )]$ transitively collapses to $$h^{N_{\eta(\alpha)}, \epsilon(\alpha)}_{n+1}[\omega\times ( \alpha\cup \{\delta(\alpha) \} )],$$ 
    where $\delta(\alpha)$ is the image of $\delta$ under the transitive collapse. Since  
    $$h^{N_{\eta(\alpha)}, \epsilon(\alpha)}_{n+1}[\omega\times ( \alpha\cup \{\delta(\alpha)\} )]$$ 
    is unbounded in $\epsilon(\alpha)$,  $h^{N_{\eta}, \epsilon}_{n+1}[\omega\times ( \alpha\cup \{\delta\} )]$ is unbounded in $\epsilon$, and, in particular, $h^{N_{\eta}, \epsilon}_{n+1}[\omega\times ( \kappa\cup \{\delta\} )]$ is unbounded in $\epsilon$. Let $W$ be the transitive collapse of $h^{N_\eta, \epsilon}_{n+1}[\omega\times \delta]$. Then $W$ is also the transitive collapse of $h^{N_{\eta(\alpha)}, \epsilon(\alpha)}_{n+1}[\omega\times \delta(\alpha)]$. As 
    $$W\in h^{N_{\eta(\alpha)}, \epsilon(\alpha)}_{n+1}[\omega\times ( \alpha\cup \{\delta(\alpha) \} )],$$ 
    we get that $W\in h^{N_{\eta}, \epsilon}_{n+1}[\omega\times ( \kappa\cup \{\delta \} )]$. Finally, we have that $\kappa \in h^{N_\eta, \epsilon}[\omega\times (\kappa \cup \{\delta \})]$. All of the above shows that $\epsilon$ is a good point of $N_{\eta}$.
    
    \item The assumption that $h^{N_\eta, \epsilon}_{n_\eta+1}[\omega\times (\alpha\cup \{\delta \} )]$ is unbounded in $\epsilon$ implies that $h^{N_{\eta(\alpha)}, \epsilon(\alpha)}_{n_\eta+1}[\omega\times (\alpha\cup \{\delta(\alpha) \} )]$ is unbounded in $\epsilon(\alpha)$, where $\delta(\alpha)$ is the image of $\delta$ under the transitive collapse. The fact that $\alpha\in h^{N_\eta, \epsilon}_{n_\eta+1}[\omega\times (\alpha \cup \{ \delta\} )]$ implies that $\alpha\in h^{N_{\eta(\alpha)}, \epsilon(\alpha)}_{n_\eta+1}[\omega\times (\alpha\cup \{\delta(\alpha) \})]$. Finally, let $W\in h^{N_\eta,\epsilon}_{n+1}[\omega\times (\alpha\cup \{ \delta \} )]$ be the transitive collapse of $h^{N_\eta,\epsilon}_{n+1}[\omega\times \delta]$. Then $W\in h^{N_{\eta(\alpha)},\epsilon(\alpha)}_{n+1}[\omega\times (\alpha\cup \{ \delta(\alpha) \} )]$, and $W$ is also the transitive collapse of $h^{N_{\eta(\alpha), \epsilon(\alpha)}}[\omega\times \delta(\alpha)]$. All of the above shows that $\epsilon(\alpha)$ is a good point of $N_{\eta(\alpha)}$.
    
    \item Since $\epsilon$ is a good point witnessed by $\alpha^\epsilon$, $h^{N_\eta,\epsilon}_{n+1}\uhr \alpha^\epsilon \times \{ \delta \} = h^{N_\eta}_{n+1}\uhr \alpha^\epsilon \times \{ \delta \}$. Therefore  $$h^{N_{\eta(\alpha)}, \epsilon(\alpha)}_{n+1} \uhr \alpha^{\epsilon} \times \{\delta(\alpha) \} = h^{N_{\eta(\alpha)}}_{n+1} \uhr \alpha^{\epsilon} \times \{\delta(\alpha) \}.$$ We should also argue that $\epsilon(\alpha)$ is the least good point $\xi\in N^{(n)}_{\eta(\alpha)}$  such that $$h^{N_{\eta(\alpha)}, \xi}_{n+1} \uhr \alpha^{\epsilon} \times \{\delta(\alpha) \} = h^{N_{\eta(\alpha)}}_{n+1} \uhr \alpha^{\epsilon} \times \{\delta(\alpha) \}.$$ 
    
    Indeed, assume that $\zeta\in N^{(n)}_{\eta}$ collapses to $\zeta(\alpha)$ which is a good point below $\epsilon(\alpha)$ satisfying $h^{N_{\eta(\alpha)}, \zeta(\alpha)}_{n+1} \uhr \alpha^{\epsilon} \times \{\delta(\alpha) \} = h^{N_{\eta(\alpha)}}_{n+1} \uhr \alpha^{\epsilon} \times \{\delta(\alpha) \}$.    
    It follows that $h^{N_{\eta}, \zeta}_{n+1} \uhr \alpha^{\epsilon} \times \{\delta \} = h^{N_{\eta}}_{n+1} \uhr \alpha^{\epsilon} \times \{\delta(\alpha) \}$. If we will prove that $\zeta$ is a good point of $N_{\eta}$, we would contradict the minimality of $\epsilon$ being the least good point $\xi$ such that $h^{N_{\eta}, \xi}_{n+1} \uhr \alpha^{\epsilon} \times \{\delta \} = h^{N_{\eta}}_{n+1} \uhr \alpha^{\epsilon} \times \{\delta(\alpha) \}$ (we have $\zeta<\epsilon$ since $\zeta(\alpha)< \epsilon(\alpha)$).

    Thus, let us argue that $\zeta$ is a good point of $N_\eta$. We know that 
    $$\kappa\in \mbox{Im}\left(  h^{N_\eta, \epsilon}_{n+1} \uhr \alpha^\epsilon\times \{\delta \} \right)  = \mbox{Im}\left(  h^{N_\eta}_{n+1} \uhr \alpha^\epsilon\times \{\delta \} \right).$$ 
    Assuming that $\kappa(\alpha)$ is the image of $\kappa$ under the transitive collapse, we get that 
    $$\kappa(\alpha)\in \mbox{Im}\left(  h^{N_{\eta(\alpha)}}_{n+1} \uhr \alpha^\epsilon\times \{\delta(\alpha) \} \right) = \mbox{Im}\left(  h^{N_{\eta(\alpha)}, \zeta(\alpha) }_{n+1} \uhr \alpha^\epsilon\times \{\delta(\alpha) \} \right).$$ 
    It follows that $\kappa\in \mbox{Im}\left(  h^{N_{\eta}, \zeta}_{n+1} \uhr \alpha^\epsilon\times \{\delta \} \right) $. Thus, by clause $1$, $\zeta$ is a good point of $N_\eta$.
\end{enumerate}
\end{proof}

\begin{definition}
 A subset $Z$ of a cardinal $\kappa$ is nowhere stationary if for every limit ordinal $\alpha < \kappa$ there is a closed unbounded set $C_\alpha \subseteq\alpha$ which is disjoint from $Z$.   
\end{definition}

\begin{lemma} \label{Lem: FineStructureSectionNowhereStatSets}
Assume that $Z$ is a set of inaccessible cardinals, such that for every inaccessible $\alpha<\kappa$, $Z\cap \alpha$ is non-stationary in $\alpha$.  Then $Z$ is nowhere stationary.
\end{lemma}

\begin{proof}
We first argue that every regular $\alpha<\kappa$ has a club $C_{\alpha}\subseteq \alpha$ disjoint from $Z$. If, by contradiction, every club in such $\alpha$ intersects $Z$, then $Z$ is unbounded in $\alpha$. In particular,  $\alpha$ by itself is strong limit, since $Z$ consists of inaccessibles. In other words, $\alpha$ is inaccessible, and thus $Z\cap \alpha$ is nonstationary, which is a contradiction.

For singular $\alpha$, if $\alpha$ is not a limit of inaccessibles, clearly it has a club disjoint from $Z$. Else, $\alpha$ is a limit of inaccessibles. If $\cf(\alpha) = \omega$, we can manually construct a club in $\alpha$ disjoint from $Z$: first, take an $\omega$-seqeunce $\la \alpha_n \colon n<\omega \ra$  of inaccessibles in $\alpha$. Since $Z$ is nowhere stationary, for each $n<\omega$ there exists a club $C_{\alpha_n} \subseteq \alpha_n$ disjoint from $Z$. Pick for each such $n<\omega$ a point $\alpha^*_n < \alpha_n$ of $C_{\alpha_n}$. Then $C_\alpha = \{\alpha^*_n \colon n<\omega  \}$ is a club in $\alpha$ disjoint from $Z$.

Finally, assume that $\omega <\cf(\alpha)< \alpha$. Pick in $\alpha$ a cofinal sequence of order type $\cf(\alpha)$ above $\cf(\alpha)$. Then every accumulation point of this club is an ordinal above $\cf(\alpha)$ with cofinality strictly below $\cf(\alpha)$. So the club of such accumulation points consists of singulars, and thus is disjoint from $Z$.
\end{proof}

\begin{proposition}\label{Lemma:NowhereStationary}
Suppose that $Z \subseteq \kappa$ is a nowhere stationary set of inaccessible cardinals. Let $\tau \in \C^\kappa$ and $N_\tau$ be the collapsing structure of $\tau$ to $\kappa$, so that $Z \in N_\tau$ and $N_\tau \models Z \text{ is nowhere stationary}$.
    Let $\alpha_0 < \kappa$ be such that $Z =h^{N_\tau}_{n_\tau+1}[\omega \times (\alpha_0\cup p^{N_\tau}_{n_\tau+1})]$. Then for every $\alpha \in Z \setminus \alpha_0$,
    $\alpha \in h^{N_\tau}_{n_\tau+1}[\omega \times (\alpha\cup p^{N_\tau}_{n_\tau+1})]$.\\
\end{proposition}

\begin{proof}
    Since $Z \subseteq \kappa$ is nowhere stationary there is a sequence of closed sets $\vec{C} = \la C_\beta \mid \beta \leq \kappa \text{ limit}\ra$ such that for every $\beta \leq \kappa$ limit, $C_\beta \subseteq \beta$ is unbounded in $\beta$ and disjoint from $Z$. Since $Z \in N_\tau$ then so is the $<_L$-minimal such sequence $\vec{C}$. Moreover $\vec{C}$ is definable from $Z$ and therefore $\vec{C} \in h^{N_\tau}_{n_\tau+1}[\omega \times (\alpha_0\cup p^{N_\tau}_{n_\tau+1})]$. 

    Fix $\alpha\in Z\setminus \alpha_0$. We argue that $\alpha\in h^{N_{\tau}}_{n_\tau+1}[\omega \times(\alpha\cup p^{N_\tau}_{n_\tau+1})]$.
    Working in $N_\tau$ we define a process that generates two finite sequences of ordinals $\vec{\alpha} = \la \alpha_i \mid i \leq N\ra$ and $\vec{\beta} = \la \beta_i \mid i < N\ra$, with the following properties:
    \begin{itemize}
        \item $\vec{\alpha}$ is a strictly decreasing sequence, and consists of ordinals greater or equal to $\alpha$,
        \item $\vec{\beta}$ consists of ordinals below $\alpha$. 
    \end{itemize}
    Start from $\alpha_0 = \kappa$. 
    Suppose $\alpha_i$ has been defined for some ordinal $i$. If $\alpha_i = \alpha$ then the process terminates. 
    Assuming $\alpha_i > \alpha$ proceed as follows:  If $\alpha_i = \alpha' + 1$ is a successor ordinal then set $\beta_i =0$ and $\alpha_{i+1} = \alpha'$. 
    Otherwise, $\alpha_i$ is a limit ordinal and $C_{\alpha_i} \subseteq \alpha_i$ is closed unbounded and disjoint from $Z$, hence $C_{\alpha_i} \cap \alpha$ is bounded in $\alpha$. 
    Define  $\beta_i = \bigcup( C_{\alpha_i} \cap \alpha)$ and 
    $\alpha_{i+1} = \min( C_{\alpha_i} \setminus (\beta_i + 1))$. \\
    Since the sequence $\vec{\alpha}$ is strictly decreasing it terminates after a finite number of steps $N < \omega$, 
    Namely $\alpha_N = \alpha$. 
    Examining the process, it is clear that the sequence $\vec{\alpha}$ is definable from $\vec{C}$ and the finite sequence $\vec{\beta}$.
    Since $\alpha_0,\vec{\beta} \subseteq \alpha$ and $\vec{C}$ is definable in $\alpha_0$, it follows that $\vec{\alpha} \in h^{N_\tau}_{n_\tau+1}[\omega\times(\alpha\cup p^{N_\tau}_{n_\tau+1})]$. In particular $\alpha = \alpha_N \in \vec{\alpha}$ belongs to $h^{N_\tau}_{n_\tau+1}[\omega\times(\alpha\cup p^{N_\tau}_{n_\tau+1})]$. 
    \end{proof}

\begin{remark} \label{Rmk: FineStructureNowhereDenseLemmaWithRestrictedHull}
    The result and proof of Proposition \ref{Lemma:NowhereStationary} remain true if each $h^{N_\tau}_{n_\tau+1}$-hull is replaced by a restricted hull $h^{N_\tau, \epsilon}_{n_\tau+1}$, for a good point $\epsilon$ of $N_\tau$. 
\end{remark}

\begin{corollary}\label{Corollary: FineStructureReflectDownGoodPts}
Let $\eta \in \S^{\kappa, \zeta,n}$ for some $\zeta\leq \kappa$ regular and $n<\omega$. Let $\delta\in (\kappa,\kappa^+)$ be such that $p^{N_\eta}_{n+1} = \{\delta\}$. \\
Then there exists $\alpha^*<\kappa$ and a very good point $\epsilon$ of $N_\eta$, such that:
\begin{enumerate}
    \item For every $Z \subseteq \kappa$ such that $Z\in h^{N_\eta}_{n+1}[\omega\times (\alpha^*\cup \{\delta\})]$ and $N_\eta \models Z \mbox{ is nowhere stationary}$, and for every $\alpha\in Z\setminus \alpha^*$,
    $\alpha\in h^{N_\eta}_{n+1}[\omega\times (\alpha\cup \{  \delta \})]$. In particular,  $h^{N_\eta}_{n+1}[\omega\times (\alpha\cup \{  \delta \})]$ transitively collapses a structure $N_{\eta(\alpha)}$ for $\eta(\alpha)\in \mathcal{C}^{\alpha}$. 
    \item Assume that $\alpha\in h^{N_\eta}_{n+1}[\omega\times (\alpha\cup \{  \delta \})]$. Then, letting $\epsilon(\alpha)$ be the image of $\epsilon$ under the transitive collapse to $N_{\eta(\alpha)}$, $\epsilon(\alpha)$ is a very good point of $N_{\eta(\alpha)}$.
    \item Assume that $\alpha\in h^{N_\eta}_{n+1}[\omega\times (\alpha\cup \{  \delta \})]$. Then, letting $\eta'(\alpha) \in c^{\alpha}_{\eta(\alpha)}$ be the club-point induced by the very good point $\epsilon(\alpha)$, $\eta'(\alpha)\notin \mathcal{S}^\alpha$.
\end{enumerate}
Furthermore, $\alpha^*, \epsilon$ can be chosen arbitrarily high, in the sense that for every $\alpha'<\kappa$ and a good point $\epsilon'$ of $N_\eta$, we can ensure $\alpha^*>\alpha', \epsilon> \epsilon'$.
\end{corollary}

\begin{proof}
    Let $\alpha', \epsilon'$ be given. By enlarging them, we can assume that  $\kappa\in h^{N_\eta, \epsilon'}_{n+1}[\omega\times ( \alpha'\cup \{ \delta' \} )]$. 

    Let $W$ be the transitive collapse of $h^{N_\eta}_{n+1}[\omega\times \delta]$. Let $\alpha_1<\kappa$ be such that $W\in h^{N_\eta}_{n+1}[\omega\times (\alpha_1\cup \{\delta \})]$. Then for every $Z$ as above with $Z\in  h^{N_\eta}_{n+1}[\omega\times (\alpha_1\cup \{\delta \})]$, and for every $\alpha\in Z\setminus \alpha_1$, we have  $\alpha\in Z\setminus \alpha_1$, $\alpha\in h^{N_\eta}_{n_\eta+1}[\omega\times (\alpha\cup \{\delta \})]$ (see Proposition \ref{Lemma:NowhereStationary}).

    Pick a very good point $\epsilon$ above $\epsilon'$ such that $\cf(\epsilon)\neq \cf(\kappa)$, and the witness to $\epsilon$ being a very good point is above $\alpha', \alpha_1$. 

    Since $\epsilon$ is a good point, $h^{N_\eta, \epsilon}[\omega\times( \kappa\cup \{\delta \})]$ is unbounded in $\epsilon$. Since $\cf(\epsilon)\neq \cf(\kappa)$, there exists $\alpha_2<\kappa$ such that $h^{N_\eta, \epsilon}[\omega\times( \alpha_2\cup \{\delta \})]$ is unbounded in $\epsilon$.

    Let $\alpha^*$ be above $\alpha_2$ and $\alpha^{\epsilon}$ (and thus, also above $\alpha',\alpha_1$).

    We argue that $\alpha^*$ is as desired. For every $\alpha\in Z\setminus \alpha^*$, $\epsilon(\alpha)$ is a very good point (where $\epsilon(\alpha)$ is the image $\epsilon(\alpha)$ of $\epsilon$ under the transitive collapse from $N_\eta$ to $N_{\eta(\alpha)}$), by Lemma \ref{Lemma: FineStructureReflectDownGoodPts}, Remark \ref{Rmk: FineStructureNowhereDenseLemmaWithRestrictedHull} and the choice of $\alpha^*$.

    Let $\eta'(\alpha)\in c^{\alpha}_{\eta(\alpha)}$ be the club-point induced by $\epsilon(\alpha)$. We argue that $\eta'(\alpha)\notin \mathcal{S}^{\alpha}$. For that, note that $N_{\eta'(\alpha)}$ is the transitive collapse of $h^{N_{\eta(\alpha)},\epsilon(\alpha)}_{n+1}[\omega\times (\alpha \cup \{\delta(\alpha) \})]$, and thus  $p^{N_{\eta'(\alpha)}}_{n+1} = \{ \delta'(\alpha) \}$, where $\delta'(\alpha)$ is the image of $\delta(\alpha)$ under the collapse. However, the structure $h^{N_{\eta(\alpha)},\epsilon(\alpha)}_{n+1}[\omega\times (\alpha \cup \{\delta(\alpha) \})]$ is by itself isomorphic to     $h^{N_\eta, \epsilon}_{n+1}[\omega \times (\alpha \cup \{ \delta \})]$. So, if we start by collapsing $h^{N_\eta, \epsilon}_{n+1}[\omega \times (\alpha \cup \{ \delta \})]$, $\delta'(\alpha)$ will be the image of $\delta$ under the collapse map. Since $\delta\in (\kappa, \kappa^{+})$, $\delta'(\alpha)$ will not be the least element outside of $h^{N_{\eta'(\alpha)}}_{n+1}[\omega\times \alpha]$. Thus, $\eta'(\alpha)\notin \mathcal{S}^{\alpha}$.
\end{proof}

Corollary \ref{Corollary: FineStructureReflectDownGoodPts} will be crucial in the proof that the forcing we will construct, realizing the ``Kunen-like" blueprint, satisfies the Iteration-Fusion property. For the proof that the final segments of this forcing are distributive, we will need the following modified version.

\begin{lemma} \label{Lemma: FineStructureInfastructureSystemForDistributivityProof}
    Assume that $\zeta<\kappa$ is a regular cardinal, and, in the case where $\kappa$ is singular, $\cf(\kappa)< \zeta$. Let $\eta\in \mathcal{S}^{\kappa, \zeta, n}$ (for some $2\leq n < \omega$). Fix an ordinal $\gamma< \zeta$. Then there exists $\alpha^*< \kappa$ and an increasing  sequence $\la \epsilon_i \colon i\leq \gamma \ra$ such that:
    \begin{enumerate}
        \item For every nowhere stationary set $Z\in h^{N_\eta}_{n+1}[\omega\times (\alpha^*\cup \{\delta \})]$ and $\alpha\in Z\setminus \alpha^*$, $\alpha\in h^{N_\eta}_{n+1}[\omega\times (\alpha\cup \{\delta \})]$. In particular, for every such $\alpha$, $h^{N_\eta}_{n+1}[\omega\times (\alpha\cup \{\delta \})]$ transitively collapses to a structure $N_{\eta(\alpha)}$ for some  $\eta(\alpha)\in \mathcal{C}^\alpha$.
        \item For every $\alpha$ as above and  $i\leq \gamma$, the image of $\epsilon_i$ under the above transitive collapse is a very good point $\epsilon_i(\alpha)$ of $N_{\eta(\alpha)}$. Furthermore, $\epsilon_i(\alpha)$ is a limit of more than $\omega_1$-many very good points.
        \item For every $i\leq \gamma$,  $\la \epsilon_j \colon j\leq i \ra\in h^{N_\eta, \epsilon_{i+1}}[\omega\times (\alpha^*\cup \{\delta \})]$.
        \item For each $\alpha$ as above and $i\leq \gamma$, $\epsilon_i(\alpha)$ induces a club point $\eta_i(\alpha)\in c^{\alpha}_{\eta(\alpha)}\setminus \mathcal{S}^\alpha$.
    \end{enumerate}
    Furthermore, given any $\alpha'<\kappa$ and a good point $\epsilon'$, we can pick $\alpha^*$ above $\alpha'$ and each $\epsilon_i$ above $\epsilon'$. 
\end{lemma}

\begin{proof}
    By increasing $\alpha'$ below $\kappa$, assume that $h^{N_\eta}_{n+1}[\omega\times \delta], \kappa\in h^{N_\eta}_{n+1}[\omega\times (\alpha' \cup \{ \delta \})]$. By increasing $\epsilon'$, assume that $\kappa\in h^{N_\eta, \epsilon'}_{n+1}[\omega\times( \alpha'\cup \{  \delta\} )]$.

    If $\kappa$ is regular, construct an increasing sequence $\la \alpha_i, \epsilon_i \colon i\leq \gamma \ra$, where each $\alpha_i$ is above $\alpha'$, and each $\epsilon_i$ is above $\epsilon'$, as follows: assuming that $i<\gamma$ and $\la \alpha_j ,\epsilon_j \colon j<i\ra$ have been constructed, let $\la \alpha_i, \epsilon_i \ra$ be coordinate-wise above $\la \sup_{j<i} \alpha_i, \sup_{j<i} \epsilon_j \ra$ such that:
    \begin{itemize}
        \item $\epsilon_i$ is a very good point, and the limit of more than $\omega_1$-many very good points.
        \item $\la \epsilon_j \colon j\leq i \ra \in h^{N_\eta, \epsilon_{j+1}}_{n_\eta}[\omega\times \left( \alpha_i \cup \{  \delta \} \right)]$.
        \item For every $\alpha$ with $\alpha\in h^{N_\eta}_{n+1}[\omega\times (\alpha\cup \{ \delta\})]$, the image $\epsilon_i(\alpha)$ of $\epsilon$ under the transitive collapse to $N_{\eta(\alpha)}$, is a very good point.
        \item $\eta_{i+1}(\alpha)\notin \mathcal{S}^\alpha$, where $\eta_{i+1}(\alpha)$ is the club point induced by $\epsilon_{i+1}(\alpha)$, as a very good point in $N_{\eta(\alpha)}$. 
    \end{itemize}
    The fact that $\kappa$ is regular and $i<\gamma$, ensures that $\sup_{j<i} \alpha_i< \kappa$. The fact that $\cf(\eta) = \zeta$ and $i<\gamma< \zeta$ ensures that there are very good points above $\sup_{j<i} \epsilon_i $. Therefore, $\la \alpha_{i+1}, \epsilon_{i+1} \ra$ as above could be found by Corollary \ref{Corollary: FineStructureReflectDownGoodPts}. Finally, take $\alpha^* = \alpha_{\gamma}$. 

    Let us proceed to the case where $\kappa$ is singular. Recall that in this case we have $\zeta> \cf(\kappa)$. Fix an increasing, cofinal sequence $\la \alpha_j \colon  j<\cf(\kappa) \ra$ in $\kappa$. Make sure that $\alpha_0 > \alpha'$.
    
    Let $\la \epsilon_i \colon i<\zeta \ra$ be the increasing enumeration of very good points of $N_{\eta}$ above $\epsilon'$, such that for every $i<\zeta$, each $\epsilon_i$ is a limit of very good points, and $\cf(\epsilon_{i})> \kappa$.\footnote{Note that the enumeration of all such very good points is external to $N_\eta$.} For every $i<\zeta$, let $j(i)< \cf(\kappa)$ be the least index such that:
    \begin{itemize}
        \item The witness for the fact that $\epsilon_i$ is a very good point is below $\alpha_{j(i)}$.
        \item $h^{N_\eta  , \epsilon_i}_{n_\eta+1}[\omega \times \left(\alpha_{j(i)} \cup \{\delta \} \right) ]$ is unbounded in $\epsilon_i$ (since $\cf(\epsilon_{i})> \kappa$, there exists $j(i)$ fulfilling this clause).
        \item $\la \epsilon_j \colon j\leq i \ra\in h^{ N_\eta , \epsilon_{i+1} }_{n_{\eta}+1}[  \omega\times (\alpha_{j(i+1)} \cup \{\delta \}) ]$ (note that, since $\epsilon_{i+1}$ is a limit of very good points, the coherence of the square sequence ensures that $\la \epsilon_j \colon j\leq i \ra \in h^{ N_\eta , \epsilon_{i+1} }_{n_{\eta}+1}[  \omega\times (\kappa\cup \{\delta \})]$. Thus, there exists some $j(i)$ fulfilling this clause).
        \end{itemize}
    Since $\zeta$ is regular above $\cf(\kappa)$, there exists $j^*<\cf(\kappa)$ such that for unboundedly many $i<\zeta$, $j(i) = j^*$. Pick any $\alpha^*<\kappa$ above $\alpha_{j^*}$. Finally, any subsequence of $\la \epsilon_i \colon i<\zeta, j(i) = j^* \ra$ of length $\gamma$ will be as desired; the least such sequence of length $\gamma$ belongs to $N_{\eta}$ since $\cf(\eta) = \zeta$.
\end{proof}

\newpage
\section{Realizing the ``Kunen-Like" Blueprint}\label{Section:FinalIteration}

In this section, we construct a forcing notion $\po = \la \po_\alpha, \name{\qo}_\alpha \colon \alpha\leq \kappa \ra$ which satisfies the ``Kunen-like" blueprint \ref{Def:KLblueprint}, over the minimal extender model $V = L[\E]$. The section is divided into three subsections. In subsection \ref{Section:GeneralMiller} we present a forcing that adds a generalized Miller subset to a Mahlo cardinal $\alpha$. Such a forcing was developed by Friedman and Zdomskyy in \cite{FriedmanZdomskyy2010}, along with an iteration theory, that ensures that the resulting posets satisfy the $C$-Fusion property. In subsection \ref{Section:LocalIteration-SelfCodingMiller} we define the iterates $\name{\qo}_\alpha$, for each Mahlo $\alpha\leq \kappa$. Each such $\name{\qo}_\alpha$ is self coding iterate of $\alpha^{++}$-many generalized Miller subsets of $\alpha$. 
In subsection \ref{Section:FinalIteration} we prove that the iteration $\po$ satisfies the remaining properties from the ``Kunen-like" blueprint, mainly the Iteration-Fusion property and the distributivity of its tails. This subsection takes advantage of the special choice of the coding stationary sets constructed in Section \ref{Section:FineStucture}.

\subsection{Generalized Miller Forcing}\label{Section:GeneralMiller}

Let $\lambda$ be a regular cardinal. We present in this subsection a forcing notion $\mathbb{M}_{\lambda}$, which is a variation of a generalized Miller forcing due to Friedman and Zdomskyy \cite{FriedmanZdomskyy2010}. Denote by $[\lambda]^{<\lambda}$ the set of increasing sequences of ordinals below $\lambda$ of length  below $\lambda$. A condition in $\mathbb{M}_{\lambda}$ is a tree $p\subseteq [\lambda]^{<\lambda}$ which satisfies--
\begin{enumerate}
    \item Closure: for every limit $\alpha<\lambda$ and $s\in [\lambda]^{\alpha}$, if every initial segment of $s$ belongs to $p$, then $s\in p$ as well.
    \item Splitting: For every $s\in p$, the set-- $$C_{p}(s)= \{ \alpha<\lambda \colon s^{\frown} \la \alpha \ra\in p \}  $$
    is either a singleton, or a co-bounded\footnote{In \cite{FriedmanZdomskyy2010}, "closed unbounded" replaces "co-bounded". For us the current version suffices.} subset of $\lambda$. In the latter case, we say that $s$ is a \textbf{splitting node} of $p$.
    \item If $s\in p\cap \kappa^{\alpha}$ for a regular cardinal $\alpha$  then $C_{p}(s) = \{0\}$. In particular, splitting nodes of $p$ must have singular height.
    \item Density of splitting nodes: every $s\in p$ has an end-extension $t\in p$, $s\subseteq t$, which is a splitting node of $p$.
    \item Splitting at singular closure points: if $\alpha<\lambda$ is singular, $s\in p\cap \alpha^{\alpha}$ (i.e., $\alpha$ is a closure point of the function $s$) and unboundedly many initial segments of $s$ are splitting nodes of $p$, then $s$ is a splitting node of $p$ as well.
\end{enumerate}

Given a pair of conditions $p,q\in \mathbb{M}_{\lambda}$, we say that $p$ extends $q$ if $p\subseteq q$. For $p\in \mathbb{M}_{\lambda}$ and $t\in p$, we denote by $(p)_t\in \mathbb{M_{\lambda}}$ the sub-tree of $p$ consisting of all the nodes which are compatible with $t$ with respect to the end-extension order.

\begin{claim}\label{Claim:M_lambda-closed}
$\mathbb{M}_{\lambda}$ is $\lambda$-closed.
\end{claim}

\begin{proof}
Let $\la p_\alpha \mid \alpha < \delta\ra$ be a $\leq$-increasing sequence in $\mathbb{M}_\lambda$ of length $\delta < \lambda$. The natural candidate  for an upper bound for the sequence is the tree  $q= \bigcap_{\alpha < \delta} p_\alpha$. It is therefore suffices to prove $q \in \mathbb{M}_\lambda$ is a condition.\\
Suppose that $s \in q$. We first show $C_q(s) \neq \emptyset$ and is either a singleton or a co-bounded subset of $\lambda$. 
Since $\lambda$ is regular and $\delta < \lambda$, it is clear that if $C_{p_\alpha}(s)$ is co-bounded in $\kappa$, for all $\alpha < \delta$, then so does $C_q(s)$. 
In the case that $C_{p_\alpha}(s) = \{x\}$ for some $\alpha < \delta$, then for all $\beta < \delta$, $s \in p_{\alpha+\beta}$ and $p_\alpha \subseteq p_{\alpha+\beta}$ ensure that $C_{p_{\alpha+\beta}}(s) = \{x\}$. Therefore $C_q(s) = \{x\}$.\\
We see that in either case, every $s \in q$ and a proper end extension $s' \in q$. Since $q = \bigcap_{\alpha < \delta} p_\alpha$ must be close to $<\lambda$-limits of its sequences, we conclude that every $s \in q$ has end-extensions $s' \in q$ of arbitrarily long lengths $<\lambda$. The same argument regarding $q = \bigcap_{\alpha < \delta} p_\alpha$ gives that the set of splitting nodes in $q$ is closed under unions of singular heights. \\
Therefore, it remains to verify the density of the splitting nodes of $q$. Fix some $t\in q$. We argue that there exists a splitting node $s$ of $q$ with $s\supseteq t$. 
To this end, we construct a strictly increasing sequence $\langle t_{i} \colon i< \delta \rangle$ of nodes in $q$, such that, for every $\alpha<\delta$, there are unboundedly many indices $i<\delta$ such that $t_{i+1}$ is a splitting node of $p_{\alpha}$. We also make sure that $\mbox{lh}(t_0)> \delta$. This would suffice, since then $t^* = \bigcup\{ t_{i} \colon i<\delta \}$ is a splitting node of $p_{\alpha}$ for every $\alpha<\delta$, and thus, it is a splitting node of $q$ (this uses the assumption that $\mbox{lh}(t_0)>\delta$, and thus $\mbox{lh}(t^*)$ is singular below $\lambda$, so the closure of splitting nodes at singular heights could be applied). 
Take any $t_0 = t\in q$ with length above $\delta$. In limit steps, assume that $i<\delta$ is limit, and let $t_{i}= \bigcup\{ t_{j} \colon j<i\}$. Since all the trees $\la p_{\alpha} \colon \alpha<\delta \ra$ are closed, $t_i \in q$.
Thus, it suffices to concentrate on successor steps. Assume that we fixed in advance a partition $\la A_{\alpha} \colon \alpha<\delta  \ra$ of $\delta$ to unbounded subsets of $\delta$. Assume that $t_i$ has been constructed, and let us construct $t_{i+1}$. Let $\alpha<\delta$ be the unique such that $i\in A_{\alpha}$. We argue that we can find $t_{i+1}\in q$ such that $t_{i}\subsetneq t_{i+1}$ and $t_{i+1}$ is a splitting node of $p_{\alpha}$. \\
Indeed, let $t_{i+1}$ be the least splitting node of $p_{\alpha}$ above $t_{i}$ ("least" in the sense that there are no splitting points  $s\in p_{\alpha}$ such that $t_{i}\subsetneq s\subsetneq t_{i+1}$). Let us prove that for every $\beta\in \delta\setminus \{\alpha \}$, $t_{i+1}\in p_{\beta}$. Fix such $\beta$. By directedness, there exists $\gamma<\delta$ such that $p_{\gamma}\subseteq p_{\alpha} \cap p_{\beta}$. But $t_{i}\in p_{\gamma}$, and thus $t_{i}$ can be extended to a node of length $\mbox{lh}(t_{i+1})$ in $p_{\gamma}\subseteq p_{\alpha}\cap p_{\beta}$. However, by the choice of $t_{i+1}$, the only possible such node is $t_{i+1}$ itself. Thus $t_{i+1}\in p_{\beta}$. 
This concludes the inductive construction and the proof.
\end{proof}

Since $\lambda$-closed posets preserve the stationarity of all stationary subsets $S \subseteq \lambda$ in $V$ we conclude that $\lambda$ remains Mahlo.  
\begin{corollary}
    Assume that $\lambda$ is Mahlo. Then $\lambda$ remains Mahlo in $V^{\mathbb{M}_{\lambda}}$.
\end{corollary}

A generic set $G\subseteq \mathbb{M}_{\lambda}$ over $V$ induces a generic branch, 
$$g = \bigcup\{s\in \left[ \lambda \right]^{<\lambda} \colon \exists p\in G, (s\in p \land  (p)_s=p) \}$$
Clearly, $g\colon \lambda\to \lambda$ and $g(\alpha) =0$ for every regular cardinal $\alpha<\lambda$.

\begin{claim}\label{Claim: GeneralizedMillerAddsDominatingGeneric}
    For every $f\colon \lambda\to \lambda$ in $V$, there exists a club $C\subseteq \lambda$ such that for every singular $\alpha\in C$, $f(\alpha)< g(\alpha)$.
\end{claim}

\begin{proof} 
    Given $f\colon \lambda\to \lambda$, let $D\subseteq \mo_\lambda$ be the dense subset of conditions $p\in G$ such that for every splitting node $s\in p$, $\min\left(C_p(s)\right)> f(\mbox{lh}(s))$. Let $q\in G\cap D$. Every initial segment of $g$ is a node of $q$. Unboundedly many  such initial segments are splitting nodes of $q$, since, otherwise, $q$ has a node that doesn't extend into a splitting node. Since splitting nodes of $q$ are closed under singular limits at closure points of $g$, there exists a club $C\in V[G]$ in $\lambda$, such that for every singular $\alpha\in C$, $g\uhr \alpha$ is a splitting node of $q$, and thus $g(\alpha)> f(\alpha)$.
\end{proof}

We now proceed and prove that for a Mahlo cardinal $\lambda$, $\mathbb{M}_{\lambda}$ has the $\lambda$ C-Fusion property. For that we need the notion of a Fusion sequence of conditions in $\mathbb{M}_{\lambda}$.

\begin{definition}\label{Def: FusionSequenceMiller}
${}$
\begin{enumerate}
    \item Assume that $p\in \mathbb{M}_{\lambda}$. For every $\alpha<\lambda$, the set $\mbox{Split}_{\alpha}(p)$ consists of nodes $s\in p$ such that--
    $$\mbox{otp}\left( \{ t\subsetneq s \colon t \mbox{ is a splitting node of } p \} \right)= \alpha$$
    and $s$ is minimal with this property\footnote{Note that such $s$ is a splitting node if and only if $\mbox{lh}(s)$ is a singular cardinal.}.    
    \item If $p,q$ are conditions and $\nu<\lambda$, we say that $p\geq_{\nu} q$ if $p\geq q$ and $\mbox{Split}_{\nu}(p) = \mbox{Split}_{\nu}(q)$. 
    \item   Fix $\nu<\lambda$.  A $\nu$-fusion sequence of conditions in $\mathbb{M}_{\lambda}$ is an increasing sequence of conditions $\la p_{\alpha} \colon \alpha<\lambda
 \ra$, such that:
\begin{enumerate}
    \item  For every $\alpha<\lambda$, $p_{\alpha+1}\geq_{\nu+\alpha} p_{\alpha}$.
    \item  For every limit $\delta<\lambda$, $p_{\delta} = \bigcap_{\alpha<\delta} p_{\alpha}$.
\end{enumerate}
A fusion sequence of conditions is a sequence which is a $\nu$-fusion sequence for some $\nu<\lambda$.
\end{enumerate}
\end{definition}

\begin{claim}
Assume that $\nu<\lambda$ and $\la p_{\alpha} \colon \alpha<\lambda \ra$ is a $\nu$-fusion sequence. Then $q = \bigcap_{\alpha<\lambda} p_{\alpha}$ is a condition, and, for every $\alpha<\lambda$, $q\geq_{\nu+\alpha} p_{\alpha}$.
\end{claim}

\begin{proof}
The property $q\geq_{\nu+\alpha} p_{\alpha}$ for $\alpha < \lambda$ is a property of trees and makes sense without the assumption of $q$ being a condition. This property is clearly holds for the tree $q = \bigcap_{\alpha<\lambda} p_{\alpha}$ and we use to verify that $q$ is a condition. We focus on verifying that the splitting nodes of $q$ are dense in $q$. The verification of the other properties is similar or simpler and left to the reader. \\

\noindent Fix $s \in q$. We split the argument into two cases.\\
Assume first that
$$\mbox{otp}\left( \{ t\subsetneq s \colon t \mbox{ is a splitting node of } p \} \right)<\nu.$$
Let $s^*$ be a minimal splitting node of $p_0$ above $s$. Then for some $\nu'\leq\nu$, $s^*\in \mbox{Split}_{\nu'}(p_0)$. But $q\geq_{\nu} p_0$, since, for every $\alpha<\lambda$, $p_{\alpha}\geq_{\nu} p_0$. Thus, $\mbox{Split}_{\nu'}(q) = \mbox{Split}_{\nu'}(p_0)$, and, in particular, $s^*\in q$. \\
Next, assume there is $\alpha<\lambda$ such that 
$$\nu+\alpha = \mbox{otp}\left( \{ t\subsetneq s \colon t \mbox{ is a splitting node of } p \} \right).$$
Let $s^*\supseteq s$ be a minimal splitting point of $p_{\alpha+1}$ above $s$. Then for every $\eta\in C_{p_{\alpha+1}}(s^*)$ and $\beta \geq \alpha+1$, ${s^*}^{\frown} \la \eta \ra\in \mbox{Split}_{\nu+\alpha+1}({p_{\alpha+1}}) = \mbox{Split}_{\nu+\alpha+1} (p_{\beta})$. Therefore,  $C_{p_{\alpha+1}}(s^*) = C_{q}(s^*)$, and in particular, $s^*$ is a splitting point of $q$.
\end{proof}

\begin{lemma} \label{Lem: MillerC-FusionProperty}
    Assume that $\lambda$ is Mahlo. Then $\mathbb{M}_{\lambda}$ satisfies the $\lambda$ C-Fusion property.
\end{lemma}

\begin{proof}
    We will prove a slightly stronger version of the C-Fusion property that will be used in Section \ref{Section:LocalIteration-SelfCodingMiller}\footnote{E.g., see Lemma \ref{Corollary: MillerApproximationOfAnOrdinal}} when we deal with back-to-back iteration of Miller and coding posets. We shall show that for every $\nu < \lambda$, together with C-Fusion parameters of $p\in \mathbb{M}_{\lambda}$, $A \subseteq H_{\lambda^{+2}}$, and a sequence of dense open sets $\vec{D} = \langle D_\alpha \colon \alpha<\lambda \rangle$, there is a condition $$q^* \geq_\nu p$$
    and continuous and internally approachable chain $\vec{X} = \la X_\alpha \mid \alpha < \lambda\ra$ of elementary substructures  $X_\alpha \elem (H_{\lambda^{+2}},\vec{D},A)$ that satisfy the requirements of the $\lambda$ C-Fusion property (i.e., requirements (a)-(c) in Definition \ref{Def:BlueprintPrelim}).\\
    
    \noindent 
    We start by taking our sequence $\vec{X} = \la X_\alpha \mid \alpha < \lambda$ to satisfy--
    \begin{itemize}
        \item $p, \mathbb{M}_{\lambda}, \vec{D}\in X_0$, and $\nu+1\subseteq X_0$. 
        \item For every $\alpha<\lambda$,  $\alpha\subseteq X_\alpha$.
        \item For every $\alpha<\lambda$, $X_{\alpha+1}$ is closed under $\alpha$-sequences of its elements. In particular $\alpha+1 \subseteq X_{\alpha+1}$.
    \end{itemize}
    We construct a $\nu$-fusion sequence $\langle p_{\alpha} \colon \alpha<\lambda \rangle$ of conditions in $\mathbb{M}_{\lambda}$, such that, for every $\alpha<\lambda$, $p_\alpha,p_{\alpha+1}\in X_{\alpha+1}$ and, if $t\in \mbox{Split}_{\nu+\alpha}(q_{\alpha+1})$, then for every $\eta\in C_{q_{\alpha+1}}(t)$,
    $$\left(q_{\alpha+1}\right)_{t^{\frown} \langle \eta \rangle }\in D^*_{\alpha}:= D_\alpha \cap \left(\bigcap\{ D \colon D\in X_{\alpha} \mbox{ is a dense open subset of } \mathbb{M}_{\lambda}\}\right).$$   
 $D^*_{\alpha}$ is dense open since $|X_\alpha|<\lambda$, by $\lambda$-closure of $\mathbb{M}_{\lambda}$. Also, $D^*_\alpha\in X_{\alpha+1}$ by internal approachability and by the fact that $D_\alpha\in X_{\alpha+1}$.\\
    We start the sequence with $p_0 = p$, and for a limit ordinal $\alpha$, take $p_\alpha = \bigcap_{\beta<\alpha}p_{\beta}$. By the Claim \ref{Claim:M_lambda-closed} and the inductive assumption for $\la p_\beta \mid \beta <\alpha\ra$,  $p_\alpha$ is a condition. We have $p_\alpha\in X_{\alpha+1}$ since $X_{\alpha+1}$ is closed under $\alpha$ sequences of its elements. Let us take care of the successor step. Assume that $p_{\alpha}\in X_{\alpha+1}$ has been constructed, and construct $p_{\alpha+1}$ inside $X_{\alpha+1}$ as follows: for every $s\in \mbox{Split}_{\nu+\alpha}( p_{\alpha} )$, and for every $\eta\in C_{p_{\alpha}}(s)$, shrink $p_{\alpha}$ above $s^{\frown} \la \eta \ra$ so it enters $D^*_{\alpha}$. Let the resulting condition be $p_{\alpha+1}$, Note that if $s$ is a splitting points of $p_\alpha$ then it remains a splitting node of $p_{\alpha+1}$. It is therefore clear that $p_{\alpha+1}\in \mo_\lambda$ and $p_{\alpha+1}\geq_{\nu+\alpha} p_{\alpha}$. Since this was done in $X_{\alpha+1}$, $p_{\alpha+1}\in X_{\alpha+1}$. \\

\noindent
    This concludes the construction. Let $q^* = \bigcap_{\alpha<\lambda} p_{\alpha}$. Then $q^*\in \mathbb{M}_{\lambda}$ as the intersection of a Fusion sequence of conditions, and $q^*\geq_{\nu} p$. Let $C = C_1 \cap C_2$, where $C_1, C_2$ are the following club subsets of $\lambda$:
    $$C_1 = \{ \alpha<\lambda \colon X_{\alpha}\cap \lambda = \alpha \}$$
    \begin{align*}
      C_2 = \{ \alpha<\lambda \colon &\mbox{for every } \beta<\alpha \mbox{ and } t\in \beta^{\beta}\cap q^*, \mbox{ either } t \mbox{ is a splitting node of }q^*, \\
      &\mbox{or the least splitting node of }(q^*)_t \mbox{ belongs to }  \alpha^{<\alpha}\cap q^* \}.  
    \end{align*}
  Let us prove that $q^*, \la X_\alpha \colon \alpha<\kappa \ra, C$ are as desired for the $C$-Fusion property:
    \begin{enumerate}
        \item Let $\alpha\in C$ be regular, and assume that $G_{X_\alpha}\subseteq X_\alpha\cap \mathbb{M}_{\lambda}$ is $X_\alpha$-generic which is compatible with $q^*$. Let us prove that $G_{X_\alpha}\cup \{ q^* \}$ has an exact upper bound. 

        Let $s\in \lambda^{<\kappa}$ be the initial segment of the generic branch determined by $G_{X_{\alpha}}$. Namely, 
        $$s = \bigcup\{ stem(p) \mid p \in G_{X_\alpha}\}.$$
        We note that $lh(s) = \alpha$. This is because for every $\beta<\alpha$, $G_{X_{\alpha}}$ meets the dense open set of conditions deciding the generic branch up to height $\beta$. Also, as $X_{\alpha}\cap \lambda = \alpha$, $G_{X_{\alpha}}$ does not include a condition whose stem extends has length $> \alpha$.\\

        \noindent 
        Note that as $\alpha$ is regular, $C_{p}(s) = \{0\}$ for all $p \in G_{X_\alpha}$. This, and the fact since $q^*$ is compatible with $G_{X_{\alpha}}$ imply that $s^{\frown} \langle 0 \rangle\in q$. \\

        \noindent
        We argue that $s\in \mbox{Split}_{\nu+\alpha}(q^*) = \mbox{Split}_{\alpha}(q^*)$\footnote{Note that $\alpha > \nu$ since $\nu\in X_\alpha$ and $\alpha\in C_1$.}. Assume otherwise. Since $\alpha$ is regular and $\mbox{lh}(s) = \alpha$, the set splitting nodes of $q^*$ along $s$ is contained in $\beta^{<\beta}$ for some $\beta<\alpha$. The regularity of $\alpha$ also implies that the set of closure points of $s$ (namely, ordinals $\beta<\alpha$ such that $s\uhr \beta\in \beta^{\beta}$) is a club in $\alpha$. Thus, we can pick a closure point $\beta^*$ of $s$ such that $\beta^* > \beta$. Then $s\uhr \beta^*$ is not a splitting node of $q^*$, and the least splitting node of $(q^*)_{s\uhr \beta^*}$ has height above $\alpha$, contradicting the fact that $\alpha\in C_2$. \\

        \noindent 
        Next, we need to show $G_{X_\alpha}\cup \{q^*\}$ has an exact upper bound, and that this bound belongs to $D_\alpha$. Since $\alpha$ is regular $(q^*)_s = (q^*)_{s \fr \la 0\ra} \in D^*_{\alpha} \subseteq D_\alpha$. Therefore, it suffices to show that $(q^*)_s$ is indeed an exact upper bound of $G_{X_\alpha}\cup \{q\}$. \\
        We first argue that $(q^*)_s$ extends every element of $G_{X_\alpha}$. Fix $q\in G_{X_{\alpha}}$. Consider the dense open set $D\in X_{\alpha}$ of conditions which either extend or are incompatible with $q$. Since $\alpha$ is a limit ordinal, there exists $\beta<\alpha$ such that $D\in X_{\beta}$. By the remarks above, there exists $\beta^* \geq \beta$ below $\alpha$ such that $s\uhr \beta^* \in \mbox{Split}_{\nu+\beta}(q^*)$. By our construction, $(p_{\beta+1})_{s\uhr (\beta^*+1)} \in D$, so $(p_{\beta+1})_{s\uhr (\beta^*+1)}$ either extends $q$ or is incompatible with it. Since both $(p_{\beta+1})_{ s\uhr \beta^*+1 }$ and   $q$ are in $G_{X_\alpha}$, $(p_{\beta+1})_{s\uhr \beta^*+1}$ extends $q$. Thus, $(q^*)_s$ extends $q$.\\
        \noindent
        Finally, to prove exactness, assume that $q'\in \mo_\lambda$ is an upper bound of  $G_{X_{\alpha}}\cup \{q^* \}$. Since $q'$ bounds $G_{X_{\alpha}}$, $s$ is the stem of $q'$, namely, $s\in q'$ and $(q')_s = q'$. This, together with the assumption that $q'\geq q^*$ ensures that $q'\geq (q^*)_s$.

        \item It remains to verify that $q^*$ is a generic condition for $X = \cup_{\alpha<\lambda} X_\alpha$. Assume that $G\subseteq \mathbb{M}_{\lambda}$ is generic over $V$ with $q^*\in G$. We argue that $G\cap X$ is $X$-generic. Assume that $D\in X$ is a dense open subset of $\mathbb{M}_{\lambda}$. Pick a regular $\alpha<\lambda$ such that $D\in X_{\alpha}$,  $\alpha$ is a closure point of the generic branch $g\in \left[ \lambda\right]^{\lambda}$ decided by $G$, and $g\uhr \alpha\in \mbox{Split}_{\alpha}(q^*)$. Note that such $\alpha$ exists since $\lambda$ remains Mahlo in $V[G]$, by Lemma \ref{Lem: MillerC-FusionProperty}. By our construction, $(p_{\alpha+1})_{g\uhr \alpha} = (p_{\alpha+1})_{g\uhr \nu+\alpha}\in D$. Therefore, $(q^*)_{g\uhr \alpha}\in D\cap G\cap X$.
    \end{enumerate}
\end{proof}

\begin{corollary}
Assume that $\lambda$ is Mahlo. Then the forcing $\mathbb{M}_{\lambda}$ preserves cardinals.    
\end{corollary}

\begin{proof}
Cardinals up to $\lambda$ are preserved by $\lambda$-closure of $\mathbb{M}_{\lambda}$. Also, since $|\mathbb{M}_{\lambda}| = \lambda^+$, cardinals above $\lambda^{+}$ are preserved. So it suffices to prove that $\lambda^{+}$ is preserved.
Assume that $p\in \mathbb{M}_{\lambda}$ and $\name{f}$ is a name for an increasing function from $\lambda$ to $\lambda^+$. By \ref{Lem: CFusionApproximation}, there exists $p^*\geq p$, a function $F\colon \lambda\to [\lambda^{+}]^{<\lambda}$ and a name $\name{C}^*$ for a club in $\lambda$, such that 
$$p^*\Vdash \forall \alpha\in \name{C}^*\cap Reg, \name{f}(\alpha)\in F(\alpha).$$ 
In $V$, let $\xi^* = \sup(\cup\mbox{Im}(F))+1 < \lambda^{+}$. We argue that $p^*\Vdash \mbox{Im}(f)\subseteq \xi^*$. Indeed, assume that $G\subseteq \mathbb{M}_{\lambda}$ is generic over $V$ with $p^*\in G$. For every $\alpha<\lambda$, let $\alpha^* = \min((C^*\cap Reg)\setminus \alpha)$. Note that $\lambda$ is Mahlo in $V[G]$, so $\alpha^*$ exists for every $\alpha<\lambda$. But then $f(\alpha)\leq f(\alpha^*)\leq \sup(F(\alpha^*))< \xi^*$.
\end{proof}

\begin{lemma} \label{Corollary: MillerApproximationOfAnOrdinal}
    Assume that $\lambda$ is Mahlo, $p\in \mo_\lambda$, $\nu<\lambda$ and $\name{\zeta}$ is a $\mo_\lambda$-name for an ordinal. Then there exists $q\geq_{\nu}p$ and a set $x\in V$ with $|x|<\lambda$ such that $q\Vdash \name{\zeta}\in \check{x}$.
\end{lemma}

\begin{proof}
    Since $\mo_\lambda$ in $\lambda$-closed and satisfies the $C$-fusion property, we can apply Lemma \ref{Lem: CFusionApproximation} on the constant function $f\colon \lambda\to \{ \zeta \}$ in $V^{\mo_\lambda}$, to find $q\geq p$ and a ground model function $F\colon \lambda\to [\mbox{Ord}]^{<\lambda}$ such that, for some $\xi<\lambda$, $q\Vdash \name{f}(\xi)\in \check{F}(\xi)$. The set $x = F(\xi)$ is as desired. Going back to the proof of Lemma \ref{Lem: CFusionApproximation}, the condition $q$ is obtained from $p$ using the $C$-fusion property; by the proof of Lemma \ref{Lem: MillerC-FusionProperty}, we can pick $q$ such that $q\geq_{\nu} p$.
\end{proof}

\subsection{Self-Coding Iteration of Miller Forcings}\label{Section:LocalIteration-SelfCodingMiller}

Fix a Mahlo cardinal $\alpha\leq \kappa$ in $L[\E]$. In this section, we define a forcing notion $\qo_\alpha$ that will play the role of the $\alpha$-th forcing in the ``Kunen-like" blueprint \ref{Def:KLblueprint}. We follow the following conventions throughout this section. 
\begin{itemize}
    \item $\qo_\alpha$ is defined in a ground model $V$ which is a generic extension of $L[\E]$ by a poset $\po_\alpha$ that satisfies the ``Kunen-like" blueprint. In particular, $\po_\alpha$ satisfies the $\alpha$ Iteration-Fusion property, $\alpha$ remains a Mahlo cardinal in $V$ by Lemma \ref{Lem: PreservationOfStationarySetsForIterationFusionProperty} and $(2^{\alpha})^V = \alpha^{+}$  by Lemma \ref{Lem: PowersetOfKappaIsPreserved}. 
  
    \item There exists in $L[\E]$ a sequence $$\vec{S}^\alpha = \la \mathcal{S}^{\alpha}(\tau, \eta, i )  \colon \tau<\alpha^{++}, \eta<\alpha^{+}, i<2 \ra $$
    of almost disjoint stationary subsets of $\alpha^{+}$, all of them contained in the stationary set $\mathcal{S}^{\alpha}\subseteq \alpha^{+}\cap \cf(\omega_1)$ that was defined over $L[\E]$ in proposition \ref{Prop: Stationary set of passive collapsing structures}. We further assume that
    \begin{itemize}
        \item $\mathcal{S}^{\alpha}$ does not reflect.
        \item For each $\tau<\alpha^{++}$ the sets in the subsequence $\la \mathcal{S}^\alpha(\tau, \eta, i) \colon \eta<\alpha^{+}, i<2 \ra$ are pairwise disjoint.
        \item For each $\tau<\alpha^{++}$,  $\eta<\alpha^+$ and $i<2$, $\min\left(\mathcal{S}^\alpha(\tau, \eta,i)\right) > \eta$.
        \item The sequence $\la \vec{S}^\alpha \colon \alpha\leq \kappa \mbox{ is Mahlo} \ra$ is uniformly definable over $L[\E]$, in the sense that there exists a formula $\varphi_{S}(\alpha, x)$ which defines it in every inner model $M$ containing $H_{\alpha^{++}}$.
    \end{itemize} 
    \item We assume that the forcing $\po_\alpha$ preserves the stationarity of stationary subsets of $\mathcal{S}^\alpha$ (this will be proved in Lemma \ref{Lem: FinalForcingPreservationOfStatSets}). So we can assume that $\mathcal{S}^\alpha(\tau, \eta,i)$ remains stationary in $V$ for every $\tau<\alpha^{++}$, $\eta<\alpha^+$ and $i<2$. 
\end{itemize}

Let us sketch some of the main ideas behind the definition of $\qo_\alpha$ in this subsection. Recall that according to the ``Kunen-like" blueprint,  $\qo_\alpha$ should satisfy the following properties in $V$:
\begin{itemize}
            \item If $\alpha <\kappa$,  $\qo_\alpha\in V_{\kappa}$,\footnote{This implies that there exists a $\po_\kappa$-name for $\qo_\alpha$ in $(V_\kappa)^{L[\E]}$.}
            
            \item $|{\qo}_\alpha|=\alpha^{++}$.
            
            \item $\qo_\alpha$ adds an increasing sequence in $<^{\alpha}_{sing}$ of length $\alpha^{++}$.
            
            \item ${\qo}_\alpha$ is $\alpha$-distributive, satisfies the $\alpha^{++}$.c.c, and preserves the Mahloness of $\alpha$,
            
            \item ${\qo}_\alpha$ has the $\alpha$ C-Fusion property,

            \item ${\qo}_\alpha$ is self coding.
\end{itemize}

\noindent
In order to ensure that the poset $\qo_\alpha$ adds an increasing sequence in $<^\alpha_{sing}$ of length $\alpha^{++}$, it includes an $\alpha^{++}$-iteration of the Miller forcing $\mo_\alpha$. By Claim \ref{Claim: GeneralizedMillerAddsDominatingGeneric}, the sequence of generic subsets of $\alpha$ produced this way will be $<^{\alpha}_{sing}$-increasing. \\
In order to achieve $\qo_\alpha$ which is self-coding, we interleave coding posets between the Miller posets, such that each coding poset codes the initial segment of the iteration up to its stage by adding a club which is disjoint from certain sets $\mathcal{S}^\alpha(\tau,\eta,i)\subseteq \alpha^+$ from the given list.\footnote{The interleaved structure of the iteration seems to be necessary since a coding poset that adds a club to $\alpha^+$ will collapse $2^\alpha$ to $\alpha^+$.}
\\
Our construction of $\qo_\alpha$ with the these two features relies on the iteration theory of higher Miller posets from  \cite{FriedmanZdomskyy2010}. The main arguments in this section combine fusion methods from \cite{FriedmanZdomskyy2010} with closure arguments of club shooting posets involving IA chains of models, that will guarantee $\qo_\alpha$ has the C-Fusion property.
The poset $\qo_\alpha$ will be equivalent to a back-to-back $\leq\alpha-$support iteration of Miller and coding posets. The back-to-back iteration will be denoted by $\hqo^\alpha$. The reason we do not take $\qo_\alpha$ to be the ``full" back-to-back iteration is that proper initial segments of this iteration have size $\alpha^{++}$, which is larger than needed and create various issues with the coding mechanism. Instead, we show that each  initial segment $\hqo^\alpha_\tau$, $\tau < \alpha^{++}$,  of $\hqo^\alpha$ has a dense subset $\qo^\alpha_\tau$ of size $|\qo^\alpha_\tau| = \alpha^+$. \\
\noindent
We proceed to define by induction on $\tau \leq \alpha^{++}$ an iteration 
$$\la \hqo^{\alpha}_\tau,\name{\ro}_\tau \mid \tau < \alpha^{++}\ra$$
together with dense subsets 
$$\qo^\alpha_\tau \subseteq \hqo^\alpha_\tau.$$
For each $\eta \leq \alpha^{++}$, 
$$ \hqo^{\alpha}_\eta=\la \hqo^{\alpha}_\tau,\name{\ro}_\tau \mid \tau < \eta\ra$$ 
is a $\leq\alpha$-support iteration of length $\eta$, where
for every $\tau< \eta$, ${\ro}_\tau \in V^{\qo^{\alpha}_\tau}$ has the form $$\ro_\tau = {\mo}_\alpha*{\co}^{\alpha}_\tau,$$ 
with $\mo_\alpha$ being the Miller poset on $\alpha$ from Section \ref{Section:GeneralMiller}, and $\co^{\alpha}_\tau \in V^{\qo^{\alpha}_\tau* \name{\mo}_\alpha}$ is a coding poset that codes a generic set of size $\alpha^+$ by adding a certain closed unbounded subset of $\alpha^+$ (see Definition \ref{Def: Definition of the Coding poset} for more details). \\

\noindent
Therefore, conditions $p \in \hqo^{\alpha}_\tau$ are sequences 
$$p = \la p(\xi) \mid \xi \in \supp(p)\ra$$ 
with $\supp(p) \in [\tau]^{\leq \alpha}$. For each $\xi \in \supp(p)$, 
$$p\uhr \xi = \la p(\eta) \mid \eta \in \supp(p) \cap \xi\ra\in \qo^{\alpha}_\xi$$ 
and 
$$p\uhr \xi \Vdash_{\hqo^{\alpha}_\xi} p(\xi) \in \name{\ro}_\xi.$$
We further write
$$p(\xi) = (\name{T}^p_\xi,\name{c}^p_\xi),$$
where 
$$p\uhr \xi \Vdash \name{T}^p_\xi \in \mo_\alpha, \text{ and }  p\uhr \xi \fr  \name{T}^p_\xi \Vdash \name{c}^p_\xi \in \name{\co}^\alpha_\xi.$$

\noindent
As mentioned above, we will introduce a dense subset $\qo^\alpha_\tau$ 
 of $\hqo^\alpha_\tau$ of size $|\qo^\alpha_\tau| = \alpha^+$ (Lemma \ref{Lem: MillerCodingSecondDensityLemma}). Assuming $\qo^{\alpha}_\tau$ is given, we define the coding posets $\co^\alpha_{\tau}$ in $V^{{\qo^{\alpha}_\tau} * \name{\mo}_\alpha}$. Since $\alpha$ is fixed in this subsection, we will use the notation $\co_\tau$ instead of  $\co^\alpha_\tau$ until the end of the subsection.

\begin{definition} \label{Def: Definition of the Coding poset}${}$\\
Let $\tau < \alpha^{++}$ and  $\qo_\tau^\alpha \subseteq \hqo^\alpha_\tau$ be a dense subset of size $|\qo^\alpha_\tau| = \alpha^+$. \\
Let $\la q^\alpha_\tau(\eta) \mid \eta < \alpha^+\ra$ be the well-order of of $\qo^\alpha_\tau$ induced by the constrcutible ordering $L[\E]$ of their $\po_{\alpha}$-names in $L[\E]$.\footnote{where the names are taken minimal with respect to the constructible order.} Let ${G}({\hqo^{\alpha}_\tau})$ denote the generic filter of $\hqo^\tau_\alpha$.\\
Let $\co_\tau$ (or $\co^\alpha_\tau$, whenever $\alpha$ is not clear from the context) be the ${\hqo^{\alpha}_{\tau}}*{\name{\mo}_{\alpha}}$-name of all closed bounded subsets $c$ of $\alpha^{+}$ which satisfy that for every $\eta < \alpha^+$,
$$
c \cap \mathcal{S}^{\alpha}(\tau, \eta, 0) = \emptyset \iff q^\tau_\alpha(\eta) \in G(\hqo^\alpha_\tau),
$$
and 
$$
c \cap \mathcal{S}^{\alpha}(\tau, \eta, 1) = \emptyset \iff q^\tau_\alpha(\eta) \not\in G(\hqo^\alpha_\tau).
$$

\noindent A closed set $d \in \co_\tau$ extends $c$ if it end-extends $c$, i.e., if $d \cap \max(c) = c$.

\end{definition}

\noindent
Our first goal is to show that the forcing $\hqo^\alpha_\tau$ is $\sigma$-closed, and preserves the stationarity of $\mathcal{S}^{\alpha}(\tau, \eta ,i)$ for every $\eta<\alpha^+$ and $i<2$. 

\begin{lemma}\label{Lem: MillerCodingSigmaClosureLemma}
    For every $\tau\leq \alpha^{++}$, $\hqo^\alpha_\tau$ is $\sigma$-closed.
\end{lemma}

\begin{proof}
    Assuming that $\la p_i \colon i<\omega \ra$ is an increasing sequence of conditions, let $p^* = \bigvee_{i<\omega} p_i$ be the condition such that $\supp(p^*) = \bigcup_{i<\omega} \supp(p_i)$, and, for every $\xi\in \supp(p^*)$,
    $$p^*\uhr \xi \Vdash \name{T}^{p^*}_\xi = \bigcap_{i<\omega, \xi\in \supp(p_i)} \name{T}^{p_i}_\xi$$
    and--
    $${p^*\uhr \xi}^{\frown}\name{T}^{p^*}_\xi\Vdash \name{c}^{p^*}_\xi = \bigcup_{i<\omega} \name{c}^{p_i}_\xi \cup \Big\{ \sup\left( \bigcup_{i<\omega} \name{c}^{p_i}_\xi \right) \big\}.$$
    We argue that $p^*\in \hqo^\alpha_\tau$. Fix $\xi\in \supp(p^*)$ and assume that $p^*\uhr \xi\in \hqo^\alpha_\xi$. First, $p^*\uhr \xi$ forces that $\name{T}^{p^*}_\xi\in \name{\mo}_\alpha$ by $\alpha$-closure of $\name{\mu_\alpha}$. Next, we show that ${p^*\uhr \xi}^{\frown}\name{T}^{p^*}_\xi$ forces that $\name{c}^{p^*}_\xi\in \name{\co}_\xi$. Assuming that $\name{c}^{p^*}_\xi\in \name{\co}_\xi$ is not equal to one of the conditions $\name{c}^{p_i}_\xi$ for $i<\omega$, $\sup\left( \bigcup_{i<\omega} \name{c}^{p_i}_\xi\right)$ is an ordinal of countable cofinality, and thus doesn't belong to any of the stationary sets $\mathcal{S}^\alpha(\tau,\eta,j)$ for $\eta<\alpha^+$ and $j<2$. Hence, ${p^*\uhr \xi}^{\frown} \name{T}^{p^*}_\xi \Vdash \name{c}^{p^*}_\xi \in \name{\co}_\xi$.
\end{proof}

\begin{lemma} \label{Lem: MillerCodingPreservationOfStationarySets}
    $\hqo^\alpha_\tau$ preserves the stationarity of $\mathcal{S}^\alpha(\tau, \eta, i)$ for every $\eta<\alpha^+$ and $i<2$.
\end{lemma}

\begin{proof}
    Let $S$ be one of the sets $\mathcal{S}^\alpha(\tau, \eta^*, i^*)$ for some $\tau^*<\alpha^+, i^*<2$.  Assume that $\name{C}$ is a $\hqo^\alpha_\tau$-name for a club subset of $\alpha^+$ and $p\in \hqo^\alpha_\tau$. We argue that there exists $q\geq p$ such that $q\Vdash \name{C}\cap S\neq \emptyset$. 

    Since the sets in $\vec{\mathcal{S}}^\alpha$ are pairwise almost disjoint, we have for every $\tau'<\tau, \eta<\alpha^+, i<2$ an ordinal $\gamma(\tau', \eta, i)<\alpha^+$ such that $$S\cap \mathcal{S}^\alpha(\tau', \eta, i)\subseteq \gamma(\tau', \eta,i).$$ Denote $\vec{\gamma} = \la \gamma(\tau', \eta, i) \colon \tau'<\tau, \eta< \alpha^+, i<2 \ra$.

    Let $X\elem H_{\theta}$ (for $\theta$ large enough) be an elementary substructure such that:
    \begin{itemize}
        \item $\hqo^\alpha_\tau, \name{C}, p,  \vec{\gamma}, \alpha, \tau \in X$.
        \item $X$ is closed under countable sequences of its elements.
        \item $\chi_X(\alpha^+):=\sup(X\cap \alpha^+)\in S$.
        \item $\chi_X(\alpha^+)\subseteq X$.
    \end{itemize}
    Such $X$ can be constructed by standard methods, by building a continuous chain $\la X_i \colon i<\alpha^+ \ra$ of elementary substructures of $H_{\theta}$, each of them of size $\alpha$, where $X_{i+1}$ contains $\chi_{X_i}(\alpha^+)$ and all the countable sequences of elements of $X_i$ for every $i<\alpha^+$. Since $S\subseteq \alpha^+$ is stationary and $\la \chi_{X_i}(\alpha^+) \colon i<\alpha^+ \ra$ is a club in $\alpha^+$, we can find limit $i<\alpha^+$ such that $X = X_i$ satisfies the above properties.\\
    Assume that $\tau'\in X\cap \tau$. For every  $\eta < \chi_{X}(\alpha+)$ and $i<2$, we have $\gamma(\tau', \eta, i)\in X\cap \alpha^+$, and thus $\chi_{X}(\alpha^+)> \gamma(\tau', \eta,i)$, namely $\chi_{X}(\alpha^+)\notin S\cap \mathcal{S}^\alpha(\tau', \eta, i)$. For every $\eta\geq \chi_{X}(\alpha^+)$ and $i<2$, we have $\chi_{X}(\alpha^+)\notin \mathcal{S}^\alpha( \tau', \eta,i)$ since $\min\left(\mathcal{S}^\alpha(\tau', \eta, i)\right) > \eta$. Over all, we proved that for every $\tau'\in X\cap \tau$, 
        $$\chi_{X}(\alpha^+)\in S\setminus \left( \bigcup_{\eta<\alpha^{+}, i<2} \mathcal{S}^\alpha(\tau', \eta, i) \right).$$
    
    \noindent
    Having this in mind, we proceed and fix a cofinal sequence $\vec{\delta} = \la \delta_i \colon i<\omega_1 \ra$ in $\chi_{X}(\alpha^+)$. By the above properties, each initial segment of $\vec{\delta}$ belongs to $X$. We build a sequence of conditions  $\vec{p} = \la p_i \colon i<\omega_1\ra$ such that:
    \begin{itemize}
        \item Every initial segment of $\vec{p}$ belongs to $X$.
        \item For every $i<\omega_1$, $p_{i+1}$ decides $\min(\name{C}\setminus (\delta_i+1))$.
        \item For every $i<\omega_1$ and $\xi\in \supp(p_{i+1})$, ${p_{i+1}\uhr \xi}^{\frown}\name{T}^{p_{i+1}}_\xi\Vdash \max(\name{c}^{p_{i+1}}_\xi) > \delta_i$.
    \end{itemize}
    We start by picking $p_0 = p$, and, in limit steps, $p_i = \bigvee_{j<i}p_j$. By $\sigma$-closure and the fact that $\la p_j \colon j<i \ra\in X$, $p_i$ is a legitimate condition of $\hqo^\alpha_\tau$ which belongs to $X$. 

    At successor steps, assuming that $p_i$ was defined, extend it to $p'_{i+1}\in X$ such that $p'_{i+1}$ decides $\min(\name{c}\setminus \delta_i)$. Finally, extend $p_{i+1}\geq p'_{i+1}$ in $X$ without extending the support or the Miller coordinates, such that, for every $\xi<\supp(p'_{i+1})$, 
    $${p_{i+1}\uhr \xi}^{\frown}\name{T}^{p_{i+1}}_\xi \Vdash \max\left(\name{c}^{p_{i+1}}_\xi\right) > \delta_i.$$

    This concludes the construction. Denote $q = \bigvee_{i<\omega_1} p_i$. We first argue that $q\in \hqo^\alpha_\tau$. Since the Miller components are $\alpha$-closed, the only potential problem might occur in the coding coordinates. Recall that for every $\xi\in \supp(q)$,   
    $${q\uhr \xi}^{\frown} \name{T}^q_\xi \Vdash \name{c}^q_\xi = \bigcup_{i<\omega_1} \name{c}^{p_i}_\xi\cup \Big\{ \sup\left( \bigcup_{i<\omega_1} \name{c}^{p_i}_\xi \right) \Big\}$$
    and, by our construction,  
    $$\sup\left( \bigcup_{i<\omega_1} \name{c}^{p_i}_\xi \right) = \sup_{i<\omega_1} \delta_i = \chi_{X}(\alpha^+).$$
    Thus, in order to prove that ${q\uhr \xi }^{\frown} \name{T}^q_\xi\Vdash \name{c}^{q}_\xi\in \name{\co}_\xi$, it suffices to prove that $\chi_{X}(\alpha^+)\notin \mathcal{S}^\alpha(\xi, \eta', i')$ for every $\eta'<\alpha^+$ and $i'<2$. 

    Let $i<\omega_1$ be the least such that $\xi\in \supp(p_i)$. Since $X\vDash |\supp(p_i)|\leq \alpha$, there exists a surjection $f\in X_i$, $f\colon \alpha\to X$, and as $\alpha\subseteq X$, $\supp(p_i) = f''\alpha\subseteq X$. It follows that $\xi\in X\cap \tau$, and, as we already established, this implies          
    $$\chi_{X}(\alpha^+)\notin  \bigcup_{\eta<\alpha^{+}, i<2} \mathcal{S}^\alpha(\xi, \eta, i).$$
    This concludes the proof that $q\in \hqo^\alpha_\tau$. Finally, $q$ forces that $\chi_{X}(\alpha^+)$ is a limit point of $\name{C}$, and thus $q\Vdash \name{C}\cap S\neq \emptyset$, as desired.
    \end{proof}

\noindent
Our next goal is to prove that for every $\tau\leq \alpha^{++}$, $\hqo^{\alpha}_{\tau}$ has desirable properties such as $\alpha$-distributivity and the $C$-fusion property. For this we introduce a number of fusion related notions and IA chains.

\begin{definition} \label{Def: PacedFusionSequences} Let $\tau\leq \alpha^{++}$. 
		\begin{enumerate}
			\item     Assume that $F\in [\tau]^{\leq \alpha}$ and $\nu<\kappa$. We define an order $\leq_{F,\nu}$ on $\hqo^{\alpha}_{\tau}$ as follows: for $p,q\in \hqo^{\alpha}_{\tau}$, set $p\leq_{F,\nu} q$ if and only if $p\leq q$, and, for every $\xi\in F\cap \supp(p)$, $q\uhr \xi \Vdash \name{T}^{p}_\xi \leq_{\nu} \name{T}^{q}_\xi$.\footnote{The order $\leq_\nu$ on $\mo_{\alpha}$ is given in Definition \ref{Def: FusionSequenceMiller}.}
   
			\item  Let $\bqo^{\alpha}_\tau\subseteq \hqo^{\alpha}_\tau$ be the set of conditions $p\in \hqo^{\alpha}_\tau$ for which there exists a sequence 
   $$\vec{\mu} = \la \mu_\xi \colon \xi\in \supp(p) \ra\subseteq \alpha^+$$ such that, for every $\xi\in \supp(p)$, 
			$$ {p\uhr \xi}^{\frown} \name{T}^{p}_\xi \Vdash_{ {\hqo^{\alpha}_{\tau}} *\mo_{\alpha} } \max(\name{c}^{p}_{\xi})< \check{\mu}_{\xi}.    $$
			We call such a sequence $\vec{\mu}$ a \textbf{local bounding function} for $p$.
   
			\item A \textbf{pacing chain} is a continuous, internally approachable sequence $\vec{X} = \la X_i \colon i\leq\alpha \ra$ elementary substructures of $H_{\theta}$, for $\theta$ large enough, such that:
			\begin{enumerate}
				\item $\hqo^{\alpha}_{\tau}, \bqo^{\alpha}_\tau, \tau, \alpha\in X_0$.
				\item For every $i<\alpha$, $|X_i|<\alpha$ (in particular, $|X_\alpha| = \alpha$).
				\item For every $i<\alpha$, $^i{X_i}\subseteq X_{i+1}$.
				\item For every $i\leq \alpha$, $\chi_{X_i}(\alpha^+):=\sup( X_i \cap \alpha^+ )\notin \mathcal{S}^{\alpha}$.
			\end{enumerate}
			\item Let $\vec{X}$ be a pacing chain and fix $\gamma\leq\alpha$. A sequence of conditions $\vec{p} = \la p_i \colon i<\gamma \ra$ is \textbf{paced} with respect to $\vec{X}$ if there exists a weakly increasing continuous sequence $\la \tau_i \colon i<\gamma \ra\subseteq \tau+1$,\footnote{A typical case will be  $\tau_i = \tau$ for every $i<\gamma$. The only point in which this will not be the case is Lemma \ref{Lem: MillerCodingDensityLemma}.} in $X_0$, such that:
			\begin{enumerate}
				\item \label{Clause: PaceSeqLocalBounds} For every $i<\gamma$, $p_i\uhr \tau_i \in \bqo^{\alpha}_{\tau_i}$.
				\item For every $i<\gamma$, if $\tau_i<\tau$ then  $p_i\uhr \tau_i \Vdash p_i\setminus \tau_i = p_0\setminus \tau_i$.
				\item For every $i<\gamma$, $\vec{p}\uhr i+1 \in X_{i+1}$.
				\item For every $i<j<\gamma$, $p_i \leq p_j$.
				\item \label{Clause: PaceSeqCharacteristicFunction} For every $i<\gamma$ and $\xi\in \supp(p_i)\cap \tau_i$, ${p_i\uhr \xi}^{\frown} \name{T}^{p_i}_{\xi} \Vdash \check{\chi}_{X_i}(\alpha^+)\in \name{c}^{p_i}_{\xi}$.
				\item \label{Clause: PaceSeqLimit} For every limit $i<\gamma$, $p_i = \bigvee_{j<i}p_j$, namely, $\supp(p_i) = \bigcup_{j<i} \supp(p_j)$, and for every $\xi\in  \supp(p_i)$,
				$$ p_i\uhr \xi \Vdash \name{T}^{p_i}_{\xi} = \bigcap_{j<i} \name{T}^{p_j}_{\xi} $$
				and--
				$$(p_i \uhr \xi) ^{\frown} \name{T}^{p_i}_\xi \Vdash \name{c}^{p_i}_\xi = \left( \bigcup_{j<i} \name{c}^{p_j}_\xi \right)\cup \Big\{\sup\left(  \bigcup_{j<i} \name{c}^{p_j}_\xi \right)\Big\}$$
			\end{enumerate}
			\item Let $\vec{X}$ be a pacing chain. A sequence of length $\alpha$, $\vec{p} = \la p_i \colon i< \alpha \ra\subseteq \hqo^{\alpha}_\tau$, is a \textbf{paced fusion sequence}  with respect to $\vec{X}$, if it is a paced sequence with respect to $\vec{X}$, and, in addition, there exists a $\subseteq$-increasing sequence of sets $\vec{F} = \la F_i \colon i\leq \alpha \ra\subseteq [\tau]^{<\alpha}$ and some $\nu<\alpha$ such that:
			\begin{enumerate}
				\item For every $i<\alpha$, $\vec{F}\uhr i+1 \in X_{i+1}$.
				\item For every $i<\alpha$, $F_i \subseteq \supp(p_i)$.
				\item For every $i<j<\alpha$, $p_{i} \leq_{F_i,\nu+i} p_j$.
				\item For every limit $i<\alpha$, $F_i = \bigcup_{j<i} F_j$.
				\item $F_\alpha = \bigcup_{i<\alpha} F_i = \bigcup_{i<\alpha} \supp(p_i)$.
			\end{enumerate}
		\end{enumerate}
	\end{definition}

	We will elaborate below how pace chains, paced sequences and paced fusion sequences are constructed. First, let us verify that paced sequences and paced fusion sequences have an exact upper bound.
	
	\begin{claim}\label{Claim:ClosureOfPacedSequences}
		Let $\gamma\leq \alpha$ be a limit ordinal. Assume that $\vec{X} = \la X_i \colon i\leq\gamma \ra$ is a pace chain. Assume that $\vec{p} = \la p_i \colon i<\gamma \ra\subseteq \hqo^{\alpha}_{\tau}$ is a paced sequence with respect to $\vec{X}$. Let $p^* = \bigvee_{i<\gamma} p_i$.
		\begin{enumerate}
			\item If $\gamma<\alpha$, then $p^*$ is an exact upper bound of $\{ p_i \colon i<\gamma \}$.
			\item If $\gamma = \alpha$ and $\vec{p}$ is a paced fusion sequence with respect to $\vec{X}$, then $p^*$ is an exact upper bound of $\{ p_i \colon i<\alpha \}$.
		\end{enumerate}
	\end{claim}
	
	\begin{proof} 
		We concentrate on the proof that $p^*$ is a condition of $\hqo^{\alpha}_\tau$ which extends $p_i$ for every $i<\gamma$. Once we prove that, the exactness easily follows. 
		
		Work by induction on  $\xi\in \supp(p^*)$, assume that  $p^*\uhr \xi\in \hqo^{\alpha}_{\xi}$ is an upper bound of $\la p_i \uhr \xi \colon i<\gamma \ra$. Let us argue that 
		$$p^*\uhr \xi \Vdash \name{T}^{p^*}_\xi\in \name{\mo}_\alpha, \quad \name{T}^{p^*}_\xi \geq_{\name{\mo}_\alpha} \name{T}^{p_i}_\xi, \text{ and }$$
		
		$${p^*\uhr \xi}^{\frown} \name{T}^{p^*}_\xi \Vdash \name{c}^{p^*}_\xi \in \co_{\xi} \mbox{ and for every } i<\gamma, \name{c}^{p^*}_\xi \geq_{\name{\co}_\xi} \name{c}^{p_i}_\xi.$$
		We remark that if $\la \tau_i \colon i<\gamma \ra$ witnesses the fact that $\vec{p}$ is a paced sequence, and $\xi\geq \sup\{ \tau_i \colon i<\gamma \}$, the above inductive step is clear since $p\uhr \xi\Vdash p^*\setminus \xi = p_0\setminus \xi$. Thus, we can consider only the case $\xi < \tau^*$ where $\tau^*$ denotes $\sup\{ \tau_i \colon i<\gamma \}$. 
		\begin{enumerate}
			\item Assume that $\gamma<\alpha$. First, $T^{p^*}_\xi = \bigcap T^{p_i}_\xi\in \mo_\alpha$ by $\alpha$-closure of $\mo_\alpha$. Moving to the coding component name in $V^{{\hqo^{\alpha}_\xi}*\mo_\alpha}$,   for every $i<\gamma$, let $\vec{\mu}^{i} = \la \mu^{i}_\eta  \colon \eta\in \supp(p_i)\cap \tau_i \ra$ be a local bounding function witnessing the fact that $p_i\uhr \tau_i\in \bqo^{\alpha}_{\tau_i}$. Since $p_i\in X_{i+1}$, we may assume  $\mu^{i}_\xi < \chi_{X_{i+1}}(\alpha^+)$ for every $i<\gamma$. Also, there exists $i_0<\gamma$ such that $\xi < \tau_{i_0}$, and for every $i_0 \leq i < \gamma$, $\chi_{X_i}(\alpha^+)\in c^{p_i}_{\xi}$. Combining the above facts, it follows that $\bigcup_{i<\gamma} c^{p_i}_{\xi}$ is an unbounded subset of $\chi_{X_\gamma}(\alpha^+)$. Thus, 
			$$c^{p^*}_{\xi}= \left(\bigcup_{i<\gamma} c^{p_i}_{\xi}\right)\cup \{ \chi_{X_\gamma}(\alpha^+) \}$$ 
			and $c^{p^*}_\xi$ is indeed a condition that extends each  $c^{p_i}_\xi$ since $\chi_{X_\gamma}(\alpha^+)\notin \mathcal{S}^{\alpha}$.
			\item Assume $\gamma = \alpha$. Let $\vec{F} = \la F_i \colon i<\alpha \ra$ be a $\subseteq$-increasing sequence of sets witnessing the fact that $\vec{p}$ is a paced fusion sequence. Let $i_0<\alpha$ be the least such that $\xi\in F_{i_0}$. In $V^{\hqo^{\alpha}_\xi}$, $\la T^{p_i}_\xi \colon i_0\leq i<\alpha\ra$ is a fusion sequence of conditions in $\mo_\alpha$. In particular, $T^{p^*}_\xi:=\bigcap T^{p_i}_\xi \in \mo_\alpha$. The argument concerning the coding component being a condition is the same as the one given just above.
		\end{enumerate}
	\end{proof}
	
	\begin{claim} \label{Claim: ExistenceOfPacedChains}
		There exists a pace chain of elementary substructures of $H_{\theta}$ for large enough $\theta$. 
	\end{claim}
	
	\begin{proof}
		Let $X^*$ be an elementary substructure of $H_{\theta}$ such that:
		\begin{itemize}
			\item $|X^*| = \alpha$.
			\item $\hqo^{\alpha}_\tau, \bqo^{\alpha}_\tau,  \tau, \alpha\in X^*$.
			\item $X^*$ is closed under $<\alpha$-sequences of its elements.
			\item $\cf(\chi_{X^*}(\alpha^+)) = \alpha$ (In particular, $\chi_{X^*}(\alpha^+)\notin \mathcal{S}^\alpha$, since $\mathcal{S}^\alpha$ consists of points of cofinality $\omega_1$). 
		\end{itemize}
		Such $X^*$ can be constructed by standard methods, by building a continuous chain of elementary substructures $\la X^*_i \colon i<\alpha \ra$ of $H_{\theta}$, each of them of size $\alpha$, where $X^*_{i+1}$ contains all the $<\alpha$-sequences of elements of $X^*_{i}$ for every $i<\alpha$. Let $X^*:= \bigcup_{i<\alpha} X^*_{i}$. Then $\cf(\chi_{X^*}(\alpha^+)) = \alpha$ and $X^*$ is indeed closed under $<\alpha$-sequences of its elements.\\

  \noindent
	As $\mathcal{S}^\alpha$ does not reflect 
 there is a club $C\subseteq \rho := \chi_{X^*}(\alpha^+)$  such that $otp(C) = \cf(\rho) = \alpha$ and  $C\cap \mathcal{S}^\alpha =\emptyset$. Fix such a club $C$ as well as an enumeration $\la a_i \mid i < \alpha\ra$ of $X^*$.
We now construct in $V$ a continuous, internally approachable sequence $\la X_i \colon i<\alpha \ra$ of elementary substructures of $X^*$ such that, for every $i<\alpha$,
		\begin{itemize}
			\item $X_i\in X^*$.
			\item $|X_i|<\alpha$.
			\item $\hqo^{\alpha}_\tau,\bqo^{\alpha}_\tau, \tau,\alpha\in X_i$.
			\item $\chi_{X_i}(\alpha^+)\in C$ (and, in particular, $\chi_{X_i}(\alpha^+)\notin \mathcal{S}^\alpha$).
			
            \item $\vec{X}\uhr i, a_i \in X_{i+1}$ and 
            ${}^i X_{i+1} \subseteq X_{i+1}$ 
		\end{itemize}
		This will suffice, since $\la X_i \colon i<\alpha \ra^{\frown} \la X^* \ra$

  will be the desired pace chain of substructures of $H_{\theta}$. 
		For the construction of $X_{0}$, pick a continuous, elementary chain $\la X^{j}_0  \colon j<\alpha \ra \in V$ of substructures of $X^*$, such that each structure in the sequence belongs to $X^*$ and $\la \chi_{X^j_0}(\alpha^+) \colon j<\alpha \ra$ is a club $D$ in $\rho$. In order to make sure that $D$ is indeed a club in $\rho$, we prescribe a cofinal sequence of order type $\alpha$ in $\rho$, and make sure that $\chi_{X^{j+1}_0}(\alpha^+)$ is above the $j$-th element in it. Since $X^{j+1}_0$ belongs to $X^*$, its ordinals do not exceed $\rho$. This ensures that $D = \la \chi_{X^{j}_0}(\alpha^+) \colon j<\alpha \ra$ is as desired, and thus there exists $j<\alpha$ such that $\chi^{j+1}_{0}(\alpha^+)\in C\cap D$. Take $X_0:= X^{j+1}_0$ and proceed as follows:\\
		At successor steps, repeat the argument: assuming that $i<\alpha$ and $X_i\elem X^*$ has been chosen, construct (in $V$) a continuous, elementary chain $\la X^{j}_{i+1} \colon j<\alpha \ra$ of substructures of $X^*$ such that $\la X_{i'} \colon i'\leq i \ra\in X^{0}_{i+1}$ (note that, since $X^*$ is closed under $<\alpha$-sequences of its elements, $\la X_{i'} \colon i'\leq i \ra\in X^*$), $X^0_{i+1}$ contains all the $\leq i$-sequences of elements of $X_i$, and $\la \chi_{X^{j}_{i+1}}(\alpha^+) \colon j<\alpha \ra$ forms a club $D_{i+1}$ in $\rho$. Then, pick $j\in D_{i+1}\cap C$ and let $X_{i+1}:= X^{j}_{i+1}$.\\
  At limit steps, we take unions: for $i<\alpha$ limit, let $X_i = \bigcup_{i'<i} X_{i'}$, and note that, by induction, $\la \chi_{ X_{i'} }(\alpha^+) \colon i'<i \ra$ is a $<\alpha$ sequence of elements of $C$, and thus its union $\chi_{X_i}(\alpha^+)$ belongs to $C$. 
	\end{proof}

	We are now ready to state the main properties of the forcings $\hqo^{\alpha}_\tau$ for $\tau\leq \alpha^{++}$, in the following sequence of Lemmas.
	
	\begin{lemma}\label{Lem: MillerCodingDensityLemma}
		$\bqo^{\alpha}_\tau$ is a dense subset of $\hqo^{\alpha}_\tau$. Furthermore, for every $F\in [\tau]^{<\alpha}$ and $\nu<\alpha$, $\bqo^{\alpha}_\tau$ is a $\leq_{F,\nu}$-dense subset of $\hqo^{\alpha}_\tau$. 
	\end{lemma}
	
	\begin{lemma}\label{Lem: MillerCodingDistributivityLemma}
		$\hqo^{\alpha}_\tau$ is $\alpha$-distributive. Furthermore, for every $F\in [\tau]^{< \alpha}$ and $\nu<\alpha$, $\la \hqo^\alpha_\tau, \leq_{F,\nu} \ra$ is $\alpha$-distributive.
	\end{lemma}
	
	\begin{lemma}\label{Lem: MillerCodingMahloLemma}
		$\hqo^{\alpha}_\tau$ preserves the Mahloness of $\alpha$.
	\end{lemma}
	
	\begin{lemma}\label{Lem: MillerCodingC-FusionLemma}
		$\hqo^{\alpha}_\tau$ satisfies the $\alpha$ $C$-Fusion property. 
	\end{lemma}
	
	\begin{lemma}\label{Lem: MillerCodingDecidingOrds}
		Suppose that $p\in \hqo^{\alpha}_\tau$, $F\in [\tau]^{< \alpha}$, $\nu< \alpha$ and $\name{\zeta}$ is a $\hqo^{\alpha}_\tau$-name for an ordinal. Then there exists $q\geq_{F,\nu} p$ and a set $x\in V$ with $|x|\leq \alpha$ such that $q\Vdash \name{\zeta}\in x$.
	\end{lemma}
	
	\begin{lemma}\label{Lem: MillerCodingSecondDensityLemma}
		Assume that $\tau<\alpha^{++}$. There exists a dense subset $\qo^{\alpha}_\tau\subseteq \bqo^{\alpha}_\tau$ of size $|\qo^{\alpha}_\tau| = \alpha^{+}$.
	\end{lemma}

	We will prove lemmas \ref{Lem: MillerCodingDensityLemma}-\ref{Lem: MillerCodingSecondDensityLemma} by induction on $\tau\leq \alpha^{++}$. Thus, we assume that the Lemmas are true for every $\tau'<\tau$, and prove that they are true for $\tau$.
	
	\begin{proof} (Lemma \ref{Lem: MillerCodingDensityLemma})\\
		Let $p\in \hqo^{\alpha}_\tau$. We argue that $p$ can be $\leq_{F,\nu}$-extended to a condition in $\bqo^{\alpha}_\tau$. We divide the proof into several cases:\\${}$

        \noindent
		\textbf{Case 1:} $\tau = \tau'+1$ is a successor ordinal. In this case, apply Corollary \ref{Corollary: MillerApproximationOfAnOrdinal} to choose  $\hqo^\alpha_{\tau'}$-names $\name{S}$ for a condition in $\mo_\alpha$ which $\leq_{\nu}$ extends $\name{T}^{p}_{\tau'}$, and $\name{\mu}$ for an ordinal below $\alpha^+$ such that-- 
		$$p\uhr \tau'\Vdash \left(\name{S}\Vdash_{\mo_\alpha} \sup(\name{c}^{p}_{\tau'})< \name{\mu}\right)$$
		By the induction hypothesis (assuming that the current Lemma combined with Lemma \ref{Lem: MillerCodingDecidingOrds} hold for $\hqo^{\alpha}_{\tau'}$), we can $\leq_{F\setminus\{\tau'\}, \nu}$-extend $p\uhr \tau'$ to a condition $q\in \bqo^{\alpha}_{\tau'}$ such that, for some $\mu^*<\alpha^+$, $q\Vdash \name{\mu}< \mu^*$. Then $q^{\frown} \name{S}^{\frown} {\name{c}^{p}_{\tau'}}$ is a $\leq_{F,\nu}$ extension of $p$ in $\bqo^{\alpha}_\tau$.
		
		${}$ \\   \noindent
		\textbf{Case 2:} $\tau$ is limit and $\cf(\tau)> \alpha$. Let $\tau'<\tau$ be an upper bound on $\supp(p)$. By the induction hypothesis, there exists $q\geq_{F\cap \tau', \nu} p$ such that $q\in \bqo^{\alpha}_{\tau'}$. Viewing $q$ as a condition of $\hqo^{\alpha}_\tau$, $p\leq_{F,\nu} q\in \bqo^{\alpha}_\tau$.

        ${}$ \\   \noindent
		\textbf{Case 3:} $\tau$ is limit and $\cf(\tau)< \alpha$. Denote $\gamma = \cf(\tau)$. Fix a continuous, increasing, cofinal sequence $\la \tau_i \colon i<\gamma \ra$ in $\tau$, and a pacing chain $\la X_i \colon i\leq \alpha \ra$ such that $p, \la \tau_i \colon i<\gamma \ra, F, \nu\in X_0$. We construct a paced sequence $\la p_i \colon i<\gamma \ra$ with respect to $\vec{X}$.
		Let $p_0 = p$. We assume for simplicity that $\tau_0 = 0$, so we have $p_0\uhr \tau_0\in \bqo^{\alpha}_{\tau_0}$.
		At limit steps, $p_i = \bigvee_{j<i} p_j$. Note first that $p_i\in X_{i+1}$ since $X_{i+1}$ contains every sequence of length $i$ of elements of $X_i$. Since $\la p_j \colon j<i \ra$ is by itself a paced sequence of conditions, 
        $p_i\in \hqo^\alpha_{\tau}$. We should verify that $p_i\uhr \tau_i\in \bqo^{\alpha}_{\tau_i}$: for every $j<i$, let 
        $$\vec{\mu}^j = \la \mu^{j}_\xi \colon \xi\in \supp(p_j)\cap \tau_j \ra $$ 
        be a local bounding function for $p_j\uhr {\tau_j}$. Define $\vec{\mu} = \la \mu_\xi \colon \xi\in \supp(p_j)\cap \tau_i \ra$ by taking, for each $\xi\in \supp(p_i)\cap \tau_i$, 
        $$\mu_\xi = \sup(\{ \mu^{j}_\xi \colon  j<i, \xi\in \supp(p_j)\cap \tau_j \})+1.$$ 
        It is not hard to verify that $\vec{\mu}$ is a local bounding function for $p_i\uhr \tau_i$.
		
		We proceed to the successor case. Assume that $p_i$ has been constructed and $p_i \in X_{i+1}$. 
		By the induction hypothesis, $p_i\uhr \tau_{i+1}$ can be $\leq_{F,\nu}$ extended inside $X_{i+1}$ to a condition $q\in \bqo^{\alpha}_{\tau_{i+1}}$. We define $p_{i+1}\geq_{F,\nu} p_i$ as follows: first, let $p_{i+1}\uhr \tau_{i+1}$ be an extension of $q$ with the same support, such that, for every $\xi\in \supp(q)$, 
        $$p_{i+1}\uhr \xi \Vdash \name{T}^{p_{i+1}}_\xi = \name{T}^{q}_\xi$$ and 
        $${p_{i+1}\uhr \xi}^{\frown} \name{T}^{p_{i+1}}_\xi \Vdash \name{c}^{p_{i+1}}_\xi = \name{c}^{q}_\xi \cup \{\chi_{ X_{i+1} }(\alpha^+)\}.$$ 
        Note that $p_{i+1}\uhr \tau_{i+1}\geq q$ is a legitimate condition because $\chi_{X_{i+1}}(\alpha^+)\notin \mathcal{S}^{\alpha}$. Second, choose $p_{i+1}\setminus \tau_{i+1}$ such that $p_{i+1}\uhr \tau_{i+1}\Vdash p_{i+1}\setminus \tau_{i+1} = p_0\setminus \tau_{i+1}$. 
		
		Note that if $\vec{\mu} = \la \mu_\xi \colon \xi\in \supp(q)\cap \tau_{i+1} \ra$ is a local bounding function for $q$ in $X_{i+1}$, then $\vec{\mu}^* = \la \mu^*_\xi \colon \xi\in \supp(p_{i+1})\cap \tau_{i+1} \ra$ is a local bounding function for $p_{i+1}\uhr \tau_{i+1}$, where $\mu^*_\xi = \max( \mu_\xi, \chi_{X_{i+1}}(\alpha^+)+1 )$. Also note that $p_{i+1}$ is chosen inside $X_{i+2}$, by internal approachability. 
		
		This concludes the inductive construction. Since $\la p_i \colon i<\gamma \ra$ is a paced sequence with respect to $\vec{X}$, $q = \bigvee_{i<\gamma} p_i$ is a condition in $\hqo^{\alpha}_\tau$. Arguing as in the limit step, it's not hard to verify that $q$ has a local bounding function, so $q\in \bqo^{\alpha}_{\tau}$. 

        ${}$ \\   \noindent
		\textbf{Case 4:} $\tau$ is limit and $\cf(\tau)= \alpha$. Fix a continuous, increasing, cofinal sequence  $\la \tau_i \colon i<\alpha \ra$ in $\tau$. We will construct a fusion pace sequence $\la p_i \colon i<\alpha \ra$ with respect to a pacing chain $\vec{X}$, witnessed by some $\subseteq$-increasing sequence $\la F_i \colon i<\alpha\ra$. We will need to integrate into the construction a bookkeeping argument in order to ensure that $\bigcup_{i<\alpha} F_i = \bigcup_{i<\alpha}\supp(p_i)$. For that, we prescribe in advance an arbitrary well order $W(x)$ of order type $\leq \alpha$ of $x$, for every set $x\in [\tau]^{\leq \alpha}$. We also fix a partition $\vec{A} = \la A_i \colon i<\alpha \ra$ of $\alpha$ to sets of size $\alpha$. Fix a pace chain $\la X_i \colon i<\alpha \ra$ with 
        $$p, \la \tau_i \colon i<\alpha \ra, \la W(x) \colon x\in [\tau]^{\leq \alpha} \ra, \vec{A}, F, \nu\in X_0.$$
		
		For the construction of the paced fusion sequence, first let $p_0 = p$, $F_0 = F$. In limit steps, $p_i = \bigvee_{j<i} p_j$ and $F_{i} = \bigcup_{j<i}F_j$ for every limit $i<\alpha$. Note that $p_i, F_i \in X_{i+1}$. Let us concentrate on successor steps. Assume that $p_i, F_i$ have been constructed in $X_{i+1}$, and let us construct $p_{i+1}, F_{i+1}$.
		
		We begin by defining $F_{i+1}$. Denote by $i^*<\alpha$ the unique index such that $i\in A_{i^*}$. If $i^* >i$, let $F_{i+1} = F_i$. Assume $i^*\leq i$ (so that $p_{i^*}$ is already defined at this point of the construction). Let $\xi^*$ be the $W( \supp(p_{i^*}) )$-minimal element of $\supp(p_{i^*})\setminus \bigcup_{j\leq i}F_j $ and set $F_{i+1} = F_i \cup \{ \xi^* \}$. Note that $F_{i+1}\in X_{i+1}$.
		
		The definition of $p_{i+1}$ is similar to the case where $\cf(\tau)<\alpha$: choose first a condition $q\in \bqo^{\alpha}_{\tau_{i+1}}$ in $X_{i+1}$, such that  
        $$q \geq_{F_{i}\cap \tau_{i+1}, \nu+i } p_{i}\uhr \tau_{i+1}.$$ 
        This is possible since $\bqo^{\alpha}_{\tau_{i+1}}$ is $\leq_{ F_{i}\cap \tau_{i+1}, \nu+i }$-dense inside $\hqo^{\alpha}_{\tau_{i+1}}$ by the induction hypothesis. Then, let $p_{i+1}\uhr \tau_{i+1}$ be an extension of $q$ with the same support, such that for every $\xi\in \supp(q)$, ${p_{i+1}\uhr \xi}\Vdash \name{T}^{p_{i+1}}_\xi = \name{T}^{q}_\xi$ and $${p_{i+1}\uhr \xi}^{\frown} \name{T}^{p_{i+1}}_\xi \Vdash \chi_{X_{i+1}}(\alpha^+)\in \name{c}^{p_{i+1}}_\xi.$$ 
        Finally, take $p_{i+1}\setminus \tau_{i+1} = p_0\setminus \tau_{i+1}$.
		The condition $p_{i+1}$ obtained this way can be chosen inside $X_{i+2}$ by internal approachability. Note also that $p_{i+1}\geq_{F_i, \nu+i} p_i$.
		
		This concludes the definition of $\la p_i \colon i<\alpha \ra$. We only need to make sure that $\bigcup_{i<\alpha} F_i = \bigcup_{i<\alpha} \supp(p_i)$. 
		
		Assume that $\xi\in \bigcup_{i<\alpha} \supp(p_i)$. Let $i^*<\alpha$ be the least such that $\xi\in \supp(p_{i^*})$. Then $S_{i^*} \setminus i^* $ is an unbounded subset of $\alpha$ of order type $\alpha$, and so, we will have some $i\in S_{i^*}  \setminus i^* $ such that $\xi^*\in F_{i+1}$ (indeed, if $\xi^*$ is the $\gamma$-th element in $\supp(p_{i^*})$ according to $W( \supp(p_{i^*}) )$, then $\xi^*$ will be in $F_{i+1}$ for $i$ which is the $\gamma$-th element of $S_{i^*} \setminus i^*$).
		
		This shows that $\la p_i \colon i<\alpha \ra$ is a paced fusion sequence, and thus it has an exact upper bound $p^* = \bigvee_{i<\alpha} p_i$. Since $p_i\uhr\tau_i\in \bqo^{\alpha}_{\tau_i}$ for every $i<\alpha$, we get  $p^*\in \bqo^{\alpha}_{\tau}$.
	\end{proof}

	\begin{proof}(Lemma \ref{Lem: MillerCodingDistributivityLemma})\\
		We will prove that $\hqo^\alpha_\tau$ is $\alpha$-distributive. The proof that $\la \hqo^\alpha_\tau, \leq_{F,\nu} \ra$ is $\alpha$-distributive is similar.
		
		Let $\gamma<\alpha$ and $\vec{D} = \la D_i \colon i<\gamma \ra$ be a sequence of dense open subsets of $\hqo^{\alpha}_\tau$. Fix a pacing chain $\vec{X} = \la X_i \colon i\leq \alpha \ra$ with $\vec{D}\in X_0$, and let us construct a paced sequence of conditions $\la p_i \colon i<\gamma \ra$, such that, for every $i<\gamma$, $p_{i+1}\in D_i$. Regarding the sequence $\la \tau_i \colon i<\gamma \ra$ required in the definition of a paced sequence, we take $\tau_i = \tau$ for every $i<\gamma$. In other words, we require, for every $i<\gamma$, $p_{i}\in \bqo^{\alpha}_{\tau}$. This is possible by Lemma \ref{Lem: MillerCodingDensityLemma}.\\
		For the construction, we pick $p_0$ to be an arbitrary element of $\bqo^{\alpha}_\tau$. At limit steps, let $p_{i} = \bigvee_{j<i} p_j$. 
  $p_i\in \bqo^{\alpha}_{\tau}$ by Claim \ref{Claim:ClosureOfPacedSequences}.
  At successor steps, assume that $p_i$ has been constructed. Choose $p_{i+1}\geq p_i$ in $X_{i+1}$ such that $p_{i+1}\in D_{i}\cap \bqo^{\alpha}_{\tau}$. By extending $p_{i+1}$ without adding points to the support or changing the Miller coordinates, we can assume that, for every $\xi\in \supp(p_{i+1})$, ${p_{i+1}\uhr \xi}^{\frown} \name{T}^{p_{i+1}}_\xi \Vdash \chi_{X_{i+1}}(\alpha^+)\in \name{c}^{p_{i+1}}_\xi$. Note that such a change can be made so it keeps $p_{i+1}$ inside $D_{i}\cap \bqo^{\alpha}_\tau$ (just add $\chi_{X_{i+1}}(\alpha^+)$ to each coding coordinate in the support of $p_{i+1}$).
		
		This concludes the inductive construction. Let $p^* = \bigvee_{j<i} p_i$. Then $p^*\in \bigcap_{j<i} D_i$, as desired.
	\end{proof}
	
	\begin{proof}(Lemma \ref{Lem: MillerCodingMahloLemma})\\
		Assume that $\name{C}$ is a $\hqo^{\alpha}_\tau$-name for a club subset of $\alpha$. We argue that every condition $p\in \hqo^{\alpha}_\tau$ can be extended to a condition $p^*$ which forces that $\name{C}$ contains a regular cardinal below $\alpha$ 
  .
		\noindent
		Let $\vec{X} = \la X_i \colon i\leq \alpha \ra$ be pacing chain with $\name{C},p\in X_0$. Since $\alpha$ is Mahlo in $V$, we can find a regular cardinal $\gamma<\alpha$ such that $X_{\gamma}\cap \alpha = \gamma$. Let us argue that $p$ can be extended to a condition $p^*$ that forces that $\gamma\in \name{C}$. \\
		For every $i<\gamma$, let $D_i$ be the dense open set of conditions that decide the $i$-th element of $\name{C}$. Then $D_i\in X_{i+1}$, and, repeating the same argument as in  the proof of Lemma \ref{Lem: MillerCodingDistributivityLemma},
  we can construct a paced sequence $\la p_i \colon i<\gamma \ra\subseteq \bqo^{\alpha}_{\tau}$ with respect to $\vec{X}\uhr \gamma+1$, such that, for every $i<\gamma$, $p_{i+1}\in D_i$.\\
		This concludes the inductive construction. Let $p^* = \bigvee_{j<i} p_i$. Then $p^*$ decides the first $\gamma$-many elements of $\name{C}$. Also, for every $i<\gamma$, $p^*$ forces that the $i$-th element of $\name{C}$ is below $\gamma$, since $\chi_{X_{\gamma}}(\alpha) = \gamma$ and $p_{i+1}\in X_{i+2}$. Thus, $p^*$ forces that the $\gamma$-th element of $\name{C}$ is $\leq \gamma$, so it is exactly $\gamma$. 
	\end{proof}
	
	We now proceed towards the proof of Lemma \ref{Lem: MillerCodingC-FusionLemma}, that is, $\hqo^{\alpha}_\tau$ satisfies the $C$-fusion property. We follow a technique due to Friedman and Zdomskyy \cite{FriedmanZdomskyy2010}. 
	
	\begin{definition}
		Let $p\in \hqo^{\alpha}_\tau$ and $F\in [\tau]^{<\alpha}$, $F\subseteq \supp(p)$. Let $\sigma\colon F\to \alpha^{<\alpha}$. We define, if possible, an extension $(p)_{\sigma}\geq p$, with the same support, in the following inductive way: assume that $\xi\in \supp(p)$ and $(p)_\sigma\uhr \xi$ has been defined. 
		\begin{itemize}
			\item If $\xi\notin F$, let $ (p)_\sigma(\xi) = p(\xi)$. 
			\item If $\xi\in F$ and $(p)_{\sigma}\uhr \xi \Vdash \sigma(\xi)\in \name{T}^{p}_\xi$, let $(p)_\sigma(\xi) = {\left(\name{T}^{p}_{\xi}\right)_{\sigma(\xi)} }^{\frown} \ \name{c}^{p}_\xi$. 
			\item Else, $(p)_\sigma\uhr \xi+1$ is not defined, and, in particular, $(p)_\sigma$ is not defined.
		\end{itemize}
		We say that $\sigma$ \emph{lies} on $p$ if $(p)_\sigma$ is defined.
	\end{definition}
	
	\begin{remark} \label{Rmk:MillerCodingStrongerFusionProperty}
		We will actually prove the following strengthening of the $\alpha$ C-Fusion property: for every $p\in \hqo^{\alpha}_\tau$, $F\in [\tau]^{< \alpha}$, $\nu<\alpha$ and sequence $\vec{D} = \la D_i \colon i<\alpha \ra$ of dense open subsets on $\hqo^{\alpha}_\tau$, there exist $q\geq_{F,\nu} p$, a pace chain $\vec{X} = \la X_i \colon i\leq \alpha \ra$ and a club $C\subseteq \alpha$, such that:
		\begin{enumerate}[label = (\alph*)]   
                \item $q$ is the exact upper bound of a $\leq_{F,\nu}$-increasing sequence of extensions of $p$, $\vec{p} = \la p_i \colon i<\alpha \ra$, which is a paced fusion sequence with respect to $\vec{X}$, witnessed by some sequence $\vec{F} = \la F_i \colon i<\alpha \ra$.
			\item For every regular $i\in C$, and an $X_i$-generic set $G_{X_i}\subseteq X_i\cap \hqo^{\alpha}_\tau$ (in $V$) whose members are compatible with $q$, define a function 
   $\sigma_{G_{X_i}}\colon F_i \to i^i$ that maps each $\xi\in F_i$ to the generic branch induced by $G_{X_i}$ at coordinate $\xi$. I.e.,  
            $$\sigma_{G_{X_i}}(\xi) = \bigcup\{ t\in i^{<i} \colon \exists r\in G_{X_i}, r\uhr \xi \Vdash t = \mbox{stem}(\name{T}^{r}_\xi)   \}.$$
            Then $\sigma_{G_{X_i}}$ lies on $q$ and $(q)_{\sigma_{G_{X_i}}}$ is the exact upper bound of $G_{X_i}\cup \{q\}$, and $(q)_{\sigma_{G_{X_i}}}\in D_i$.	
		\item $q$ is a generic condition for $X_{\alpha}$.
            \item $\supp(q) = \bigcup_{i<\alpha} F_i \subseteq X_{\alpha}\cap \alpha^{++}$.\footnote{In fact, we can further have equality, see also footnote \ref{Footnote: MillerCodingFusion}.}
            \item $q\in \bqo^{\alpha}_\tau$. Moreover, the constant function $\vec{\mu} = \la \mu_{\xi}  \colon \xi\in \supp(q)\ra$ defined by setting for every $\xi\in \supp(q)$, $\mu_\xi = \chi_{X_\alpha}(\alpha^{+})$, is a local bounding function for $q$.
            \item $X_{\alpha}\cap \alpha^{+} = \chi_{X_\alpha}(\alpha^{+})$.
		\end{enumerate}
		This stronger version of the $\alpha$ C-Fusion property will be used in the proof of Lemmas \ref{Lem: MillerCodingDecidingOrds} and \ref{Lem: MillerCodingSecondDensityLemma}.
	\end{remark}
	
	\begin{proof}(Lemma \ref{Lem: MillerCodingC-FusionLemma})\\
		Assume that $\vec{D} = \la D_i \colon i<\alpha \ra$ is a sequence of dense open subsets of $\hqo^{\alpha}_\tau$, and let $p\in \hqo^{\alpha}_\tau$ be a condition. We will show that for every $F\in [\tau]^{\leq \alpha}$, $\nu<\alpha$, there exists $q\geq_{F,\nu} p$, a continuous sequence $\la X_i \colon i<\alpha \ra$ and a club $C\subseteq \alpha$ witnessing the $C$-fusion property. We will construct a paced fusion sequence $\la p_i \colon i<\alpha \ra$, $\la F_i \colon i<\alpha \ra$ with respect to some pace chain $\vec{X}$.\footnote{As before, the conditions in the pace sequence will belong to $\bqo^{\alpha}_\tau$, namely, the sequence $\la  \tau_i \colon i<\alpha\ra$ will be the constant sequence with value $\tau$.}
		
		As in the proof of case 4 in lemma \ref{Lem: MillerCodingDensityLemma}, we fix, for the bookkeeping argument, a sequence $\vec{W} = \la W(x) \colon x\in [\max(\tau, \alpha^+)]^{\leq \alpha} \ra$ of well orders, each $W(x)$ is a well order of $x$, and a sequence $\vec{A} = \la A_i \colon i<\alpha \ra$ which is a partition of $\alpha$ to sets of size $\alpha$.
		
		First, let $\vec{X} = \la X_i \colon i\leq \alpha \ra$ be a pace chain such that $p, \vec{D}, F, \nu, \vec{W}, \vec{A} \in X_0$. \footnote{Note that $X_{\alpha}$ satisfies that $\chi_{X_{\alpha}}(\alpha^{+})\subseteq X_{\alpha}$, since, for every $\beta< \chi_{X_{\alpha}}(\alpha^{+})$, there exists $j<\alpha+$ such that $\beta<\chi_{X_j}(\alpha^+)$, and thus $\beta\in X_i$ for every $i$ above the maximum of $j$ and the index of $\beta$ in $W\left( \chi_{X_j}(\alpha^+) \right)$. This shows that the last clause of Remark \ref{Rmk:MillerCodingStrongerFusionProperty} holds.} Let $p_0 = p$, $F_0 = F$. In limit steps, we will take $p_i = \bigvee_{j<i} p_j$ and $F_i = \bigcup_{j<i} F_j$. In successor steps, we define $F_{i+1}$ as defined in the parallel step of the proof of Lemma \ref{Lem: MillerCodingDensityLemma}. As for $p_{i+1}$, pick $p_{i+1}\in X_{i+2}$ such that $p_{i+1}\geq_{F_i, \nu+i} p_i$, and $p_{i+1}\in \bqo^\alpha_\tau$ is $\leq_{F_i, \nu+i}$-generic over $X_{i+1}$, namely--
        $$p_{i+1} \in \bqo^\alpha_\tau \cap \bigcap\{ E\in X_{i+1} \colon E \mbox{ is a }\leq_{F_i, \nu+i}\mbox{-dense open subset of }\hqo^\alpha_\tau \}.$$
        Such $p_{i+1}$ exists by Lemmas \ref{Lem: MillerCodingDensityLemma} and \ref{Lem: MillerCodingDistributivityLemma}. In particular, for every $\gamma< \chi_{X_{i+1}}(\alpha^+)$, $p_{i+1}$ meets the $\leq_{F_i, \nu+i}$-dense open set of conditions $r\in \hqo^\alpha_\tau$ such that, for every $\xi\in \supp(r)$, ${r\uhr \xi}^{\frown}\name{T}^r_\xi \Vdash \sup(\name{c}^r_\xi)> \gamma$. It follows that, for every $\xi\in \supp(p_{i+1})$, ${p_{i+1}\uhr \xi }^{\frown} \name{T}^{p_{i+1}}_\xi \Vdash \chi_{X_{i+1}}(\alpha^+)\in \name{c}^{p_{i+1}}_\xi$. This concludes the construction of $p_{i+1}$. Before we proceed, we mention a useful property that $p_{i+1}$ has.

        \begin{claim} \label{Claim: MillerCodingFusionLemmaFirstClaim}
            Assume that $\sigma \colon F_i \to \alpha^\alpha$ belongs to $X_{i+1}$, and there exists $r\geq p_{i+1}$ such that:
            \begin{enumerate}
			\item $\sigma$ lies on $r$.
			\item $(r)_\sigma = r$.
			\item For every $\xi\in F_i$, $$r\uhr \xi \Vdash "\sigma(\xi) \mbox{ is an immediate successor of an element in }  \mbox{Split}_{\nu+i}\left(\name{T}^{p_{i+1}}_\xi\right)".$$
            If $\mbox{lh}(\sigma(\xi))$ is regular for every $\xi\in F_i$, we can replace this assumption with the simpler assumption that for every $\xi\in F_i$, 
            $r\uhr \xi \Vdash "\sigma(\xi) \in \mbox{Split}_{\nu+i}\left(\name{T}^{p_{i+1}}_\xi\right)"$.
            \end{enumerate}
            Then $\sigma$ lies on $p_{i+1}$, $r =(p_{i+1})_\sigma$  satisfies properties 1-3 above with respect to $\sigma$, and moreover, 
            $$\left( p_{i+1} \right)_\sigma \in D^*_i:= D_i\cap \left( \bigcap\{D\in X_i\colon D \mbox{ is a dense-open subset of }\hqo^{\alpha}_\tau \}\right).$$ 
        \end{claim}
        \begin{proof}
            Let $E$ be the set of conditions $s\in \hqo^{\alpha}_{\tau}$ for which one of the following hold:
		  \begin{enumerate}
			\item There is no $r\geq s$ that satisfies  properties $1-3$ above, with respect to $\sigma$.
			\item $\sigma$ lies on $s$ and $(s)_\sigma$ extends a condition $r \geq p_{i+1}$ satisfying properties 1-3 above, and $r$ belongs to $D^*_i$.
		\end{enumerate}
  $E \in X_{i+1}$ as $X_i \in X_{i+1}$, so it is sufficient to prove that $E$ is $\leq_{F_i, \nu+i}$-dense open above $p_i$. To see that this suffices, note that $E\in X_{i+1}$ and $p_{i+1}\geq_{F_i, \nu+i} p_i$ is $\leq_{F_i, \nu+i}$-generic over $X_{i+1}$, and thus $p_{i+1}\in E$. Therefore, by the assumptions of the claim, we have that $(p_{i+1})_\sigma$ satisfies properties 1-3 with respect to $\sigma$, and $(p_{i+1})_\sigma\in D^*_{i}$, as desired.

            It's not hard to see that $E$ is open. Thus, let us concentrate on the proof of the $\leq_{F_i, \nu+i}$-density of $E$. Assume that $s$ is a condition, $s\geq_{F_{i}, \nu+i} p_{i}$ , and let us argue that there exists $s^* \geq_{ F_{i}, \nu+i } s$ in $E$. Assume that there exists $r\geq s$ such that properties $1-3$ above hold. Since $D^*_i$ is dense open (by $\alpha$-distributivity of $\hqo^\alpha_\tau$), we can find $r^*\geq r$ in $D^*_i$. Note that $\sigma$ lies on $r^*$ and $(r^*)_{\sigma} = r^*$. Now, let $s^*\geq_{F_{i}, \nu+i} s$ be the amalgamation between $r^*, s$, defined as follows: $s^*$ has the same support as $r^*$, and, for every $\xi$ in this support,
			$$
			s^*\uhr \xi \Vdash \name{T}^{s^*}_{\xi} = 
			\left\{
			\begin{array}{ll}
				\left( \name{T}^{s}_\xi             \setminus  \left(\name{T}^{s}_\xi\right)_{ \sigma(\xi) } \right)\cup \name{T}^{r^*}_\xi & \mbox{if } \xi\in F_{i} \mbox{ and } r^*\uhr \xi \in \name{G}(\hqo^\alpha_\xi) \\
				
				\name{T}^{r^*}_\xi & \mbox{if } \xi\notin F_{i} \mbox{ and } r^*\uhr \xi \in \name{G}(\hqo^\alpha_\xi)  \\
				\name{T}^{s}_\xi & \mbox{if }  r^*\uhr \xi \notin \name{G}(\hqo^\alpha_\xi)  
			\end{array} 
			\right.
			$$
			where, in the notations above, $G(\hqo^\alpha_\xi)$ denotes the generic set up to coordinate $\xi$.          Also, set--
			$$
			{s^*\uhr \xi }^{\frown} \name{T}^{s^*}_{\xi} \Vdash \name{c}^{s^*}_{\xi} = 
			\left\{
			\begin{array}{ll}
				\name{c}^{r^*}_\xi & \mbox{if } r^*\uhr \xi \in \name{G}(\hqo^\alpha_\xi) \\
				\name{c}^{s}_\xi & \mbox{else}  
			\end{array} 
			\right.
			$$
			Let us justify that indeed $s^*\geq_{F_{i}, \nu+i} s$. Assume that $\xi\in F_i$. We argue that $s^*\uhr \xi \Vdash \name{T}^{s^*}_\xi \geq_{ \nu+i } \name{T}^{s}_\xi$. This is clear if $r^*\uhr \xi\notin G(\hqo^{\alpha}_\xi)$, so assume $r^*\uhr \xi\in G(\hqo^{\alpha}_\xi)$. In this case, the tree $\name{T}^{s^*}_\xi$ is   obtained by shrinking $\name{T}^s_\xi$ only above a successor of its splitting node in $\mbox{Split}_{\nu+i}\left( \name{T}^{p_i}_\xi \right)=\mbox{Split}_{\nu+i}\left( \name{T}^{s}_\xi \right)$. Thus, 
            this is an $\leq_{\nu+i}$-extension.\footnote{If $\sigma(\xi)$ is regular for every $\xi\in F_i$, we can shrink $\name{T}^{p_i}_{\xi}$ above $\sigma(\xi)$ (instead shrinking above one of its immediate successors) and get a $\leq_{\nu+i}$-extension, since $\sigma(\xi)$ is not a splitting node.}\\
			It remains to check that $\sigma$ lies on $s^*$, and $(s^*)_\sigma \geq r^*$. Since $D^*_{i}$ is open and contains $r^*$, this also ensures that $(s^*)_\sigma\in D^*_i$. 
			We argue by induction that for every $\xi\in \supp(s^*)$, if $(s^*)_{\sigma}\uhr \xi$ is defined, then it extends $r^*\uhr \xi$, and  $(s^*)_{\sigma}\uhr \xi \Vdash \sigma(\xi)\in \name{T}^{s^*}_\xi$. Since this is clear for $\xi\notin F_i$, we can assume $\xi\in F_i$. Take any $\eta<\xi$. By induction, $(s^*)_{\sigma}\uhr \eta$ is defined and extends $r^*\uhr \eta$. So $(s^*)_{\sigma}\uhr \eta$ forces that $r^*\uhr \eta\in \name{G}(\hqo^\alpha_\eta)$, and, whether or not $\eta\in F_i$ holds, we have, by our construction, $(s^*)_\sigma\uhr \eta \Vdash \left(\name{T}^{s^*}_{\eta}\right)_{\sigma(\eta)}\subseteq \name{T}^{r^*}_{\eta}$, and ${(s^*)_\sigma \uhr \eta}^{\frown} \name{T}^{(s^*)_\sigma}_\eta \Vdash \name{c}^{s^*}_\eta = \name{c}^{r^*}_\eta$. Since this is true for every $\eta<\xi$, we get that $(s^*)_\sigma\uhr \xi$ extends $r^*\uhr \xi$.    Having this, it follows that $(s^*)_{\sigma}\uhr \xi \Vdash r^*\uhr \xi\in \name{G}(\hqo^{\alpha}_\xi)$, and thus $(s^*)_{\sigma}\uhr \xi \Vdash \sigma(\xi) \in \name{T}^{s^*}_\xi $, as desired.
		\end{proof}

		This concludes the construction of the sequences $\vec{p} = \la p_i \colon i<\alpha \ra$, $\vec{F} = \la F_i \colon i<\alpha \ra$. By the parallel argument from the proof of case 4 of Lemma \ref{Lem: MillerC-FusionProperty}, $\bigcup_{i<\alpha}\supp(p_i) =\bigcup_{i<\alpha} F_i $. So $\vec{p}$ is indeed a paced fusion sequence with respect to $\vec{X}$, witnessed by $\vec{F}$. 

        Let $q = \bigvee_{i<\alpha} p_i$. $q\in \bqo^\alpha_\tau$ as a limit of a paced fusion sequence of conditions. We are now ready to finish the proof of Lemma \ref{Lem: MillerCodingC-FusionLemma} with the following two claims: the first isolates a club $C\subseteq \alpha$, and the second shows that $q, \vec{X}, C$ are witnesses for the $\alpha$ C-Fusion property.

        \begin{claim} \label{Claim: MillerCodingFusionLemmaSecondClaim}
            There exists a club $C\subseteq \alpha$ such that, for every $i\in C$,
            \begin{enumerate}
                \item $X_i \cap \alpha = i$.
                \item $F_i \subseteq X_i$.\footnote{\label{Footnote: MillerCodingFusion}We can actually strengthen this to $F_i = X_i\cap \alpha^{++}$, since $\{ i<\alpha \colon F_i = X_i\cap \alpha^{++} \}$ is a club in $\alpha$. To see this, note first that for every $i<\alpha$, $X_i\cap \alpha^{++}\subseteq \supp(p_{i+1})$ by $\leq_{F_i, \nu+i}$-genericity of $p_{i+1}$ over $X_i$. Since $|X_i|<\alpha$ and $\bigcup_{j<i}\supp(p_j) = \bigcup_{j<i} F_j$, we deduce that for some $j<\alpha$, $X_i\cap \alpha^{++} \subseteq F_j$. Now, assuming that we have established that $\{ i<\alpha \colon F_i\subseteq X_i \}$ is a club, shrink it by intersecting with the set of closure points of the map sending $i$ to the least $j$ with $X_i\cap \alpha^{++}\subseteq F_j$.}
                \item For every $j<i$, $\xi\in X_j$ and $t\in j^j$, 
                \begin{align*}
                    p_i\uhr \xi\Vdash "&\mbox{if } t\in \name{T}^{p_i}_\xi \mbox{ then either } t \mbox{ is a splitting node of } \name{T}^{p_i}_\xi \\
                    &\mbox{or the least splitting node of } \left(\name{T}^{p_i}_\xi\right)_{t} \mbox{ belongs to }i^{<i}". 
                \end{align*}
            \end{enumerate}        
        \end{claim}
        \begin{proof}
            It's enough to make sure that each clause of the three could be fulfilled on a club. The first clause is standard. For the second clause, consider the map sending each $j<\alpha$ to the least $i(j)<\alpha$ such that $F_j\subseteq X_{i(j)}$. If this map is a well defined map from $\alpha$ to $\alpha$, then every closure point $i$ of it satisfies $F_i\subseteq X_i$. To justify that this map is well defined, assume that $j<\alpha$. For every $\xi\in F_j$, let $\gamma(\xi)<\alpha$ be its index in the well order $W(F_j)$. Then $F_{j}\subseteq X_{i(j)}$ where  $i(j)=\max\left(j,\sup_{\xi\in F_j} \gamma(\xi) \right)+1<\alpha$. \\
            \noindent
            Finally, we justify that the last clause holds on a club. Fix some $j<\alpha$. For every $\xi\in X_j$ let
            \begin{align*}
                E(\xi) = \{ r\in \hqo^\alpha_\tau \colon &\mbox{there exists } i< \alpha \mbox{ such that } r\uhr \xi\Vdash "\mbox{if } t\in \name{T}^{p_j}_\xi\cap j^j, \mbox{ then either }\\
                &t \mbox{ is a splitting node of }\name{T}^{p_j}_\xi, \mbox{ or the least splitting node of }\\
                &\left( \name{T}^{p_j}_\xi \right)_t \mbox{ belongs to }i^{<i}."  \}    
            \end{align*}   
            By using the inductive assumption of Lemma \ref{Lem: MillerCodingDecidingOrds} for  $\hqo^\alpha_{\xi}$ to determine a suitable $i < \alpha$ in $V$, we see that $E(\xi)$ is $\leq_{F_j, \nu+j}$-dense open above $p_j$. By distributivity, $E = \bigcap_{\xi\in X_j} E(\xi)$ is $\leq_{F_j,\nu+j}$-dense open and belongs to $X_{j+1}$. Thus $p_{j+1}$ meets $E$. Let $i^*(j)<\alpha$ be minimal $i$ such that for every $\xi\in X_j$ and $t\in j^j$, $p_{j+1}\uhr \xi$ forces that if $t\in \name{T}^{q_j}_\xi$, then $t$ is either a splitting node, or the least splitting node above $t$ belongs to $i^*(j)^{<i^*(j)}$. We argue that every closure point $i$ of the map $j\mapsto i^*(j)$ satisfies the third clause.\\
            \noindent
            Indeed, fix such a closure point $i<\alpha$. Assume that $j<i$, $\xi\in X_j$, $t\in j^j$ and $G(\qo^\alpha_\xi)$ is $\hqo^\alpha_\xi$-generic over $V$ with $p_{i}\uhr \xi\in G(\qo^\alpha_\xi)$. Assume that in $V[G(\qo^\alpha_\xi)]$, $t\in {T}^{p_i}_\xi$ is not a splitting node. Since $t\in j^j$ and $T^{p_j}_\xi \leq_{\nu+j} T^{p_i}_\xi$, we have $t\in {T}^{p_j}_\xi$. Since $p_i \geq p_{j+1}$ and $p_{j+1}\in E(\xi)$, we have $p_i\in E(\xi)$. This, combined with the fact that $p_i\uhr \xi\in G(\qo^\alpha_\xi)$, implies that in $V[G(\qo^\alpha_\xi)]$, the least splitting node above $t$ in ${T}^{p_j}_\xi$ belongs to $i(j)^{<i(j)}$. The same splitting node is the least splitting node above $t$ in ${T}^{p_i}_\xi$  (again, because $T^{p_j}_\xi\leq_{\nu+j} T^{p_i}_\xi$) and indeed it belongs to $i^{<i}$.
        \end{proof}

        \begin{claim}
        The following properties hold for  $q = \bigvee_{i<\alpha} p_i$:
        \begin{enumerate}
            \item For every regular $i\in C$ and an $X_i$-generic set $G_{X_i}$ whose members are compatible with $q$, $G_{X_i}\cup \{ q \}$ has an exact upper bound, and this exact upper bound belongs to $D_i$. 
            \item $q$ is a generic condition for $X_\alpha$. 
        \end{enumerate}
        \end{claim}
        \begin{proof}${}$
            \begin{enumerate}
			\item Let $\sigma = \sigma_{G_{X_i}}\colon F_i \to i^i$ be a function, such that for every $\xi\in F_i$, $\sigma_{G_{X_i}}(\xi)$ is the initial segment of the generic branch decided by $G_{X_i}$ at coordinate $\xi$. Namely,
			$$\sigma_{G_{X_i}}(\xi) = \bigcup\{ t\colon \exists r\in G_{X_i}, r\uhr \xi \Vdash \left( \name{T}^{r}_\xi \right)_{t} = \name{T}^{r}_\xi \}.$$
			Since $G_{X_i}$ is $X_i$-generic and $X_i\cap \alpha = i$, we have for every $\xi\in F_i$, $\sigma(\xi)\in i^i$. Since every sequence of length $i$ whose elements are in $X_i$ belongs to $X_{i+1}$, $\sigma\in X_{i+1}$. 

            We argue that $\sigma$ lies on $q$, and $(q)_\sigma$ belongs to $D_i$ and is an exact upper bound of $G_{X_i}\cup \{ q \}$. \\
            The fact that $q$ is compatible with all members of $G_{X_i}$ ensures that for every $\xi \in F_i$, all initial segments of $\sigma(\xi)$ are forced to be in $T^q_{\xi}$. Since $T^q_{\xi}$ is closed the same holds for $\sigma(\xi)$, and thus $\sigma$ lies on $q$. 
            
            In order to prove that $(q)_{\sigma}\in D_i$, we prove first that $\sigma$ lies on $p_{i+1}$, and $r = (p_{i+1})_{\sigma}$ satisfies the assumptions of Claim \ref{Claim: MillerCodingFusionLemmaFirstClaim}.
            \begin{itemize}
                \item $\sigma$ lies on $p_{i+1}$: this follows from the fact $\sigma\colon F_i \to i^i$ lies on $q$, and $q \geq_{F_i,\nu+i} p_{i+1}$.
                \item $(r)_\sigma = r$: immediate from the fact that $r = (p_{i+1})_{\sigma}$.
                \item For every $\xi\in F_i$, $p_{i+1}\uhr \xi\Vdash \sigma(\xi)\in \mbox{Split}_{\nu+i}\left( \name{T}^{p_i}_\xi \right)$: assume that $\xi\in F_i$ and, by contradiction,  there exists some $j<i$ and an extension $s$ of $r\uhr \xi$ that forces that the splitting nodes of $\name{T}^{p_i}_\xi$ along $\sigma(\xi)$ are contained in $j^j$. By extending $j$ if necessary, we can assume that $j$ is a closure point of $\sigma(\xi)$, and denote $t = \sigma(\xi)\uhr j \in j^j$. So $s$ extends $p_i\uhr \xi$, forces that $t\in \name{T}^{p_i}_\xi$ and $t$ is not a splitting point. Since $i\in C$, it follows that the least splitting node of $\name{T}^{p_i}_\xi$ above $t$ belongs to $i^{<i}$. This contradicts the fact that $r\uhr \xi$ forces that all the splitting nodes of $\name{T}^{p_i}_\xi$ along $\sigma(\xi)$ are initial segments of $t$.
                
            \end{itemize}
   
            By Claim \ref{Claim: MillerCodingFusionLemmaFirstClaim}, $(p_{i+1})_{\sigma} \in D_i$. Since $D_i$ is open, it follows that $(q)_{\sigma}\in D_i$. It remains to verify that $(q)_{\sigma}$ is the exact upper bound of $G_{X_i}\cup \{ q \}$.\\
            
            \noindent
            In order to argue that it is an upper bound, we show that for every $s\in G_{X_i}$, $s\leq (q)_{\sigma}$. Indeed, given such $s$, the dense set $D(s)$ of conditions which are either incompatible or extend $s$ belongs to $X_{i}$. Thus, $D(s)\subseteq D^*_j$ for some $j<i$ with $X_j\cap \alpha = j$.\footnote{Since $i$ is regular, $\la X_k \colon k<i \ra$ is continuous and $X_i\cap \alpha = i$, there exists such $j<i$.} By repeating a similar argument as above, $\sigma'$ lies on $p_{j+1}$, where, for every $\xi\in F_j$,  $\sigma'\colon F_j\to \alpha^\alpha$ is given by taking $\sigma'(\xi)$ to be the initial segment of  $\sigma(\xi)$ which is the immediate successor of an element in $\mbox{Split}_j\left( \name{T}^{q}_\xi \right)$. Appealing to Claim \ref{Claim: MillerCodingFusionLemmaFirstClaim}, $(p_{j+1})_{\sigma'}\in D(s)$, and as both $(p_{j+1})_{\sigma'}, s$  belong to $G_{X_i}$, they must be compatible. Thus, $(p_{j+1})_{\sigma'}$ extends $s$, and, in particular, $(q)_{\sigma}$ extends $s$. \\
            
            \noindent
            The only thing left to be checked is that $(q)_{\sigma}$ is an exact upper bound of $G_{X_i}\cup \{ q\}$. Assume that $q'$ is an upper bound of $G_{X_i}\cup \{ q\}$. Since $q'$ extends each condition in $G_{X_i}$, $\sigma$ lies on $q'$ and $(q')_{\sigma} = q'$. Since, in addition, $q'$ extends $q$, we can deduce that $q'\geq (q)_\sigma$, as desired.

            \item Assume $G\subseteq \hqo^{\alpha}_\tau$ is generic over $V$ with $q\in G$. We argue that $G\cap X_{\alpha}$ is $X_{\alpha}$-generic. Let  $D\in X_\alpha$ be a dense open set. In $V[G]$, denote by $E$ the set of all $i<\alpha$ such that for every $\xi\in F_i$, $i$ is a closure point of the function in $\alpha^\alpha$ given by the generic branch of $G$ at coordinate $\xi$. It is not hard to verify that $E$ is a club in $\alpha$, and, since $\alpha$ remains Mahlo in $V[G]$ (by Lemma \ref{Lem: MillerCodingMahloLemma}) we can find a regular $i\in C\cap E$ with $D\in X_i$.
            
            Let $\sigma \colon F_i \to i^i$ be the function which maps each $\xi\in F_i$ to the generic branch induced by $G$ at coordinate $\xi$, restricted to $i$. Then $\sigma\in X_{i+1}$, and, by the above construction, $(p_{i+1})_{\sigma}\in D\cap G \cap X_{\alpha}$.
		\end{enumerate}
        \end{proof}
        This concludes the proof of Lemma \ref{Lem: MillerCodingC-FusionLemma}.
 	\end{proof}

	\begin{proof}(Lemma \ref{Lem: MillerCodingDecidingOrds})\\
		Fix $F,\nu$, $p\in \hqo^{\alpha}_\tau$, and a $\hqo^{\alpha}_\tau$-name $\zeta$ for an ordinal. Since $\hqo^{\alpha}_\tau$ is $\alpha$-distributive and satisfies the $\alpha$ C-fusion property, we can apply Lemma \ref{Lem: CFusionApproximation} (on the constant function $f\colon \alpha\to \{ \zeta \}$ in $V^{\hqo^{\alpha}_\tau}$), to find $q\geq p$ and a set $x$ of ordinals of size $|x|<\alpha$, such that $q \Vdash \name{\zeta}\in \check{x}$. Moreover, as our proof of Lemma \ref{Lem: MillerCodingC-FusionLemma} can be applied with respect to the suborder $\geq_{F,\nu}$, we may assume that $q\geq_{F,\nu} p$.  
	\end{proof}
	
 \begin{definition} \label{Def: MillerCodingDefinitionOfTheDenseSubset}
		We define $\qo^{\alpha}_{\tau}\subseteq \hqo^{\alpha}_{\tau}$ to be the set of conditions $q\in \hqo^{\alpha}_{\tau}$ for which there exists a sequence $\la F_i \colon i<\alpha \ra\subseteq [\supp(q)]^{< \alpha}$, such that:
		
  \begin{enumerate} 
        \item $q\in \bqo^{\alpha}_\tau$, witnessed by a local bounding function $\vec{\mu} = \la \mu_\xi \colon \xi\in \supp(q) \ra$.
        \item Given $\xi\in \supp(q)$, $t\in \alpha^{<\alpha}$, $\beta<\mu_\xi$ and $r\geq q$, such that $r\uhr \xi\parallel \check{t}\in \name{T}^{q}_\xi$ and ${r\uhr \xi}^{\frown}\name{T}^{r}_\xi \parallel \check{\beta}\in \name{c}^{q}_\xi$, there exists a regular cardinal $i<\alpha$ and $\sigma \colon F_i \to i^i$ such that:
    
        \begin{enumerate}
            \item $\sigma$ lies on $q$.
            \item $(q)_\sigma \uhr \xi$ decides the $\hqo^\alpha_\xi$-statements $``\check{t}\in \name{T}^{q}_\xi "$ 
            \item ${(q)_\sigma\uhr \xi}^{\frown}\name{T}^{(q)_{\sigma}}_\xi$ decides the $\hqo^\alpha_\xi*\mo_\alpha$-statement $``\check{\beta}\in \name{c}^{q}_\xi"$.
            \item $(q)_\sigma, r$ are compatible.
        \end{enumerate}
        \item For every $q'\in \bqo^{\alpha}_{\tau}$, if $q\geq q'$ and $q'\geq q$ then the least $\po_\alpha$-name for $q$ is lexicographically below the least $\po_\alpha$-name for $q'$, where on each coordinate $\xi<\alpha$, the coordinates are ordered with respect to the canonical well-order of $L[\E]$.
    \end{enumerate}
    
\end{definition}
	
	\begin{proof}(Lemma \ref{Lem: MillerCodingSecondDensityLemma})\\
		We argue that $\qo^{\alpha}_{\tau}$ is a dense subset of $\hqo^{\alpha}_{\tau}$, and $|\qo^{\alpha}_\tau| = \alpha^+$. 
		\begin{enumerate}
			\item $\qo^{\alpha}_\tau$ is dense: assume that $p\in \hqo^{\alpha}_\tau$. We argue that there exists $q\geq p$ in $\qo^{\alpha}_\tau$.

            Apply the $C$-fusion property (on an arbitrary sequence $\vec{D} = \la D_i \colon i<\alpha \ra$ of dense open sets) to find an extension $q\geq p$, $q\in \bqo^{\alpha}_\tau$ and some club in $\alpha$, witnessing the $\alpha$ C-Fusion property in the stronger sense of Remark \ref{Rmk:MillerCodingStrongerFusionProperty}. In particular, $q = \bigvee_{i<\alpha} q_i$, where $\la q_i \colon i<\alpha \ra$ is a paced fusion sequence with respect to a pacing chain $\vec{X} = \la X_i \colon i\leq \alpha \ra$ and some sequence $\vec{F} = \la F_i \colon i<\alpha \ra$. 
            

            We argue that $q$ satisfies clauses $1-3$ above. The first clause was already verified. We proceed to the second clause.
            
            Fix $\xi\in \supp(q)$, $t\in \alpha^{<\alpha}$, $\beta<\chi_{X_{\alpha}}(\alpha^{+})$ and some $r\geq q$ as in clause 2 of Definition \ref{Def: MillerCodingDefinitionOfTheDenseSubset}. Pick $j<\alpha$ high enough so that:
            \begin{itemize}
                \item $ t\in j^{<j} $.
                \item $\xi \in X_j\cap F_j$ (this holds for large enough $j<\alpha$ since $\supp(q) = \bigcup_{i<\alpha} F_i\subseteq X_{\alpha}\cap \alpha^{++}$).
                \item $\beta \in X_j\cap \chi_{X_j}(\alpha^+)$ (this holds for large enough $j<\alpha$ since $\chi_{X_\alpha}(\alpha^+)\subseteq X_\alpha$).
            \end{itemize}

            Let $G\subseteq \hqo^{\alpha}_{\tau}$ be a generic set with $r\in G$. By Lemma \ref{Lem:IA-Fusion-NewClubGeneric}, combined with the fact that $\alpha$ remains Mahlo in $V[G]$, there exists a regular $i<\alpha$ above $j$ such that $G_{X_i} := G\cap X_i\in V$ is $X_i$-generic and $(q)_\sigma$ is an exact upper bound of $(G_{X_i})\cup\{q\}$, where $\sigma = \sigma_{G_{X_i}}$ is as in Remark \ref{Rmk:MillerCodingStrongerFusionProperty}. $(q)_\sigma$ and $r$ are compatible, since both belong to $G$. Also,  $(q)_\sigma$ meets every dense open subset of $\hqo^{\alpha}_\tau$ which belongs to $X_i$. In order to verify that $q$ satisfies clause 2 in the definition of $\qo^\alpha_\tau$, it remains to check that $(q)_\sigma \uhr \xi \parallel \check{t} \in \name{T}^{q_{i}}_\xi$ and ${(q)_\sigma\uhr \xi}^{\frown} \name{T}^{(q)_\sigma}_\xi \parallel \check{\beta}\in \name{c}^{q}_\xi$. \\
            \noindent
            We first verify that $(q)_\sigma \uhr \xi \parallel \check{t} \in \name{T}^{q_{i}}_\xi$. By the discussion above, for every $j<j'<i$, $(q)_\sigma\uhr\xi \parallel \check{t} \in \name{T}^{q_{j'}}_\xi$. If there exists $j<j'<i$ such that $(q)_\sigma \Vdash \check{t}\notin \name{T}^{q_{j'}}_\xi$, then $(q)_\sigma\Vdash \check{t}\notin \name{T}^{q}_\xi$. Assume otherwise. Then $(q)_\sigma\Vdash \check{t}\in \name{T}^{q_{i}}_\xi$.

                Let us argue that there exists $j<j^*$ and $t^*\in (j^*)^{j^*}$ which extends $t$, such that $(q)_\sigma \uhr \xi \Vdash \check{t}^* \in \mbox{Split}_{j^*}\left( \name{T}^{q_{j^*}}_\xi \right)$. This will suffice, since $(q)_\sigma\uhr \xi \Vdash \name{T}^{q_{j^*}}_\xi \leq_{j^*} \name{T}^{q}_\xi$ and thus it will follow that $(q)_\sigma \Vdash \check{t}^*\in \name{T}^{q}_\xi$, and in particular, $(q)_\sigma \Vdash \check{t}\in \name{T}^{q}_\xi$.

                To show the existence of $t^*, j^*$ as above, we construct in $X_i$ an increasing sequence $\la t_n \colon n<\omega \ra$ of extensions of $t$ in $T^{p_i}_\xi$,  a sequence of ordinals $\la j_n \colon n<\omega \ra$ below $i$, and an increasing sequence of conditions $\la p_n \colon n<\omega \ra\subseteq G\cap X_i$, such that $j_n \geq \mbox{lh}(t_n)$, and $p_{n+1}\Vdash \check{t}_{n+1}\in \mbox{Split}_{j_n}( T^{q_{ j_n }}_\xi )\cap X_i$. Start with $j_0 =j$, $t_0 = t$. Assuming $t_n, j_n, p_n$ were picked, consider the set--
                $$ E = \{ r\in \hqo^\alpha_\tau \colon \mbox{for some } s\in \alpha^\alpha \mbox{ which extends } t_n, \  r\uhr \xi \Vdash \check{s}\in \mbox{Split}_{j_n}\left(T^{q_{j_n}}_\xi\right)  \}\in X_i$$
                which is a dense open subset of $\hqo^\alpha_\tau$ above $p_n$, and thus, it is met by some condition  $p_{n+1}\in G\cap X_i$ above $p_n$. Thus, we can pick $t_{n+1}\in X_i$ to be a witness $s$ for the fact that $p_{n+1}\in E$. So $p_{n+1}\Vdash 
                \check{t}_{n+1}\in \mbox{Split}_{j_n}\left({T}^{q_{j_n}}_\xi \right)$. Finally, take $j_{n+1}<i$ above $j_n$ such that $t_{n+1}\in {j_{n+1}}^{j_{n+1}}$. This describes the construction of the sequence $\la t_n \colon n<\omega \ra$. Setting $t^* = \cup_{n<\omega} t_n\in X_i$ and $j^* = \sup_{n<\omega}j_n< i$, we have $\mbox{lh}(t^*) = j^*$, 
                and, since $(q)_\sigma$ extends all the conditions $\la p_n \colon n<\omega \ra$, $(q)_\sigma \Vdash \check{t}^*\in \mbox{Split}_{j^*}\left(T^{q_{j^*}}_{\xi}\right)$. This concludes the proof that $(q)_\sigma \uhr \xi \parallel \check{t} \in \name{T}^{q_{i}}_\xi$.\\
                \noindent
                Next, let us verify that ${(q)_\sigma\uhr \xi}^{\frown} \name{T}^{(q)_\sigma}_\xi \parallel \check{\beta}\in \name{c}^{q}_\xi$. The dense open subset of ${\hqo^\alpha_\xi}*\name{\mo}_\alpha$ consisting of conditions deciding whether $\check{\beta}\in \name{c}^{q_j}_\xi$ belongs to $X_i$, since $\beta,\xi, q_j\in X_i$. Thus ${(q)_\sigma\uhr \xi}^{\frown} \name{T}^{(q)_\sigma}_\xi\parallel \check{\beta}\in \name{c}^{q_j}_\xi$. But $${(q)_\sigma\uhr \xi}^{\frown} \name{T}^{(q)_\sigma}_\xi\Vdash \name{c}^{q_j}_\xi = \name{c}^{q}_\xi \cap \max{\name{c}^{q_j}_\xi} \mbox{ and } \check{\beta}<\chi_{X_{j}}(\alpha^+)\leq \max{\name{c}^{q_j}_\xi}$$
                and thus we have ${(q)_\sigma\uhr \xi}^{\frown} \name{T}^{(q)_\sigma}_\xi \parallel \check{\beta}\in \name{c}^{q}_\xi$.

            Finally, let us check that clause $3$ is fulfilled. If there exists $q'$ such that $q'\leq q\leq q'$ and $q'$ has a $\po_\alpha$-name which is enumerated below the minimum $\po_\alpha$-name for $q$ (with respect to the constructibility order of $L[\E]$), replace $q$ with $q'$. We should argue that $q'$ still satisfies the previous clauses. First, note that $q'\in \bqo^{\alpha}_\tau$ with the same local bounding function $\vec{\mu}$, since $q'\leq q$. Next, it's not hard to see that for every regular $i<\alpha$ and $\sigma\colon F_i\to i^i$ lying on $q$, $\sigma$ lies on $q'$ as well, and $(q')_\sigma \leq (q)_\sigma \leq (q')_\sigma$. In particular, every statement forced by $(q)_\sigma$ is also forced by $(q')_\sigma$. Also, every condition $r$ which is compatible with $(q)_\sigma$, is also compatible with $(q')_\sigma$. This shows that $q'$ satisfies the second clause.

        \item $|\qo^{\alpha}_\tau| = \alpha^{+}$: we argue that every $q\in \mathbb{D}^\alpha_\tau$ is uniquely determined by the following parameters:
        \begin{itemize}
            \item The support of $q$, $\supp(q)\in [\tau]^{\leq \alpha}$.
            \item A local bounding function for $q$,  $\vec{\mu} = \la \mu_\xi \colon \xi\in \supp(q) \ra\in (\alpha^{+})^{\supp(q)}$.
            \item A sequence $\vec{F} = \la F_i \colon i<\alpha \ra\in  ([\supp(q)]^{<\alpha})^{\alpha}$.
            \item For every $\xi\in \supp(q)$ and $t\in \alpha^{<\alpha}$, the function $F_{\xi,t}$ whose domain is  $\Sigma:=\bigcup_{i<\alpha}\left(i^i\right)^{F_i}$, where, for every $\sigma\in \Sigma$,
            
            $$F_{\xi,t}(\sigma)=
            \begin{cases} 
            0 & \sigma \mbox{ lies on } q \mbox{ and } (q)_{\sigma}\Vdash \check{t}\in \name{T}^{q}_{\xi} \\
            1 & \sigma \mbox{ lies on } q \mbox{ and }(q)_{\sigma}\Vdash \check{t}\notin \name{T}^{q}_{\xi} \\ 
            2 & \sigma \mbox{ does not lie on } q, \mbox{ or }\sigma \mbox{ lies on } q \mbox{ and }\neg(q)_{\sigma}\parallel \check{t}\in \name{T}^{q}_{\xi}. \\
            \end{cases}$$

            \item For every $\xi\in \supp(q)$ and $\beta< \mu_\xi$, the function $G_{\xi,\beta}$ whose domain is  $\Sigma$, where, for every $\sigma\in \Sigma$,
            
            $$G_{\xi,\beta}(\sigma)=
            \begin{cases} 
            0 & \sigma \mbox{ lies on } q \mbox{ and } (q)_{\sigma}\Vdash \check{\beta}\in \name{c}^{q}_{\xi} \\
            1 & \sigma \mbox{ lies on } q \mbox{ and }(q)_{\sigma}\Vdash \check{\beta}\notin \name{c}^{q}_{\xi} \\ 
            2 & \sigma \mbox{ does not lie on } q, \mbox{ or }\sigma \mbox{ lies on } q \mbox{ and }\neg(q)_{\sigma}\parallel \check{t}\in \name{T}^{q}_{\xi}. \\
            \end{cases}$$
            \end{itemize}

        Since $\tau<\alpha^{++}$ and $2^{\alpha} = \alpha^{+}$ holds in $V$,  the set of all sequences of possible parameters as above has size $\alpha^{+}$.
        
        Let us argue that given $q,q'\in \mathbb{D}^{*\alpha}_{\tau}$ with the same parameters (support, local bounding function, $\vec{F}$, and functions $F_{\xi,t}, G_{\xi, \beta}$), $q= q'$. We prove this by showing that $q\leq q'$ and $q'\leq q$.

        Let us show that $q\geq q'$, and by symmetry, we deduce that $q'\geq q$ as well. We prove by induction on $\xi\in \supp(q)=\supp(q')$ that  $q\uhr \xi\geq q'\uhr \xi$ implies--
        \begin{itemize}
            \item $q\uhr \xi\Vdash \name{T}^{q}_\xi\geq \name{T}^{q'}_\xi$.
            \item ${q\uhr \xi}^{\frown}\name{T}^{q}_\xi \Vdash \name{c}^{q}_\xi \geq \name{c}^{q'}_\xi$.
        \end{itemize} 
        Fix such $\xi$. Assume by contradiction that there exist $r\geq q$ and $t\in \alpha^{<\alpha}$ such that $r\uhr \xi\Vdash \check{t}\in \name{T}^{q}_\xi\setminus \name{T}^{q'}_\xi$. Since $q\in \qo^{\alpha}_\xi$, there are regular $i<\alpha$ and $\sigma\colon F_i \to i^i$ which lies on $q$, such that $(q)_\sigma$ and $r$ are compatible, and $(q)_{\sigma}\uhr \xi\Vdash \check{t}\in \name{T}^{q}_\xi$. But $q,q'$ share the same function $F_{\xi,t}(\sigma)$, so the above $\sigma$ lies on $q'$, and $(q')_{\sigma}\uhr \xi\Vdash \check{t}\in \name{T}^{q'}_\xi$. But, since $q\uhr \xi \geq q'\uhr \xi$ we get $(q)_\sigma \uhr \xi \geq (q')_{\sigma}\uhr \xi$. So $(q)_\sigma\Vdash  \check{t}\in \name{T}^{q}_\xi \cap \name{T}^{q'}_\xi$. This contradicts the fact that $r, (q)_\sigma$ are compatible.

        It's left to show that ${q\uhr \xi}^{\frown} \name{T}^{q}_\xi \Vdash \name{c}^{q}_\xi \geq \name{c}^{q'}_\xi$. Assume by contradiction that there exists $r\geq q$ and an ordinal $\beta$ such that ${r\uhr \xi}^{\frown}\name{T}^{r}_\xi \Vdash \beta\leq\max\left( \name{c}^{q'}_\xi \right) \mbox{ and } \beta\in \name{c}^{q}_\xi \triangle \name{c}^{q'}_\xi$. Note that such $\beta$ must satisfy $\beta<\mu_\xi$. Let $i<\alpha$ be regular and $\sigma\colon F_i\to i^i$ lying on $q$, such that ${(q)_\sigma \uhr \xi}^{\frown}\name{T}^{(q)_\sigma}_\xi$ is compatible with $r$ and decides whether $\check{\beta}\in \name{c}^{q}_\xi$. Since $G_{\xi,\beta}(\xi)\neq 2$, $\sigma$ lies on $q'$ and ${(q')_\sigma \uhr \xi}^{\frown}\name{T}^{(q')_\sigma}_\xi$ also decides whether $\beta\in \name{c}^{q'}_\xi$, in the same way. As we already proved that $q\uhr \xi\Vdash \name{T}^{q}_\xi\geq \name{T}^{q'}_\xi$, we have ${(q)_\sigma \uhr \xi}^{\frown}\name{T}^{(q)_\sigma}_\xi\geq {(q')_\sigma \uhr \xi}^{\frown}\name{T}^{(q')_\sigma}_\xi$. Thus ${(q)_\sigma \uhr\xi}^{\frown} \name{T}^{(q)_\sigma}_\xi \Vdash \check{\beta}\notin \name{c}^{q}_\xi \triangle \name{c}^{q'}_{\xi}$, contradicting the fact that $(q)_\sigma,r$ are compatible.
        \end{enumerate}
	\end{proof}

This concludes the inductive proof of Lemmas \ref{Lem: MillerCodingDensityLemma}-\ref{Lem: MillerCodingSecondDensityLemma}. 

Finally, let us define the dense subset We conclude this subsection by proving that $\qo_\alpha$ satisfies all the properties from the ``Kunen-like" Blueprint \ref{Def:KLblueprint}.

\begin{theorem}\label{Thm:MillerCodingFinalForcing}
    Let $\qo_\alpha\subseteq \hqo^\alpha_{\alpha^{++}}$ be the set of all the conditions in $\qo^\alpha_\tau$ for some $\tau<\alpha^{++}$. Then:
    \begin{itemize}
            \item $|\qo_\alpha| =\alpha^{++}$ and adds a $<^\alpha_{Sing}$-increasing sequence of length $\alpha^{++}$,
            
            \item ${\qo}_\alpha$ is $\alpha$-distributive, satisfies the $\alpha^{++}$.c.c, and preserves the Mahloness of $\alpha$,
            
            \item ${\qo}_\alpha$ has the $\alpha$ C-Fusion property,

            \item ${\qo}_\alpha$ is self coding, by the almost disjoint sequence of stationary sets $$\vec{\mathcal{S}}^{\alpha} = \la \mathcal{S}^{\alpha}(\tau, \eta, i) \colon \tau<\alpha^{++}, \eta<\alpha^{+}, i<2 \ra\subseteq \alpha^{+}\cap\cf(\omega_1).$$ 
            Furthermore, the self coding is uniformly definable over $(H_{\alpha^{++}})^{L[\E]}$, in the sense that there exists a parameter-free formula $\varphi_{\qo}(\alpha,x,y)$ that defines the following over every inner model containing $(H_{\alpha^{++}})^{L[\E]}$:
            \begin{itemize}
                \item an enumeration $\la \name{q}_\eta\colon \eta< \alpha^{+} \ra$ of $\po_\alpha$-names for the conditions in $\qo_\alpha$.
                \item the sequence $\vec{\mathcal{S}}^{\alpha}$ of stationary subsets of $\alpha^{+}$, used for coding.
            \end{itemize}       

        \end{itemize}
\end{theorem}

\begin{proof}
    Since $\qo_\alpha$ is a dense subset of $\hqo^\alpha_{\alpha^{++}}$, it adds a $<^\alpha_{Sing}$-increasing sequence of Miller subset of $\alpha$; furthermore, it preserves the Mahloness of $\alpha$, and doesn't new add $<\alpha$-sequences of ordinals. Thus, $\qo_\alpha$ is $\alpha$-distributive. $\qo_\alpha$ satisfies the $\alpha$-fusion property, since, as seen in the proof of Lemma \ref{Lem: MillerCodingSecondDensityLemma}, every exact upper bound of a paced fusion sequence belongs to it (regardless of the initial sequence $\vec{D}$ of dense open sets). 

    We argue that $\qo_\alpha$ satisfies the $\alpha^{++}-c.c.$. Assume that $\la p_\gamma \colon \gamma<\alpha^{++} \ra$ is a sequence of conditions in $\qo_\alpha$. By applying the $\Delta$-System Lemma (which is possible since $2^\alpha = \alpha^+$ over $V$), we can assume that $\la \supp(p_\gamma) \colon \gamma< \alpha^{++} \ra$ is a delta-system with root $s\in [\alpha^{++}]^{\leq \alpha}$. Let $\tau = \sup(s)$. Then $\la q_\gamma \uhr \tau \colon \gamma<\alpha^{++} \ra\subseteq \qo^\alpha_\tau$ and $|\qo^\alpha_\tau|= \alpha^{+}$. Thus, by picking $\gamma\neq \gamma'$ with $q_\gamma\uhr \tau, q_{\gamma'}\uhr \tau$ compatible, we get that $q_\gamma, q_{\gamma'}$ are compatible as well.

    It remains to check that $\qo_\alpha$ is self coding.  For every $p\in \qo_\alpha$, let $\tau<\alpha^{++}$ be the least such that $p\in \qo^\alpha_\tau$, and let $\eta<\alpha^+$ be the enumeration of $p$ in the well-order on $\qo^\alpha_\tau$ induced by minimal $\po_\alpha$-names, as described in Definition \ref{Def: Definition of the Coding poset}. Those minimal $\po_\alpha$-names belong to $(H_{\alpha^{++}})^{L[\E]}$ (this follows since conditions in $\qo_\alpha$ belong to $(H_{\alpha^{++}})^{V}$, and since $|\po_\alpha| = \alpha^+$, they have $\po_\alpha$-names in $(H_{\alpha^{++}})^{L[\E]}$). This shows that the list minimal $\po_\alpha$-names for conditions of $\qo_\alpha$ can be enumerated in a definable way over $(H_{\alpha^{++}})^{L[\E]}$. Also, the stationary sets used for coding, and their correspondence with elements of $\qo_\alpha$, are all definable over $H_{\alpha^{++}}$. This shows that  In order to complete the self-coding verification, it's enough to argue that--
    $$ 0_{\qo_\alpha}\Vdash p\in G(\qo_\alpha) \iff \mathcal{S}^{\alpha}(\tau, \eta ,0) \mbox{ is nonstationary} $$
    and, similarly--
    $$ 0_{\qo_\alpha}\Vdash p\notin G(\qo_\alpha) \iff \mathcal{S}^{\alpha}(\tau, \eta ,1) \mbox{ is nonstationary}. $$
    Assuming that $p\in G(\qo_\alpha)$, the generic coding poset at coordinate $\tau$ adds club disjoint from $\mathcal{S}^\alpha(\tau, \eta, 0)$. Thus, we concentrate on the other direction and prove that--
    $$0_{\qo_\alpha}\Vdash p\notin G(\qo_\alpha) \implies \mathcal{S}^\alpha(\tau, \eta, 0) \mbox{ is stationary}.$$
    First, by Lemma \ref{Lem: MillerCodingPreservationOfStationarySets},  $\hqo^\alpha_\tau$ preserves the stationarity of $\mathcal{S}^\alpha(\tau, \eta, 0)$. Once we will prove that $\mathbb{R}_\tau:= \mo_\alpha * \name{\co}_\tau$ preserves the stationarity of $\mathcal{S}^\alpha(\tau, \eta,0)$, the same argument of Lemma \ref{Lem: MillerCodingPreservationOfStationarySets} shows that $\mathcal{S}^\alpha(\tau,\eta,0)$ is preserved in $V^{\qo_\alpha}$. Thus, it suffices to prove that, over $V^{\hqo^\alpha_\tau}$, $\mathbb{R}_\tau$ preserves the stationarity of $\mathcal{S}^\alpha(\tau, \eta, 0)$. To show this, assume that $r\in \ro_\tau$ and $\name{C}$ is a $\ro_\tau$-name for a club subset of $\alpha^+$. We argue that $r$ can be extended to a condition which forces that $\name{C}\cap \mathcal{S}^\alpha(\tau, \eta,0)\neq \emptyset$. 
    
    Take an elementary substructure $X\elem (H_\theta)^{ V^{\hqo^\alpha_\tau} }$ (for $\theta$ large enough) such that:
    \begin{itemize}
        \item $\ro_\tau, r, \name{C}, \alpha, \tau\in X$.
        \item $|X|=\alpha$.
        \item $X$ is closed under countable sequences of its elements.
        \item $\chi_{X}(\alpha^+)\in \mathcal{S}^\alpha(\tau, \eta,0)$.
        \item $\chi_{X}(\alpha^+)\subseteq X$.
    \end{itemize}
    Such $X$ can be constructed as in the proof of Lemma \ref{Lem: MillerCodingPreservationOfStationarySets}. Note that $\alpha$ remains Mahlo in $V^{\qo^\alpha_\tau}$, so the closure under countable sequences can be ensured.
    
    Fix a cofinal sequence $\la \delta_i \colon i<\omega_1 \ra$ in $\chi_{X}(\alpha^+)$. In $X$, construct an increasing sequence $\la r_i = (T_i, \name{c}_i) \colon i<\omega_1 \ra$ of conditions in $\ro_\tau$: take $r_0 = r$, and at every limit stage $i<\omega_1$, use the $\sigma$-closure of $\ro_\tau$ and the closure under countable sequences of $X$, to pick an upper bound $p_i\in X$ for $\la p_j \colon j<i \ra$. At successor stages, make sure that $r_{i+1} \parallel \min(\name{C}\setminus \delta_i)$ and $T_{i+1}\Vdash \max(\name{c}_{i+1})> \delta_i$.

    This concludes the inductive construction. Let $r^* = (T^*,\name{c}^*)$, where $T^* = \bigcap_{i<\omega_1} T_{i}$, and $\name{c}^* = \left(\bigcup_{i<\omega_1} \name{c}_{i}\right) \cup \chi_{X}(\alpha^+)$. We argue that $r^*\in \ro$. First, $T^*\in \mo_\alpha$ by $\alpha$-closure of $\mo_\alpha$. So it suffices to justify that $T^*\Vdash \name{c}^*\in \name{\co}_\tau$. Indeed, 
    $\chi_{X}(\alpha^+)\notin \mathcal{S}^\alpha(\tau, \eta', j)$ whenever $\eta'\neq \eta$ or $(\eta'=\eta \land j=1)$,  since  $\chi_{X}(\alpha^+)\in \mathcal{S}^\alpha(\tau, \eta,0)$ and $\la \mathcal{S}^\alpha(\tau, \eta', j) \colon \eta'<\alpha^+, j<2 \ra$ are pairwise disjoint.

    Finally, $r^*$ forces that $\chi_{X}(\alpha^+)$ is a limit point of $\name{C}$. Thus $r^*\Vdash \name{C}\cap \mathcal{S}^\alpha(\tau, \eta, 0)\neq \emptyset$, as desired.
\end{proof}

\subsection{The Final Forcing $\po = \po_{\kappa+1}$ and its Properties}\label{subsection:final}

In this section we conclude the construction of the forcing notion $\po$ over $L[\E]$ which realizes the ``Kunen-like" blueprint. 

We define a nonstationary support iteration of length $\kappa+1$,  $\la \po_\alpha, \name{\qo}_\alpha \colon \alpha\leq \kappa+1 \ra$. For every non-Mahlo $\alpha<\kappa$, $\name{\qo}_\alpha$ is taken to be the trivial forcing. For every Mahlo $\alpha\leq \kappa$, $\name{\qo}_\alpha$ is taken to be the forcing given in the statement of Theorem \ref{Thm:MillerCodingFinalForcing}. Recall that such a forcing $\name{\qo}_\alpha$ depends on a sequence $\mathcal{S}^\alpha$ of almost disjoint stationary sets in $V^{\po_\alpha}$. Moreover, as required in the definition of the ``Kunen-like" blueprint, the forcings $\qo_\alpha$ (and, consequently, the stationary sets $\vec{\mathcal{S}}^\alpha$) should be uniformly defined, in the sense of Clause 4 in Definition \ref{Def:FM-blueprint}. 

Once we will provide the proper definition of the above stationary sets, we will let $\po = \po_{\kappa+1}$. 

This section is devoted to the proof that the properties of the Friedman-Magidor and ``Kunen-like" blueprints are fulfilled by $\po$. The structure of this section is as follows:
\begin{itemize}
    \item In Lemma \ref{Lem: FinalForcingCinstructionOfStationarySets} we will define the stationary sets $\vec{\mathcal{S}}^\alpha$ for every Mahlo $\alpha\leq \kappa$. By that, we complete the definition of $\po_{\kappa+1}$. We remark that the sets are stationary in $L[\E]$. In Lemma \ref{Lem: FinalForcingPreservationOfStatSets} we prove that, for every $\alpha\leq \kappa$ Mahlo, the stationarity of the sets $\vec{\mathcal{S}}^\alpha$ is preserved by the forcing $\po_\alpha$.
    \item In Lemma \ref{Lem: FinalForcingFusionLemma} we prove that, for every Mahlo $\lambda\leq \kappa$, $\po_\lambda$ satisfies the $\lambda$ Iteration-Fusion property.
    \item In Lemma \ref{Lem: FinalForcingDistributivityLemma} we prove that for every $\lambda\leq \kappa$, $\Vdash_{\po_\lambda} \left( \po/ \po_{\lambda} \right)$ is $\lambda$-distributive.
    \item With the Lemmas established, we prove the our main Theorem \ref{Thm:Main}.
\end{itemize}

\begin{lemma} \label{Lem: FinalForcingCinstructionOfStationarySets}
    For every $\alpha\leq \kappa$, there exists in $L[\E]$ a sequence $$\vec{\mathcal{S}}^\alpha = \la \mathcal{S}^\alpha(\tau, \eta,i) \colon \tau<\alpha^{++}, \eta<\alpha^+, i<2 \ra$$
    of almost disjoint stationary subsets of $\mathcal{S}^\alpha$ (see Proposition \ref{Prop: Stationary set of passive collapsing structures}) such that,
    \begin{enumerate}
        \item For every $\tau<\alpha^{++}$, $\la \mathcal{S}^\alpha(\tau, \eta,i) \colon \eta<\alpha^+, i<2 \ra$ are fully pairwise disjoint.
        \item For every $\tau<\alpha^{++}, \eta<\alpha^{+}$ and $ i<2$, $\min(\mathcal{S}^\alpha(\tau,\eta,i))> \eta$.
        \item The sequence is uniformly  defined, in the sense that  there exists a formula $\varphi_{\mathcal{S}}(x,y)$ such that $\varphi_{\mathcal{S}}(\alpha, y)$ defines $\vec{\mathcal{S}}^\alpha$ in every inner model that contains $H_{\alpha^{++}}$.
    \end{enumerate}
\end{lemma}
\begin{proof}
Recall that for every $\alpha\leq \kappa$, $L[\E]$ contains a  $\Diamond\left(\mathcal{S}^\alpha \right)$-sequence which is definable over $H_{\alpha^{++}}$ from $\alpha$. Denote the 
least such a sequence in the canonical well ordering of $L[\E]$ by $\la A^{\alpha}_\xi \colon \xi<\alpha^+ \ra$. For every subset $B\subseteq \alpha^+$, let 
$$\mathcal{S}^\alpha(B) = \{ \xi\in \mathcal{S}^\alpha \colon  B\cap \xi = A^\alpha_\xi \}.$$
Then $\la \mathcal{S}^\alpha(B) \colon B\in \mathcal{P}\left( 
\alpha^+ \right) \ra$ is a sequence of almost disjoint stationary subsets of $\mathcal{S}^\alpha$. The almost disjointness follows since, for every pair of subsets $B\neq B'$ of $\alpha^+$, $\mathcal{S}^\alpha(B) \cap \mathcal{S}^\alpha(B')\subseteq \min B\triangle B'$.

Finally, pick the least constructible bijection from $(\mathcal{P}(\alpha^+))^{L[\E]}$ to $\alpha^{++}\times \alpha^+\times 2$, and use it to re-index the sequence $\la \mathcal{S}^\alpha(B) \colon B\subseteq \alpha^+ \ra$ in a definable way, in the form $\la \mathcal{S}^\alpha(\tau, \eta,i) \colon \tau<\alpha^{++}, \eta<\alpha^+, i<2 \ra$. We can further assume that for every $\tau<\alpha^{++}$, 
$$\la  \mathcal{S}^\alpha(\tau, \eta,i) \colon \eta<\alpha^+, i<2 \ra $$
are pairwise disjoint and for each $\eta<\alpha^{+}, i<2$, $\min(\mathcal{S}^\alpha(\tau,\eta,i))> \eta$, by definably modifying the sequence. More formally, replace each $\mathcal{S}^\alpha(\tau, \eta,i)$ with $\mathcal{S}^{\alpha}(\tau, \eta, i)\setminus \left(\gamma(\tau,\eta,i)\cup (\eta+1)\right)$, where--
$$ \gamma(\tau,\eta, i) = \sup\Big\{ \mathcal{S}^\alpha(\tau, \eta ,i)\cap \mathcal{S}^\alpha(\tau, \eta',j) \colon (\eta',j)<_{lex} (\eta,i)  \Big\}+1 $$
is an ordinal below $\alpha^+$ by almost disjointness.
\end{proof}

We now proceed towards the proofs of Lemmas \ref{Lem: FinalForcingFusionLemma}, \ref{Lem: FinalForcingDistributivityLemma}. Let us fix some notations. Assume that $p\in \po$ (or $p\in \po_{\lambda}$ for some $\lambda\leq \kappa+1$) and $\alpha \in \supp(p)$ is Mahlo. Then 
$$p\uhr \alpha\Vdash p(\alpha) \in \name{\qo}_\alpha.$$
In particular, using the notations of the previous section, for every $\xi<\alpha^{++}$, 
$$ p\uhr \alpha \Vdash_{\po_\alpha}  \left(  p(\alpha)\uhr \xi \Vdash_{\qo_\alpha \uhr \xi}  \  p(\alpha)(\xi) = \la \name{T}^{p(\alpha)}_\xi , \name{c}^{p(\alpha)}_\xi \ra \right) $$
where $\name{T}^{p(\alpha)}_\xi$ is forced to be a condition in the generalized Miller forcing in $V^{ \po_\alpha * {\name{Q}^\alpha\uhr \xi} }$, namely--
$$ {p\uhr \alpha}^{\frown} p(\alpha)\uhr \xi \Vdash \name{T}^{p(\alpha)}_\xi \in \name{\mo}_\alpha $$
and $\name{c}^{p(\alpha)}_\xi$ is forced to be a condition in a coding forcing in $V^{ \po_\alpha * {\name{Q}^\alpha\uhr \xi * \name{\mo}_\alpha } }$, namely--
$$ {p\uhr \alpha}^{\frown} \left({p(\alpha)\uhr \xi}^{\frown} \name{T}^{p(\alpha)}_\xi\right) \Vdash \name{c}^{p(\alpha)}_\xi \in \name{\co}^\alpha_\xi $$

\begin{definition}
Assume that $\lambda\leq \kappa+1$ and $p\in \po_{\lambda}$. A \textbf{local bounding function} for $p$ is a sequence $\vec{\mu}^p = \la \mu^p_{\alpha} \colon \alpha\in \supp(p) \ra \in \prod_{\alpha\in \supp(p)} \alpha^+$, such that for every Mahlo $\alpha\in \supp(p)$,
$$ p\uhr \alpha \Vdash \forall \xi\in \supp(p(\alpha)), {p(\alpha)\uhr \xi}^{\frown} \name{T}^{p(\alpha)}_\xi \Vdash \sup\left( \name{c}^{ p(\alpha) }_\xi \right) < \mu^p_\alpha. $$
    
\end{definition}
In other words, for every $\alpha\in \supp(p)$, $p\uhr \alpha$ forces that the constant function whose value is $\mu^p_\alpha$ is a local bounding function for $p(\alpha)\in \name{\qo}_\alpha$, in the sense of Definition \ref{Def: PacedFusionSequences}. In the following Lemma, we argue that the subset of conditions equipped with a local bounding function is dense.

\begin{lemma}\label{Lem: FinalForcingDensityLemma}
    Assume that $\lambda\leq \kappa+1$. Let $\bar{\po}_\lambda\subseteq \po_\lambda$ be the set of conditions having a local bounding function. Then $\bar{P}_\lambda$ is dense.
\end{lemma}

\begin{proof}
    We work by induction on $\lambda$.

    Assume that $\lambda = \alpha+1$ is successor. Let $p\in \po_\lambda$. In $V^{\po_{\alpha}}$, $p(\alpha)\in \qo_{\alpha}$, and, by Lemma \ref{Lem: MillerCodingSecondDensityLemma}, $p(\alpha)$ has a local bounding function in the sense of Definition \ref{Def: PacedFusionSequences}. In particular, there exists a $\po_{\alpha}$-name for an ordinal $\name{\mu} < {\alpha}^+$ such that, for every $\xi\in \supp(p(\alpha))$, 
    $${p\uhr \alpha}^{\frown} {p(\alpha)\uhr \xi }^{\frown} \name{T}^{p(\alpha)}_\xi \Vdash \sup\left(\name{c}^{p(\alpha)}_\xi\right) < \name{\mu}. $$
    Now extend $q\geq p\uhr \alpha$ such that $q\in \bar{\po}_{\alpha}$ and $q$ decides the value of $\mu$. Then $q^{\frown} p(\alpha) $ extends $p$ and belongs to $\bar{\po}_\lambda$.

    Assume that $\lambda$ is limit. We first address the case where $\lambda$ is a regular cardinal, and then make the needed adjustments to the singular case.  
    
    We construct sequences--
\begin{itemize}
    \item $\la \alpha_i \colon i<\lambda \ra $ an increasing, continuous, cofinal sequence in $\lambda$.
    \item $\la p_{i} \colon i<\lambda \ra $ a fusion sequence of conditions in $\po_\lambda$, namely, increasing in the order of $\po_\lambda$, and satisfying that for every $i<j$, $p_j\uhr \alpha_i = p_i \uhr \alpha_i$. 
    \item $\la C_i \colon i <\lambda \ra $, decreasing with respect to inclusion, each $C_i$ is a club disjoint from the support of $p_i$.
\end{itemize}

We will also make sure during the construction that, for every $i<\lambda$, $p_{i}\uhr \alpha_i \in \bar{\po}_{\alpha_i}$,  $\alpha_i\notin \mbox{supp}( p_{i+1})$ and $p_{i+1}\uhr \alpha_{i+1} \Vdash p_{i+1}\setminus \alpha_{i+1} = p\setminus \alpha_{i+1}$.

Start with $p_0 = p$, $C_0$ any club disjoint from its support, and $\alpha_0 = 0$.

Successor steps: assume that the $p_i,C_i, \alpha_i$ were defined. Let $\alpha_{i+1}$ be the first element in $C_i$ strictly above $\alpha_i$. We define $p_{i+1}$. First, define $p_{i+1}\uhr \alpha_i = p_{i}\uhr \alpha_i$. $\alpha_i$ itself is outside the support of $p_{i+1}$. Denote $I_i  =\left( \alpha_i, \alpha_{i+1} \right)$, and take $p_{i+1}\uhr I_i$ to be an extension of  $p_i \uhr I_i$, such that, for every $\beta$ in its support, there exists an ordinal $\mu^{p_{i+1}}_\beta < \beta^+$ such that-- 
$$p_{i+1} \uhr \beta \Vdash \forall \xi\in \supp\left( p_{i+1}(\beta)  \right), \     { p_{i+1}(\beta)\uhr \xi }^{\frown}  \name{  T  }^{p_{i+1}(\beta)}_\xi \Vdash \sup\left( \name{c}^{p_{i+1}(\beta)}_\xi \right) < \mu^{p_{i+1}}_{\beta}.$$
This is possible by the induction hypothesis. The sequence of bounds $\la \mu^{ p_{i+1} }_\beta  \colon \beta\in \supp( p_{i+1} ) \cap I_{i}  \ra $ is a $\po_{\alpha_i+1}$-name, but since $|\po_{\alpha_i+1}| = (\alpha_{i})^{++}$, we can assume as well that the sequence is known in $V$ (by replacing each $\mu^{p_{i+1}}_\beta$ with the supremum of all the possible values of $\name{\mu}^{p_{i+1}}_\beta$ as forced by conditions in $\po_{\alpha_i+1}$). Finally, choose $p_{i+1}\setminus \alpha_{i+1}$ such that $p_{i+1}\uhr \alpha_{i+1} \Vdash p_{i+1}\setminus \alpha_{i+1} = p\setminus \alpha_{i+1} $. This concludes the definition of $p_{i+1}$, and, by standard arguments, it has a nowhere stationary support. Let $C_{i+1}$ be obtained from $C_{i}$ by shrinking it so it is disjoint from this support.

We move on to the limit case. Assume that $i<\alpha$ and  $\la p_j, C_j , \alpha_j \colon j< i\ra$ were defined. Let $ \alpha_i = \sup\{ \alpha_j \colon j<i\} $. We define $p_i$. First, take $p_i \uhr \alpha_i = \bigcup_{ j<i } p_{j}\uhr \alpha_j+1 $. $\alpha_i$ itself will be outside the support of $p_i$, as well as all the ordinals $\la \alpha_j \colon j<i \ra$. Finally, set $p_i\setminus \alpha_i = p\setminus \alpha_i$. Then $p_i\uhr \alpha_i \in \bar{\po}_{ \alpha_i }$, since, for every $j<i$, we have by induction $p_j \uhr \alpha_j \in \bar{\po}_{\alpha_j}$. Finally, take $C_i = \bigcap_{j<i} C_j$. 

This concludes the construction. We finally take $p^*$ to be the natural limit of the fusion sequence, namely $p^* = \bigcup_{i<\lambda} p_{i}\uhr \alpha_i+1$. Then $p^* \in \bar{\po}_{\lambda}$, as desired. 

For the case where $\lambda$ is singular, fix an increasing, cofinal sequence $\la \lambda_i \colon i< \cf(\lambda) \ra$ in $\lambda$. Repeat the above argument, only construct sequences of length $\cf(\lambda)$ rather than $\lambda$, and choose the sequence of ordinals $\la \alpha_i \colon i<\cf(\lambda) \ra $ such that, for every $i<\cf(\lambda)$, $\alpha_{i+1}$ is the least element of $C_i$ above $\max{\alpha_i, \lambda_i}$. 
\end{proof}

\begin{lemma} \label{Lem: FinalForcingFusionLemma}
    For every $\lambda\leq \kappa$ Mahlo, $\po_\lambda$ satisfies the $\lambda$ Iteration-Fusion property.
\end{lemma}

\begin{proof}
Assume otherwise, and let $p\in \po_{\lambda}$ be a condition and $\vec{D} = \la D_i \colon i<\lambda \ra$ be a sequence such that $\la p, \vec{D} \ra$ is the 
least counterexample. In other words, $\la p, \vec{D} \ra$ is chosen constructibly minimal such that $\vec{D}$ is a sequence of dense open subsets of $\po_\lambda$, for which there are no $p^*\geq p$ and  club $C\subseteq \lambda$ of coordinates $i\in C$ such that the set--
$$\{ s\in \po_{i+1} \colon s^{\frown} p^*\setminus i+1 \in D_i \}$$
is a dense open subset of $\po_{i+1}$. 

All the relevant parameters $\lambda, \po_{\lambda}, \bar{\po}_{\lambda}, p, \vec{D}$ belong to $J_{\lambda^{+3}}$,\footnote{Actually, all the relevant parameters above already belong to $J^{\E}_{\lambda^{++}}$. We chose $J^{\E}_{\lambda^{+3}}$ is order to use Proposition \ref{Prop: Stationary set of passive collapsing structures on an arbitrary cofinality}.} and are lightface definable there: $\lambda$ itself  is definable as the third-to-last cardinal; the forcings $\po_{\lambda}, \bar{\po}_{\lambda}$ are definable from $\lambda$ and from the sequence of canonical stationary sets $\la \mathcal{S}^{\mu} \colon  \mu\leq \lambda \mbox{ is Mahlo} \ra$, which is lighface definable; finally, the pair $(p,\vec{D})$ is definable as the constructible-least counterexample for the Iteration-Fusion property. Let $2 \leq n<\omega$ be such that the above parameters, and the fact that $(p,\vec{D})$ is a counterexample for the Iteration-Fusion property, are $\Sigma_n$ definable over $J^{\E}_{\lambda^{+3}}$ without parameters. 

Apply Proposition \ref{Prop: Stationary set of passive collapsing structures on an arbitrary cofinality} to find an ordinal $\eta\in (\lambda, \lambda^+)$ such that $\eta\in \mathcal{S}^{\lambda, \lambda, n}$. In other words, $\cf(\eta) = \lambda$, $n_\eta \geq n$, and the collapsing structure $N_\eta$ can be weakly $r\Sigma_{n+1}$-embedded into $J^{\E}_{\lambda^{+3}}$ by an embedding $\pi$, such that all the above parameters belong to $\mbox{Im}(\pi)$. Let $\delta\in (\lambda, \lambda^{+})$ be such that $p^{N_\eta}_{n_\eta+1} = \{\delta \} $.

Let $\po'_\lambda, \bar{\po}'_{\lambda} , \vec{D}' = \la D'_i \colon i<\lambda \ra, p' \in N^{\E}_{\eta}$ be the preimages under $\pi$ of $\po_{\lambda}, \bar{\po}_\lambda, \vec{D}, p$, respectively. Note that $p' = p$ as $\mbox{crit}(\pi)> \lambda$. Also, $\po'_{\lambda}$ is a nonstationary support iteration of the form $\la \po'_{\alpha}, \name{\qo}'_{\alpha} \colon \alpha \leq \kappa' \ra$ in $N_{\eta}$. Since $\mbox{crit}(\pi)\geq \lambda$, we have, for every $\alpha< \lambda$, $\po'_\alpha = \po_\alpha$. It follows that $\po'_\lambda\subseteq \po_{\lambda}$.

Our goal will be to produce a condition $p^* \in \po_\lambda$ extending $p$ and a club $C\subseteq \lambda$, such that for every $i\in C$, the set $\{ s\in \po_{i+1} \colon s^{\frown} p^*\setminus i+1\in D_i \}$ is a dense subset of $\po_{i+1}$. This will contradict our initial assumptions.

In order to produce such $p^*$, We construct sequences $\vec{p} = \la p_i \colon i\leq \lambda \ra$, $\vec{C} = \la C_i \colon i<\lambda \ra$, $\vec{\alpha} = \la \alpha_i \colon i<\lambda \ra$, $\vec{\alpha}^* = \la \alpha^*_i \colon i<\lambda \ra$ and $\vec{\epsilon} = \la \epsilon_i \colon i<\lambda \ra $, each strict initial segment of those sequences is inside $N_{\eta}$, such that:
\begin{enumerate}
    \item The sequences $\la \alpha_i \colon i<\lambda \ra$, $\la \alpha^*_i \colon i<\lambda \ra$ are increasing sequences of ordinals, both cofinal in $\lambda$. Moreover,
        \begin{enumerate}
            \item The sequence $\la \alpha_i \colon i<\lambda \ra$ is also continuous. 
            \item The sequences interleave, in the sense that for every $i<\lambda$, $\alpha_i \leq \alpha^*_i < \alpha_{i+1}$.
            \item For every $\alpha\geq \alpha^*_i$ in $\supp(p_{i})$, $\alpha\in h^{N_{\eta}}_{n_\eta+1}[\omega\times (\alpha\cup \{\delta \})]$.\\In particular, By Lemma \ref{Lem: FineStructureCorrespondenceBetweenkappaalpha},\footnote{By picking $\alpha_0$ high enough, we will make sure that the transitive collapse of  $h^{N_\eta}_{n_\eta+1}[\omega\times \delta]$ belongs to $h^{N_\eta}_{n_\eta+1}[ \omega\times ( \alpha_0 \cup \{\delta \} ) ]$, which is required in order to apply Lemma \ref{Lem: FineStructureCorrespondenceBetweenkappaalpha}.} there exists $\eta(\alpha)\in (\alpha,\alpha^+)$ such that $h^{N_{\eta}}_{n_\eta+1}[\omega\times (\alpha\cup \{\delta \})]$ transitively collapses to a structure $N_{\eta(\alpha)}$. 
        \end{enumerate}
    \item The sequence $\la \epsilon_i \colon i<\lambda \ra$ is an increasing sequence of very good points in $N_{\eta}$, such that:
    \begin{enumerate}
        \item For every $i<\lambda$, and for every $\alpha\geq \alpha^*_{i}$ in $\supp(p_{i})$, $\epsilon_{i}$ defines a corresponding very good point $\epsilon_{i}(\alpha)$ of $N_{\eta(\alpha)}$, which is associated to a club point $\eta_{i}(\alpha)\in c^{\alpha}_{\eta(\alpha)}\setminus \mathcal{S}^\alpha$. 
        \item For every $i<\lambda$ and $\alpha<\lambda$ with $\alpha\in h^{N^{\E}_{\eta}}_{n_\eta+1}[\omega\times (\alpha\cup \{\delta \})]$, 
        $$\la \epsilon_j \colon j<i \ra, \la \eta_j(\alpha) \colon j\leq i \ra  \in h^{N_{\eta}, \epsilon_{i+1} }_{n_\eta+1}[\omega\times (\alpha\cup \{\delta\})].$$
    \end{enumerate}
    \item The sequence $\la p_i \colon i<\lambda \ra$ is a sequence of conditions in $\bar{\po}'_{\lambda}$, such that:
    \begin{enumerate}
        \item For every $i<j<\lambda$, $p_i \leq p_j$ and $p_j \uhr (\alpha_{i}+1) = p_{i}\uhr (\alpha_{i}+1)$.
        \item For every $i<\lambda$ and Mahlo $\alpha >\alpha_i$ in $\supp(p_i)$, 
        $$p_{i+1}\uhr \alpha \Vdash \ \  \forall \xi< \alpha^{++}, \ \  {p_{i+1}(\alpha)\uhr \xi }^{\frown} \name{T}^{p_{i+1}(\alpha)}_\xi \Vdash \eta_{i}(\alpha)\in \name{c}^{p_{i+1}(\alpha)}_\xi. $$
        \item For every $i<\lambda$,  $\la p_j \colon j\leq i \ra \in h^{N_{\eta}, \epsilon_{i}}_{n+1}[\omega \times (\alpha^*_{i}\cup \{\delta \})]$
    \end{enumerate}
    \item The sequence $\la C_i \colon i<\lambda \ra$ is decreasing, continuous sequence of club subsets of $\lambda$, such that:
    \begin{enumerate}
        \item For every limit $i<\lambda$, $C_i = \bigcap_{j<i} C_j$.
        \item For every $i<\lambda$, $C_{i+1}$ is the constructibly least club subset of $C_i$ disjoint from $\mbox{supp}(p_{i+1})$. 
        \item For every $i<j$, $\alpha_{j}\in C_i$.
    \end{enumerate}
\end{enumerate}

Begin by picking $p_0$ being the constructibly least extension of $p$ in $\bar{\po}'_{\lambda}$. Take $\alpha_0<\lambda$ be the least such that the transitive collapse of  $h^{N_\eta}_{n_\eta+1}[\omega\times \delta]$ belongs to $h^{N_\eta}_{n_\eta+1}[ \omega\times ( \alpha_0 \cup \{\delta \} ) ]$. Let $C_0$ be the constructibly least club disjoint from $\supp(p_0)$. By Corollary \ref{Corollary: FineStructureReflectDownGoodPts}, there exists $\alpha^*<\lambda$ and a good point $\epsilon$ such that:
\begin{itemize}
    \item  For every $\alpha\geq \alpha^*$ in $\supp(p_{0})$, $\alpha\in h^{N_\eta}_{n+1}[\omega\times (\alpha\cup \{ \delta\})]$. In particular, $h^{N_\eta,\epsilon}_{n+1}[\omega\times (\alpha\cup \{ \delta \})]$ is collapsed to a structure $N_{\eta(\alpha)}$ for some $\eta(\alpha)\in \mathcal{C}^\alpha$.
    \item For every $\alpha$ as above, the image of $\epsilon$ under the above collapse is a very good point $\epsilon(\alpha)$ of $N_{\eta(\alpha)}$, which induces a club point $\eta_0(\alpha)\in c^{\alpha}_{\eta(\alpha)}\setminus \mathcal{S}^{\alpha}$.
    \item $p_0 \in h^{N_\eta, \epsilon}[\omega\times (\alpha^*\cup \{\delta \})]$
    \item $\epsilon$ is a limit of more than $\omega_1$-many very good points of $N_\eta$.\footnote{This clause can be ensured since $\eta\in \mathcal{S}^{\lambda, \lambda, n}$, and in particular $\cf(\eta)> \omega_1$.}
\end{itemize} 
Let $\alpha^*_0\in C_0$ be the least such $\alpha^*$, and let $\epsilon_0$ be the least very good point $\epsilon$ as above with respect to $\alpha_0$. Denote by $\eta_0(\alpha)$ the club point induced by $\epsilon_0(\alpha)$. 

We now proceed to the successor case. Assume that $p_i, C_i, \alpha_i, \alpha^*_i, \epsilon_i$ were defined. Define first $\alpha_{i+1} = \min\left(C_i\setminus (\alpha^*_i +1)\right)$.
Let $\vec{\mu}^i = \la \mu^i_\alpha \colon \alpha\in \supp(p_i) \ra$ be the constructibly least local bounding function for $p_i$. Since $p_i \in h^{N_{\eta}, \epsilon_{i}}_{n_\eta+1}[\omega \times (\alpha^*_i\cup \{ \delta\})]$, we have $\vec{\mu}^i\in h^{N_{\eta},\epsilon_i}_{n_\eta+1}[\omega \times (\alpha^*_i\cup \{\delta \})]$, and in particular, for every $\alpha\in \supp(p_i)\setminus \alpha_i$, $\mu^i_\alpha\in h^{N_{\eta},\epsilon_i}_{n_\eta+1}[\omega\times (\alpha\cup \{\delta \})]\cap \alpha^+$, namely $\mu^i_\alpha < \eta_i(\alpha)$. Thus, we can define an extension $p'_{i}\geq p_i$ with the same support and the same Miller coordinates, by adding $\eta_i(\alpha)$ to each of the coding components of $p_i(\alpha)$, for every $\alpha> \alpha_{i+1}$ in $\supp(p_i)$. More formally, $\supp(p'_i) = \supp(p_i)$, $p'_i \uhr \alpha_{i+1}+1 = p_i \uhr \alpha_{i+1}+1$, and, for every $\alpha\in \supp(p_i)\setminus \alpha_{i+1}+1$, it is forced by $p'_i\uhr \alpha$ that for every $\xi<\alpha^{++}$,
$${p'_{i}(\alpha)\uhr \xi } \Vdash \name{T}^{p'_i(\alpha)}_\xi = \name{T}^{p_i(\alpha)}_\xi$$
and-- 
$${p'_{i}(\alpha)\uhr \xi }^{\frown}\name{T}^{p'_i(\alpha)}_\xi \Vdash \name{c}^{p'_i(\alpha)}_\xi = \name{c}^{p_i(\alpha)}_\xi\cup \{ \eta_i(\alpha) \}.$$
Note that $p'_i\in \bar{\po}'_{\lambda}$ since for every $\alpha\in \supp(p_i)\setminus \alpha_{i+1}+1$, $\eta_i(\alpha)\notin \mathcal{S}^\alpha$, and $\la \eta_{i}(\alpha) \colon \alpha\in \supp(p_i)\setminus \alpha_{i+1}+1 \ra\in N_{\eta}$. Furthermore,  $\la \eta_{i}(\alpha) \colon \alpha\in \supp(p_i)\setminus \alpha_{i+1}+1 \ra \in h^{N_\eta}_{n_\eta+1}[\omega\times (\alpha_{i+1}\cup \{ \delta\})]$ and $p'_i\in h^{N_{\eta}}_{n_\eta+1}[\omega\times (\alpha_{i+1}\cup \{\delta \} )]$, since $\alpha_{i+1}$ itself belongs to $\supp(p_i)\setminus \alpha^*_i$.  
. Finally, let $p_{i+1}\geq p'_i$ be the constructibly least extension of $p'_i$, such that $p_{i+1}\uhr \alpha_{i+1}+1 = p_i \uhr \alpha_{i+1}+1$, and--
$$ p_{i+1}\uhr \alpha_{i+1}+1 \Vdash \exists s\in G( \po_{\alpha_{i+1}+1} ), \ s^{\frown} p_{i+1}\setminus (\alpha_{i+1}+1)\in D'_i. $$
This concludes the construction of $p_{i+1}$. To finish the successor step, let $ C_{i+1} $ be the constructibly minimal club disjoint from $\supp(p_{i+1})$, and pick the least $\alpha^*_{i+1} \geq \alpha_{i+1}$ and a very good point $\epsilon_{i+1} > \epsilon_i$ such that:
\begin{itemize}
    \item  For every $\alpha\geq \alpha^*_{i+1}$ in $\supp(p_{i+1})$, $\alpha\in h^{N_\eta}_{n_\eta+1}[\omega\times (\alpha\cup \{ \delta\})]$. In particular, $h^{N_\eta,\epsilon}[\omega\times (\alpha\cup \{ \delta \})]$ is collapsed to a structure $N_{\eta(\alpha)}$ for some $\eta(\alpha)\in \mathcal{C}^\alpha$.
    \item The image of $\epsilon_{i+1}$ under the above collapse is a very good point $\epsilon_{i+1}(\alpha)$ of $N_{\eta(\alpha)}$,  which is a limit of (more than $\omega_1$) very good points.
    \item $\epsilon_{i+1}(\alpha)$ induces a club point $\eta_{i+1}(\alpha)\notin \mathcal{S}^{\alpha}$.
    \item $p_{i+1}, \la \epsilon_j \colon j\leq i \ra \in h^{N_\eta, \epsilon_{i+1}}_{n_\eta+1}[\omega\times (\alpha^*_{i+1}\cup \{\delta \})]$.\\
    Additionaly, since $\epsilon_{i+1}$ is a limit of very good points, we get, by coherence of the canonical square sequence, that each $\epsilon_{j}$ (for $j\leq i$) collapses to a very good point in the transitive collapse of $h^{N_\eta, \epsilon_{i+1}}_{n_\eta+1}[\omega\times (\alpha^*_{i+1}\cup \{\delta \})]$. In particular, for every $\alpha\in \supp(p_{i+1})\setminus  \alpha^*_{i+1}+1$, 
    $$\la \eta_j(\alpha) \colon j\leq i \ra \in h^{N_\eta, \epsilon_{i+1}}_{n_\eta+1}[\omega\times \left( \alpha\cup \{ \delta \} \right)].$$
\end{itemize} 

Finally, let us proceed to the limit case. Assume that $i < \lambda$ and 
$$\la p_j \colon j<i \ra, \la C_j \colon j<i \ra, \la \alpha_j, \alpha^*_j \colon j<i \ra \text{ and }
\la \epsilon_j \colon j<i \ra$$ have all been constructed, and let us construct $p_i, C_i, \alpha_i, \epsilon_i$. First take $\alpha_i = \sup\{\alpha_j \colon j<i \}$. We define $p_i$ as follows: 
\begin{itemize}
    \item $\supp(p_i) = \bigcup_{j<i}\supp(p_j)$. In particular, $\alpha_i\notin \supp(p_i)$ as $\alpha_i \in \bigcap_{j<i}C_j$.
    \item For every $j<i$, $p_i \uhr \alpha_j+1 = p_j \uhr \alpha_j+1$.
    \item For every $\alpha\in \supp(p_i)\setminus \alpha_i+1$, $$p_i \uhr \alpha \Vdash p_{i}(\alpha) = \bigvee_{j<i} p_j(\alpha)$$
where $\bigvee_{j<i}p_j(\alpha)$ is the least upper bound of $\la p_j(\alpha) \colon j< i \ra$ in the forcing $\name{\qo}_\alpha$, as defined in Definition \ref{Def: PacedFusionSequences}. 
\end{itemize}
Let us justify that the above requirements define a legitimate condition in $\po'_{\lambda}$. It's clear that $p_{i}\uhr \alpha_{i}+1\in \po_{\alpha_{i}+1}$, so let us concentrate on $p_{i}\setminus \alpha_i+1$.  
Assume $\alpha\in \supp(p_i)\setminus \alpha_i+1$. Then 
$$ p_i\uhr \alpha \Vdash \forall \xi<\alpha^{++}, p_i(\alpha)\uhr \xi \Vdash \name{T}^{p_i(\alpha)}_{\xi} = \bigcap_{ j<i, \alpha\in \supp(p_j) } \name{T}^{p_j(\alpha) }_\xi   $$ 
where $\name{T}^{p_i(\alpha)}_\xi $ is a legitimate condition since $\mo_\alpha$ is $\alpha$-closed, and $\alpha> i$. Also, 
\begin{align*}
   p_i\uhr \alpha \Vdash \forall \xi<\alpha^{++}, 
p_i(\alpha)\uhr \xi ^{\frown}\name{T}^{p_i(\alpha)}_{\xi}  \Vdash \name{c}^{p_i(\alpha)}_{\xi}= & \left(\bigcup_{ j<i, \alpha\in \supp(p_j) } \name{c}^{p_j(\alpha) }_\xi \right)\cup  \Big\{\sup\left(\bigcup_{j<i}  \name{c}^{p_j(\alpha) }_\xi \right) \Big\} \\
=& \left(\bigcup_{ j<i, \alpha\in \supp(p_j) } \name{c}^{p_j(\alpha) }_\xi \right)\cup  \Big\{\sup_{j<i} \eta_j(\alpha) \Big\}. 
\end{align*}
Let us justify that for every $\alpha\in \supp(p_i)\setminus \alpha_i+1$,  $\name{c}^{p_i(\alpha)}_\xi$ is a legitimate condition in the $\alpha$-th coding coordinate. First, note that $\la \eta_j(\alpha) \colon j<i \ra$ is bounded in $\eta(\alpha)$: if this wasn't the case, we would get that $\la \eta_j \colon j<i \ra$ is unbounded in $\eta$, where $\eta_j\in c^{\kappa}_{\eta}$ is the club point induced by the very good point $\epsilon_j$ of $N_\eta$. But this as a contradiction to the choice of $\eta$ in $\mathcal{S}^{\lambda,\lambda,n}$ (we have $\cf(\eta) = \lambda$). Now, since $\la \eta_j(\alpha) \colon j<i \ra$ is bounded in $\eta(\alpha)$, we have $\sup_{j<i} \eta_j(\alpha) \in c^{\alpha}_{\eta(\alpha)}$. By the definition of $\mathcal{S}^\alpha$, the only possible element of $\mathcal{S}^{\alpha}$ which is a limit point of $c^{\alpha}_{\eta(\alpha)}$, is the $\omega_1$-th limit point of $c^{\alpha}_{\eta(\alpha)}$ (see Corollary \ref{Cor: FineStructureS^alphaDoesNotReflect}). However, each  $\epsilon_{j}$ (for $j<i$) was chosen such that it is a limit point of more than $\omega_1$-many very good points. So $\sup_{j<i} \eta_j(\alpha)\notin \mathcal{S}^\alpha$.

We conclude the limit step, as in the previous steps, by picking $\alpha^*_{i}, \epsilon_{i}$ such that:
\begin{itemize}
\item  For every $\alpha\geq \alpha^*_{i}$ in $\supp(p_{i})$, $\alpha\in h^{N_\eta}_{n_\eta+1}[\omega\times (\alpha\cup \{ \delta\})]$. In particular, $h^{N_\eta,\epsilon}[\omega\times (\alpha\cup \{ \delta \})]$ is collapsed to a structure $N_{\eta(\alpha)}$ for some $\eta(\alpha)\in \mathcal{C}^\alpha$.
\item The image of $\epsilon_{i}$ under the above collapse is a very good point $\epsilon_{i}(\alpha)$ of $N_{\eta(\alpha)}$, which is a limit of more than $\omega_1$-many very good points.
\item $\epsilon_{i}(\alpha)$ induces a club point $\eta_{i}(\alpha)\notin \mathcal{S}^{\alpha}$.
\item $p_{i}, \la \epsilon_j \colon j < i \ra \in h^{N_\eta, \epsilon_{i}}_{n_\eta+1}[\omega\times (\alpha^*_{i}\cup \{\delta \})]$.
\end{itemize} 

This concludes the inductive construction. The sequence $\la p_i \colon i<\lambda \ra$ is external to $N_{\eta}$, but each of its elements is a condition in $\po'_\lambda\subseteq \po_\lambda$. Thus, Let $p^*\in \po_{\lambda}$ be a condition such that $\supp(p^*) = \bigcup_{i<\lambda}\supp(p_i)$, and for every $i<\lambda$, $p^*\uhr \alpha_{i}+1 = p_{i}\uhr \alpha_{i}+1$. Note that $\{ \alpha_i \colon i<\lambda \}$ form a club disjoint from $\supp(p^*)$. 

For every $i<\lambda$, $N_{\eta}$ sees the condition $p_{i+1}$ as the condition with the property that
$$ \{ s\in \po_{\alpha_i+1} \colon s^{\frown} p^*\setminus \alpha_i+1 \in D'_i \} $$
is a dense subset of $\po'_{\alpha_i+1}$. In particular, by applying $\pi \colon N_\eta\to L[\E]$, the same condition $p_{i+1} = \pi(p_{i+1})$, viewed as a condition in $\po_\lambda$, has the property that 
$$ \{ s\in \po_{\alpha_i+1} \colon s^{\frown} p_{i+1}\setminus \alpha_i+1 \in D_i \} $$
is a dense subset of $\po_{\alpha_i+1}$. 

Since $p^*$ extends each condition $p_i$ ($i<\lambda$), $p^*$ has the property that for every $i<\lambda$, 
$$ \{ s\in \po_{\alpha_i+1} \colon s^{\frown} p^*\setminus \alpha_i+1 \in D_i \} $$
is a dense subset of $\po_{\alpha_i+1}$. This, contradicts our initial assumption by contradiction.
\end{proof}

\begin{lemma}\label{Lem: FinalForcingDistributivityLemma}
    Assume that $\lambda\leq \kappa+1$ is a cardinal and $\nu<\lambda$. Then $0_{\po_{\nu}}\Vdash \po_\lambda \setminus \po_{\nu}$ is $\nu$-distributive.
\end{lemma}

\begin{proof}
    We prove by induction on $\lambda$ that for every $\nu \leq \lambda\leq \kappa+1$, $\po_{\lambda}\setminus \po_{\nu}$ is $\nu$-distributive.

    We begin by justifying the successor case. Assume that $\lambda = \lambda'+1$. If $\lambda' = \nu$, $\po_{\lambda}\setminus \nu$ is the forcing $\name{\qo}_{\nu}$ over $V^{\po_{\nu}}$, which is known to be $\nu$-distributive. If  $\lambda'> \nu$, $\po_{\lambda}\setminus {\nu}$ is $\nu$-distributive since it can be factored to the form $\po_{\lambda'} \setminus \nu * \name{\qo}_{\lambda'}$, where  $\po_{\lambda'}\setminus \nu$ is $\nu$-distributive (by induction) and $\name{\mathbb{Q}}_{\lambda'}$ is $\lambda'$-distributive.

    Next, we take care of the case where $\lambda$ is a limit ordinal which is not a limit of Mahlo cardinals. In this case, $\po_{\lambda} = \po_{\lambda'}$ where $\lambda'<\lambda$ is the supremum of the Mahlo cardinals below $\lambda$, and the induction hypothesis can be applied to show that $\po_{\lambda'}\setminus \nu$ is $\nu$-distributive.

    We proceed to the main case and assume that $\lambda$ is a limit of Mahlo cardinals. Assume by contradiction that $0_{\po_{\nu}}$ does not force that $\po_\lambda \setminus \nu$ is $\nu$-distributive. So there are $s\in \po_{\nu}$, a  cardinal $\gamma<\nu$ and $\vec{\name{D}} = \la \name{D}_i \colon i<\gamma \ra$ which form a counterexample for distributivity, in the sense that $s$ forces that $\la \name{D}_i \colon i<\gamma \ra$ is a sequence of dense open subsets of $\po_{\lambda}\setminus \nu$ with an empty intersection. Assume that the triple $\la s, \gamma, \vec{\name{D}} \ra$ is the constructibly minimal such counterexample in $L[\E]$. Let $2\leq n<\omega$ be such thay $\lambda, \po_\lambda, \bar{\po}_\lambda, s, \vec{D}$ are $\Sigma_n$-definable over $J_{\lambda^{+3}}$ without parameters, and the statement that $(s,\gamma,\vec{D})$ is a counterexample for distributivity is $\Sigma_n$-expressible. Let $\delta\in (\lambda, \lambda^+)$ be such that $p^{N_\eta}_{n+1} = \{ \delta \}$.

    Let $\zeta$ be a regular cardinal in the interval $(\gamma, \lambda)$. In the case where $\lambda$ is singular, choose $\zeta$ such that $\cf(\lambda) < \zeta$. Apply Proposition \ref{Prop: Stationary set of passive collapsing structures on an arbitrary cofinality} to find $\eta\in \mathcal{S}^{\lambda, \zeta, n}$ such that its collapsing structure $N^{\E}_{\eta}$ can be $r\Sigma_n$-embedded into $J_{\lambda^{+3}}$, where $\lambda, \po_{\lambda}, \po'_\lambda, s, \vec{D}$ belong to the range of the embedding. Let $\po'_\lambda, \bar{\po}'_\lambda, s, \vec{\name{D}}'$ be the preimages of $\po_\lambda, \bar{\po}_\lambda, s, \vec{\name{D}}$ under the above embedding.

    Our goal is to construct in $N_{\tau}$ a condition of $\po'_\lambda\setminus \nu$ which belongs to $\bigcap_{i<\gamma} \name{D}'_i$, using a similar argument as in Lemma \ref{Lem: FinalForcingFusionLemma}. This was previously done by constructing an increasing sequence of conditions $\la p_i \colon i\leq \gamma \ra$, where each $p_{i+1}$ is obtained from $p_i$ by adding points to the coding components of $p_i(\alpha)$, for every $\alpha$ above a carefully chosen $\alpha^*_i$. However, for the current lemma, it will be important for us that the sequence $\la \alpha^*_i \colon i< \gamma \ra$ is bounded in $\lambda$, and this cannot be easily guaranteed for singular $\lambda$.\footnote{We remark that, for the ``Kunen-like" blueprint, the only case that matters to us is where $\lambda$ is Mahlo. However, the inductive nature of the distributivity proof forces us to deal also with the cases where $\lambda$ is singular.}  To overcome this obstacle, we replace the sequence $\la \alpha^*_i \colon i<\gamma \ra$ with a single ordinal, by applying Lemma \ref{Lemma: FineStructureInfastructureSystemForDistributivityProof} to find a sequence of very good points $\la \epsilon_i \colon i\leq\gamma \ra$ in $N_{\eta}$ and a single  $\alpha^*\in (\nu, \lambda)$, such that--
    \begin{itemize}
        \item For every nowhere stationary $Z\in h^{N^{\E}_\eta}_{n_\eta+1}[\omega\times (\alpha^*\cup \{\delta \})]$ of inaccessible cardinals and  $\alpha\in Z\setminus \alpha^*$, $\alpha\in h^{N^{\E}_\eta}_{n_\eta+1}[\omega\times (\alpha\cup \{\delta \})]$. In particular, $h^{N_\eta}_{n_\eta+1}[\omega\times (\alpha\cup \{\delta \})]$ transitively collapses to $N_{\eta(\alpha)}$ for some $\eta(\alpha)\in \mathcal{C}^\alpha$.\\
        By increasing $\alpha^*$ if necessary, we can assume that $\alpha^* \in h^{N_\eta}_{n_\eta+1}[\omega\times (\alpha^*\cup \{\delta \})]$.
        \item For every $\alpha < \lambda$ satisfying $\alpha\in h^{N_{\eta}}_{n+1}[\omega\times (\alpha\cup \{\delta \})]$, there are sequence $\la  \epsilon_i(\alpha) \colon i\leq \gamma\ra$ and $\la \eta_i(\alpha) \colon i\leq \gamma \ra$, such that the former sequence consists of very good points of $N_{\eta(\alpha)}$, and the latter sequence consists of the club points induced by them. Moreover:
        \begin{itemize}
            \item Each $\epsilon_i$ is a limit of more than $\omega_1$-many very good points.
            \item $\eta_i(\alpha)\notin \mathcal{S}^\alpha$ for every $i\leq \gamma$.
            \item For every $i<\gamma$, $\la \epsilon_j \colon j< i \ra\in h^{N_\eta, \epsilon_{i}}_{n+1}[\omega\times ( \alpha^*\cup \{\delta \} )]$.
        \end{itemize}
    \end{itemize}
     
\noindent
Construct an increasing sequence $\la p_i \colon i<\gamma \ra\subseteq \bar{\po}'_\lambda$, such that:
\begin{enumerate}
    \item For every $i<\gamma$, $\la p_j\colon j\leq i \ra \in h^{ N^{\E}_\eta ,\epsilon_{i}}_{n+1}[\omega \times (\alpha^*\cup \{\delta\})]$.
    \item For every $i<\gamma$, $p_i \uhr \alpha^* = p_0\uhr \alpha^*$.
    \item For every $i<\gamma$, there exists a dense open set $D^*_{i}\subseteq \po_{\left[\alpha_0, \alpha^*\right)}$ above $p_0\uhr \left[ \alpha_0, \alpha^*\right)$, such that, for every $q\in D^*_{i}$, $q^{\frown} p_{i+1}\setminus \alpha^* \in \name{D}'_{i}$.
    \item For every $i<\gamma $,  $p_{i+1}\uhr \alpha$ forces that for every $\xi<\alpha^{++}$, $${p_{i+1}(\alpha)\uhr \xi}^{\frown} \name{T}^{ p_{i+1}(\alpha) }_\xi \Vdash  {\eta}_i(\alpha)  \in \name{c}^{p_{i+1}(\alpha)}_\xi. $$
    \item For every limit $i\leq \gamma$, $p_i = \bigvee_{j<i}p_j$, namely $p_i$ is the condition such that $\supp(p_i) = \bigcup_{j<i}\supp(p_j)$, and, for every $\alpha\in \supp(p_i)$, 
    $$p_i \uhr \alpha \Vdash p_{i}(\alpha) = \bigvee_{j<i} p_j(\alpha)$$
    where $\bigvee_{j<i}p_j(\alpha)$ is the least upper bound of $\la p_j(\alpha) \colon j< i \ra$ in the forcing $\name{\qo}_\alpha$, as defined in Definition \ref{Def: PacedFusionSequences}.
\end{enumerate}
For simplicity, let us denote for every $i\leq \gamma$, $X_i = h^{ N^{\E}_\eta ,\epsilon_{i}}_{n_\eta+1}[\omega \times (\alpha^*\cup \{\delta\})]$.\\
We begin the construction by picking the constructibly least condition $p_0 \in \po'_{\lambda}$ such that $p_0\uhr \nu = s$. Assuming that $p_i\in \bar{\po}'_{\lambda}\setminus \alpha^*$ has been constructed and belongs to $X_{i}$, let us construct $p_{i+1}\in X_{i+1}$. First note that $\supp(p_i)\in X_{i}$ is a nowhere stationary set of inaccessible cardinals, and thus, for every $\alpha\in \supp(p_i)\setminus \alpha^*$, $\alpha\in h^{N_\eta, \epsilon_{i}}_{n_\eta+1}[\omega\times (\alpha\cup \{\delta \})]$. This, together with the fact that $\la \epsilon_i \colon i\leq \gamma \ra\in X_{i+1}$, shows that the sequence $\la \eta_i(\alpha) \colon \alpha\in \supp(p_i)\setminus \alpha^* \ra$ could be entirely defined inside $X_{i+1}$.\footnote{This follows from the coherence of the square sequence.} 

Let $\vec{\mu}^{p_{i}} = \la \mu^{p_{i}}_\alpha \colon \alpha\in \supp(p_{i}) \ra$ be the constructibly-least local bounding function of $p_i$. Then $\vec{\mu}^{p_i}\in X_{i}$. This, combined with the fact that for every  $\alpha \in \supp(p_i)\setminus \alpha^*$, $\alpha\in h^{N_{\eta}, \epsilon_{i}}_{n+1}[\omega\times (\alpha\cup \{\delta \})]$, implies that for every such $\alpha$, $\mu^{p_i}_{\alpha}\in h^{N_{\eta}, \epsilon_{i}}_{n+1}[\omega\times (\alpha\cup \{\delta \})] \cap \alpha^+ $, namely $\mu^{p_i}_{\alpha} < \eta_{i}(\alpha)$. Thus, we can extend $p_i$ to a condition $p'_i$ by adding the points ${\eta}_i(\alpha)$ to the coding coordinates at each $\alpha \geq \alpha^*$. More formally, we define $p'_{i}\geq p_i$ with the same support, such that $p'_i\uhr \alpha^* = p \uhr \alpha^*$, and, for every $\alpha \geq \alpha^*$, $p'_{i}\uhr \alpha$ forces that for every $\xi<\alpha^{++}$,
$$ p'_i(\alpha)\uhr \xi \Vdash \name{T}^{p'_i(\alpha)}_\xi = \name{T}^{p_i(\alpha)}_\xi$$
and
$$ {p'_i(\alpha)\uhr \xi}^{\frown} T^{p'_i(\alpha)}_\xi\Vdash \name{c}^{p'_i(\alpha)}_\xi = \name{c}^{p_i(\alpha)}_\xi \cup \{ \eta_i(\alpha) \}.$$
The name $\name{c}^{p'_i(\alpha)}_{\xi}$ defined above is forced to be a legitimate condition of $\co^\alpha_\xi$ since $\eta_i(\alpha)\notin \mathcal{S}^\alpha$. As mentioned above, $\la \eta_{i}(\alpha) \colon \alpha\in \supp(p_i) \ra$ belongs to $X_{i+1}$, and thus $p'_i\in X_{i+1}$.

\noindent
Let $p_{i+1} \geq p'_{i}$ be an extension such that $p_{i+1}\uhr \alpha^* = p_0\uhr \alpha^*$, and 
$$p_0\uhr \alpha^* \Vdash p_{i+1}\setminus \alpha^* \in \name{D}'_i \cap \left(\bar{\po}'_{\lambda}\setminus \alpha^*\right).$$ 
Such $p_{i+1}$ exists by Lemma \ref{Lem: FinalForcingDensityLemma} and since $\name{D}'_i\subseteq \po'_\lambda\setminus \alpha^*$ is dense open. We can also pick $p_{i+1}$ inside $X_{i+1}$, being constructibly minimal with the above properties. 

Finally, let $D^*_i\subseteq \po_{\left[\nu,\alpha^*\right)}$ be the set of conditions $r$ extending $p_0\uhr \left[ \nu, \alpha^* \right)$, such that $p_0\uhr \nu\Vdash r^{\frown} p_{i+1}\in \name{D}'_i$. It's not hard to see that $D^*_i$ is a dense open subset of $\po_{\left[\nu, \alpha^*\right)}$. This concludes the successor step.\\

Assume now that $i\leq\gamma$ is limit, and $\la p_{j} \colon j<i \ra$ have been defined. Let $p_i = \bigvee_{j<i}p_j$, namely $p_i$ is the condition such that $\supp(p_i) = \bigcup_{j<i}\supp(p_j)$, and, for every $\alpha\in \supp(p_i)$, 
$$p_i \uhr \alpha \Vdash p_{i}(\alpha) = \bigvee_{j<i} p_j(\alpha)$$
where $\bigvee_{j<i}p_j(\alpha)$ is the least upper bound of $\la p_j(\alpha) \colon j< i \ra$ in the forcing $\name{\qo}_\alpha$, as defined in Definition \ref{Def: PacedFusionSequences}.
In other words, 
$$ p_i\uhr \alpha \Vdash \forall \xi<\alpha^{++}, 
 p_i(\alpha)\uhr \xi \Vdash \name{T}^{p_i(\alpha)}_{\xi} = \bigcap_{ j<i, \alpha\in \supp(p_j) } \name{T}^{p_j(\alpha) }_\xi   $$
 where $\name{T}^{p_i(\alpha)}_\xi $ is a legitimate condition since $\mo_\alpha$ is $\alpha$-closed, and--
$$ p_i\uhr \alpha \Vdash \forall \xi<\alpha^{++}, 
p_i(\alpha)\uhr \xi ^{\frown}\name{T}^{p_i(\alpha)}_{\xi}  \Vdash \name{c}^{p_i(\alpha)}_{\xi}=  \left(\bigcup_{ j<i, \alpha\in \supp(p_j) } \name{c}^{p_j(\alpha) }_\xi \right)\cup  \Big\{\bigcup_{j<i} \eta_{j}(\alpha) \Big\}  $$
where $\name{c}^{p_i(\alpha)}_\xi$ is a legitimate condition since $\sup_{j<i}\eta_j(\alpha)=\eta_{i}(\alpha)\notin \mathcal{S}^\alpha$. 

Finally, we need to verify $\la p_j \colon j\leq i \ra \in X_{i}$. It suffices to prove that $\la p_j \colon j<i \ra\in X_{i}$, since the above construction of $p_i$ from $\la p_j \colon j<i \ra$ can be done inside $X_{i}$. For the proof that $\la p_j \colon j<i \ra\in X_{i}$, we rely on the fact that $\la \epsilon_{j}\colon j<i \ra\in X_{i}$, and, in each step in the construction, $p_{j+1}$ is chosen in a constructible way from $p_j$, $\eta_j(\alpha)$ and $D_{j}$, which are all available inside $X_{i}$.

This concludes the inductive construction. Consider the condition $p_{\gamma}\in \po'_\lambda$ obtained in the final limit step. $p_{\gamma}\uhr \nu = s$ forces that $\la \name{D}^*_i \colon i\leq \gamma \ra$ are dense open subsets of $\po_{\left[ \nu ,\alpha^*\right)}$ above $p_{\gamma}\uhr \left[ \nu, \alpha^* \right)$. By applying the induction hypothesis for the forcing $\po_{\left[ \nu, \alpha^* \right)}$, it is forced by $p_{\gamma}\uhr \nu$ that there exists $r\in \bigcap_{i<\gamma} D^*_{i}$ which extends $p_{\gamma}\uhr \left[ \nu, \alpha^* \right)$. Now, extend $p_{\gamma}$ in the interval $\left[ \nu, \alpha^* \right)$ to such $r$, and let $p^*$ be the resulting condition. Then $p^*\uhr \nu = s$, and forces that $p^*\setminus \nu\in \bigcap_{i<\gamma} \name{D}'_i$, contradicting our initial assumption.\end{proof}

\begin{lemma} \label{Lem: FinalForcingPreservationOfStatSets}
    For every $\lambda\leq \kappa$ Mahlo and regular $\delta<\lambda$, the forcing $\po_\lambda$ preserves the stationarity of stationary subsets of $\mathcal{S}^{\lambda} \subseteq \lambda^+\cap \cf(\delta)$.
\end{lemma}

\begin{proof}
Assume otherwise. Let $(p, S, \name{C})$ be the constructibly least counterexample, in the sense that $S\subseteq \mathcal{S}^{\lambda}$ is a stationary set in $L[\E]$, $p\in \po_{\lambda}$ and $\name{C}$ is $\po_{\lambda}$-name for a club subset of $\lambda$, which is forced by $p$ to be disjoint from $S$.

Pick $\eta\in S$. Let $\pi \colon N_\eta\to J^{\E}_{\lambda^{+3}}$ with $\mbox{crit}(\pi)> \lambda$ be $r\Sigma_{n+1}$-preserving, where $n = n_\eta \geq 2$. Then $\lambda, \po_{\lambda}, \bar{P}_\lambda, p, \name{C}$ belong to $\mbox{Im}(\pi)$. Let $\po'_\lambda,  \bar{\po}'_{\lambda}, p', \name{C}'$ be the images of the relevant parameters under the transitive collapse. Note that $\name{C}'$ is unbounded in $\eta$ since $\eta= (\alpha^+)^{N_\eta}$. 

We construct a sequences $\la p_i \colon i<\omega_1 \ra$ of conditions of $\bar{\po}'_\lambda$,  $\la \alpha^*_i \colon i<\omega_1 \ra$ of ordinals below $\lambda$, and $\la \epsilon_i \colon i<\omega_1 \ra$ of very good points of $N_\eta$, such that:
\begin{enumerate}
    \item  For every $i<j<\omega_1$, $p_i \leq p_j$ and $p_j \uhr (\alpha^*_{i}+1) = p_{i}\uhr (\alpha^*_{i}+1)$.
    \item For every $i<\omega_1$ and $\alpha\geq \alpha^*_i$ in $\supp(p_{i})$, $\alpha\in h^{N_{\eta}}_{n_\eta+1}[\omega\times (\alpha\cup \{\delta \})]$.
    \item For every $i<\omega_1$, $\la p_j \colon j\leq i \ra \in h^{N_\eta, \epsilon_i}_{n+1}[\omega\times \left( \alpha^*_i\cup \{\delta \} \right)]$.
    \item For every $i<\lambda$ and Mahlo $\alpha >\alpha_i$ in $\supp(p_i)$, 
        $$p_{i+1}\uhr \alpha \Vdash \ \  \forall \xi< \alpha^{++}, \ \  {p_{i+1}(\alpha)\uhr \xi }^{\frown} \name{T}^{p_{i+1}(\alpha)}_\xi \Vdash \eta_{i}(\alpha)\in \name{c}^{p_{i+1}(\alpha)}_\xi $$ 
    where $\eta_i(\alpha)\notin \mathcal{S}^\alpha$ is the club-point associated with a very good point $\epsilon_i(\alpha)$, which is the image of $\epsilon_i$ under the transitive collapse of $h^{N_\eta}_{n_\eta+1}[ \omega\times (\alpha\cup \{  \delta\}) ]$ to $N_{\eta(\alpha)}$.
    \item For every $i<\omega_1$, $p_{i+1}\uhr \alpha^*_{i+1}$ forces that $p_{i+1}\setminus \alpha^*_{i+1}$ decides that $\eta_i$-th element of $\name{C}'$, where $\eta_i$ is the club point that corresponds to the very good point $\epsilon_i$.
\end{enumerate}
Let $p_0 = p'$. Let $\alpha_0, \epsilon_0$ be chosen using Corollary \ref{Corollary: FineStructureReflectDownGoodPts} such that $p_0 \in h^{N_{\eta}, \epsilon_0}_{n_\eta+1}[\omega\times \left( \alpha_0 \cup \{ \delta\} \right)]$. 

Assuming that $p_i, \alpha_i, \epsilon_i$ were chosen, construct $p'_i$ with the same support, such that $p'_i\uhr \alpha^*_i = p \uhr \alpha^*_i$, and, for every $\alpha \geq \alpha^*_i$, $p'_{i}\uhr \alpha$ forces that for every $\xi<\alpha^{++}$,
$$ p'_i(\alpha)\uhr \xi \Vdash \name{T}^{p'_i(\alpha)}_\xi = \name{T}^{p_i(\alpha)}_\xi$$
and
$$ {p'_i(\alpha)\uhr \xi}^{\frown} T^{p'_i(\alpha)}_\xi\Vdash \name{c}^{p'_i(\alpha)}_\xi = \name{c}^{p_i(\alpha)}_\xi \cup \{ \eta_i(\alpha) \}.$$
The name $\name{c}^{p'_i(\alpha)}_{\xi}$ defined above is forced to be a legitimate condition of $\co^\alpha_\xi$ since $\eta_i(\alpha)\notin \mathcal{S}^\alpha$. $\eta_{i}(\alpha)$ could be added above $\name{c}^{p_i(\alpha)}_{\xi}$ (for each $\xi<\alpha^{++}$) since the least local bounding function $\vec{\mu}^{p_i}$ for $p_i$ belongs to $h^{N_\eta, \epsilon_{i}}_{n+1}[\omega\times (\alpha^*_i \cup \{ \delta\})]$; in particular, for every $\alpha \in \supp(p_i)\setminus \alpha^*_i$, $\vec{\mu}^{p_i}(\alpha)\in \alpha^+ \cap h^{N_\eta, \epsilon_i}_{n+1}[\omega\times (\alpha\cup \{\delta \})] = \eta_i(\alpha)$. Finally, let $p_{i+1}$ be the constructibly least extension of $p'_i$ in $N_\eta$, that agrees with $p'_i$ up to $\alpha^*_{i+1}$, and such that $p_{i+1}\setminus \alpha_{i+1}$ decides the $\eta_i$-th element of $\name{C}'$. Once $p_{i+1}$ has been constructed, use Corollary \ref{Corollary: FineStructureReflectDownGoodPts} to find $\alpha_{i+1}> \alpha_i$ and $\epsilon_{i+1}> \epsilon_i$ such that $p_{i+1}\in h^{N_\eta, \epsilon_{i+1}}_{n_\eta+1}[\omega\times (\alpha_{i+1} \cup \{\delta \} )]$. This concludes the successor step in the construction.

In limit steps, use the fact that $\po_\lambda$ is $\omega_1$-closed to take an upper bound. More specifically, assume that $i\leq\omega_1$ is limit, and $\la p_{j} \colon j<i \ra$ have been defined. Let $p_i = \bigvee_{j<i}p_j$, namely $p_i$ is the condition such that $\supp(p_i) = \bigcup_{j<i}\supp(p_j)$, and, for every $\alpha\in \supp(p_i)$, 
$$p_i \uhr \alpha \Vdash p_{i}(\alpha) = \bigvee_{j<i} p_j(\alpha)$$
where $\bigvee_{j<i}p_j(\alpha)$ is the least upper bound of $\la p_j(\alpha) \colon j< i \ra$ in the forcing $\name{\qo}_\alpha$, as defined in Definition \ref{Def: PacedFusionSequences}. By Lemma \ref{Lem: MillerCodingSigmaClosureLemma}, $\bigvee_{j<i}p_j(\alpha)$ exists. Once $p_i$ has been constructed, choose $\alpha_i \geq \sup_{j<i} \alpha_j$ and $\epsilon_i > \sup_{j<i} \epsilon_j$ such that $\epsilon_j$ is a very good point fulfilling the conditions of Corollary \ref{Corollary: FineStructureReflectDownGoodPts}, chosen high enough such that $p_{i}\in h^{N_\eta, \epsilon_{i}}_{n_\eta+1}[\omega\times (\alpha_{i} \cup \{\delta \} )]$.

This concludes the construction. Let $\alpha^* = \sup_{i<\omega_1}\alpha_i$. Let $p^* = \bigcup_{i<\omega_1}p_i$ ($p^*$ is external to $N_\eta$). We first argue that $p^*\in \po_{\lambda}$. For every $\alpha<\alpha^*$, $p^*\uhr \alpha = p_i\uhr \alpha_i$, where $i<\omega_1$ is the least such that $\alpha_i> \alpha$. Thus, it suffices to check that $p^*\setminus \alpha^*$ is a legitimate condition. Note that, since $\eta\in \mathcal{S}^\lambda$, the very good points $\la \epsilon_i \colon i<\omega_1 \ra$ are cofinal in $N^{(n_\eta)}_\eta\cap \mbox{Ord}$. In other words, $\eta(\alpha) = \sup_{i<\omega_1}\eta_i(\alpha)$, for every $\alpha\in \supp(p^*)\setminus \alpha^* $. It follows that, for every $\alpha\geq \alpha^*$, $p^*\uhr \alpha$ forces that $p^*_\alpha$ is a condition, such that for every $\xi<\alpha^{++}$ in its support,
$${p^*(\alpha)\uhr \xi}^{\frown} \name{T}^{p^*(\alpha)}_\xi \Vdash \name{c}^{p^*(\alpha)}_\xi = \left(\bigcup_{ i<\omega_1, \alpha\in \supp(p_i) } \name{c}^{p_i(\alpha) }_\xi \right)\cup  \Big\{ \eta(\alpha) \Big\} $$
which is a legitimate condition since $\eta(\alpha)\notin \mathcal{S}^\alpha$ by Lemma \ref{Lem: FineStructureCorrespondenceBetweenkappaalpha}.

This shows that $p^*\in \po_\lambda$. By our construction, for every $i<\omega_1$, $p^*\uhr \alpha_{i+1} = p_{i+1}\uhr \alpha_{i+1}$ forces that $p_{i+1}\setminus \alpha_{i+1}$ decides the $\eta_i$-th element in $\name{C}$. Let $\gamma_i$ be the least ordinal below $\alpha^+$ such that $p_{i+1} \Vdash \name{C}(\check{\eta}_i)< \check{\gamma_i}$. This is a a $\Pi_2$-statement that defines $\gamma_i$, with parameters in the image of $\pi$. Since $\pi$ is weakly $r\Sigma_{n+1}$-preserving for some $n\geq 2$, and the parameters $\alpha, p_{i+1}, \check{\eta_i}$ belong to the image of $\pi$ (since $\mbox{crit}(\pi) = \eta$), we have that $\gamma_i\in \mbox{Im}(\pi)$. But $\mbox{Im}(\pi)\cap \alpha^+ = \eta$, so $\gamma_i<\eta$. 

It follows that $p^*$ forces that the $\eta_i$-th element of $\name{C}$ is less than $\eta$, for every $i<\omega_1$. Therefore, $p^*$ forces that $\eta\in \name{C}$, a contradiction.\end{proof}

We are ready to prove the main theorem of this work.
\begin{proof}(Theorem \ref{Thm:Main})\\
Let $L[\E]$ be a minimal extender model with a $(\kappa,\kappa^{++})$-extender $E$ on a measurable cardinal $\kappa$. In particular, the sequence $\E$ is a set and $E$ is its last extender.
Let $\po = \po_{\kappa+1}$ be the nonstationary support iteration described at the beginning of this section (\ref{subsection:final}). 
We need to verify that it satisfies the ``Kunen-like" blueprint, from Definition \ref{Def:KLblueprint}, which extends the Friedman-Magidor blueprint from Definition \ref{Def:FM-blueprint}.
It is clear from the definition of $\po$ that it satisfies assumption (FM1). Theorem \ref{Thm:MillerCodingFinalForcing} shows it satisfies assumption (FM2). Lemmas \ref{Lem: FinalForcingFusionLemma} and \ref{Lem: FinalForcingDistributivityLemma} secure assumption $(FM3)$. 
Regarding (FM4), Theorem \ref{Thm:MillerCodingFinalForcing} establishes the existence of a uniform definition $\varphi_\qo$ for the witnesses for the self coding of $\name{\qo}_\alpha$. Choosing
$\vec{f}^\alpha = \la f^\alpha_\tau \mid \tau < \alpha^{++}\ra \subseteq {}^\alpha \alpha$ to be the $<_{Sing}^\alpha$-increasing sequence to be $\qo_\alpha$-induced sequence of Miller generic branches, provides a definition $\varphi_F$, whose uniformity follows from the uniformity of $\varphi_{\qo}$. A uniform definition $\varphi_{\po}$ for $\po$ is an immediate consequene of the definition of $\po_{\alpha+1}$ as a nonstationary support iteration of the posets $\qo_\alpha$, given by $\varphi_{\qo}$.\\
It remains to verify the ``Kunen-Like" additional assumptions.
Theorem \ref{Thm:MillerCodingFinalForcing} 
proves that  $\vec{f}^\alpha = \la f^\alpha_\tau \mid \tau < \alpha^{++}\ra \subseteq {}^\alpha \alpha$ is $<^\alpha_{Sing}$-increasing, which validates assumption (KL1).
Assumption (KL2) for $\po$ is an immediate consequence of 
our choice of the stationary sets 
$$\vec{S}^\alpha = \la \mathcal{S}^\alpha(\tau,\eta,i): \tau < \alpha^{++}, \eta < \alpha^+,i<2\ra \in L[\E]$$ from Definition \ref{Lem: FinalForcingCinstructionOfStationarySets}.
We conclude that $\po$ satisfies the ``Kunen-Like" blueprint. 
Let $G \subseteq \po$ be a generic filter over $V = L[\E]$. By Theorem \ref{Thm:KLblueprint} $V[G] = L[\E][G]$ is a ``Kunen-Like" model where the single measurable cardinal $\kappa$ satisfies $2^\kappa = \kappa^{++}$.  Since $\E,G$ are sets, letting $U$ denote the single normal ultrafilter on $\kappa$ in $V[G]$, we can express $V[G]$ as $L[A,U]$ for a set of ordinals $A$.
\end{proof}
\newpage

\bibliographystyle{plain}
\bibliography{Bibliography}

\begin{thebibliography}{10}

\bibitem{MR2299507}
Arthur~W. Apter, James Cummings, and Joel~David Hamkins.
\newblock Large cardinals with few measures.
\newblock {\em Proc. Amer. Math. Soc.}, 135(7):2291--2300, 2007.

\bibitem{MR820124}
Stewart Baldwin.
\newblock The {$\triangleleft\,$}-ordering on normal ultrafilters.
\newblock {\em J. Symbolic Logic}, 50(4):936--952, 1985.

\bibitem{MR3397347}
Omer Ben-Neria.
\newblock The structure of the {M}itchell order---{II}.
\newblock {\em Ann. Pure Appl. Logic}, 166(12):1407--1432, 2015.

\bibitem{MR3544708}
Omer Ben-Neria.
\newblock The structure of the {M}itchell order---{I}.
\newblock {\em Israel J. Math.}, 214(2):945--982, 2016.

\bibitem{MR1257466}
James Cummings.
\newblock Possible behaviours for the {M}itchell ordering.
\newblock {\em Ann. Pure Appl. Logic}, 65(2):107--123, 1993.

\bibitem{MR1312304}
James Cummings.
\newblock Possible behaviours for the {M}itchell ordering. {II}.
\newblock {\em J. Symbolic Logic}, 59(4):1196--1209, 1994.

\bibitem{10.1305/ndjfl/1125409326}
James Cummings.
\newblock {Notes on Singular Cardinal Combinatorics}.
\newblock {\em Notre Dame Journal of Formal Logic}, 46(3):251 -- 282, 2005.

\bibitem{CumForMag}
James Cummings, Matthew Foreman, and Menachem Magidor.
\newblock Squares, scales and stationary reflection.
\newblock {\em Journal of Mathematical Logic}, 1(01):35--98, 2001.

\bibitem{ForMagZem}
Matthew Foreman, Menachem Magidor, and Martin Zeman.
\newblock Games with filters.
\newblock {\em arXiv preprint arXiv:2009.04074}, 2020.

\bibitem{FriMag09}
Sy-David Friedman and Menachem Magidor.
\newblock The number of normal measures.
\newblock {\em The Journal of Symbolic Logic}, 74(3):1069--1080, 2009.

\bibitem{FriedmanThompson}
Sy-David Friedman and Katherine Thompson.
\newblock Perfect trees and elementary embeddings.
\newblock {\em The Journal of Symbolic Logic}, 73(3):906--918, 2008.

\bibitem{FriedmanZdomskyy2010}
Sy-David Friedman and Lyubomyr Zdomskyy.
\newblock Measurable cardinals and the cofinality of the symmetric group.
\newblock {\em Fundamenta Mathematicae}, 207(2):101--122, 2010.

\bibitem{GoldbergUABook}
Gabriel Goldberg.
\newblock {\em The ultrapower axiom}, volume~10 of {\em De Gruyter Series in
  Logic and its Applications}.
\newblock De Gruyter, Berlin, [2022] \copyright 2022.

\bibitem{MR520190}
A.~Kanamori and M.~Magidor.
\newblock The evolution of large cardinal axioms in set theory.
\newblock In {\em Higher set theory ({P}roc. {C}onf., {M}ath.
  {F}orschungsinst., {O}berwolfach, 1977)}, volume 669 of {\em Lecture Notes in
  Math.}, pages 99--275. Springer, Berlin, 1978.

\bibitem{Kanamori}
Akihiro Kanamori.
\newblock Perfect-set forcing for uncountable cardinals.
\newblock {\em Annals of Mathematical Logic}, 19(1):97--114, 1980.

\bibitem{KanamoriBook}
Akihiro Kanamori.
\newblock {\em The higher infinite: large cardinals in set theory from their
  beginnings}.
\newblock Springer Science \& Business Media, 2008.

\bibitem{MR277381}
K.~Kunen and J.~B. Paris.
\newblock Boolean extensions and measurable cardinals.
\newblock {\em Ann. Math. Logic}, 2(4):359--377, 1970/71.

\bibitem{KunenModel}
Kenneth Kunen.
\newblock Some applications of iterated ultrapowers in set theory.
\newblock {\em Ann. Math. Logic}, 1:179--227, 1970.

\bibitem{MR3274970}
Jeffrey~Scott Leaning and Omer Ben-Neria.
\newblock Disassociated indiscernibles.
\newblock {\em MLQ Math. Log. Q.}, 60(6):389--402, 2014.

\bibitem{Mitchell1974innermodelsforcoherentsequences}
William~J Mitchell.
\newblock Sets constructible from sequences of ultrafilters.
\newblock {\em The Journal of Symbolic Logic}, 39(1):57--66, 1974.

\bibitem{MR716621}
William~J. Mitchell.
\newblock Sets constructed from sequences of measures: revisited.
\newblock {\em J. Symbolic Logic}, 48(3):600--609, 1983.

\bibitem{Mitchell-HB}
William~J Mitchell.
\newblock Beginning inner model theory.
\newblock In {\em Handbook of set theory}, pages 1449--1495. Springer, 2009.

\bibitem{schimmerlingZemanSquare}
Ernest Schimmerling and Martin Zeman.
\newblock Square in core models.
\newblock {\em Bulletin of Symbolic Logic}, 7(3):305--314, 2001.

\bibitem{schindler2006iteratesofthecoremodel}
Ralf Schindler.
\newblock Iterates of the core model.
\newblock {\em The Journal of Symbolic Logic}, 71(1):241--251, 2006.

\bibitem{SchindlerZeman2009FineStruture}
Ralf Schindler and Martin Zeman.
\newblock Fine structure.
\newblock In {\em Handbook of set theory}, pages 605--656. Springer, 2009.

\bibitem{Silver-KunenModel}
Jack Silver.
\newblock The consistency of the gch with the existence of a measurable
  cardinal.
\newblock {\em Axiomatic Set Theory, Part 1}, 1:391, 1971.

\bibitem{solovay1974SCHabovestronglycompact}
Robert~M Solovay et~al.
\newblock Strongly compact cardinals and the gch.
\newblock In {\em Proceedings of the Tarski symposium}, volume~25, pages
  365--372. Proceedings of Symposia in Pure Mathematics, 1974.

\bibitem{Steel-HB}
John~R Steel.
\newblock An outline of inner model theory.
\newblock In {\em Handbook of set theory}, pages 1595--1684. Springer, 2009.

\bibitem{MR1286010}
J{\v i}r\'i Witzany.
\newblock Any behaviour of the {M}itchell ordering of normal measures is
  possible.
\newblock {\em Proc. Amer. Math. Soc.}, 124(1):291--297, 1996.

\bibitem{Zeman-Book}
Martin Zeman.
\newblock {\em Inner models and large cardinals}, volume~5.
\newblock Walter de Gruyter, 2011.

\end{thebibliography}

\end{document}